\documentclass{aims}
\usepackage{amsmath}
  \usepackage{paralist}
  \usepackage{longtable}
  \usepackage{graphics} %% add this and next lines if pictures should be in esp format
  \usepackage{epsfig} %For pictures: screened artwork should be set up with an 85 or 100 line screen
\usepackage{graphicx}  \usepackage{epstopdf}%This is to transfer .eps figure to .pdf figure; please compile your paper using PDFLeTex or PDFTeXify.
 \usepackage[colorlinks=true]{hyperref}
\hypersetup{urlcolor=blue, citecolor=red}

\newtheorem{theorem}{Theorem}[section]

\newtheorem{lemma}[theorem]{Lemma}

\theoremstyle{definition}
\newtheorem{definition}[theorem]{Definition}
\newtheorem{remark}{Remark}

\title[	The algebraic classification of nilpotent commutative  algebras] %Use the shortened version of the full title
      {	The algebraic classification of nilpotent commutative  algebras}

\author[Doston Jumaniyozov,
Ivan   Kaygorodov and  Abror Khudoyberdiyev]{}


 \email{jumaniyozovdoston50@gmail.com}
 \email{kaygorodov.ivan@gmail.com}
 \email{abror.khudoyberdiyev@mathinst.uz}

\begin{document}

\maketitle

% Enter the first author's name and address:
\centerline{\scshape Doston Jumaniyozov}
\medskip
{\footnotesize
% please put the address of the first author

} % Do not forget to end the {\footnotesize by the sign }

\medskip

\centerline{\scshape 
Ivan   Kaygorodov }
\medskip
{\footnotesize
 % please put the address of the second  and third author

   }
   
   \medskip
   
\centerline{\scshape 
  Abror Khudoyberdiyev}
\medskip
{\footnotesize
 % please put the address of the second  and third author

}

\bigskip

%The abstract of your paper
\begin{abstract}
This paper is devoted to the complete algebraic classification of complex $5$-dimensional nilpotent commutative algebras.
Our method of classification is based on the standard method of classification of central extensions of smaller nilpotent commutative algebras and the    recently obtained classification of complex $5$-dimensional nilpotent commutative $\mathfrak{CD}$-algebras. 
\end{abstract}

\section*{Introduction}

The algebraic classification (up to isomorphism) of algebras of dimension $n$ from a certain variety
defined by a certain family of polynomial identities is a classic problem in the theory of non-associative algebras.
There are many results related to the algebraic classification of small-dimensional algebras in many varieties of
non-associative algebras \cite{      ck13,  degr3, usefi1,   degr2,   ha16,  jkk19, kkl20, kkl21, kv16}.
So, algebraic classifications of 
$2$-dimensional algebras \cite{kv16,petersson},
$3$-dimensional evolution algebras \cite{ccsmv},
$3$-dimensional anticommutative algebras \cite{japan},
$4$-dimensional   division algebras \cite{erik,Ernst},
$4$-dimensional nilpotent algebras \cite{kkl21}
and 
$6$-dimensional anticommutative nilpotent algebras \cite{kkl20} have been given.
In the present paper, we give the algebraic classification of
$5$-dimensional nilpotent commutative  algebras.
The variety of commutative algebras is defined by the following identity:
$xy = yx.$
It contains commutative $\mathfrak{CD}$-algebras, Jordan algebras, mock-Lie algebras and commutative associative algebras as subvarieties.
On the other hand, it is a principal part in the varieties of weakly associative algebras and flexible algebras.

The algebraic study of central extensions of associative  and non-associative algebras has been an important topic for years 
(see, for example, \cite{hac16,ss78} and references therein).
Our method for classifying nilpotent commutative algebras is based on the calculation of central extensions of nilpotent algebras of smaller dimensions from the same variety (first, this method has been developed by Skjelbred and Sund for Lie algebra case in   \cite{ss78}) and 
the classifications of all 
complex $5$-dimensional nilpotent commutative (non-Jordan) $\frak{CD}$-algebras \cite{jkk19};
nilpotent Jordan (non-associative) algebras \cite{ha16};
and nilpotent associative commutative algebras \cite{maz80}.
 
	\newpage
\section{The algebraic classification of nilpotent commutative algebras}
\subsection{Method of classification of nilpotent algebras}

Throughout this paper, we use the notations and methods well written in \cite{hac16},
which we have adapted for the commutative case with some modifications.
Further in this section we give some important definitions.

Let $({\bf A}, \cdot)$ be a complex  commutative algebra
and $\mathbb V$ be a complex  vector space. The $\mathbb C$-linear space ${\rm Z^{2}}\left(
\bf A,\mathbb V \right) $ is defined as the set of all bilinear maps $\theta  \colon {\bf A} \times {\bf A} \longrightarrow {\mathbb V}$ such that $ \theta(x,y)=\theta(y,x) .$
These elements will be called {\it cocycles}. For a
linear map $f$ from $\bf A$ to  $\mathbb V$, if we define $\delta f\colon {\bf A} \times
{\bf A} \longrightarrow {\mathbb V}$ by $\delta f  (x,y ) =f(xy )$, then $\delta f\in {\rm Z^{2}}\left( {\bf A},{\mathbb V} \right) $. We define ${\rm B^{2}}\left({\bf A},{\mathbb V}\right) =\left\{ \theta =\delta f\ : f\in {\rm Hom}\left( {\bf A},{\mathbb V}\right) \right\} $.
%One can easily check that ${\rm B^{2}}(\bf A,\mathbb V)$ is a linear subspace of ${\rm Z^{2}}\left( {\bf A},{\mathbb V}\right) $; %its elements are called {\it coboundaries}.
We define the {\it second cohomology space} ${\rm H^{2}}\left( {\bf A},{\mathbb V}\right) $ as the quotient space ${\rm Z^{2}}
\left( {\bf A},{\mathbb V}\right) \big/{\rm B^{2}}\left( {\bf A},{\mathbb V}\right) $.
% The equivalence class of $%
%\theta \in {\rm Z^{2}}\left( {\bf A},{\mathbb V}\right) $ will be denoted by $\left[
%\theta \right] \in {\rm H^{2}}\left( {\bf A},{\mathbb V}\right) $.

Let $\operatorname{Aut}({\bf A}) $ be the automorphism group of  ${\bf A} $ and let $\phi \in \operatorname{Aut}({\bf A})$. For $\theta \in
{\rm Z^{2}}\left( {\bf A},{\mathbb V}\right) $ define  the action of the group $\operatorname{Aut}({\bf A}) $ on ${\rm Z^{2}}\left( {\bf A},{\mathbb V}\right) $ by 
\begin{center}$\phi \theta (x,y)
=\theta \left( \phi \left( x\right) ,\phi \left( y\right) \right) $. 
\end{center} 
It is easy to verify that
${\rm B^{2}}\left( {\bf A},{\mathbb V}\right) $ is invariant under the action of $\operatorname{Aut}({\bf A}).$
So, we have an induced action of  $\operatorname{Aut}({\bf A})$  on ${\rm H^{2}}\left( {\bf A},{\mathbb V}\right)$.

Let $\bf A$ be a commutative algebra of dimension $m$ over  $\mathbb C$ and ${\mathbb V}$ be a $\mathbb C$-vector
space of dimension $k$. For the bilinear map $\theta$, define on the linear space ${\bf A}_{\theta } = {\bf A}\oplus {\mathbb V}$ the
bilinear product `` $\left[ -,-\right] _{{\bf A}_{\theta }}$'' by $\left[ x+x^{\prime },y+y^{\prime }\right] _{{\bf A}_{\theta }}=
xy +\theta(x,y) $ for all $x,y\in {\bf A},x^{\prime },y^{\prime }\in {\mathbb V}$.
The algebra ${\bf A}_{\theta }$ is called a $k$-{\it dimensional central extension} of ${\bf A}$ by ${\mathbb V}$. One can easily check that ${\bf A_{\theta}}$ is a commutative algebra if and only if $\theta \in {\rm Z^2}({\bf A}, {\mathbb V})$.

Call the
set $\operatorname{Ann}(\theta)=\left\{ x\in {\bf A}:\theta \left( x, {\bf A} \right)=0\right\} $
the {\it annihilator} of $\theta $. We recall that the {\it annihilator} of an  algebra ${\bf A}$ is defined as
the ideal $\operatorname{Ann}(  {\bf A} ) =\left\{ x\in {\bf A}:  x{\bf A}=0\right\}$. Observe
that
$\operatorname{Ann}\left( {\bf A}_{\theta }\right) =(\operatorname{Ann}(\theta) \cap\operatorname{Ann}({\bf A}))
\oplus {\mathbb V}$.

\

The following result shows that every algebra with a non-zero annihilator is a central extension of a smaller-dimensional algebra.

\begin{lemma}
Let ${\bf A}$ be an $n$-dimensional commutative algebra such that  
\begin{center}$\dim (\operatorname{Ann}({\bf A}))=m\neq0$. 
\end{center}
Then there exists, up to isomorphism, a unique $(n-m)$-dimensional commutative algebra ${\bf A}'$ and a bilinear map $\theta \in {\rm Z^2}({\bf A'}, {\mathbb V})$ with $\operatorname{Ann}({\bf A'})\cap\operatorname{Ann}(\theta)=0$, where $\mathbb V$ is a vector space of dimension m, such that ${\bf A} \cong {{\bf A}'}_{\theta}$ and
${\bf A}/\operatorname{Ann}({\bf A})\cong {\bf A}'$.
\end{lemma}

\begin{proof}
	Let ${\bf A}'$ be a linear complement of $\operatorname{Ann}({\bf A})$ in ${\bf A}$. Define a linear map $P \colon {\bf A} \longrightarrow {\bf A}'$ by $P(x+v)=x$ for $x\in {\bf A}'$ and $v\in\operatorname{Ann}({\bf A})$, and define a multiplication on ${\bf A}'$ by $[x, y]_{{\bf A}'}=P(x y)$ for $x, y \in {\bf A}'$.
	For $x, y \in {\bf A}$, we have
	\[P(xy)=P((x-P(x)+P(x))(y- P(y)+P(y)))=P(P(x) P(y))=[P(x), P(y)]_{{\bf A}'}. \]
	
	Since $P$ is a homomorphism, $P({\bf A})={\bf A}'$  and ${\bf A}'$ is a commutative algebra and also
	${\bf A}/\operatorname{Ann}({\bf A})\cong {\bf A}'$, which gives us the uniqueness. Now, define the map $\theta \colon {\bf A}' \times {\bf A}' \longrightarrow\operatorname{Ann}({\bf A})$ by $\theta(x,y)=xy- [x,y]_{{\bf A}'}$.
	Thus, ${\bf A}'_{\theta}$ is ${\bf A}$ and therefore $\theta \in {\rm Z^2}({\bf A'}, {\mathbb V})$ and $\operatorname{Ann}({\bf A'})\cap\operatorname{Ann}(\theta)=0$.
\end{proof}

\

\begin{definition}
Let ${\bf A}$ be an algebra and $I$ be a subspace of $\operatorname{Ann}({\bf A})$. If ${\bf A}={\bf A}_0 \oplus I$
then $I$ is called an {\it annihilator component} of ${\bf A}$.
A central extension of an algebra $\bf A$ without annihilator component is called a {\it non-split central extension}.
\end{definition}

Our task is to find all central extensions of an algebra $\bf A$ by
a space ${\mathbb V}$.  In order to solve the isomorphism problem we need to study the
action of $\operatorname{Aut}({\bf A})$ on ${\rm H^{2}}\left( {\bf A},{\mathbb V}
\right) $. To do that, let us fix a basis $e_{1},\ldots ,e_{s}$ of ${\mathbb V}$, and $
\theta \in {\rm Z^{2}}\left( {\bf A},{\mathbb V}\right) $. Then $\theta $ can be uniquely
written as $\theta \left( x,y\right) =
\displaystyle \sum_{i=1}^{s} \theta _{i}\left( x,y\right) e_{i}$, where $\theta _{i}\in
{\rm Z^{2}}\left( {\bf A},\mathbb C\right) $. Moreover, $\operatorname{Ann}(\theta)=\operatorname{Ann}(\theta _{1})\cap\operatorname{Ann}(\theta _{2})\cap\ldots \cap\operatorname{Ann}(\theta _{s})$. Furthermore, $\theta \in
{\rm B^{2}}\left( {\bf A},{\mathbb V}\right) $\ if and only if all $\theta _{i}\in {\rm B^{2}}\left( {\bf A},
\mathbb C\right) $.
It is not difficult to prove (see \cite[Lemma 13]{hac16}) that given a commutative algebra ${\bf A}_{\theta}$, if we write as
above $\theta \left( x,y\right) = \displaystyle \sum_{i=1}^{s} \theta_{i}\left( x,y\right) e_{i}\in {\rm Z^{2}}\left( {\bf A},{\mathbb V}\right) $ and
$\operatorname{Ann}(\theta)\cap \operatorname{Ann}\left( {\bf A}\right) =0$, then ${\bf A}_{\theta }$ has an
annihilator component if and only if $\left[ \theta _{1}\right] ,\left[
\theta _{2}\right] ,\ldots ,\left[ \theta _{s}\right] $ are linearly
dependent in ${\rm H^{2}}\left( {\bf A},\mathbb C\right) $.

Let ${\mathbb V}$ be a finite-dimensional vector space over $\mathbb C$. The {\it Grassmannian} $G_{k}\left( {\mathbb V}\right) $ is the set of all $k$-dimensional
linear subspaces of $ {\mathbb V}$. Let $G_{s}\left( {\rm H^{2}}\left( {\bf A},\mathbb C\right) \right) $ be the Grassmannian of subspaces of dimension $s$ in
${\rm H^{2}}\left( {\bf A},\mathbb C\right) $. There is a natural action of $\operatorname{Aut}({\bf A})$ on $G_{s}\left( {\rm H^{2}}\left( {\bf A},\mathbb C\right) \right) $.
Let $\phi \in \operatorname{Aut}({\bf A})$. For $W=\left\langle
\left[ \theta _{1}\right] ,\left[ \theta _{2}\right] ,\dots,\left[ \theta _{s}
\right] \right\rangle \in G_{s}\left( {\rm H^{2}}\left( {\bf A},\mathbb C
\right) \right) $ define $\phi W=\left\langle \left[ \phi \theta _{1}\right]
,\left[ \phi \theta _{2}\right] ,\dots,\left[ \phi \theta _{s}\right]
\right\rangle $. We denote the orbit of $W\in G_{s}\left(
{\rm H^{2}}\left( {\bf A},\mathbb C\right) \right) $ under the action of $\operatorname{Aut}({\bf A})$ by $\operatorname{Orb}(W)$. Given
\[
W_{1}=\left\langle \left[ \theta _{1}\right] ,\left[ \theta _{2}\right] ,\dots,
\left[ \theta _{s}\right] \right\rangle ,W_{2}=\left\langle \left[ \vartheta
_{1}\right] ,\left[ \vartheta _{2}\right] ,\dots,\left[ \vartheta _{s}\right]
\right\rangle \in G_{s}\left( {\rm H^{2}}\left( {\bf A},\mathbb C\right)
\right),
\]
we easily have that if $W_{1}=W_{2}$, then $ \bigcap\limits_{i=1}^{s}\operatorname{Ann}(\theta _{i})\cap \operatorname{Ann}\left( {\bf A}\right) = \bigcap\limits_{i=1}^{s}
\operatorname{Ann}(\vartheta _{i})\cap\operatorname{Ann}( {\bf A}) $, and therefore we can introduce
the set
\[
{\bf T}_{s}({\bf A}) =\left\{ W=\left\langle \left[ \theta _{1}\right] ,
\dots,\left[ \theta _{s}\right] \right\rangle \in
G_{s}\left( {\rm H^{2}}\left( {\bf A},\mathbb C\right) \right) : \bigcap\limits_{i=1}^{s}\operatorname{Ann}(\theta _{i})\cap\operatorname{Ann}({\bf A}) =0\right\},
\]
which is stable under the action of $\operatorname{Aut}({\bf A})$.

\

Now, let ${\mathbb V}$ be an $s$-dimensional linear space and let us denote by
${\bf E}\left( {\bf A},{\mathbb V}\right) $ the set of all {\it non-split $s$-dimensional central extensions} of ${\bf A}$ by
${\mathbb V}$. By above, we can write
\[
{\bf E}\left( {\bf A},{\mathbb V}\right) =\left\{ {\bf A}_{\theta }:\theta \left( x,y\right) = \sum_{i=1}^{s}\theta _{i}\left( x,y\right) e_{i} \ \ \text{and} \ \ \left\langle \left[ \theta _{1}\right] ,\left[ \theta _{2}\right] ,\dots,
\left[ \theta _{s}\right] \right\rangle \in {\bf T}_{s}({\bf A}) \right\} .
\]
We also have the following result, which can be proved as in \cite[Lemma 17]{hac16}.

\begin{lemma}

Let ${\bf A}_{\theta },{\bf A}_{\vartheta }\in {\bf E}\left( {\bf A},{\mathbb V}\right) $. Suppose that $\theta \left( x,y\right) =  \displaystyle \sum_{i=1}^{s}
\theta _{i}\left( x,y\right) e_{i}$ and $\vartheta \left( x,y\right) =
\displaystyle \sum_{i=1}^{s} \vartheta _{i}\left( x,y\right) e_{i}$.
Then the commutative algebras ${\bf A}_{\theta }$ and ${\bf A}_{\vartheta } $ are isomorphic
if and only if
$$\operatorname{Orb}\left\langle \left[ \theta _{1}\right] ,
\left[ \theta _{2}\right] ,\dots,\left[ \theta _{s}\right] \right\rangle =
\operatorname{Orb}\left\langle \left[ \vartheta _{1}\right] ,\left[ \vartheta
_{2}\right] ,\dots,\left[ \vartheta _{s}\right] \right\rangle.$$
\end{lemma}

This shows that there exists a one-to-one correspondence between the set of $\operatorname{Aut}({\bf A})$-orbits on ${\bf T}_{s}\left( {\bf A}\right) $ and the set of
isomorphism classes of ${\bf E}\left( {\bf A},{\mathbb V}\right) $. Consequently we have a
procedure that allows us, given a commutative algebra ${\bf A}'$ of
dimension $n-s$, to construct all non-split central extensions of ${\bf A}'$. This procedure is:

\begin{enumerate}
	\item For a given commutative algebra ${\bf A}'$ of dimension $n-s $, determine ${\rm H^{2}}( {\bf A}',\mathbb {C}) $, $\operatorname{Ann}({\bf A}')$ and $\operatorname{Aut}({\bf A}')$.
	
	\item Determine the set of $\operatorname{Aut}({\bf A}')$-orbits on ${\bf T}_{s}({\bf A}') $.
	
	\item For each orbit, construct the commutative algebra associated with a
	representative of it.
\end{enumerate}

\subsection{Reduction to non-$\frak{CCD}$-algebras}
The idea of the definition of a $\mathfrak{CD}$-algebra comes from the following property of Jordan and Lie algebras: {\it the commutator of any pair of multiplication operators is a derivation}.
 This gives three identities of degree four,  which reduce to only one identity of degree four in the commutative or anticommutative case.
 Namely, a commutative algebra is a commutative $\frak{CD}$-algebra ($\frak{CCD}$-algebra) if it satisfies
 the following identity:
\[((xy)a)b  + ((xb)a)y + x((yb)a)  = ((xy)b)a +  ((xa)b)y+ x((ya)b).\]

The above described method gives all commutative ($\mathfrak{CCD}$- and non-$\mathfrak{CCD}$-)  algebras. But we are interested in developing this method in such a way that it only gives non-$ \mathfrak{CCD} $ commutative algebras, because the classification of all $ \mathfrak{CCD} $-algebras is done in \cite{jkk19}. Clearly, any central extension of a commutative non-$ \mathfrak{CCD} $-algebra is a non-$ \mathfrak{CCD} $-algebra. But a $ \mathfrak{CCD} $-algebra may have extensions which are not $ \mathfrak{CCD} $-algebras. More precisely, let $\mathfrak{D}$ be a $ \mathfrak{CCD} $-algebra and $\theta \in {\rm Z_\mathfrak{C}^2}(\mathfrak{D}, {\mathbb C}).$ Then ${\mathfrak{D}}_{\theta }$ is a $ \mathfrak{CCD} $-algebra if and only if
\[\theta(x,y)=\theta(y,x),\]
\[\theta((xy)a,b)+\theta((xb)a,y)+\theta(x,(yb)a)= \theta((xy)b,a)+\theta((xa)b,y)+\theta(x,(ya)b).\]
for all $x,y,a,b\in {\mathfrak{D}}.$ Define the subspace ${\rm Z_\mathfrak{D}^2}({\mathfrak{D}},{\mathbb C})$ of ${\rm Z_\mathfrak{C}^2}({\bf  \mathfrak{D}},{\mathbb C})$ by
\begin{equation*}
{\rm Z_\mathfrak{D}^2}({\mathfrak{D}},{\mathbb C}) =\left\{\begin{array}{c} \theta \in {\rm Z_\mathfrak{C}^2}({\mathfrak{D}},{\mathbb C}) : \theta(x,y)=\theta(y,x),\\ \theta((xy)a,b)+\theta((xb)a,y)+\theta(x,(yb)a)=\\ 
\multicolumn{1}{r}{\theta((xy)b,a)+\theta((xa)b,y)+\theta(x,(ya)b)}\\
\text{ for all } x, y,a,b\in {\mathfrak{D}}\end{array}\right\}.
\end{equation*}

Observe that ${\rm B^2}({ \mathfrak{D}},{\mathbb C})\subseteq{\rm Z_\mathfrak{D}^2}({\mathfrak{D}},{\mathbb C}).$
Let ${\rm H_\mathfrak{D}^2}({\mathfrak{D}},{\mathbb C}) =%
{\rm Z_\mathfrak{D}^2}({\mathfrak{D}},{\mathbb C}) \big/{\rm B^2}({\mathfrak{D}},{\mathbb C}).$ Then ${\rm H_\mathfrak{D}^2}({\mathfrak{D}},{\mathbb C})$ is a subspace of $%
{\rm H_\mathfrak{C}^2}({\mathfrak{D}},{\mathbb C}).$ Define
\[{\bf R}_{s}({\mathfrak{D}}) =\left\{ {\bf W}\in {\bf T}_{s}({\mathfrak{D}}) :{\bf W}\in G_{s}({\rm H_\mathfrak{D}^2}({\mathfrak{D}},{\mathbb C}) ) \right\}, \]
\[	{\bf U}_{s}({\mathfrak{D}}) =\left\{ {\bf W}\in {\bf T}_{s}({\mathfrak{D}}) :{\bf W}\notin G_{s}({\rm H_\mathfrak{D}^2}({\mathfrak{D}},{\mathbb C}) ) \right\}.\]

Then ${\bf T}_{s}({\mathfrak{D}}) ={\bf R}_{s}(
{\mathfrak{D}})$ $\mathbin{\mathaccent\cdot\cup}$ ${\bf U}_{s}(
{\mathfrak{D}}).$ The sets ${\bf R}_{s}({\mathfrak{D}}) $
and ${\bf U}_{s}({\mathfrak{D}})$ are stable under the action
of $\operatorname{Aut}({\mathfrak{D}}).$ Thus, the commutative algebras
corresponding to the representatives of $\operatorname{Aut}({\mathfrak{D}}) $%
-orbits on ${\bf R}_{s}({\mathfrak{D}})$ are $ \mathfrak{CCD} $-algebras,
while those corresponding to the representatives of $\operatorname{Aut}({\mathfrak{D}}%
) $-orbits on ${\bf U}_{s}({\mathfrak{D}})$ are not
$ \mathfrak{CCD} $-algebras. Hence, we may construct all non-split commutative non-$ \mathfrak{CCD}$-algebras $%
\bf{A}$ of dimension $n$ with $s$-dimensional annihilator
from a given commutative algebra $\bf{A}%
^{\prime }$ of dimension $n-s$ in the following way:

\begin{enumerate}
	\item If $\bf{A}^{\prime }$ is non-$ \mathfrak{CCD} $, then apply the procedure.
	
	\item Otherwise, do the following:
	
	\begin{enumerate}
		\item Determine ${\bf U}_{s}(\bf{A}^{\prime })$ and $%
		\operatorname{Aut}(\bf{A}^{\prime }).$
		
		\item Determine the set of $\operatorname{Aut}(\bf{A}^{\prime })$-orbits on ${\bf U%
		}_{s}(\bf{A}^{\prime }).$
		
		\item For each orbit, construct the commutative algebra corresponding to one of its
		representatives.
	\end{enumerate}
\end{enumerate}

\subsection{Notations}
Let us introduce the following notations. Let ${\bf A}$ be a nilpotent algebra with
a basis $e_{1},e_{2}, \ldots, e_{n}.$ Then by $\Delta_{ij}$\ we will denote the
bilinear form
$\Delta_{ij}:{\bf A}\times {\bf A}\longrightarrow \mathbb C$
with 
$\Delta_{ij}(e_{l},e_{m})  = \delta_{il}\delta_{jm},$
if $i\leq j$ and $l\leq m.$
The set $\left\{ \Delta_{ij}:1\leq i\leq j\leq n\right\}$ is a basis for the linear space of
bilinear forms on ${\bf A},$ so every $\theta \in
{\rm Z^2}({\bf A},\bf \mathbb V )$ can be uniquely written as $
\theta = \displaystyle \sum_{1\leq i\leq j\leq n} c_{ij}\Delta _{{i}{j}}$, where $
c_{ij}\in \mathbb C$.
Let us fix 
complex number  
$\eta_k$ \, ($\eta_k^k=-1,$  $\eta_k^l\neq1$ for $0<l<k$).
For denote our algebras, we will use the following notations:
$$\begin{array}{lll}
\mathbf{N}^{\Xi}_{j}& \mbox{---}& j\mbox{th }5\mbox{-dimensional   family of}\\
&& \mbox{commutative non-$ \mathfrak{CCD} $-algebras with parametrs $\Xi$.} \\
\mathbf{N}^{i}_{j}& \mbox{---}& j\mbox{th }i\mbox{-dimensional   non-$ \mathfrak{CCD} $-algebra.} \\
\mathbf{N}^{i*}_{j}& \mbox{---}& j\mbox{th }i\mbox{-dimensional  $\mathfrak{CCD}$-algebra.}
\end{array}$$

\begin{remark}
All families of algebras from our final list do not have intersections,
but inside some families of algebras there are isomorphic algebras.
All isomorphisms between algebras from a certain family of algebras constucted from the representative $\nabla(\Sigma)$ are given in the list of distinct orbit representations.
The notation $\langle \nabla(\Xi) \rangle^{O(\Xi_1)=O(\Xi_2)}$ represents
that  
the elements $\langle \nabla(\Xi_1) \rangle$ and $\langle \nabla(\Xi_2) \rangle$ have the same orbit.

\end{remark}

\subsection{$ 3 $- and $ 4 $-dimensional commutative algebras.} 
Thanks to \cite{fkkv} we have the complete classification of complex $4$-dimensional nilpotent commutative algebras. 
It will be re-written by some different way for separating $\mathfrak{CCD}$- and non-$\mathfrak{CCD}$-algebras.

\begin{longtable}{ll lllll|l}
$\mathbf{N}_{01}^{3*}, \mathbf{N}_{01}^{4*}$&$:$& $e_{1}e_{1}=e_{2}$& & &&& $\mathrm{H}^2_{\mathfrak{C}} =\mathrm{H}^2_{\mathfrak{D}}$  \\
\hline
$\mathbf{N}_{02}^{3*},\mathbf{N}_{02}^{4*}$&$:$& $e_{1}e_{1}=e_{2}$ & $e_{1}e_{2}=e_{3}$ &&&&
$\mathrm{H}^2_{\mathfrak{C}}  \neq \mathrm{H}^2_{\mathfrak{D}}$ \\
\hline
$\mathbf{N}_{03}^{3*},\mathbf{N}_{03}^{4*}$&$:$& $e_{1}e_{2}=e_{3}$ & &&&&
$\mathrm{H}^2_{\mathfrak{C}}=\mathrm{H}^2_{\mathfrak{D}}$\\
\hline
$\mathbf{N}_{04}^{3*}, \mathbf{N}_{04}^{4*}$&$:$& $e_{1}e_{1}=e_{2}$ & $e_{2}e_{2}=e_{3}$& &&&
$\mathrm{H}^2_{\mathfrak{C}}  \neq \mathrm{H}^2_{\mathfrak{D}}$ \\
\hline
 
$\mathbf{N}_{05}^{4*}$&$:$& $e_1e_1=e_2$ & $e_1e_3=e_4$&&&&
$\mathrm{H}^2_{\mathfrak{C}}=\mathrm{H}^2_{\mathfrak{D}}$\\
\hline
$\mathbf{N}_{06}^{4*}$&$:$& $e_1e_1=e_2$ & $e_3e_3=e_4$&&&&
$\mathrm{H}^2_{\mathfrak{C}}=\mathrm{H}^2_{\mathfrak{D}}$\\
\hline
$\mathbf{N}_{07}^{4*}$&$:$& $e_1e_1=e_4$ & $e_2e_3=e_4$&&&&
$\mathrm{H}^2_{\mathfrak{C}}=\mathrm{H}^2_{\mathfrak{D}}$\\
\hline
$\mathbf{N}_{08}^{4*}$&$:$& $e_1e_1=e_2$ & $e_1e_2=e_3$ & $e_2e_2=e_4$&&&
$\mathrm{H}^2_{\mathfrak{C}}\neq \mathrm{H}^2_{\mathfrak{D}}$\\
\hline
${\mathbf N}^{4*}_{09}$ &$:$&  $e_1 e_1 = e_2$ &  $e_2 e_3=e_4$&&&&
$\mathrm{H}^2_{\mathfrak{C}}\neq \mathrm{H}^2_{\mathfrak{D}}$ \\
\hline
${\mathbf N}^{4*}_{10}$ &$:$&  $e_1 e_1=e_2$ & $e_1 e_2=e_4$ & $e_3e_3=e_4$ &&&
$\mathrm{H}^2_{\mathfrak{C}}\neq \mathrm{H}^2_{\mathfrak{D}}$ \\
\hline
${\mathbf N}^{4*}_{11}$ &$:$& $e_1 e_1=e_2$  & $e_1 e_3=e_4$  & $e_2 e_2=e_4$&&&
$\mathrm{H}^2_{\mathfrak{C}}\neq \mathrm{H}^2_{\mathfrak{D}}$ \\
\hline
${\mathbf N}^{4*}_{12}$ &$:$&  $e_1 e_1 = e_2$  & $e_2 e_2=e_4$ & $e_3 e_3= e_4$&&&
$\mathrm{H}^2_{\mathfrak{C}}\neq \mathrm{H}^2_{\mathfrak{D}}$\\
\hline
${\mathbf N}^{4*}_{13}(\lambda)$ &$:$&  $e_1 e_1 = e_2$ & $e_1e_2=e_3$ & $e_1e_3=e_4$ & $e_2 e_2=\lambda e_4$&&
$\mathrm{H}^2_{\mathfrak{C}}\neq \mathrm{H}^2_{\mathfrak{D}}$ \\
\hline
${\mathbf N}^{4*}_{14}$ &$:$&  $e_1 e_2 = e_3$ &  $e_1 e_3=e_4$ &&&&
$\mathrm{H}^2_{\mathfrak{C}}\neq \mathrm{H}^2_{\mathfrak{D}}$ \\
\hline
${\mathbf N}^{4*}_{15}$ &$:$&  $e_1 e_2=e_3$ & $e_1 e_3=e_4$ & $e_2e_2=e_4$  &&&
$\mathrm{H}^2_{\mathfrak{C}}\neq \mathrm{H}^2_{\mathfrak{D}}$\\
\hline
${\mathbf N}^{4*}_{16}$ &$:$& $e_1 e_2=e_3$  & $e_1 e_3=e_4$  & $e_2 e_3=e_4$&&&
$\mathrm{H}^2_{\mathfrak{C}}\neq \mathrm{H}^2_{\mathfrak{D}}$ \\
\hline
${\mathbf N}^{4*}_{17}$ &$:$&  $e_1 e_2 = e_3$ & $e_3 e_3= e_4$&&&&
$\mathrm{H}^2_{\mathfrak{C}}\neq \mathrm{H}^2_{\mathfrak{D}}$\\
\hline
${\mathbf N}^{4*}_{18}$ &$:$&  $e_1 e_1 = e_4$ & $e_1e_2=e_3$ & $e_3 e_3= e_4$&&&
$\mathrm{H}^2_{\mathfrak{C}}\neq \mathrm{H}^2_{\mathfrak{D}}$\\
\hline
${\mathbf N}^{4*}_{19}$ &$:$&  $e_1 e_1= e_4$ & $e_1 e_2 = e_3$ & $e_2 e_2= e_4$ & $e_3 e_3= e_4$ && 
$\mathrm{H}^2_{\mathfrak{C}}\neq \mathrm{H}^2_{\mathfrak{D}}$\\
\hline
${\mathbf N}^{4}_{01}$ &$:$&  $e_1 e_1=e_2$ & $e_1 e_2=e_3$ & $e_2e_3=e_4$ && \\
\hline
${\mathbf N}^{4}_{02}$ &$:$& $e_1 e_1=e_2$  & $e_1 e_2=e_3$  & $e_1e_3=e_4$ & $e_2 e_3=e_4$&& \\
\hline
${\mathbf N}^{4}_{03}$ &$:$&  $e_1 e_1 = e_2$  & $e_1 e_2=e_3$ & $e_3 e_3= e_4$&&\\
\hline
${\mathbf N}^{4}_{04}$ &$:$&  $e_1 e_1 = e_2$  & $e_1 e_2=e_3$ & $e_2e_2=e_4$ & $e_3 e_3= e_4$&&\\
\hline
${\mathbf N}^{4}_{05}$ &$:$& $e_1 e_1=e_2$ & $e_1 e_3 = e_4$ &  $e_2 e_2=e_3$&& \\
\hline
${\mathbf N}^{4}_{06}$ &$:$& $e_1 e_1=e_2$ & $e_1e_2=e_4$ & $e_1 e_3 = e_4$ &  $e_2 e_2=e_3$&&  \\
\hline
${\mathbf N}^{4}_{07}$ &$:$&  $e_1 e_1=e_2$ & $e_2 e_2=e_3$ & $e_2 e_3=e_4$ &&  \\
\hline
${\mathbf N}^{4}_{08}$ &$:$&  $e_1 e_1=e_2$ & $e_1e_3=e_4$ & $e_2 e_2=e_3$ & $e_2 e_3=e_4$  && \\
\hline
${\mathbf N}^{4}_{09}$ &$:$& $e_1 e_1=e_2$  & $e_2 e_2=e_3$  & $e_3 e_3=e_4$&& \\
\hline
${\mathbf N}^{4}_{10}$ &$:$&  $e_1 e_1 = e_2$ & $e_2e_2=e_3$ & $e_1e_2=e_4$ &$ e_3 e_3= e_4$&&\\
\hline
${\mathbf N}^{4}_{11}(\lambda)$ &$:$&  
$e_1 e_1 = e_2$ & $e_1e_2=\lambda e_4$ & $e_2 e_2=e_3$ &&&\\
&&$e_2e_3= e_4$ & $e_3 e_3= e_4$&&&&\\
\end{longtable} 

\section{Central extensions of $ 3 $-dimensional nilpotent commutative algebras}

\subsection{\textbf{$ 2 $-dimensional central extension of $ \mathbf{N}_{02}^{3*} $.}} Here we will collect all information about $ \mathbf{N}_{02}^{3*}: $

$$
\begin{array}{|l|l|l|l|}
\hline
\text{ }  & \text{ } & \text{Cohomology} & \text{Automorphisms} \\
\hline
{\mathbf{N}}^{3*}_{02} & 
\begin{array}{l}e_1e_1=e_2 \\ e_1e_2=e_3
\end{array}
&
\begin{array}{lcl}
\mathrm{H}^2_{\mathfrak{D}}(\mathbf{N}^{3*}_{02})&=&\langle [\Delta_{13}], [\Delta_{22}] \rangle,\\
\mathrm{H}^2_{\mathfrak{C}}(\mathbf{N}^{3*}_{02})&=&\mathrm{H}^2_{\mathfrak{D}}(\mathbf{N}^{3*}_{02})\oplus\\
&&\langle [\Delta_{23}], [\Delta_{33}] \rangle \\
\end{array}
& \phi=\begin{pmatrix}
x&0&0\\
y&x^2&0\\
z&2xy&x^3
\end{pmatrix}\\
\hline
\end{array}$$

Let us use the following notations:
\begin{align*} \nabla_1=[\Delta_{13}], \quad \nabla_2=[\Delta_{22}], \quad \nabla_3=[\Delta_{23}], \quad  \nabla_4=[\Delta_{33}].
\end{align*}
Take $ \theta=\sum\limits_{i=1}^{4}\alpha_i\nabla_i\in\mathrm{H}^2_{\mathfrak{C}}(\mathbf{N}^{3*}_{02}) .$ Since
	
$$\phi^T\begin{pmatrix}
0&0&\alpha_1\\
0&\alpha_2&\alpha_3\\
\alpha_1&\alpha_3&\alpha_4
\end{pmatrix}\phi=
\begin{pmatrix}
\alpha^*&\alpha^{**}&\alpha^*_1\\
\alpha^{**}&\alpha^*_2&\alpha^*_3\\
\alpha^*_1&\alpha^*_3&\alpha^*_4
\end{pmatrix},$$	
we have
\begin{longtable}{ll}
$\alpha_1^*=(\alpha_1x+\alpha_3y+\alpha_4z)x^3$, &
$\alpha_2^*=(\alpha_2x^2+4\alpha_3xy+4\alpha_4y^2)x^2$,\\
$\alpha_3^*=(\alpha_3x+2\alpha_4y)x^4$,&
$\alpha_4^*=\alpha_4x^6.$
\end{longtable}

%\subsubsection{$ 2 $-dimensional central extensions.}
We are interested  only in $ (\alpha_3,\alpha_4)\neq(0,0) $
  and consider the vector space generated by the following two cocycles:
$$ \theta_1=\alpha_1\nabla_1+\alpha_2\nabla_2+\alpha_3\nabla_3+\alpha_4\nabla_4 \ \ \text{and} \ \  \theta_2=\beta_1\nabla_1+\beta_2\nabla_2+\beta_3\nabla_3.$$
Thus, we have
\begin{longtable}{ll}
$\alpha_1^*=(\alpha_1x+\alpha_3y+\alpha_4z)x^3,$ & $\beta_1^*=(\beta_1x+\beta_3y)x^3,$\\
$\alpha_2^*=(\alpha_2x^2+4\alpha_3xy+4\alpha_4y^2)x^2,$ & $\beta_2^*=(\beta_2x+4\beta_3y)x^3,$\\
$\alpha_3^*=(\alpha_3x+2\alpha_4y)x^4,$ & $\beta_3^*=\beta_3x^5.$\\
$\alpha_4^*=\alpha_4x^6.$
\end{longtable}	

Consider the following cases.

\begin{enumerate}
	\item$ \alpha_4\neq0, $ then:
\begin{enumerate}

\item $ \beta_3=0, \beta_2\neq 0, \beta_1=0 ,$ then by choosing 
$x=2 \alpha_4^2, $
$y=-\alpha_3 \alpha_4,$
$z=\alpha_3^2-2 \alpha_1 \alpha_4,$
we have the representatives $\left\langle \nabla_4, \nabla_2 \right\rangle;$

\item\label{caso2} $ \beta_3=0, \beta_2\neq 0, \beta_1\neq 0 ,$ then by choosing 
\begin{center}
$x=2 \alpha_4^2 \beta_2, $
$y=-\alpha_3 \alpha_4 \beta_2,$
$z=\alpha_3^2 (-2 \beta_1+\beta_2)+2 \alpha_4 (\alpha_2 \beta_1-\alpha_1 \beta_2),$
\end{center}
we have the representatives 
$\left\langle \nabla_4, \nabla_1 +\alpha \nabla_2 \right\rangle_{\alpha\neq 0};$

\item $ \beta_3=0, \beta_2= 0, \beta_1\neq 0 ,$ then by choosing 
$y=-\frac{x \alpha_3}{2 \alpha_4},$
we have two representatives 
$\left\langle \nabla_4, \nabla_1 \right\rangle$ and 
$\left\langle \nabla_2+\nabla_4, \nabla_1 \right\rangle,$ depending on 
$\alpha_3^2=\alpha_2 \alpha_4$ or not. 
The first representative will be joint with the family from the case (\ref{caso2});

	\item $ \beta_3\neq0, 4\alpha_2\beta_3^2=4\beta_2\alpha_3\beta_3-\beta_2^2\alpha_4, \beta_2=4\beta_1,$ then by choosing 
	\begin{center}$ x=4\beta_3\alpha_4, y=-\beta_2\alpha_4, z=\beta_2\alpha_3-4\alpha_1\beta_3,$ \end{center} 
	we have the   representative $\left\langle \nabla_4, \nabla_3 \right\rangle; $
	
	\item $ \beta_3\neq0, 4\alpha_2\beta_3^2=4\beta_2\alpha_3\beta_3-\beta_2^2\alpha_4, \beta_2\neq4\beta_1,$ then by choosing 
	\begin{center}$ x=\frac{4\beta_1-\beta_2}{4\beta_3}, y=\frac{\beta_2^2-4\beta_1\beta_2}{16\beta_3^2}, z=\frac{(4\beta_1-\beta_2)(8\beta_1\alpha_3\beta_3-4\beta_1\beta_2\alpha_4-8\alpha_1\beta_3^3+\beta_2^2\alpha_4)}{32\beta_3^3\alpha_4},$ \end{center}
	we have the    representative 
	$\left\langle \nabla_4, \nabla_1+\nabla_3 \right\rangle;$
	
	\item $ \beta_3\neq0, 4\alpha_2\beta_3^2\neq4\beta_2\alpha_3\beta_3-\beta_2^2\alpha_4,$ then by choosing 
	\begin{center}
	$ x=\sqrt{\frac{4\alpha_2\beta_3^2-4\beta_2\alpha_3\beta_3+\beta_2^2\alpha_4}{4\beta_3^2\alpha_4}},$
	$ y=-\frac{ \beta_2\sqrt{\alpha_4 \beta_2^2-4 \alpha_3 \beta_2 \beta_3+4 \alpha_2 \beta_3^2}}{8\beta_3^2\sqrt{\alpha_4}}, $  $z=\frac{(8\beta_1\alpha_3\beta_3-4\beta_1\beta_2\alpha_4-8\alpha_1\beta_3^3+\beta_2^2\alpha_4)\sqrt{4\alpha_2\beta_3^2-4\beta_2\alpha_3\beta_3+\beta_2^2\alpha_4}}{16\beta_3^3\alpha_4\sqrt{\alpha_4}},$
	\end{center} we have the family of representatives $\left\langle \nabla_2+\nabla_4, \alpha\nabla_1+\nabla_3 \right\rangle. $
	
\end{enumerate}	

\item $ \alpha_4=0, \alpha_3\neq0,$  then we  may suppose that $ \beta_3=0$ and

\begin{enumerate}
	\item\label{case2.a} if $ \beta_1\neq0, \beta_2=4\beta_1, \alpha_2=4\alpha_1, $ then by choosing $ x=\alpha_3, y=-\alpha_1, z=0, $ we have the   representative $\left\langle \nabla_3, \nabla_1+4\nabla_2 \right\rangle;$
	
	\item if $\beta_1\neq0, \beta_2=4\beta_1, \alpha_2\neq4\alpha_1, $ then by choosing $ x=\frac{\alpha_2-4\alpha_1}{\alpha_3}, y=\frac{4\alpha^2_1-\alpha_1\alpha_2}{\alpha_3^2}, z=0, $ we have the  representative $\left\langle -24(\nabla_2+\nabla_3), \nabla_1+4\nabla_2 \right\rangle;$
	
	\item if $ \beta_1\neq0, \beta_2\neq4\beta_1 ,$ then by choosing $ x=\alpha_3(\beta_2-4\beta_1), y=\beta_1\alpha_2-\alpha_1\beta_2, z=0,$ we have the family of representatives $\left\langle \nabla_3, \nabla_1+\alpha\nabla_2 \right\rangle_{\alpha\neq4},$ which will be jointed with the case (\ref{case2.a});	

          \item if $\beta_1=0,$ then we have the   representative  $\left\langle -3\nabla_3, \nabla_2 \right\rangle$.
\end{enumerate}
\end{enumerate}	

Summarizing, we have the following distinct orbits:

\begin{center}
$
\left\langle  \nabla_1, \nabla_2+\nabla_4 \right\rangle, \, 
\left\langle  \nabla_1+4\nabla_2, -24(\nabla_2+\nabla_3) \right\rangle, \,  
\left\langle \nabla_1+\lambda\nabla_2, \nabla_3 \right\rangle, \,  
\left\langle  \nabla_1 +\lambda \nabla_2,  \nabla_4\right\rangle, 
$
\\ 
$
\left\langle \alpha\nabla_1+\nabla_3, \nabla_2+\nabla_4 \right\rangle, \, 
\left\langle  \nabla_1+\nabla_3, \nabla_4 \right\rangle, \, 
\left\langle \nabla_2, -3 \nabla_3 \right\rangle, \, 
 \left\langle \nabla_2, \nabla_4  \right\rangle, \,  
\left\langle  \nabla_3, \nabla_4 \right\rangle. 
$
\end{center}

Note that the algebras constructed from the orbits $\left\langle  \nabla_1+4\nabla_2, -24(\nabla_2+\nabla_3) \right\rangle,$ $\left\langle \nabla_1 + \lambda \nabla_2, \nabla_3 \right\rangle,$ $\left\langle  \nabla_1 +\alpha \nabla_2,  \nabla_4\right\rangle, 
$
$\left\langle \nabla_2, -3 \nabla_3 \right\rangle$ and 
$ \left\langle \nabla_2, \nabla_4  \right\rangle$
are parts of some families of algebras which found below.
Hence, we have the following new algebras:

\begin{longtable}{llllllll}

${\mathbf{N}}_{12}$& $:$& 
$e_1e_1=e_2$ &  $e_1e_2=e_3$ & $e_1e_3=e_4$ & $e_2e_2=e_5$&  $e_3e_3=e_5$
\\
 
${\mathbf{N}}_{168}^4$ & $:$& 
$e_1e_1=e_2$ &  $e_1e_2=e_3$ &$e_1e_3=e_4$ & $e_2e_2=4e_4-24e_5$ & $e_2e_3=-24e_5$ 
\\
${\mathbf{N}}_{170}^{\lambda,  0}$ & $:$& 
$e_1e_1=e_2$ &  $e_1e_2=e_3$ & $e_1e_3=e_4$ & $e_2e_2=\lambda e_4$ & $e_2e_3=e_5$
\\
 
${\mathbf{N}}_{184}^{\lambda,0}$  & $:$& 
$e_1e_1=e_2$ &  $e_1e_2=e_3$ & $e_1e_3= e_4$ & $e_2e_2=\lambda e_4$ & $e_3e_3=e_5$\\
${\mathbf{N}}_{13}^{\alpha}$ & $:$& 
$e_1e_1=e_2$ &  $e_1e_2=e_3$ & $e_1e_3=\alpha e_4$ \\
& &$e_2e_2=e_5$ &$e_2e_3=e_4$ &
$e_3e_3=e_5$
\\
${\mathbf{N}}_{14}$ & $:$& 
$e_1e_1=e_2$ &  $e_1e_2=e_3$ & $e_1e_3=e_4$ & $e_2e_3=e_4$ &$e_3e_3=e_5$
\\
${\mathbf{N}}_{76}^{-1}$ & $:$& 
$e_1e_1=e_2$ &  $e_1e_2=e_3$ & $e_2e_2=e_4$ & $e_2e_3=-3 e_5$ 
\\
 ${\mathbf{N}}_{80}^{0}$  & $:$& 
$e_1e_1=e_2$ &  $e_1e_2=e_3$ & $e_2e_2=e_4$ & $e_3e_3=e_5$ 
\\
${\mathbf{N}}_{15}$ & $:$& 
$e_1e_1=e_2$ &  $e_1e_2=e_3$   & $e_2e_3=e_4$ &$e_3e_3=e_5$
\\
\end{longtable}

\subsection{$ 2 $-dimensional central extension of $ \mathbf{N}_{04}^{3*} $.}Here we will collect all information about $ \mathbf{N}_{04}^{3*}: $
	$$
	\begin{array}{|l|l|l|l|}
%	\hline
%	\text{ }  & \text{ } & \text{Cohomology} & \text{Automorphisms} \\
	\hline
	{\mathbf{N}}^{3*}_{04} &
	\begin{array}{l}e_1e_1=e_2 \\ e_2e_2=e_3
	\end{array}
	&
	\begin{array}{lcl}
	\mathrm{H}^2_{\mathfrak{D}}(\mathbf{N}^{3*}_{04})&=&\langle [\Delta_{12}] \rangle,\\
	\mathrm{H}^2_{\mathfrak{C}}(\mathbf{N}^{3*}_{04})&=&\mathrm{H}^2_{\mathfrak{D}}(\mathbf{N}^{3*}_{04})\oplus \\&&
	\langle [\Delta_{13}], [\Delta_{23}] ,[\Delta_{33}]\rangle
	\end{array}&  
\phi=\begin{pmatrix}
x&0&0\\
0&x^2&0\\
z&0&x^4
\end{pmatrix}\\
\hline

\end{array}$$

Let us use the following notations:
\begin{align*}\nabla_1=[\Delta_{12}], \quad \nabla_2=[\Delta_{13}], \quad \nabla_3=[\Delta_{23}], \quad  \nabla_4=[\Delta_{33}]. \end{align*}

Take $ \theta=\sum\limits_{i=1}^{4}\alpha_i\nabla_i\in\mathrm{H}^2_{\mathfrak{C}}(\mathbf{N}^{3*}_{04}) .$  Since
$$\phi^T\begin{pmatrix}
0&\alpha_1&\alpha_2\\
\alpha_1&0&\alpha_3\\
\alpha_2&\alpha_3&\alpha_4
\end{pmatrix}\phi=
\begin{pmatrix}
\alpha^*&\alpha^{*}_1&\alpha^*_2\\
\alpha^{*}_1&\alpha^{**}&\alpha^*_3\\
\alpha^*_2&\alpha^*_3&\alpha^*_4
\end{pmatrix},$$	
we have

\begin{longtable}{ll}
$\alpha_1^*=(\alpha_1x+\alpha_3z)x^2$,&
$\alpha_2^*=(\alpha_2x+\alpha_4z)x^4$,\\
$\alpha_3^*=\alpha_3x^5$,&
$\alpha_4^*=\alpha_4x^8.$
\end{longtable}

%\subsubsection{$ 2 $-dimensional central extensions.}
Consider the following cases:

\begin{enumerate}
	\item $ \alpha_4\neq0, $ then consider the vector space generated by the following two cocycles:
	$$ \theta_1=\alpha_1\nabla_1+\alpha_2\nabla_2+\alpha_3\nabla_3+\alpha_4\nabla_4 \ \ \text{and} \ \  \theta_2=\beta_1\nabla_1+\beta_2\nabla_2+\beta_3\nabla_3.$$
	Thus, we have
	\begin{longtable}{ll}
$\alpha_1^*=(\alpha_1x+\alpha_3z)x^2,$ & $\beta_1^*=(\beta_1x+\beta_3z)x^2,$\\
$\alpha_2^*=(\alpha_2x+\alpha_4z)x^4,$ & $\beta_2^*=\beta_2x^5,$\\
$\alpha_3^*=\alpha_3x^6,$ & $\beta_3^*=\beta_3x^6.$\\
$\alpha_4^*=\alpha_4x^8.$
	\end{longtable}	
	
	Then we consider the following subcases:

\begin{enumerate}
	\item $\beta_3=0, \alpha_3=0,$ then we have:
	
	\begin{enumerate}
		\item if $ \beta_1=0, \alpha_1=0, $ then we have the   representative $ \left\langle \nabla_4,\nabla_2 \right\rangle; $
		
		\item if $ \beta_1=0, \alpha_1\neq0, $ then by choosing $ x=\sqrt[5]{ \alpha_1{\alpha_4}^{-1}},$ we have the  representative $ \left\langle \nabla_1+\nabla_4,\nabla_2 \right\rangle;$
		
		\item if $ \beta_1\neq0, \beta_2=0, $ then by choosing 
		$x=1$ and $z = - \alpha_2 \alpha_4^{-1},$ we have the   representative $ \left\langle \nabla_4,\nabla_1 \right\rangle; $
		
		\item\label{1-1.a.iv} if $ \beta_1\neq0, \beta_2\neq0, $ then by choosing $ x=\sqrt{{\beta_1}{\beta_2}^{-1}}$ and $ z=\frac{\alpha_1\beta_2-\beta_1\alpha_2}{\alpha_4 \sqrt{\beta_1\beta_2}},$ we have the  representative $ \left\langle \nabla_4,\nabla_1+\nabla_2 \right\rangle. $		
				
	\end{enumerate}

	\item $\beta_3=0, \alpha_3\neq0,$ then we have:
	
	\begin{enumerate}
		\item if $ \beta_2=0, $ then by choosing $ x=\sqrt{{\alpha_3}{\alpha_4}^{-1}}$ and $ 
		z=-{\alpha_2}\sqrt{{\alpha_3} \alpha_4^{-3}},$ we have the  representative $ \left\langle \nabla_3+\nabla_4,\nabla_1 \right\rangle; $
		
		\item if $ \beta_2\neq0, \beta_1=0,$ then by choosing 
		\begin{center}$ x=\sqrt{{\alpha_3}{\alpha_4}^{-1}}$ and $
		z=-{\alpha_1}\sqrt{{\alpha_3}^{-1} \alpha_4^{-1}},$
		\end{center} 
		we have the representative 
		$ \left\langle \nabla_3+\nabla_4,\nabla_2 \right\rangle; $
		
		\item if $ \beta_2\neq0, \beta_1\neq0, \beta_2\alpha_3=\beta_1\alpha_4,$ then by choosing $ x=\sqrt{{\alpha_3}{\alpha_4}^{-1}}$ and $z=0, $ we have the family of representatives $ \left\langle \alpha\nabla_1+\nabla_3+\nabla_4,\nabla_1+\nabla_2 \right\rangle;$
		
		\item if $ \beta_2\neq0, \beta_1\neq0, \beta_2\alpha_3\neq\beta_1\alpha_4,$ then by choosing 
		\begin{center}$ x=\sqrt{\frac{\beta_1}{\beta_2}}$ and $ z=\frac{(\alpha_1\beta_2-\beta_1\alpha_2)\sqrt{\beta_1}}{(\beta_1\alpha_4-\beta_2\alpha_3)\sqrt{\beta_2}},$ \end{center} we have the family of representatives $ \left\langle \alpha\nabla_3+\nabla_4,\nabla_1+\nabla_2 \right\rangle_{\alpha\neq0,1}, $
		which will be jointed with the case (\ref{1-1.a.iv}).
		
	\end{enumerate}

\item $\beta_3\neq0, \alpha_3=0,$ then we have:
 
	\begin{enumerate}
 	\item if $\beta_2=0, \alpha_1 \neq0,$ 
		then by choosing $ x=\sqrt[5]{ \alpha_1  \alpha_4^{-1}}$ and $z=-\alpha_2 \sqrt[5]{\alpha_1 \alpha_4^{-6}},$  we have the family of representatives $ \left\langle \nabla_1+\nabla_4,\alpha\nabla_1+\nabla_3 \right\rangle; $
		
 		\item if $\beta_2=0, \alpha_1=0,$ then by choosing $ z=-\frac{\alpha_2x}{\alpha_4},$ we have two representatives 
$\left\langle \nabla_4,\nabla_3 \right\rangle$ or $ \left\langle \nabla_4,\nabla_1+\nabla_3 \right\rangle  $ depending on whether $ \beta_1\alpha_4=\alpha_2 \beta_3$  or not;
		
		\item if $\beta_2\neq0,$ then by choosing $ x={\frac{\beta_2}{\beta_3}}$ and $ z=-\frac{\alpha_2\beta_2}{\beta_3\alpha_4},$ we have the family of representatives $ \left\langle \alpha\nabla_1+\nabla_4,\beta\nabla_1+\nabla_2+\nabla_3 \right\rangle. $		
				
	\end{enumerate}		
 		
\end{enumerate}

\item $ \alpha_4=0, \alpha_3\neq0 $,  then we may suppose that $ \beta_3=0. $ Thus, we have
\begin{longtable}{ll}
$\alpha_1^*=(\alpha_1x+\alpha_3z)x^2,$ & $\beta_1^* =\beta_1x^3,$\\
$\alpha_2^*=\alpha_2x^5,$ & $\beta_2^*=\beta_2x^5,$\\
$\alpha_3^*=\alpha_3x^6,$
\end{longtable} 	
and consider the following subcases:
\begin{enumerate}
	\item $ \beta_2=0, $ then  we have two  representatives $\left\langle \nabla_3,\nabla_1 \right\rangle $ or $\left\langle \nabla_2+\nabla_3,\nabla_1 \right\rangle,$ depending on whether $ \alpha_2=0 $ or not;
	
	\item $ \beta_2\neq0, \alpha_2=0,$ then by choosing $ z=-\frac{\alpha_1x}{\alpha_3},$ we have two  representatives $\left\langle \nabla_3,\nabla_2 \right\rangle $ or $\left\langle \nabla_3,\nabla_1+\nabla_2 \right\rangle,$ depending on whether $ \beta_1=0 $ or not.
	
\end{enumerate}	

\item $ \alpha_4=0, \alpha_3=0, \beta_3=0, \beta_2=0, \alpha_2\neq0, $ then we have the representative  $\left\langle \nabla_2,\nabla_1 \right\rangle.$ 	
\end{enumerate}

Summarizing, we have the following distinct orbits:
\begin{center}
$
\left\langle \nabla_1, \nabla_2 \right\rangle, $ $
\left\langle \nabla_1, \nabla_2+\nabla_3 \right\rangle,$ $
\left\langle \nabla_1, \nabla_3 \right\rangle,$ $
\left\langle \nabla_1, \nabla_3+\nabla_4  \right\rangle, $ $ \left\langle \nabla_1, \nabla_4 \right\rangle,$ 
$\left\langle \nabla_1+\nabla_2, \alpha\nabla_1+\nabla_3+\nabla_4 \right\rangle^{O(\alpha)=O(-\alpha)}, $ $
\left\langle \nabla_1+\nabla_2,\nabla_3  \right\rangle, $ $
\left\langle  \nabla_1+\nabla_2, \alpha \nabla_3 + \nabla_4 \right\rangle_{\alpha\neq 1},$ $
\left\langle \beta\nabla_1+\nabla_2+\nabla_3, \alpha \nabla_1 + \nabla_4 \right\rangle, $ $ 
\left\langle \alpha\nabla_1+\nabla_3, \nabla_1+\nabla_4 \right\rangle^{O(\alpha)=O(-\eta_3 \alpha)=O(\eta_3^2 \alpha)}, $ $
\left\langle  \nabla_1+\nabla_3, \nabla_4 \right\rangle,$
$\left\langle \nabla_1+\nabla_4, \nabla_2 \right\rangle,$ $
\left\langle \nabla_2, \nabla_3 \right\rangle, $ $
\left\langle \nabla_2, \nabla_3+\nabla_4  \right\rangle, $ $ \left\langle  \nabla_2, \nabla_4 \right\rangle,$ $ 
\left\langle \nabla_3, \nabla_4 \right\rangle.$
\end{center}	
Note that, the orbit $\left\langle \nabla_1, \nabla_2 \right\rangle$ after a change of the basis of the constructed algebra gives a part of the family 
${\mathbf{N}}_{79}^{\alpha},$ which will be found below.
Hence, we have the following new algebras:

\begin{longtable}{lllllllllllll}
${\mathbf{N}}_{76}^{0}$ & $:$ & 
$e_1e_1=e_2$ & $e_1e_2=e_3$& $e_1e_4=e_5$& $e_2e_2=e_4$ 
\\

${\mathbf{N}}_{16}$ & $:$ & 
$e_1e_1=e_2$ & $e_1e_2=e_4$& $e_1e_3=e_5$& $e_2e_2=e_3$ &$e_2e_3=e_5$
\\

${\mathbf{N}}_{17}$ & $:$ & 
$e_1e_1=e_2$ & $e_1e_2=e_4$& $e_2e_2=e_3$ &$e_2e_3=e_5$
\\

${\mathbf{N}}_{18}$ & $:$ & 
$e_1e_1=e_2$ &$e_1e_2=e_4$ & $e_2e_2=e_3$ & $e_2e_3=e_5$& $e_3e_3=e_5$
\\

${\mathbf{N}}_{19}$ & $:$ & 
$e_1e_1=e_2$ & $e_1e_2=e_4$& $e_2e_2=e_3$ & $e_3e_3=e_5$
\\

${\mathbf{N}}_{20}^{\alpha}$ & $:$ & 
$e_1e_1=e_2$ & $e_1e_2=e_4+\alpha e_5$& $e_1e_3=e_4$\\
&& $e_2e_2=e_3$ &$e_2e_3=e_5$ & $e_3e_3=e_5$
\\

${\mathbf{N}}_{21}$ & $:$ & 
$e_1e_1=e_2$ & $e_1e_2=e_4$& $e_1e_3=e_4$& $e_2e_2=e_3$ & $e_2e_3=e_5$
\\

${\mathbf{N}}_{22}^{\alpha\neq 1} $ & $:$ & 
$e_1e_1=e_2$ & $e_1e_2=e_4$& $e_1e_3=e_4$\\
&& $e_2e_2=e_3$ & $e_2e_3=\alpha e_5$ &$e_3e_3=e_5$
\\

${\mathbf{N}}_{23}^{\alpha, \beta}$ & $:$ & 
$e_1e_1=e_2$ & $e_1e_2=\beta e_4 +\alpha e_5$& $e_1e_3=e_4$\\
&& $e_2e_2=e_3$ &$e_2e_3=e_4$& $e_3e_3=e_5$

\\

${\mathbf{N}}_{24}^{\alpha}$ & $:$ & 
$e_1e_1=e_2$ & $e_1e_2=\alpha e_4+e_5$& $e_2e_2=e_3$ & $e_2e_3=e_4$ & $e_3e_3=e_5$ 
\\

${\mathbf{N}}_{25}$ & $:$ & 
$e_1e_1=e_2$ & $e_1e_3=e_4$& $e_2e_2=e_3$ & $e_2e_3=e_4$ & $e_3e_3=e_5$
\\

${\mathbf{N}}_{26}$ & $:$ & 
$e_1e_1=e_2$ & $e_1e_2=e_4$ & $e_1e_3=e_5$& $e_2e_2=e_3$ & $e_3e_3=e_4$
\\

${\mathbf{N}}_{27}$ & $:$ & 
$e_1e_1=e_2$ & $e_1e_3=e_4$& $e_2e_2=e_3$ & $e_2e_3=e_5$
\\

${\mathbf{N}}_{28}$ & $:$ & 
$e_1e_1=e_2$ & $e_1e_3=e_4$& $e_2e_2=e_3$ & $e_2e_3=e_5$& $e_3e_3=e_5$
\\

${\mathbf{N}}_{29}$ & $:$ & 
$e_1e_1=e_2$ & $e_1e_3=e_4$& $e_2e_2=e_3$ & $e_3e_3=e_5$
\\

${\mathbf{N}}_{30}$ & $:$ & 
$e_1e_1=e_2$ & $e_2e_2=e_3$ & $e_2e_3=e_4$ & $e_3e_3=e_5$
 	\end{longtable}
%All these algebras are non-isomorphic, excepting 

%\begin{center}${\mathbf{N}}_{23}^{\alpha} \cong %{\mathbf{N}}_{23}^{-\alpha}, \ 
%{\mathbf{N}}_{27}^{\alpha} \cong {\mathbf{N}}_{27}^{ \xi_5 %\alpha}. 
%$
%\end{center}

\section{Central extensions of nilpotent  $\mathfrak{CCD}$-algebras.}

\subsection{$ 1 $-dimensional central extensions of $ {\mathbf N}_{02}^{4*} $.} Here we will collect all information about $ {\mathbf N}_{02}^{4*}: $

	$$
\begin{array}{|l|l|l|l|}
%\hline
%\text{ }  & \text{ } & \text{Cohomology} & \text{Automorphisms} \\ 
\hline
{\mathbf{N}}^{4*}_{02} & 
\begin{array}{l} e_1e_1=e_2\\
e_1e_2=e_3
\end{array}
&
\begin{array}{lcl}
\mathrm{H}^2_{\mathfrak{D}}(\mathbf{N}^{4*}_{02})&=&\\
\multicolumn{3}{r}{\langle  
[\Delta_{13}],[\Delta_{22}], [\Delta_{14}],[\Delta_{24}],[\Delta_{44}]
 \rangle}\\
\mathrm{H}^2_{\mathfrak{C}}(\mathbf{N}^{4*}_{02})&=&
\mathrm{H}^2_{\mathfrak{D}}(\mathbf{N}^{4*}_{02})\oplus\\ 
&& \langle [\Delta_{23}], [\Delta_{33}], [\Delta_{34}] \rangle 
\end{array}
 &
 \phi=\begin{pmatrix}
x&0&0&0\\
q&x^2&0&0\\
w&2xq&x^3&r\\
e&0&0&t
\end{pmatrix}\\
\hline

\end{array}$$

Let us use the following notations:
\begin{longtable}{llll}
$\nabla_1=[\Delta_{13}],$&$  \nabla_2=[\Delta_{14}], $&$ \nabla_3=[\Delta_{22}], $&$  \nabla_4=[\Delta_{23}],$ \\ 
$\nabla_5=[\Delta_{24}],  $&$ \nabla_6=[\Delta_{33}],  $&$  \nabla_7=[\Delta_{34}], $&$ \nabla_8=[\Delta_{44}].$
\end{longtable}

Take $ \theta=\sum\limits_{i=1}^{8}\alpha_i\nabla_i\in\mathrm{H}^2_{\mathfrak{C}}(\mathbf{N}^{4*}_{02}) .$  Since
$$\phi^T\begin{pmatrix}
0&0&\alpha_1&\alpha_2\\
0&\alpha_3&\alpha_4&\alpha_5\\
\alpha_1&\alpha_4&\alpha_6&\alpha_7\\
\alpha_2&\alpha_5&\alpha_7&\alpha_8

\end{pmatrix}\phi=
\begin{pmatrix}
\alpha^*&\alpha^{**}&\alpha^{*}_1&\alpha^*_2\\
\alpha^{**}&\alpha^{*}_3&\alpha^*_4&\alpha^*_5\\
\alpha^{*}_1&\alpha^*_4&\alpha^*_6&\alpha^*_7\\
\alpha^*_2&\alpha^*_5&\alpha^*_7&\alpha^*_8
\end{pmatrix},$$
we have
\begin{longtable}{lcl}
$\alpha_1^*$ & $=$ & $(\alpha_1x+\alpha_4q+\alpha_6w+\alpha_7e)x^3,$ \\
$\alpha_2^*$ & $=$ & $(\alpha_1x+\alpha_4q+\alpha_6w+\alpha_7e)r+(\alpha_2x+\alpha_5q+\alpha_7w+\alpha_8e)t,$\\

$\alpha_3^*$ & $=$ & $(\alpha_3x^2+4\alpha_4xq+4\alpha_6q^2)x^2,$\\
$\alpha_4^*$ & $=$ & $(\alpha_4x+2\alpha_6q)x^4,$\\

$\alpha_5^*$ & $=$ & $(\alpha_4r+\alpha_5t)x^2+2(\alpha_6r+\alpha_7t)xq,$\\
$\alpha_6^*$ & $=$ & $\alpha_6x^6,$\\

$\alpha_7^*$ & $=$ & $(\alpha_6r+\alpha_7t)x^{3},$\\
$\alpha_8^*$ & $=$ & $\alpha_6r^2+2\alpha_7rt+\alpha_8t^2.$
\end{longtable}

We interested in $ (\alpha_2,\alpha_5,\alpha_7,\alpha_8)\neq(0,0,0,0) $ and $ (\alpha_4,\alpha_6,\alpha_7)\neq(0,0,0) .$ Let us consider the following cases:
\begin{enumerate}
	\item $ \alpha_6=0, \alpha_7=0,$ then $ \alpha_4\neq0 $ and we have the following subcases:
	\begin{enumerate}
		\item $ \alpha_8=0, \alpha_2\alpha_4-\alpha_1\alpha_5=0, $ then we have a split extension;
		
		\item $ \alpha_8=0, \alpha_2\alpha_4-\alpha_1\alpha_5\neq0, \alpha_3=4\alpha_1, $ then by choosing 
		\begin{center}$ x=\sqrt[4]{\alpha_2\alpha_4-\alpha_1\alpha_5}, t=\alpha_4^2, r=-\alpha_4\alpha_5, q=-\frac{\alpha_1\sqrt[4]{\alpha_2\alpha_4-\alpha_1\alpha_5}}{\alpha_4}, $ \end{center}
		we have the representative $ \left\langle \nabla_2+\nabla_4 \right\rangle; $
		
		\item $ \alpha_8=0, \alpha_2\alpha_4-\alpha_1\alpha_5\neq0, \alpha_3\neq4\alpha_1, $ then by choosing 
		\begin{center}$ x=\frac{\alpha_3-4\alpha_1}{\alpha_4}, t=\frac{(\alpha_3-4\alpha_1)^4}{\alpha_4^2(\alpha_2\alpha_4-\alpha_1\alpha_5)}, r=\frac{\alpha_5(\alpha_3-4\alpha_1)^4}{\alpha_4^3(\alpha_1\alpha_5-\alpha_2\alpha_4)}, q=\frac{4\alpha^2_1-\alpha_1\alpha_3}{\alpha^2_4},$
		\end{center}we have the representative $ \left\langle \nabla_2+\nabla_3+\nabla_4 \right\rangle; $
		
		\item $ \alpha_8\neq0, \alpha_3=4\alpha_1, $ then by choosing 
		\begin{center}$ x=\alpha_4\alpha_8, t=\alpha^3_4\alpha^2_8, q=-\alpha_1\alpha_8, r=-\alpha_4^2\alpha_5\alpha_8^2, e=\alpha_1\alpha_5-\alpha_2\alpha_4, $\end{center}
		we have the representative  $ \left\langle \nabla_4+\nabla_8 \right\rangle; $
		
		\item $ \alpha_8\neq0, \alpha_3\neq4\alpha_1, $ then by choosing  \begin{center}$ x=\frac{\alpha_3-4\alpha_1}{\alpha_4},$
		$t=\frac{\sqrt{(\alpha_3-4\alpha_1)^5}}{\alpha_4^2\sqrt{\alpha_8}},$ $q=\frac{4\alpha^2_1-\alpha_1\alpha_3}{\alpha_4^2},$ $r=-\frac{\alpha_5\sqrt{(\alpha_3-4\alpha_1)^5}}{\alpha_4^3\sqrt{\alpha_8}},$ $e=\frac{(4 \alpha_1- \alpha_3) (\alpha_2 \alpha_4-\alpha_1 \alpha_5)}{\alpha_4^2\alpha_8}, $\end{center}
		we have the representative  $ \left\langle \nabla_3+\nabla_4+\nabla_8 \right\rangle. $
				
	\end{enumerate}

     \item $ \alpha_6=0, \alpha_7\neq0, $ then we have the following subcases:

    \begin{enumerate}
   	   \item $ \alpha_4=0, \alpha_3=0, $ then by choosing
   	   \begin{center}$ x=2\alpha_7^2,$ 
   	   $q=-\alpha_5\alpha_7,$ 
   	   $e=-2\alpha_1\alpha_7,$ $w=\alpha_5^2+2\alpha_1\alpha_8-2\alpha_2\alpha_7,$ $ t=-2\alpha_7,$ $ r=\alpha_8, $ 
   	   \end{center}
   	   we have the representative  $ \left\langle \nabla_7 \right\rangle; $
   	
   	   \item $ \alpha_4=0, \alpha_3\neq0, $ then by choosing
   	  \begin{center}$ 
   	   x=1,$ $q=-\frac{\alpha_5}{2\alpha_7},$ $ e=- \frac{\alpha_1}{\alpha_7},$  $w=\frac{\alpha_5^2+2\alpha_1\alpha_8-2\alpha_2\alpha_7}{2 \alpha_7^2},$ $t=\frac{\alpha_3}{\alpha_7},$ $ r=-\frac{\alpha_3\alpha_8}{2\alpha_7^2}, $ \end{center}
   	   we have the representative  $ \left\langle \nabla_3+\nabla_7 \right\rangle; $
   	
   	   \item $ \alpha_4\neq0, \alpha_3\alpha_7^2-2\alpha_4\alpha_5\alpha_7+\alpha_4^2\alpha_8=0, $ then by choosing 
   	    \begin{center}
   	    $ x=\sqrt{\alpha_7}, t=\alpha_4,$ 
   	       $e=\frac{\alpha_3-4 \alpha_1}{4\sqrt{\alpha_7} },$ $
   r=-\frac{\alpha_4\alpha_8}{2\alpha_7},$ $
q=-\frac{\alpha_3\sqrt{\alpha_7}}{4\alpha_4},$ $
w=\frac{4 \alpha_1 \alpha_4 \alpha_8-4 \alpha_2 \alpha_4 \alpha_7+\alpha_3 (\alpha_5 \alpha_7-\alpha_4 \alpha_8)}{4\alpha_4\sqrt{\alpha_7^3}},$   	       
   	   \end{center} we have the representative  $ \left\langle \nabla_4+\nabla_7 \right\rangle; $
   	
   	   \item $ \alpha_4\neq0, \alpha_3\alpha_7^2-2\alpha_4\alpha_5\alpha_7+\alpha_4^2\alpha_8\neq0, $ then by choosing 
\begin{center}
$x=-\frac{\alpha_3}{2 \alpha_4}+\frac{\alpha_5}{\alpha_7}-\frac{\alpha_4 \alpha_8}{2 \alpha_7^2},$
$q=\frac{\alpha_3 (\alpha_3 \alpha_7^2-2 \alpha_4 \alpha_5 \alpha_7+\alpha_4^2 \alpha_8)}{8 \alpha_4^2 \alpha_7^2},$
$w=\frac{(\alpha_3 \alpha_7^2-2 \alpha_4 \alpha_5 \alpha_7+\alpha_4^2 \alpha_8) (4 \alpha_2 \alpha_4 \alpha_7-4 \alpha_1 \alpha_4 \alpha_8+\alpha_3 (-\alpha_5 \alpha_7+\alpha_4 \alpha_8))}{8 \alpha_4^2 \alpha_7^4},$
$e=\frac{(4 \alpha_1-\alpha_3) (\alpha_3 \alpha_7^2-2 \alpha_4 \alpha_5 \alpha_7+\alpha_4^2 \alpha_8)}{8 \alpha_4 \alpha_7^3},$
$t=\frac{(\alpha_3 \alpha_7^2-2 \alpha_4 \alpha_5 \alpha_7+\alpha_4^2 \alpha_8)^2}{4 \alpha_4 \alpha_7^5},$
$r=-\frac{\alpha_8 (\alpha_3 \alpha_7^2-2 \alpha_4 \alpha_5 \alpha_7+\alpha_4^2 \alpha_8)^2}{8 \alpha_4 \alpha_7^6}, $
\end{center}   	   
   	   
   	   we have the representative  $ \left\langle \nabla_4+ \nabla_5+\nabla_7 \right\rangle. $   	
   	      	
    \end{enumerate}

   \item $ \alpha_6\neq0, $ then we have the following subcases:

   \begin{enumerate}
   	\item $ \alpha_6\alpha_8-\alpha_7^2=0, \alpha_5\alpha_6-\alpha_4\alpha_7=0, \alpha_2\alpha_6-\alpha_1\alpha_7=0, $ then we have a split extension;
   	
   	\item $ \alpha_6\alpha_8-\alpha_7^2=0, \alpha_5\alpha_6-\alpha_4\alpha_7=0, \alpha_2\alpha_6-\alpha_1\alpha_7\neq0, \alpha_3\alpha_6-\alpha^2_4=0,$ then by choosing 
   	\begin{center}$x=1, t=\frac{\alpha_6^2}{\alpha_2\alpha_6-\alpha_1\alpha_7}, q=-\frac{\alpha_4}{2\alpha_6}, r=\frac{\alpha_6\alpha_7}{\alpha_1\alpha_7-\alpha_2\alpha_6}, e=0, w=\frac{\alpha^2_4-2\alpha_1\alpha_6}{\alpha_6},$
   	\end{center}
   	we have the representative  $ \left\langle \nabla_2+\nabla_6 \right\rangle; $
   	
   	\item $ \alpha_6\alpha_8-\alpha_7^2=0, \alpha_5\alpha_6-\alpha_4\alpha_7=0, \alpha_2\alpha_6-\alpha_1\alpha_7\neq0, \alpha_3\alpha_6-\alpha^2_4\neq0,$ then by choosing 
  	\begin{center}$ x=\sqrt{\frac{\alpha_3\alpha_6-\alpha_4^2}{\alpha_6^2}}, $
$ t=\frac{\sqrt{(\alpha_3 \alpha_6-\alpha_4^2)^5}}{\alpha_6^3 (\alpha_2 \alpha_6-\alpha_1 \alpha_7)}, $
$ q=-\frac{\alpha_4 \sqrt{\alpha_3 \alpha_6-\alpha_4^2}}{2 \alpha_6^2}, $
$ r=-\frac{\sqrt{(\alpha_3 \alpha_6-\alpha_4^2)^5}  \alpha_7}{\alpha_6^4 (\alpha_2 \alpha_6-\alpha_1 \alpha_7)}, e=0,$   	\end{center} 
and $w=\frac{(\alpha_4^2-2 \alpha_1 \alpha_6) \sqrt{\alpha_3 \alpha_6-\alpha_4^2}}{2 \alpha_6^3},$
we have the representative  $ \left\langle \nabla_2+\nabla_3+\nabla_6 \right\rangle;$
   	
   	\item $ \alpha_6\alpha_8-\alpha_7^2=0, \alpha_5\alpha_6-\alpha_4\alpha_7\neq0, 2\alpha_6(\alpha_2\alpha_6-\alpha_1\alpha_7)=\alpha_4(\alpha_5\alpha_6-\alpha_4\alpha_7),$ then by  choosing 
   	\begin{center}
   	$ t=\frac{\alpha_6^2}{\alpha_5\alpha_6-\alpha_4\alpha_7}x^4,$ $q=-\frac{\alpha_4}{2\alpha_6}x,$ $r=\frac{\alpha_6\alpha_7}{\alpha_4\alpha_7-\alpha_5\alpha_6}x^4,$ $e=0,$ $w=\frac{\alpha^2_4-2\alpha_1\alpha_6}{\alpha_6}x,$ 
   	\end{center} 
   	we have the representatives $ \left\langle \nabla_5+\nabla_6 \right\rangle $ and $ \left\langle \nabla_3+\nabla_5+\nabla_6 \right\rangle  $ depending on whether $ \alpha_3\alpha_6=\alpha_4^2$ or not;
   	
   	\item $ \alpha_6\alpha_8-\alpha_7^2=0, \alpha_5\alpha_6-\alpha_4\alpha_7\neq0, 2\alpha_6(\alpha_2\alpha_6-\alpha_1\alpha_7)\neq\alpha_4(\alpha_5\alpha_6-\alpha_4\alpha_7),$ then by  choosing $ x=\frac{2\alpha_6(\alpha_2\alpha_6-\alpha_1\alpha_7)-\alpha_4(\alpha_5\alpha_6-\alpha_4\alpha_7)}{2\alpha_6^2(\alpha_5\alpha_6-\alpha_4\alpha_7)},$ 
   	$t=\frac{\alpha_6^2}{\alpha_5\alpha_6-\alpha_4\alpha_7}x^4,$ $q=-\frac{\alpha_4}{2\alpha_6}x,$ $r=\frac{\alpha_6\alpha_7}{\alpha_4\alpha_7-\alpha_5\alpha_6}x^4,$ $e=0,$ $w=\frac{\alpha^2_4-2\alpha_1\alpha_6}{\alpha_6}x,$ we have the representative $ \left\langle \nabla_2+\alpha\nabla_3+\nabla_5+\nabla_6 \right\rangle;$
   	
   	\item $ \alpha_6\alpha_8-\alpha_7^2\neq0, \alpha_5\alpha_6-\alpha_4\alpha_7=0,$ then by  choosing 
   	\begin{center}
   	$ t=\frac{\alpha_6x^3}{\sqrt{\alpha_6\alpha_8-\alpha^2_7}}, $
   	$q=-\frac{\alpha_4x}{2\alpha_6},$ 
   	$r=-\frac{\alpha_7x^3}{\sqrt{\alpha_6\alpha_8-\alpha^2_7}}, $
   	$e=\frac{(\alpha_1\alpha_7-\alpha_2\alpha_6)x}{\alpha_6\alpha_8-\alpha_7^2}, $
   	$w=\Big(\frac{ \alpha_4^2}{2 \alpha_6^2}+ \frac{\alpha_1 \alpha_8-\alpha_2 \alpha_7}{\alpha_7^2-\alpha_6 \alpha_8} \Big)x,$ 
   	\end{center} 
   	we have the representatives $ \left\langle \nabla_6+\nabla_8 \right\rangle  $ and $ \left\langle \nabla_3+\nabla_6+\nabla_8 \right\rangle  $ depending on whether $ \alpha_3\alpha_6=\alpha_4^2$ or not.
   	
   	\item $ \alpha_6\alpha_8-\alpha_7^2\neq0, \alpha_5\alpha_6-\alpha_4\alpha_7\neq0,$ then by  choosing 
   	
   \begin{center}
   $ x=\frac{\alpha_5 \alpha_6-\alpha_4 \alpha_7}{\sqrt{\alpha_6^2 (\alpha_6 \alpha_8-\alpha_7^2)}},$ 
   $t=\frac{(\alpha_5 \alpha_6-\alpha_4 \alpha_7)^3}{\alpha_6^2 (\alpha_7^2-\alpha_6 \alpha_8)^2},$ 
   $q=\frac{\alpha_4 (\alpha_4 \alpha_7-\alpha_5 \alpha_6)}{2 \alpha_6 \sqrt{\alpha_6^2 (\alpha_6 \alpha_8-\alpha_7^2)}},$ 
   $r=\frac{\alpha_7 (\alpha_4 \alpha_7-\alpha_5 \alpha_6)^3}{\alpha_6^3 (\alpha_7^2-\alpha_6 \alpha_8)^2},$ 
   $e=\frac{\alpha_6 (\alpha_5 \alpha_6-\alpha_4 \alpha_7) (\alpha_4 \alpha_5 \alpha_6-\alpha_4^2 \alpha_7+2 \alpha_6 (-\alpha_2 \alpha_6+\alpha_1 \alpha_7))}{2 \alpha_6^3 \sqrt{(\alpha_6 \alpha_8-\alpha_7^2)^3}},$ 
   $w=\frac{\alpha_6 (\alpha_5 \alpha_6-\alpha_4 \alpha_7) (\alpha_4^2 \alpha_8-\alpha_4 \alpha_5 \alpha_7+2 \alpha_6 (\alpha_2 \alpha_7-\alpha_1 \alpha_8))}{2 \alpha_6^3 \sqrt{(\alpha_6 \alpha_8-\alpha_7^2)^3}},$
   \end{center}
   we have the representative $ \left\langle \alpha\nabla_3+\nabla_5+\nabla_6+\nabla_8 \right\rangle.$
   	  	
   \end{enumerate}
\end{enumerate}
Summarizing, we have the following distinct orbits
\begin{center}
$
\left\langle \nabla_2+ \nabla_3+ \nabla_4\right\rangle,$
$\left\langle \nabla_2+ \alpha \nabla_3+ \nabla_5+\nabla_6\right\rangle,$ 
$\left\langle \nabla_2+\nabla_3+\nabla_6\right\rangle,$
$\left\langle \nabla_2+ \nabla_4 \right\rangle,$
$\left\langle \nabla_2+ \nabla_6\right\rangle,$
$\left\langle \nabla_3+ \nabla_4+ \nabla_8\right\rangle,$
$\left\langle \nabla_3+ \nabla_5+\nabla_6\right\rangle,$
$\left\langle \alpha\nabla_3+\nabla_5+\nabla_6+\nabla_8 \right\rangle,$
$\left\langle \nabla_3+\nabla_6+\nabla_8\right\rangle,$
$\left\langle \nabla_3+ \nabla_7 \right\rangle,$ 
$\left\langle \nabla_4+ \nabla_5+ \nabla_7\right\rangle,$ 
$\left\langle \nabla_4+ \nabla_7\right\rangle, $
$\left\langle \nabla_4+ \nabla_8\right\rangle,$
$\left\langle \nabla_5+ \nabla_6 \right\rangle,$ 
$\left\langle \nabla_6+ \nabla_8\right\rangle,$ 
$\left\langle \nabla_7\right\rangle,$
 \end{center}
which gives the following new algebras:

\begin{longtable}{llllllllllllllllll}

${\mathbf{N}}_{31}$ & $:$ & 
$e_1e_1=e_2$ & $e_1e_2=e_3$ & $e_1e_4=e_5$ & $e_2e_2=e_5$& $e_2e_3=e_5$
\\

${\mathbf{N}}_{32}^{\alpha}$ & $:$ & 
$e_1e_1=e_2$ & $e_1e_2=e_3$ & $e_1e_4=e_5$ &$e_2e_2=\alpha e_5$ & $e_2e_4=e_5$ & $e_3e_3=e_5$
\\

${\mathbf{N}}_{33}$ & $:$ & 
$e_1e_1=e_2$ & $e_1e_2=e_3$ & $e_1e_4=e_5$ & $e_2e_2=e_5$ & $e_3e_3=e_5$
\\

${\mathbf{N}}_{34}$ & $:$ & 
$e_1e_1=e_2$ & $e_1e_2=e_3$ & $e_1e_4=e_5$ & $e_2e_3=e_5$
\\

${\mathbf{N}}_{35}$ & $:$ & 
$e_1e_1=e_2$ & $e_1e_2=e_3$ & $e_1e_4=e_5$ & $e_3e_3=e_5$
\\

${\mathbf{N}}_{36}$ & $:$ & 
$e_1e_1=e_2$ & $e_1e_2=e_3$ & $e_2e_2=e_5$ & $e_2e_3=e_5$ & $e_4e_4=e_5$
\\

${\mathbf{N}}_{37}$ & $:$ & 
$e_1e_1=e_2$ & $e_1e_2=e_3$ & $e_2e_2=e_5$ & $e_2e_4=e_5$ & $e_3e_3=e_5$
\\

${\mathbf{N}}_{38}^{\alpha}$ & $:$ & 
$e_1e_1=e_2$ & $e_1e_2=e_3$ & $e_2e_2=\alpha e_5$ & $e_2e_4=e_5$ & $e_3e_3=e_5$ & $e_4e_4=e_5$
\\

${\mathbf{N}}_{39}$ & $:$ & 
$e_1e_1=e_2$ & $e_1e_2=e_3$ & $e_2e_2=e_5$ & $e_3e_3=e_5$ & $e_4e_4=e_5$
\\

${\mathbf{N}}_{40}$ & $:$ & 
$e_1e_1=e_2$ & $e_1e_2=e_3$ & $e_2e_2=e_5$ & $e_3e_4=e_5$
\\

${\mathbf{N}}_{41}$ & $:$ & 
$e_1e_1=e_2$ & $e_1e_2=e_3$ &  $e_2e_3=e_5$ & $e_2e_4=e_5$ & $e_3e_4=e_5$
\\

${\mathbf{N}}_{42}$ & $:$ & 
$e_1e_1=e_2$ & $e_1e_2=e_3$ & $e_2e_3=e_5$ & $e_3e_4=e_5$
\\

${\mathbf{N}}_{43}$ & $:$ & 
$e_1e_1=e_2$ & $e_1e_2=e_3$ & $e_2e_3=e_5$ & $e_4e_4=e_5$
\\

${\mathbf{N}}_{44}$ & $:$ & 
$e_1e_1=e_2$ & $e_1e_2=e_3$ & $e_2e_4=e_5$ & $e_3e_3=e_5$
\\

${\mathbf{N}}_{45}$ & $:$ & 
$e_1e_1=e_2$ & $e_1e_2=e_3$ & $e_3e_3=e_5$ & $e_4e_4=e_5$
\\

${\mathbf{N}}_{46}$ & $:$ & 
$e_1e_1=e_2$ & $e_1e_2=e_3$ & $e_3e_4=e_5$

\end{longtable}

\subsection{$ 1 $-dimensional central extensions of $ {\mathbf N}_{04}^{4*} $.} Here we will collect all information about $ {\mathbf N}_{04}^{4*}: $
$$
\begin{array}{|l|l|l|l|}
%\hline
%\text{ }  & \text{ } & \text{Cohomology} & \text{Automorphisms} \\
\hline
{\mathbf{N}}^{4*}_{04} & 
\begin{array}{l}
e_1e_1=e_2 \\ 
e_2e_2=e_3
\end{array}
&
\begin{array}{lcl}\mathrm{H}^2_{\mathfrak{D}}(\mathbf{N}^{4*}_{04})&=&\\
\multicolumn{3}{r}{\langle [\Delta_{12}],[\Delta_{14}], [\Delta_{24}],[\Delta_{44}] \rangle},\\
\mathrm{H}^2_{\mathfrak{C}}(\mathbf{N}^{4*}_{04})&=&\mathrm{H}^2_{\mathfrak{D}}(\mathbf{N}^{4*}_{04})\oplus\\
\multicolumn{3}{r}{\langle [\Delta_{13}], [\Delta_{23}], [\Delta_{33}], [\Delta_{34}] \rangle}
\end{array} 
&
  \phi=\begin{pmatrix}
x&0&0&0\\
0&x^2&0&0\\
y&0&x^4&r\\
z&0&0&t
\end{pmatrix}\\
\hline

\end{array}$$

Let us use the following notations:
\begin{longtable}{llll}
$\nabla_1=[\Delta_{12}],$&$ \nabla_2=[\Delta_{13}], $&$ \nabla_3=[\Delta_{14}],$&$  \nabla_4=[\Delta_{23}],$\\
$\nabla_5=[\Delta_{24}], $&$ \nabla_6=[\Delta_{33}], $&$  \nabla_7=[\Delta_{34}],$&$ \nabla_8=[\Delta_{44}]. $
\end{longtable}

Take $ \theta=\sum\limits_{i=1}^{8}\alpha_i\nabla_i\in\mathrm{H}^2_{\mathfrak{C}}(\mathbf{N}^{4*}_{04}) .$  Since
$$\phi^T\begin{pmatrix}
0&\alpha_1&\alpha_2&\alpha_3\\
\alpha_1&0&\alpha_4&\alpha_5\\
\alpha_2&\alpha_4&\alpha_6&\alpha_7\\
\alpha_3&\alpha_5&\alpha_7&\alpha_8

\end{pmatrix}\phi=
\begin{pmatrix}
\alpha^*&\alpha^{*}_1&\alpha^{*}_2&\alpha^*_3\\
\alpha^{*}_1&\alpha^{**}&\alpha^*_4&\alpha^*_5\\
\alpha^{*}_2&\alpha^*_4&\alpha^*_6&\alpha^*_7\\
\alpha^*_3&\alpha^*_5&\alpha^*_7&\alpha^*_8
\end{pmatrix},$$

we have
\begin{longtable}{ll}
$\alpha_1^*=(\alpha_1x+\alpha_4y+\alpha_5z)x^2,$&
$\alpha_2^*=(\alpha_2x+\alpha_6y+\alpha_7z)x^4,$\\
$\alpha_3^*=(\alpha_2x+\alpha_6y+\alpha_7z)r+(\alpha_3x+\alpha_7y+\alpha_8z)t,$&
$\alpha_4^*=\alpha_4x^6,$\\
$\alpha_5^*=(\alpha_4r+\alpha_5t)x^2,$&
$\alpha_6^*=\alpha_6x^8,$\\
$\alpha_7^*=(\alpha_6r+\alpha_7t)x^{4},$&
$\alpha_8^*=\alpha_6r^2+2\alpha_7rt+\alpha_8t^2.$
\end{longtable}

We interested in $ (\alpha_3,\alpha_5,\alpha_7,\alpha_8)\neq(0,0,0,0) $ and $ (\alpha_2,\alpha_4,\alpha_6,\alpha_7)\neq(0,0,0,0) .$ Let us consider the following cases:

\begin{enumerate}
	\item $ \alpha_6=0, \alpha_7=0, \alpha_4=0, $ then $ \alpha_2\neq0 $ and we have the following cases:
	\begin{enumerate}
		\item if $ \alpha_8=0, \alpha_5=0, $ then by choosing $t=1$ and $ r=-\frac{\alpha_3}{\alpha_2},$ we have a split extension;
		
		\item if $ \alpha_8=0, \alpha_5\neq0, $ then by choosing 
		
		\begin{center}$ x=\alpha_2\alpha_5,$ $t=\alpha_2^4\alpha_5^2,$ $z=-\alpha_1\alpha_2$, $r=-\alpha_2^3\alpha_3\alpha_5^2,$ $y=0,$
		\end{center} we have the representative $ \left\langle \nabla_2+\nabla_5 \right\rangle; $
		
		\item if $ \alpha_8\neq0, \alpha_5=0, \alpha_1=0,$ then by choosing 
		
		\begin{center}
		$x=\alpha_3,$ $t={\sqrt{\alpha_2}}\alpha_8^2,$ $z=-{\alpha_3},$ $r=0,$ $y=0,$
		\end{center} we have the representative 
		$ \left\langle \nabla_2+\nabla_8 \right\rangle; $
		
		\item if $ \alpha_8\neq0, \alpha_5=0, \alpha_1\neq0,$ then by choosing 
		\begin{center}$ x=\sqrt{{\alpha_1}{\alpha_2}^{-1}},$ $t=\sqrt[4]{\alpha_1^5\alpha_2^{-3}}\sqrt{\alpha_8^{-1}},$ $z=-\sqrt{\alpha_1\alpha_2^{-1}} \alpha_3\alpha_8^{-1},$ $r=0,$ $y=0,$
		\end{center}
		we have the representative $ \left\langle \nabla_1+\nabla_2+\nabla_8 \right\rangle;$
		
		\item if $ \alpha_8\neq0, \alpha_5\neq0,$ then by choosing 
		\begin{center}$ x={\frac{\alpha^2_5}{\alpha_2\alpha_8}}, t={\frac{\alpha_5^5}{\alpha_2^2\alpha^3_8}}, z=-{\frac{\alpha_1\alpha_5}{\alpha_2\alpha_8}}, r=\frac{\alpha_5^4(\alpha_1\alpha_8-\alpha_3\alpha_5)}{\alpha_2^3\alpha_8^3}, y=0,$
		\end{center}
		we have the representative $ \left\langle \nabla_2+\nabla_5+\nabla_8 \right\rangle. $
				
	\end{enumerate}

	\item $ \alpha_6=0, \alpha_7=0, \alpha_4\neq0, $ then by choosing $ r=-\frac{\alpha_5}{\alpha_4}t, y=-\frac{\alpha_1x+\alpha_5z}{\alpha_4}, $ we have $ \alpha^*_1=\alpha^*_5=0 .$ Now we can suppose that $ \alpha_1=0, \alpha_5=0,$ and we have the following subcases:
	
	\begin{enumerate}
		\item if $ \alpha_8=0, \alpha_3=0, $ then  we have a split extension;
		
		\item if $ \alpha_8=0, \alpha_3\neq0, \alpha_2=0, $ then by choosing 
	$ x=\alpha_3,$ $y=0,$ $z=0,$ $r=0,$ $t=\alpha_3^4,$
		we have the representative $ \langle \nabla_3+\nabla_4 \rangle;$
		
		\item if $ \alpha_8=0, \alpha_3\neq0, \alpha_2\neq0, $ then by choosing 
		 $ x=\frac{\alpha_2}{\alpha_4},$ $y=0,$ $z=0,$ $r=0, t=\frac{\alpha_2^5}{\alpha_3\alpha_4^4}, $ we have the representative $ \langle \nabla_2+\nabla_3+\nabla_4 \rangle;$
		
		\item if $ \alpha_8\neq0, \alpha_2=0, $ then by choosing $ x=1, y=0, z=-\frac{\alpha_3}{\alpha_8}, r=0, t=\sqrt{\frac{\alpha_4}{\alpha_8}},$ we have the representative $ \langle \nabla_4+\nabla_8 \rangle;$
		
		\item if $ \alpha_8\neq0, \alpha_2\neq0, $ then by choosing $ x=\frac{\alpha_3}{\alpha_4}, y=0, z=-\frac{\alpha_2\alpha_3}{\alpha_4\alpha_8}, r=0, t=\frac{\alpha_2^3}{\sqrt{\alpha_4^5\alpha_8}},$ we have the representative $ \langle \nabla_2+\nabla_4+\nabla_8 \rangle.$
						
	\end{enumerate}

	\item  $ \alpha_6=0, \alpha_7\neq0, $ then by choosing $ r=-\frac{\alpha_8t}{2\alpha_7}, 
	y=-\frac{(\alpha_2\alpha_8-\alpha_3\alpha_7)x}{\alpha_7^2}, z=-\frac{\alpha_2x}{\alpha_7},$ 
	we have $ \alpha^*_2=\alpha_3^*=\alpha^*_8=0.$ Now we can suppose that 
	$ \alpha_2=0,$ $\alpha_3=0,$ $\alpha_8=0, $ and consider the following cases:
		\begin{enumerate}
		\item if $ \alpha_4=0, \alpha_5=0, \alpha_1=0, $ then we have the representative $ \langle \nabla_7  \rangle;$
		
		\item if $ \alpha_4=0, \alpha_5=0, \alpha_1\neq0, $ then by choosing $ x=\frac{1}{\alpha_7}, t=\alpha_1, y=0, z=0, r=0,$ we have the representative $ \langle \nabla_1+\nabla_7  \rangle;$
		
		\item if $ \alpha_4=0, \alpha_5\neq0, \alpha_1=0, $ then by choosing $ x=\sqrt{\frac{\alpha_5}{\alpha_7}}, t=1, y=0, z=0, r=0, $ we have the representative $ \langle \nabla_5+\nabla_7  \rangle;$
		
		\item if $ \alpha_4=0, \alpha_5\neq0, \alpha_1\neq0, $ then by choosing $ x=\sqrt{\frac{\alpha_5}{\alpha_7}}, t=\sqrt{\frac{\alpha_1^2}{\alpha_5\alpha_7}}, y=0, z=0, r=0, $ we have the representative $ \langle \nabla_1+\nabla_5+\nabla_7  \rangle;$
		
		\item if $ \alpha_4\neq0, \alpha_5=0, \alpha_1=0,$ then by choosing $ x=\sqrt{\alpha_7}, t=\alpha_4, y=0, z=0, r=0, $ we have the representative $ \langle \nabla_4+\nabla_7  \rangle;$
		
		\item if $ \alpha_4\neq0, \alpha_5=0, \alpha_1\neq0, $ then by choosing $ x=\sqrt[3]{\frac{\alpha_1}{\alpha_4}}, t=\sqrt[3]{\frac{\alpha_1\alpha_4^2}{\alpha^3_7}}, y=0, z=0, r=0, $ we have the representative $ \langle \nabla_1+\nabla_4+\nabla_7  \rangle;$
		
		\item if $ \alpha_4\neq0, \alpha_5\neq0, $ then by choosing $ x=\sqrt{\frac{\alpha_5}{\alpha_7}}, t=\sqrt{\frac{\alpha_4\alpha_5}{\alpha^2_7}}, y=0, z=0, r=0, $ we have the representative $ \langle \alpha\nabla_1+\nabla_4+\nabla_5+\nabla_7  \rangle.$						
						
	\end{enumerate}

    \item $\alpha_6\neq0,$ then by choosing  $ r=-\frac{\alpha_7t}{\alpha_6}, y=-\frac{\alpha_2x+\alpha_7z}{\alpha_6}, $ we have $ \alpha^*_2=\alpha_7^*=0 .$ Now we can suppose that $ \alpha_2=0, \alpha_7=0,$ and we have: 

\begin{enumerate}

\item if $\alpha_8=0, \alpha_5=0,$ then $\alpha_3\neq0$ and we have the following subcases:

\begin{enumerate}
\item $\alpha_4=0, \alpha_1=0,$ then by choosing $x=\alpha_3, t=\alpha_3^6\alpha_6, y=0, z=0, r=0,$ we have the representative  $ \langle \nabla_3+\nabla_6  \rangle;$

\item $\alpha_4=0, \alpha_1\neq0,$ then by choosing $x=\sqrt[5]{\frac{\alpha_1}{\alpha_6}}, t=\sqrt[5]{\frac{\alpha_1^7}{\alpha_3^5\alpha^2_6}}, y=0, z=0, r=0,$ we have the representative  $ \langle \nabla_1+\nabla_3+\nabla_6  \rangle;$

\item $\alpha_4\neq0,$ then by choosing $x=\sqrt{\frac{\alpha_4}{\alpha_6}}, t=\sqrt[5]{\frac{\alpha_4^7}{\alpha_3^2\alpha^5_6}}, y=0, z=0, r=0,$ we have the representative  $ \langle \alpha\nabla_1+\nabla_3+\nabla_4+\nabla_6  \rangle.$
\end{enumerate}

\item $\alpha_8=0, \alpha_5\neq0,$ then we have the following subcases:

\begin{enumerate}
\item $\alpha_4=0, \alpha_3=0,$ then by choosing $x=\sqrt[6]{\frac{\alpha_5}{\alpha_6}}, t=1, z=-\frac{\alpha_1}{\sqrt[6]{\alpha_5^5\alpha_6}},y=0,  r=0,$ we have the representative  $ \langle \nabla_5+\nabla_6  \rangle;$

\item $\alpha_4=0, \alpha_3\neq0,$ then by choosing 
$x=\frac{\alpha_3}{\alpha_5}, 
t=\alpha_3^6\alpha_5^{-7}\alpha_6, z=-\frac{\alpha_1\alpha_3}{\alpha_5^2}, y=0,  r=0,$ we have the representative  $ \langle \nabla_3+\nabla_5+\nabla_6  \rangle;$

\item $\alpha_4\neq0,$ then by choosing $x=\sqrt{\frac{\alpha_4}{\alpha_6}}, t=\frac{\alpha_4^3}{\alpha_5\alpha^2_6}, z=-\alpha_1\alpha_5^{-1}\sqrt{\alpha_4\alpha_6^{-1}}, y=0,  r=0,$ we have the representative  $ \langle \alpha\nabla_3+\nabla_4+\nabla_5+\nabla_6  \rangle.$
\end{enumerate}

\item  $\alpha_8\neq0,$ then we have the following subcases:

\begin{enumerate}
\item $\alpha_5=0, \alpha_4=0, \alpha_1=0,$ then by choosing $x=1, t=\sqrt{\frac{\alpha_6}{\alpha_8}}, z=-\frac{\alpha_3}{\alpha_8},y=0,  r=0,$ we have the representative  $ \langle \nabla_6+\nabla_8  \rangle;$

\item $\alpha_5=0, \alpha_4=0, \alpha_1\neq0,$ then by choosing 
\begin{center}$x=\sqrt[5]{\alpha_1 \alpha_6^{-1}},$ 
$t=\sqrt[10]{\alpha_1^8 \alpha_6^{-3}\alpha_8^{-5}},$ $z=-\alpha_3\alpha_8^{-1}\sqrt[5]{\alpha_1\alpha_6^{-1}},$ $y=0,$  $r=0,$
\end{center}
we have the representative  $ \langle \nabla_1+\nabla_6+\nabla_8  \rangle;$

\item $\alpha_5=0, \alpha_4\neq0,$ then by choosing 
\begin{center}
$x=\sqrt{\alpha_4 \alpha_6^{-1}},$ 
$t=\alpha_4^2 \sqrt{ \alpha_6^{-3} \alpha_8^{-1}},$ 
$z=-\alpha_3{\alpha_8^{-1}}\sqrt{\alpha_4 \alpha_6^{-1}},$ $y=0,$  $r=0,$ 
\end{center}
we have the representative  $ \langle \alpha\nabla_1+\nabla_4+\nabla_6+\nabla_8  \rangle;$

\item $\alpha_5\neq0,$ then by choosing $x=\sqrt[4]{\frac{\alpha^2_5}{\alpha_6\alpha_8}},$ $t=\frac{\alpha_5^2}{\sqrt{\alpha_6\alpha_8^3}},$ 
$z=-\frac{\alpha_3\sqrt{\alpha_5}}{\sqrt[4]{\alpha_6\alpha_8^5}},$ $y=0,$  $r=0,$ we have the representative  $ \langle \alpha\nabla_1+\beta\nabla_4+\nabla_5+\nabla_6+\nabla_8  \rangle.$

\end{enumerate}
\end{enumerate}
\end{enumerate}

Summarizing, we have the following distinct orbits:
\begin{center}
$
\left\langle \nabla_1+ \nabla_2+\nabla_8\right\rangle,$ 
$\left\langle \nabla_1+ \nabla_3+\nabla_6\right\rangle,$
$\left\langle \alpha\nabla_1+ \nabla_3+\nabla_4+\nabla_6\right\rangle^{O(\alpha)=O(-\alpha)},$ 
$ 
\left\langle  \alpha\nabla_1+ \beta \nabla_4+\nabla_5+\nabla_6+ \nabla_8\right\rangle^{O(\alpha,\beta)=O(-\alpha,\beta)=O(\pm i \alpha, -\beta)},$ 
$\left\langle \alpha\nabla_1+ \nabla_4+\nabla_6+ \nabla_8\right\rangle^{O(\alpha)=O(-\alpha)},$ 
$\left\langle \nabla_1+ \nabla_4+\nabla_7\right\rangle,$ 
$\left\langle \alpha\nabla_1+ \nabla_4+\nabla_5+\nabla_7\right\rangle^{O(\alpha)=O(-\alpha)},$
$\left\langle \nabla_1+ \nabla_5+\nabla_7\right\rangle,$
$\left\langle \nabla_1+ \nabla_6+\nabla_8\right\rangle, $   
$ 
\left\langle \nabla_1+\nabla_7\right\rangle,$
$\left\langle \nabla_2+\nabla_3+ \nabla_4 \right\rangle,$
$\left\langle \nabla_2+ \nabla_4+\nabla_8\right\rangle,$ 
$\left\langle \nabla_2+ \nabla_5 \right\rangle,$
$\left\langle \nabla_2+ \nabla_5+ \nabla_8\right\rangle, $ 

$
\left\langle \nabla_2+ \nabla_8 \right\rangle,$
$\left\langle \nabla_3+\nabla_4\right\rangle,$
$\left\langle \alpha\nabla_3+\nabla_4+\nabla_5+\nabla_6 \right\rangle^{O(\alpha)=O(-\alpha)},$
$\left\langle \nabla_3+ \nabla_5+\nabla_6\right\rangle,$ 
$\left\langle \nabla_3+\nabla_6 \right\rangle,$
$\left\langle \nabla_4+ \nabla_7\right\rangle,$
$\left\langle \nabla_4+ \nabla_8\right\rangle,$
$\left\langle \nabla_5+ \nabla_6\right\rangle,$
$\left\langle \nabla_5+\nabla_7 \right\rangle,$
$\left\langle \nabla_6+ \nabla_8\right\rangle,$
$\left\langle \nabla_7\right\rangle.
$
\end{center}

Hence, we have the following new algebras:

\begin{longtable}{lllllllllllllll}

${\mathbf{N}}_{47}$ & $:$ & 
$e_1e_1=e_2$ & $e_1e_2=e_5$ & $e_1e_3=e_5$ &$ e_2e_2=e_3$ & $e_4e_4=e_5$
\\

${\mathbf{N}}_{48}$ & $:$ & 
$e_1e_1=e_2$ & $e_1e_2=e_5$ & $e_1e_4=e_5$ &$ e_2e_2=e_3$ & $e_3e_3=e_5$
\\

${\mathbf{N}}_{49}^{\alpha}$ & $:$ & 
$e_1e_1=e_2$ & $e_1e_2=\alpha e_5$ & $e_1e_4=e_5$ &$ e_2e_2=e_3$ & $e_2e_3=e_5$ & $e_3e_3=e_5$
\\

${\mathbf{N}}_{50}^{\alpha, \beta}$ & $:$ & 
$e_1e_1=e_2$ & $e_1e_2=\alpha e_5$ &$ e_2e_2=e_3$ & $e_2e_3=\beta e_5$ \\
&& $e_2e_4=e_5$ &$e_3e_3=e_5$ & $e_4e_4=e_5$
\\
${\mathbf{N}}_{51}^{\alpha}$ & $:$ & 
$e_1e_1=e_2$ &$e_1e_2=\alpha e_5$ &$ e_2e_2=e_3$ & $e_2e_3=e_5$ & $e_3e_3=e_5$ & $e_4e_4=e_5$
\\

${\mathbf{N}}_{52}$ & $:$ & 
$e_1e_1=e_2$ & $e_1e_2=e_5$ &$ e_2e_2=e_3$ & $e_2e_3=e_5$ & $e_3e_4=e_5$
\\

${\mathbf{N}}_{53}^{\alpha}$ & $:$ & 
$e_1e_1=e_2$ & $e_1e_2=\alpha e_5$ &$ e_2e_2=e_3$ & $e_2e_3=e_5$ & $e_2e_4=e_5$ & $e_3e_4=e_5$
\\

${\mathbf{N}}_{54}$ & $:$ & 
$e_1e_1=e_2$ & $e_1e_2=e_5$ &$ e_2e_2=e_3$ & $e_2e_4=e_5$ & $e_3e_4=e_5$
\\

${\mathbf{N}}_{55}$ & $:$ & 
$e_1e_1=e_2$ & $e_1e_2=e_5$ &$ e_2e_2=e_3$ & $e_3e_3=e_5$ & $e_4e_4=e_5$
\\

${\mathbf{N}}_{56}$ & $:$ & 
$e_1e_1=e_2$ & $e_1e_2=e_5$ &$ e_2e_2=e_3$ & $e_3e_4=e_5$
\\

${\mathbf{N}}_{57}$ & $:$ & 
$e_1e_1=e_2$ &$e_1e_3=e_5$ & $e_1e_4=e_5$& $ e_2e_2=e_3$ & $e_2e_3=e_5$
\\

${\mathbf{N}}_{58}$ & $:$ & 
$e_1e_1=e_2$ &$e_1e_3=e_5$ &$ e_2e_2=e_3$ & $e_2e_3=e_5$ & $e_4e_4=e_5$
\\

${\mathbf{N}}_{59}$ & $:$ & 
$e_1e_1=e_2$ &$e_1e_3=e_5$ &$ e_2e_2=e_3$ & $e_2e_4=e_5$
\\

${\mathbf{N}}_{60}$ & $:$ & 
$e_1e_1=e_2$ &$e_1e_3=e_5$ &$ e_2e_2=e_3$ & $e_2e_4=e_5$ & $e_4e_4=e_5$
\\

${\mathbf{N}}_{61}$ & $:$ & 
$e_1e_1=e_2$ &$e_1e_3=e_5$ &$ e_2e_2=e_3$ & $e_4e_4=e_5$ \\

${\mathbf{N}}_{62}$ & $:$ & 
$e_1e_1=e_2$ &$e_1e_4=e_5$ &$ e_2e_2=e_3$ & $e_2e_3=e_5$\\

${\mathbf{N}}_{63}^{\alpha}$ & $:$ & 
$e_1e_1=e_2$ &$e_1e_4=\alpha e_5$ &$ e_2e_2=e_3$ & $e_2e_3=e_5$ & $e_2e_4=e_5$ & $e_3e_3=e_5$\\

${\mathbf{N}}_{64}$ & $:$ & 
$e_1e_1=e_2$ &$e_1e_4=e_5$ &$ e_2e_2=e_3$ & $e_2e_4=e_5$ & $e_3e_3=e_5$\\

${\mathbf{N}}_{65}$ & $:$ & 
$e_1e_1=e_2$ &$e_1e_4=e_5$ &$ e_2e_2=e_3$ & $e_3e_3=e_5$\\

${\mathbf{N}}_{66}$ & $:$ & 
$e_1e_1=e_2$  &$ e_2e_2=e_3$ &$e_2e_3=e_5$ & $e_3e_4=e_5$\\

${\mathbf{N}}_{67}$ & $:$ & 
$e_1e_1=e_2$  &$ e_2e_2=e_3$ &$e_2e_3=e_5$ & $e_4e_4=e_5$\\

${\mathbf{N}}_{68}$ & $:$ & 
$e_1e_1=e_2$  &$ e_2e_2=e_3$ &$e_2e_4=e_5$ & $e_3e_3=e_5$\\

${\mathbf{N}}_{69}$ & $:$ & 
$e_1e_1=e_2$  &$ e_2e_2=e_3$ &$e_2e_4=e_5$ & $e_3e_4=e_5$\\

${\mathbf{N}}_{70}$ & $:$ & 
$e_1e_1=e_2$  &$ e_2e_2=e_3$ &$e_3e_3=e_5$ & $e_4e_4=e_5$\\

${\mathbf{N}}_{71}$ & $:$ & 
$e_1e_1=e_2$  &$ e_2e_2=e_3$ &$e_3e_4=e_5$ &
\\

\end{longtable}
%All these algebras are non-isomorphic, excepting 

%\begin{center}
%$
%%{\mathbf{N}}_{52}^{\alpha} \cong {\mathbf{N}}_{52}(-\alpha), \ 
%{\mathbf{N}}_{53}^{\alpha, \beta} \cong {\mathbf{N}}_{53}( -\alpha,\beta)\cong {\mathbf{N}}_{53}( \pm i \alpha,-\beta),$  

%${\mathbf{N}}_{54}^{\alpha} \cong {\mathbf{N}}_{54}(-\alpha), \ 
%{\mathbf{N}}_{56}^{\alpha} \cong {\mathbf{N}}_{56}(-\alpha), \ 
%{\mathbf{N}}_{66}^{\alpha} \cong {\mathbf{N}}_{66}(-\alpha). $
%\end{center}

\subsection{$ 1 $-dimensional central extensions of $ {\mathbf N}_{08}^{4*} $.} Here we will collect all information about $ {\mathbf N}_{08}^{4*}: $
$$
\begin{array}{|l|l|l|l|}
%\hline
%\text{ }  & \text{ } & \text{Cohomology} & \text{Automorphisms}\\
\hline
{\mathbf{N}}^{4*}_{08} &  
\begin{array}{l}
e_1e_1=e_2 \\
e_1e_2=e_3 \\  
e_2e_2=e_4
\end{array}
&
\begin{array}{lcl}
\mathrm{H}^2_{\mathfrak{D}}(\mathbf{N}^{4*}_{08})&=&\\ 
\multicolumn{3}{r}{\langle [\Delta_{13}],[\Delta_{14}]+3[\Delta_{23}]\rangle}\\
\mathrm{H}^2_{\mathfrak{C}}(\mathbf{N}^{4*}_{08})&=&\mathrm{H}^2_{\mathfrak{D}}(\mathbf{N}^{4*}_{08})\oplus \\
\multicolumn{3}{r}{\langle [\Delta_{14}], [\Delta_{24}], [\Delta_{33}], [\Delta_{34}], [\Delta_{44}] \rangle}
\end{array} & 
  \phi=\begin{pmatrix}
x&0&0&0\\
y&x^2&0&0\\
z&2xy&x^3&0\\
t&y^2&x^2y&x^4
\end{pmatrix}\\
\hline

\end{array}$$

Let us use the following notations:
\begin{longtable}{llll}
$\nabla_1=[\Delta_{13}],$ & $\nabla_2=[\Delta_{14}]+3[\Delta_{23}],$& $\nabla_3=[\Delta_{14}],$&  $\nabla_4=[\Delta_{24}],$ \\
$\nabla_5=[\Delta_{33}],$&  $\nabla_6=[\Delta_{34}],$ &  $\nabla_7=[\Delta_{44}].$
\end{longtable}

Take $ \theta=\sum\limits_{i=1}^{7}\alpha_i\nabla_i\in\mathrm{H}^2_{\mathfrak{C}}(\mathbf{N}^{4*}_{08}) .$  Since
$$\phi^T\begin{pmatrix}
0&0&\alpha_1&\alpha_2+\alpha_3\\
0&0&3\alpha_2&\alpha_4\\
\alpha_1&3\alpha_2&\alpha_5&\alpha_6\\
\alpha_2+\alpha_3&\alpha_4&\alpha_6&\alpha_7

\end{pmatrix}\phi=
\begin{pmatrix}
\alpha^*&\alpha^{**}&\alpha^{*}_1&\alpha_2^*+\alpha^*_3\\
\alpha^{**}&\alpha^{***}&3\alpha^*_2&\alpha^*_4\\
\alpha^{*}_1&3\alpha^*_2&\alpha^*_5&\alpha^*_6\\
\alpha^*_2+\alpha_3^*&\alpha^*_4&\alpha^*_6&\alpha^*_7
\end{pmatrix},$$

we have

\begin{longtable}{lcl}
$\alpha_1^*$ & $=$ & $(\alpha_1x+3\alpha_2y+\alpha_5z+\alpha_6t)x^3+((\alpha_2+\alpha_3)x+\alpha_4y+\alpha_6z+\alpha_7t)x^2y,$\\

$\alpha_2^*$& $=$ & $\frac{1}{3}(3\alpha_2x^3+(\alpha_4+2\alpha_5)x^2y+3\alpha_6xy^2+\alpha_7y^3)x^2,$\\

$\alpha_3^*$& $=$& $((\alpha_2+\alpha_3)x+\alpha_4y+\alpha_6z+\alpha_7t)x^4-$\\
&& \multicolumn{1}{r}{$\frac{1}{3}(3\alpha_2x^3+(\alpha_4+2\alpha_5)x^2y+3\alpha_6xy^2+\alpha_7y^3)x^2,$}\\

$\alpha_4^*$& $=$& $(\alpha_4x^2+2\alpha_6xy+\alpha_7y^2)x^4,$\\
$\alpha_5^*$&$=$&$(\alpha_5x^2+2\alpha_6xy+\alpha_7y^2)x^4,$\\
$\alpha_6^*$&$=$&$(\alpha_6x+\alpha_7y)x^{6},$\\
$\alpha_7^*$&$=$&$\alpha_7x^8.$
\end{longtable}

We are interested in $ (\alpha_3,\alpha_4, \alpha_5, \alpha_6,\alpha_7)\neq(0,0,0,0,0), $ 
$ (\alpha_2+\alpha_3,\alpha_4, \alpha_6, \alpha_7)\neq(0,0,0,0) $ and  
$ (\alpha_1,\alpha_2, \alpha_5, \alpha_6)\neq(0,0,0,0).$ Let us consider the following cases:

\begin{enumerate}
	\item $ \alpha_7=0, \alpha_6=0, \alpha_5=0, \alpha_4=0, $ then $ \alpha_3\neq0, $ $ \alpha_2+\alpha_3 \neq 0$ and 
 $ (\alpha_1,\alpha_2)\neq(0,0).$
 	
	\begin{enumerate}
		\item\label{blabla1.a} if $  \alpha_2\neq-\frac{\alpha_3}{4}, $ then by choosing $ x=4\alpha_2+\alpha_3, y=-\alpha_1, $ we have the representative $ \left\langle \alpha\nabla_2+\nabla_3 \right\rangle_{\alpha\neq0,-\frac{1}{4},-1};  $
		
		\item if $  \alpha_2=-\frac{\alpha_3}{4}, $ then we have the representatives 
		\begin{center}$ \left\langle -\frac{1}{4}\nabla_2+\nabla_3 \right\rangle $ and $ \left\langle \nabla_1-\frac{1}{4}\nabla_2+\nabla_3 \right\rangle $\end{center} depending on $ \alpha_1=0 $ or not.

	\end{enumerate}

	\item $ \alpha_7=0, \alpha_6=0, \alpha_5=0, \alpha_4\neq0, $ then by choosing $y= -\frac {3\alpha_2}{\alpha_4}x, $ we have $\alpha_2^*=0.$
This we can suppose $\alpha_2=0,$ which implies $\alpha_1\neq 0$ and choosing $x  = \sqrt {\alpha_1 \alpha_4^{-1}}, $  
we have the representative $ \left\langle \nabla_1+\alpha\nabla_3+\nabla_4 \right\rangle.  $

	\item $ \alpha_7=0, \alpha_6=0, \alpha_5\neq0. $
	
	\begin{enumerate}
				
		\item\label{etotcase2} if $ \alpha_4=0,$ then $\alpha_2\neq-\alpha_3 $ and choosing 
\begin{center}		$ x=\frac{\alpha_2+\alpha_3}{\alpha_5},$ $y=\frac{3\alpha_2\alpha_3+3\alpha_3^2}{2\alpha_5^2},$ $z=-\frac{(\alpha_2+\alpha_3)(2\alpha_1\alpha_5+12\alpha_2\alpha_3+3\alpha_3^2)}{4\alpha_5^3},$
\end{center}
we have the representative $ \left\langle \nabla_2+\nabla_5 \right\rangle;$
		
		\item if $ \alpha_4\neq0, \alpha_4\neq\alpha_5, 2(\alpha_2\alpha_5-\alpha_2\alpha_4+\alpha_3\alpha_5)+\alpha_3\alpha_4=0,$ then by choosing 
	\begin{center}$ x=2(\alpha_4-\alpha_5), y=3\alpha_3, z=0, t=0, $
	\end{center}
	we have the representative $ \left\langle \alpha\nabla_4+\nabla_5 \right\rangle_{\alpha\neq0,1};$
		
		\item\label{etotcase1} if $ \alpha_4\neq0, \alpha_4\neq\alpha_5, 2(\alpha_2\alpha_5-\alpha_2\alpha_4+\alpha_3\alpha_5)+\alpha_3\alpha_4\neq0,$ then by choosing 
		\begin{center}$ x=\frac{2(\alpha_2\alpha_5-\alpha_2\alpha_4+\alpha_3\alpha_5)+\alpha_3\alpha_4}{2(\alpha_5^2-\alpha_4\alpha_5)},$ 
$y=\frac{3\alpha_3(2(\alpha_2\alpha_5-\alpha_2\alpha_4+\alpha_3\alpha_5)+\alpha_3\alpha_4)}{2\alpha_5(\alpha_5-\alpha_4)^2},$
$z=-\frac{(2 \alpha_2 (\alpha_4-\alpha_5)-\alpha_3 (\alpha_4+2 \alpha_5)) (4 \alpha_1 (\alpha_4-\alpha_5)^2-24 \alpha_2 \alpha_3 (\alpha_4-\alpha_5)+3 \alpha_3^2 (\alpha_4+2 \alpha_5))}{8 (\alpha_4-\alpha_5)^3 \alpha_5^2}, t=0, $ 
		\end{center}
		we have the family of representatives $ \left\langle \nabla_2+\alpha\nabla_4+\nabla_5 \right\rangle_{\alpha\neq0,1}$;
		
\item if $ \alpha_4\neq0, \alpha_4=\alpha_5, $ then by choosing $ y=-\frac{\alpha_2x}{\alpha_5}$ and  $z=\frac{(\alpha_3\alpha_5-\alpha_1\alpha_5+3\alpha_2^2)x}{\alpha_5^2}, $ we have the representatives $ \left\langle \nabla_4+\nabla_5 \right\rangle$ and $ \left\langle \nabla_3+\nabla_4+\nabla_5 \right\rangle$ depending on whether $ \alpha_3=0 $ or not.	
Note that $ \left\langle \nabla_4+\nabla_5 \right\rangle =\left\langle \nabla_2+\nabla_4+\nabla_5 \right\rangle$ and it    will be jointed with the family from the case (\ref{etotcase1}).
		
	\end{enumerate}

	\item if $ \alpha_7=0, \alpha_6\neq0, $ then by choosing  
$x=1,$
$y = \frac{
  \sqrt{(  \alpha_4 + 2   \alpha_5)^2 - 36   \alpha_2 \alpha_6}-  \alpha_4 - 2   \alpha_5   }{6 \alpha_6},$ 
\begin{center} $z = y^2 - \frac{ \alpha_3}{\alpha_6} + \frac{2 y ( \alpha_5-\alpha_4)}{3 \alpha_6}$ and  
 $t = -\frac{x^2 \alpha_1 + x y (4 \alpha_2 + \alpha_3) + x z \alpha_5 + 
   y (y \alpha_4 + z \alpha_6)}{ \alpha_6)}, $
\end{center}
we have $ \alpha_1^*=\alpha_2^*=\alpha_3^*=0. $ Now we can suppose that $ \alpha_1=0, \alpha_2=0, \alpha_3=0, $ and we have the following cases:
	
	\begin{enumerate}
		\item if $ \alpha_4=0, \alpha_5=0, $ then by choosing $ x=1, y=0, z=0, t=0, $ we have the representative $ \left\langle \nabla_6 \right\rangle;$
		
		\item\label{caseqq} if $ \alpha_4=0, \alpha_5\neq0, $ then by choosing 
$ x=-\frac{4 \alpha_5}{3 \alpha_6}, y=\frac{8 \alpha_5^2}{9 \alpha_6^2}, z=0, t=0, $ 
we have the representative $ \left\langle \nabla_4+\frac{1}{4}\nabla_5+\nabla_6 \right\rangle;$
		
		\item if $ \alpha_4\neq0, $ then by choosing $ x=\frac{\alpha_4}{\alpha_6}, y=0, z=0, t=0, $ we have the family of representatives $ \left\langle \nabla_4+\alpha\nabla_5+\nabla_6 \right\rangle,$
		which will be jointed with the representative from the case (\ref{caseqq}).
					
	\end{enumerate}
	
	\item if $ \alpha_7\neq0, $ then by choosing $x=1,$ $ y=-\frac{\alpha_6}{\alpha_7},$ $t=\frac{  2 \alpha_6^3+2 (\alpha_4-\alpha_5) \alpha_6 \alpha_7-3 \alpha_3 \alpha_7^2 }{3\alpha_7^3}$ and $z=0,$ we have $ \alpha_3^*=0, \alpha_6^*=0. $ Now we can suppose that $ \alpha_3=0, \alpha_6=0, $ and we have the following cases:	
	
	\begin{enumerate}
	\item if $ \alpha_5\neq0, $ then by choosing $ x=\sqrt{\alpha_5\alpha_7^{-1}}, y=0, 
	z=\sqrt{\alpha_1^2 \alpha_5^{-1}\alpha_7^{-1}}, t=0, $ we have the family  of representatives $ \left\langle \alpha\nabla_2+\beta\nabla_4+\nabla_5+\nabla_7 \right\rangle;$

		\item if $ \alpha_5=0, \alpha_2=0, $ then $\alpha_1\neq 0$  and we have the family of representatives $ \left\langle \nabla_1+\alpha\nabla_4+\nabla_7 \right\rangle;$
		
		\item if $ \alpha_5=0,  \alpha_2\neq0, $ then by choosing $ x=\sqrt[3]{\alpha_2\alpha_7^{-1}}, y=0, z=0, t=0, $ we have the family of representatives $ \left\langle \alpha\nabla_1+ \nabla_2+\beta\nabla_4+\nabla_7 \right\rangle.  $

	\end{enumerate}
		
\end{enumerate}

Summarizing all cases we have the following distinct orbits
\begin{center}

$\left\langle \nabla_1-\frac{1}{4}\nabla_2+\nabla_3 \right\rangle,$
$\left\langle \alpha\nabla_1+ \nabla_2+\beta\nabla_4+\nabla_7 \right\rangle^{O(\alpha,\beta)=O(-\eta_3 \alpha, \eta_3^2\beta)=O(\eta_3^2\alpha,-\eta_3 \beta)},$
$ 
\left\langle \nabla_1+\alpha\nabla_3+\nabla_4 \right\rangle^{O(\alpha)=O(- \alpha)},$
$\left\langle \nabla_1+\alpha\nabla_4+\nabla_7 \right\rangle^{O(\alpha)=O(- \alpha)},$ 
$\left\langle  \alpha\nabla_2+\nabla_3 \right\rangle_{\alpha\neq0,-1},$ 
$\left\langle \nabla_2+\alpha\nabla_4+\nabla_5 \right\rangle,$
$\left\langle \alpha\nabla_2+\beta\nabla_4+\nabla_5+\nabla_7 \right\rangle^{O(\alpha,\beta)=O(- \alpha,\beta)},$
$\left\langle \nabla_3+\nabla_4+\nabla_5 \right\rangle,$
$
\left\langle \alpha\nabla_4+\nabla_5 \right\rangle_{\alpha\neq0,1},$
$\left\langle \nabla_4+\alpha\nabla_5+\nabla_6 \right\rangle,$
$\left\langle \nabla_6 \right\rangle,
$
\end{center}
which gives the following new algebras:

\begin{longtable}{llllllllllllllllll}
${\mathbf{N}}_{72}$ & $:$ &
$e_1e_1=e_2$ & $e_1e_2=e_3$ & $e_1e_3=e_5$ \\
&& $e_1e_4=\frac{3}{4}e_5$ &$e_2e_2=e_4$ & $e_2e_3=-\frac{3}{4}e_5$
\\

${\mathbf{N}}_{73}^{\alpha, \beta}$ & $:$ &
$e_1e_1=e_2$ & $e_1e_2=e_3$ & $e_1e_3=\alpha e_5$ &$e_1e_4=e_5$\\ &&$e_2e_2=e_4$  & $e_2e_3=3e_5$ &
$e_2e_4=\beta e_5$ & $e_4e_4=e_5$
\\

${\mathbf{N}}_{74}^{\alpha}$ & $:$ &
$e_1e_1=e_2$ & $e_1e_2=e_3$ & $e_1e_3=e_5$ \\
&& $e_1e_4=\alpha e_5$ &$e_2e_2=e_4$ & $e_2e_4=e_5$

\\

${\mathbf{N}}_{75}^{\alpha}$ & $:$ &
$e_1e_1=e_2$ & $e_1e_2=e_3$ & $e_1e_3=e_5$ \\
&&$e_2e_2=e_4$ & $e_2e_4=\alpha e_5$ & $e_4e_4=e_5$
\\

${\mathbf{N}}_{76}^{\alpha\neq 0,-1}$ & $:$ &
$e_1e_1=e_2$ & $e_1e_2=e_3$ & \multicolumn{2}{l}{$e_1e_4=(1+\alpha) e_5$} \\&&$e_2e_2=e_4$ &$e_2e_3=3\alpha e_5$
\\

${\mathbf{N}}_{77}^{\alpha} $ & $:$ &
$e_1e_1=e_2$ & $e_1e_2=e_3$ & $e_1e_4=e_5$ &$e_2e_2=e_4$ \\
&& $e_2e_3=3e_5$  & $e_2e_4=\alpha e_5$ & $e_3e_3=e_5$
\\

${\mathbf{N}}_{78}^{\alpha, \beta}$ & $:$ &
$e_1e_1=e_2$ & $e_1e_2=e_3$ & $e_1e_4=\alpha e_5$ &$e_2e_2=e_4$ \\
&& $e_2e_3=3\alpha e_5$  & $e_2e_4=\beta e_5$ & $e_3e_3=e_5$ & $e_4e_4=e_5$
\\

${\mathbf{N}}_{79}$ & $:$ &
$e_1e_1=e_2$ & $e_1e_2=e_3$ &$e_1e_4=e_5$ \\& &$e_2e_2=e_4$ & $e_2e_4=e_5$ & $e_3e_3=e_5$
\\

${\mathbf{N}}_{80}^{\alpha \neq 0,1}$ & $:$ &
$e_1e_1=e_2$ & $e_1e_2=e_3$ &$e_2e_2=e_4$ & $e_2e_4=\alpha e_5$ & $e_3e_3=e_5$
\\

${\mathbf{N}}_{81}^{\alpha}$ & $:$ &
$e_1e_1=e_2$ & $e_1e_2=e_3$ &$e_2e_2=e_4$ \\
& & $e_2e_4=e_5$ &$e_3e_3=\alpha e_5$ &$e_3e_4=e_5$
\\

${\mathbf{N}}_{82}$ & $:$ &
$e_1e_1=e_2$ & $e_1e_2=e_3$ &$e_2e_2=e_4$ & $e_3e_4=e_5$
\\

\end{longtable}

%All these algebras are non-isomorphic, excepting 

%\begin{center}

%{\mathbf{N}}_{76}^{\alpha, \beta} \cong {\mathbf{N}}_{76}^{\xi_3 %\alpha, \xi_3^2 \beta}\cong {\mathbf{N}}_{76}^{\xi_3^2 \alpha, %\xi_3 \beta}, $  
%${\mathbf{N}}_{77}^{\alpha} \cong {\mathbf{N}}_{77}^{-\alpha}, %\quad   
%{\mathbf{N}}_{78}^{\alpha} \cong {\mathbf{N}}_{78}^{-\alpha}, %\quad 
%{\mathbf{N}}_{81}^{\alpha, \beta} \cong %{\mathbf{N}}_{81}^{-\alpha,\beta}. $
%\end{center}

\subsection{$ 1 $-dimensional central extensions of $ {\mathbf N}_{09}^{4*} $.} Here we will collect all information about $ {\mathbf N}_{09}^{4*}: $
$$
\begin{array}{|l|l|l|l|}
%\hline
%\text{ }  & \text{ } & \text{Cohomology} &\text{Automorphisms} \\
\hline
{\mathbf{N}}^{4*}_{09} & 
\begin{array}{l}
e_1e_1=e_2 \\ 
e_2e_3=e_4
\end{array}
&
\begin{array}{lcl}
\mathrm{H}^2_{\mathfrak{D}}(\mathbf{N}^{4*}_{09})&=&
\\
\multicolumn{3}{r}{\langle [\Delta_{12}],[\Delta_{13}],[\Delta_{22}],[\Delta_{33}]\rangle}\\
\mathrm{H}^2_{\mathfrak{C}}(\mathbf{N}^{4*}_{09})&=&\mathrm{H}^2_{\mathfrak{D}}(\mathbf{N}^{4*}_{09})\oplus\\
\multicolumn{3}{r}{\langle [\Delta_{14}],[\Delta_{24}], [\Delta_{34}], [\Delta_{44}] \rangle}
\end{array}& 
 \phi=\begin{pmatrix}
x&0&0&0\\
0&x^2&0&0\\
0&0&r&0\\
t&0&s&x^2r
\end{pmatrix}\\
\hline

\end{array}$$

Let us use the following notations:
\begin{longtable}{llll}
$\nabla_1=[\Delta_{12}], $&$ \nabla_2=[\Delta_{13}], $&$ \nabla_3=[\Delta_{14}], $&$  \nabla_4=[\Delta_{22}],$\\
$\nabla_5=[\Delta_{24}],  $&$ \nabla_6=[\Delta_{33}],  $&$  \nabla_7=[\Delta_{34}], $&$ \nabla_8=[\Delta_{44}]. $
\end{longtable}

Take $ \theta=\sum\limits_{i=1}^{8}\alpha_i\nabla_i\in\mathrm{H}^2_{\mathfrak{C}}(\mathbf{N}^{4*}_{09}) .$  Since

$$\phi^T\begin{pmatrix}
0&\alpha_1&\alpha_2&\alpha_3\\
\alpha_1&\alpha_4&0&\alpha_5\\
\alpha_2&0&\alpha_6&\alpha_7\\
\alpha_3&\alpha_5&\alpha_7&\alpha_8

\end{pmatrix}\phi=
\begin{pmatrix}
\alpha^*&\alpha_1^{*}&\alpha^{*}_2&\alpha^*_3\\
\alpha_1^{*}&\alpha_4^{*}&\alpha^{**}&\alpha^*_5\\
\alpha^{*}_1&\alpha^{**}&\alpha^*_6&\alpha^*_7\\
\alpha_3^*&\alpha^*_5&\alpha^*_7&\alpha^*_8
\end{pmatrix},$$
we have
\begin{longtable}{ll}
$\alpha_1^*=(\alpha_1x+\alpha_5t)x^2,$&
$\alpha_2^*=(\alpha_2x+\alpha_7t)r+(\alpha_3x+\alpha_8t)s,$\\

$\alpha_3^*=(\alpha_3x+\alpha_8t)x^2r,$&
$\alpha_4^*=\alpha_4x^4,$\\

$\alpha_5^*=\alpha_5x^4r,$&
$\alpha_6^*=(\alpha_6r+\alpha_7s)r+(\alpha_7r+\alpha_8s)s,$\\

$\alpha_7^*=(\alpha_7r+\alpha_8s)x^2r,$&
 $\alpha_8^*=\alpha_8 r^2 x^4.$
\end{longtable}

We are interested in $ (\alpha_3,\alpha_5,\alpha_7,\alpha_8)\neq(0,0,0,0) $ . Let us consider the following cases:

\begin{enumerate}
	\item  $ \alpha_8=0, \alpha_7=0, \alpha_5=0, $ then $ \alpha_3\neq0 $ and we have
	
	\begin{enumerate}
		\item if $ \alpha_1=0, \alpha_4=0, \alpha_6=0, $ then by choosing 
		$ x=1, r=\alpha_3, s=-\alpha_2, t=0, $ 
		we have the representative $ \left\langle \nabla_3 \right\rangle; $
		
		\item if $ \alpha_1=0, \alpha_4=0, \alpha_6\neq0, $ then by choosing 
		$ x=\alpha_6, r=\alpha_3\alpha_6^2, s=-\alpha_2\alpha_6^2, t=0,$ 
		we have the representative $ \left\langle \nabla_3+\nabla_6 \right\rangle; $
		
		\item if $ \alpha_1=0, \alpha_4\neq0, \alpha_6=0, $ then by choosing 
		$x=\alpha_3^2, r=\alpha_3\alpha_4, s=-\alpha_2\alpha_4, t=0,$ 
		we have the representative $ \left\langle \nabla_3+\nabla_4 \right\rangle; $
		
		\item if $ \alpha_1=0, \alpha_4\neq0, \alpha_6\neq0, $ then by choosing 
$x={\alpha_3}^{-1}\sqrt{\alpha_4\alpha_6},$ $r={\alpha_3}^{-2}{\sqrt{\alpha_4^3\alpha_6}},$ $s=-\alpha_2{\alpha_3}^{-3}\sqrt{\alpha_4^3\alpha_6},$ $
t=0,$
we have the representative $ \left\langle \nabla_3+\nabla_4+\nabla_6 \right\rangle;$						
		
		\item if $ \alpha_1\neq0, \alpha_4=0, \alpha_6=0, $ then by choosing
	 $ x=1,$ $ r={\alpha_1}{\alpha_3}^{-1},$ $ s=-{\alpha_1\alpha_2}{\alpha^{-2}_3},$ $ t=0, $
	 we have the representative $ \left\langle \nabla_1+\nabla_3 \right\rangle; $
		
		\item if $ \alpha_1\neq0, \alpha_4=0, \alpha_6\neq0, $ then by choosing 
	$x=\sqrt[3]{\alpha_1\alpha_6 \alpha_3^{-2}},$ 
	$r=\alpha_1\alpha_3^{-1},$ 
	$s=-\alpha_1\alpha_2\alpha^{-2}_3,$ $
	t=0,$
	we have the representative $ \left\langle \nabla_1+\nabla_3+\nabla_6 \right\rangle;$		
		
		\item if $ \alpha_1\neq0, \alpha_4\neq0,  $ then by choosing \begin{center}
$ x= \alpha_1 \alpha_4^{-1},$ 
$r=\alpha_1\alpha_3^{-1},$ 
$s=-\alpha_1\alpha_2 \alpha^{-2}_3,$ $
t=0,$\end{center}
 we have the  family of  representatives $ \left\langle \nabla_1+\nabla_3+\nabla_4+\alpha\nabla_6 \right\rangle. $		
				
	\end{enumerate}

	\item $ \alpha_8=0, \alpha_7=0, \alpha_5\neq0 $ and we have

	\begin{enumerate}
	\item if $ \alpha_3=0, \alpha_2=0, \alpha_4=0, \alpha_6=0,$ then by choosing 
	$ r=1, x=\alpha_5, t=-\alpha_1, s=0,$
	we have the representative $ \left\langle \nabla_5 \right\rangle;$
	
	\item if $ \alpha_3=0, \alpha_2=0, \alpha_4=0, \alpha_6\neq0, $ then by choosing 
	   $ x=\alpha_5\alpha_6,$ 
	   $r=\alpha_5^5\alpha_6^3, s=0, t=-\alpha_1\alpha_6,$
	   we have the representative $ \left\langle \nabla_5+\nabla_6 \right\rangle;$
	
	\item if $ \alpha_3=0, \alpha_2=0, \alpha_4\neq0, \alpha_6=0, $ then by choosing  
	  $ x=1, r=\alpha_4{\alpha_5}^{-1}, t=-{\alpha_1}{\alpha_5}^{-1}, s=0,$
	  we have the representative $ \left\langle \nabla_4+\nabla_5 \right\rangle;$
	
	\item if $ \alpha_3=0, \alpha_2=0, \alpha_4\neq0, \alpha_6\neq0, $ then by choosing 
  \begin{center}  $x={\sqrt[4]{\alpha_4\alpha_6\alpha_5^{-2}}},$ $r={\alpha_4}{\alpha_5}^{-1},$ $t=-{\alpha_1\sqrt[4]{\alpha_4\alpha_6 \alpha_5^{-6}}}, s=0,$\end{center}
	    we have the representative $ \left\langle \nabla_4+\nabla_5+\nabla_6 \right\rangle;$						
	
	\item if $ \alpha_3=0, \alpha_2\neq0, \alpha_4=0, \alpha_6=0, $ then by choosing 
\begin{center}	   $ r=1, x=\sqrt[3]{{\alpha_2}{\alpha_5}^{-1}}, t=-\alpha_1\sqrt[3]{\alpha_2\alpha_5^{-4}}, s=0,$\end{center}
	    we have the representative $ \left\langle \nabla_2+\nabla_5 \right\rangle;$
	
	\item if $ \alpha_3=0, \alpha_2\neq0, \alpha_4=0, \alpha_6\neq0, $ then by choosing 
	\begin{center}$ x=\sqrt[3]{{\alpha_2}{\alpha_5}^{-1}},$ $r=\alpha_6^{-1}\sqrt[3]{\alpha_2^{4}\alpha_5^{-1}},$ $t=-\alpha_1\sqrt[3]{\alpha_2\alpha_5^{-4}}, s=0,$\end{center}
	we have the representative $ \left\langle \nabla_2+\nabla_5+\nabla_6 \right\rangle;$
	
	\item if $ \alpha_3=0, \alpha_2\neq0, \alpha_4\neq0, $ then by choosing 
\begin{center}	    $ x=\sqrt[3]{\alpha_2\alpha_5^{-1}}, 
	    r={\alpha_4}{\alpha_5}^{-1}, t=-\alpha_1\sqrt[3]{\alpha_2 \alpha_5^{-4}}, s=0,$
	    \end{center}
	    	we have the  family of  representatives $ \left\langle \nabla_2+\nabla_4+\nabla_5+\alpha\nabla_6 \right\rangle;$		
	
	\item if $ \alpha_3\neq0, \alpha_4=0, \alpha_6=0, $ then by choosing 
$ x={\alpha_3}{\alpha_5}^{-1}, r=\alpha_3, s=-\alpha_2, t=-{\alpha_1\alpha_3}{\alpha_5^{-2}},$
	we have the representative $ \left\langle \nabla_3+\nabla_5 \right\rangle;$	
	
	\item if $ \alpha_3\neq0, \alpha_4=0, \alpha_6\neq0, $ then by choosing \begin{center}
$ x={\alpha_3}{\alpha_5}^{-1}, 
	r={\alpha_3^4}{\alpha_5^{-3}\alpha_6^{-1}}, s=-{\alpha_2\alpha_3^3}{\alpha_5^{-3}\alpha_6^{-1}},
	t=-{\alpha_1\alpha_3}{\alpha_5^{-2}},$\end{center}
	we have the representative $ \left\langle \nabla_3+\nabla_5+\nabla_6 \right\rangle;$		
	
	\item if $ \alpha_3\neq0, \alpha_4\neq0, $ then by choosing 	    $x={\alpha_3}{\alpha_5}^{-1}, 
	    r={\alpha_4}{\alpha_5}^{-1}, 
	    s=-{\alpha_2\alpha_4}{\alpha_3^{-1}\alpha_5^{-1}}, t=-{\alpha_1\alpha_3}{\alpha_5^{-2}},$
	    we have the  family of  representatives $ \left\langle \nabla_3+\nabla_4+\nabla_5+\alpha\nabla_6 \right\rangle. $

\end{enumerate}

	\item  $ \alpha_8=0, \alpha_7\neq0, $ then by choosing 
$ x=2 \alpha_7^2,$ $ t=\alpha_3 \alpha_6-2 \alpha_2 \alpha_7,$ $s=-\alpha_6,$ $r=2 \alpha_7,$
we have $ \alpha_2^*=\alpha_6^*=0. $ 
Now we can suppose that $ \alpha_2=0$ and $\alpha_6=0,$ then for $s=0$ and $t=0,$ we have:
\begin{enumerate}
	\item if $ \alpha_1=0, \alpha_3=0, \alpha_4=0, \alpha_5=0,$ then by choosing 
	    $ r=1, x=1, $
 we have the representative $ \left\langle \nabla_7 \right\rangle; $
	
	\item if $ \alpha_1=0, \alpha_3=0, \alpha_4=0, \alpha_5\neq0, $ then by choosing 
$ x=\alpha_7, r=\alpha_5\alpha_7,$ 
we have the representative $ \left\langle \nabla_5+\nabla_7 \right\rangle;$
	
	\item if $ \alpha_1=0, \alpha_3=0, \alpha_4\neq0, \alpha_5=0, $ then by choosing 
$ x=\sqrt{\alpha_7}, r=\sqrt{\alpha_4}, $
 we have the representative $ \left\langle \nabla_4+\nabla_7 \right\rangle; $
	
	\item if $ \alpha_1=0, \alpha_3=0, \alpha_4\neq0, \alpha_5\neq0, $ then by choosing $ x=\alpha_4\sqrt{\alpha_7 \alpha_5^{-3}}, r={\alpha_4}{\alpha_5}^{-1},$
we have the representative $ \left\langle \nabla_4+\nabla_5+\nabla_7 \right\rangle;$						
	
	\item if $ \alpha_1=0, \alpha_3\neq0, \alpha_5=0, $ then by choosing $ r=\alpha_3, x=\alpha_5, $
	we have the  family of representatives $ \left\langle \nabla_3+\alpha\nabla_4+\nabla_7 \right\rangle;$
	
	\item if $ \alpha_1=0, \alpha_3\neq0, \alpha_5\neq0, $ then by choosing $ x={\alpha_3}{\alpha_5}^{-1}, r={\alpha_3^2}{\alpha_5^{-1}\alpha_7^{-1}}, $
 we have the  family of representatives $ \left\langle \nabla_3+\alpha\nabla_4+\nabla_5+\nabla_7 \right\rangle;$
	
	\item if $ \alpha_1\neq0, \alpha_3=0, \alpha_4=0, \alpha_5=0,$ then by choosing $ x=\alpha_1\alpha_7, r=\alpha_1,  $
	 we have the representative $ \left\langle \nabla_1+\nabla_7 \right\rangle;$		
	
	\item if $ \alpha_1\neq0, \alpha_3=0, \alpha_4=0, \alpha_5\neq0,$ then by choosing $ x=\sqrt[3]{{\alpha_1\alpha_7}{\alpha^{-2}_5}}, r=\sqrt[3]{{\alpha_1^2}{\alpha_5^{-1}\alpha_7^{-1}}},  $
	we have the representative $ \left\langle \nabla_1+\nabla_5+\nabla_7 \right\rangle;$	
	
	\item if $ \alpha_1\neq0, \alpha_3=0, \alpha_4\neq0, $ then by choosing $ x={\alpha_1}{\alpha_4}^{-1}, r={\alpha_1}{\sqrt{\alpha_4^{-1}\alpha_7^{-1}}},$
 we have the  family of representatives $ \left\langle \nabla_1+\nabla_4+\alpha\nabla_5+\nabla_7 \right\rangle;$		
	
	\item if $ \alpha_1\neq0, \alpha_3\neq0, $ then by choosing 
 $x={\alpha_1\alpha_7}{\alpha_3^{-2}}, 
	r={\alpha_1}{\alpha_3}^{-1}, $ 
	we have the family of representatives $ \left\langle \nabla_1+\nabla_3+\alpha\nabla_4+\beta\nabla_5+\nabla_7 \right\rangle. $

\end{enumerate}
	\item $ \alpha_8\neq0, $ then by choosing 
	    $x=\alpha_8, t=- \alpha_3, s=- \alpha_7, r=\alpha_8,$
	we have $ \alpha_3^*=\alpha_7^*=0. $ Now we can suppose that $ \alpha_3=0$ and  $\alpha_7=0,$ then for $s=0$ and $t=0,$ we have
\begin{enumerate}
	\item if $ \alpha_1=0, \alpha_2=0, \alpha_4=0, \alpha_6=0,$ then we have the representatives $ \left\langle \nabla_8 \right\rangle $ and $ \left\langle \nabla_5+\nabla_8 \right\rangle, $ depending on whether $ \alpha_5=0 $ or not;
	
	\item if $ \alpha_1=0, \alpha_2=0, \alpha_4=0, \alpha_6\neq0, $ then we have the representatives $ \left\langle \nabla_6+\nabla_8 \right\rangle $ and $ \left\langle \nabla_5+\nabla_6+\nabla_8 \right\rangle, $ depending on whether $ \alpha_5=0 $ or not;
	
	\item if $ \alpha_1=0, \alpha_2=0, \alpha_4\neq0, \alpha_6=0, $ then we have the representatives $ \left\langle \nabla_4+\nabla_8 \right\rangle  $ and $ \left\langle \nabla_4+\nabla_5+\nabla_8 \right\rangle, $ depending on whether $\alpha_5=0$ or not;
	
	\item if $ \alpha_1=0, \alpha_2=0, \alpha_4\neq0, \alpha_6\neq0, $ then by choosing $ x=\sqrt[4]{{\alpha_6}{\alpha_8}^{-1}}, 
	r=\sqrt{{\alpha_4}{\alpha_8}^{-1}} ,$ we have the representative $ \left\langle \nabla_4+\alpha\nabla_5+\nabla_6+\nabla_8 \right\rangle;$						
	
	\item if $ \alpha_1=0, \alpha_2\neq0, \alpha_4=0, \alpha_5=0,$ then we have the representatives $ \left\langle \nabla_2+\nabla_8 \right\rangle $ and $ \left\langle \nabla_2+\nabla_6+\nabla_8 \right\rangle, $ depending on whether $ \alpha_6=0 $ or not;
	
	\item if $ \alpha_1=0, \alpha_2\neq0, \alpha_4=0, \alpha_5\neq0, $ then by choosing $ x=\sqrt[3]{{\alpha_2}{\alpha_5}^{-1}}, r={\alpha_5}{\alpha_8}^{-1}, $ we have the representative $ \left\langle \nabla_2+\nabla_5+\alpha\nabla_6+\nabla_8 \right\rangle; $
	
	\item if $ \alpha_1=0, \alpha_2\neq0, \alpha_4\neq0,$ then by choosing $ x=\sqrt[6]{{\alpha_2^2 \alpha_4^{-1}\alpha_8^{-1}}},$ $r=\sqrt{{\alpha_4}{\alpha_8}^{-1}},$ we have the representative $ \left\langle \nabla_2+\nabla_4+\alpha\nabla_5+\beta\nabla_6+\nabla_8 \right\rangle; $		
	
	\item if $ \alpha_1\neq0, \alpha_2=0, \alpha_4=0, \alpha_5=0,$  then we have the representatives $ \left\langle \nabla_1+\nabla_8 \right\rangle $ and $ \left\langle \nabla_1+\nabla_6+\nabla_8 \right\rangle $ depending on whether $ \alpha_6=0 $ or not;	
	
	\item if $ \alpha_1\neq0, \alpha_2=0, \alpha_4=0, \alpha_5\neq0, $ then by choosing $ x={\alpha_1\alpha_8}{\alpha_5^{-2}}, r={\alpha_5}{\alpha_8}^{-1},$ we have the representative $ \left\langle \nabla_1+\nabla_5+\alpha\nabla_6+\nabla_8 \right\rangle;$		
	
	\item if $ \alpha_1\neq0, \alpha_2=0, \alpha_4\neq0, $ then by choosing $ x={\alpha_1}{\alpha_4}^{-1}, r=\sqrt{{\alpha_4}{\alpha_8}^{-1}},$ we have the representative $ \left\langle \nabla_1+\nabla_4+\alpha\nabla_5+\beta\nabla_6+\nabla_8 \right\rangle;$
	
	\item if $ \alpha_1\neq0, \alpha_2\neq0, $ then by choosing $ x=\sqrt[5]{{\alpha_2^2}{\alpha_1^{-1}\alpha_8^{-1}}}, r=\sqrt[5]{{\alpha_1^4}{\alpha_2^{-3}\alpha_8^{-1}}},$ we have the representative $ \left\langle \nabla_1+\nabla_2+\alpha\nabla_4+\beta\nabla_5+\gamma\nabla_6+\nabla_8 \right\rangle.$

\end{enumerate}

\end{enumerate}

Summarizing, we have the following distinct orbits:
\begin{center}

$\left\langle \nabla_1+\nabla_2+\alpha\nabla_4+\beta\nabla_5+\gamma\nabla_6+\nabla_8  \right\rangle
^{
{\tiny\begin{array}{l}
O(\alpha,\beta,\gamma)=O(-\eta_5\alpha,-\eta_5^3\beta,-\eta_5\gamma)=\\
O(\eta^2_5\alpha,-\eta_5\beta,\eta_5^2\gamma)=O(-\eta^3_5\alpha,\eta_5^4\beta,-\eta_5^3\gamma)=\\
O(\eta_5^4\alpha,\eta_5^2\beta,\eta_5^4\gamma)
\end{array}} },$
$\left\langle \nabla_1+\nabla_3\right\rangle,$ 
$\left\langle \nabla_1+\nabla_3+\alpha\nabla_4+\beta \nabla_5+\nabla_7\right\rangle,$ 
$\left\langle \nabla_1+\nabla_3+\nabla_4+\alpha\nabla_6  \right\rangle,$ 
$\left\langle \nabla_1+\nabla_3 +\nabla_6\right\rangle,$
$\left\langle \nabla_1+\nabla_4+\alpha\nabla_5+\beta \nabla_6+\nabla_8\right\rangle
^{O(\alpha,\beta)=O(-\alpha,\beta)},$ 
$\left\langle \nabla_1+\nabla_4+\alpha\nabla_5+\nabla_7  \right\rangle^{O(\alpha,\beta)=O(-\alpha,\beta)},$ 
$\left\langle \nabla_1+\nabla_5 +\alpha \nabla_6+\nabla_8\right\rangle,$ 
$\left\langle \nabla_1+\nabla_5+\nabla_7\right\rangle,$
$\left\langle \nabla_1+\nabla_6+\nabla_8  \right\rangle,$ 
$\left\langle \nabla_1+\nabla_7 \right\rangle,$  
$\left\langle \nabla_1+\nabla_8\right\rangle, $
$\left\langle \nabla_2+\nabla_4+\nabla_5+\alpha\nabla_6\right\rangle^{
O(\alpha)=O(\eta_3\alpha)=O(\eta_3^2\alpha)},$ 
$\left\langle \nabla_2+\nabla_4+\alpha\nabla_5+\beta\nabla_6+\nabla_8  \right\rangle,$ 
$\left\langle \nabla_2+\nabla_5 \right\rangle,$
$\left\langle \nabla_2+\nabla_5+\nabla_6\right\rangle,$ 
$\left\langle \nabla_2+\nabla_5+\alpha\nabla_6+\nabla_8\right\rangle^
{{\tiny \begin{array}{l}
O(\alpha,\beta)=O(-\alpha,\beta)=O(\alpha,\eta_3^2\beta)=
O(-\alpha,\eta_3^2\beta)=\\
O(-\alpha,-\eta_3\beta)=
O(\alpha,-\eta_3\beta)\end{array}}},$

$\left\langle \nabla_2+\nabla_6+\nabla_8  \right\rangle,$ 
$\left\langle \nabla_2+\nabla_8 \right\rangle,$
$\left\langle \nabla_3\right\rangle,$ 
$\left\langle \nabla_3+\nabla_4\right\rangle,$
$\left\langle \nabla_3+\nabla_4+\nabla_5+\alpha\nabla_6\right\rangle,$

$\left\langle \nabla_3+\alpha\nabla_4+\nabla_5+\nabla_7  \right\rangle,$ 
$\left\langle \nabla_3+\nabla_4+\nabla_6 \right\rangle,$
$\left\langle \nabla_3+\nabla_5\right\rangle,$ 
$\left\langle \nabla_3+\nabla_5+\nabla_6\right\rangle,$
$\left\langle \nabla_3+\nabla_6\right\rangle,$

$\left\langle \nabla_4+\nabla_5 \right\rangle,$ 
$\left\langle \nabla_4+\nabla_5+\nabla_6 \right\rangle,$
$\left\langle \nabla_4+\alpha\nabla_5+\nabla_6+\nabla_8\right\rangle
^{O(\alpha)=O(-\alpha)},$ 
$\left\langle \nabla_4+\nabla_5+\nabla_7\right\rangle,  $
$\left\langle \nabla_4+\nabla_5+\nabla_8\right\rangle,$
$\left\langle \nabla_4+\nabla_7 \right\rangle,$
$\left\langle \nabla_4+\nabla_8\right\rangle,$
$\left\langle \nabla_5\right\rangle,$ 
$\left\langle \nabla_5+\nabla_6\right\rangle, $
$\left\langle \nabla_5+\nabla_6+\nabla_8\right\rangle,$
$\left\langle \nabla_5+\nabla_7 \right\rangle,$
$\left\langle \nabla_5+\nabla_8\right\rangle,$ 
$\left\langle \nabla_6+\nabla_8\right\rangle,$
$\left\langle \nabla_7\right\rangle,$ 
$\left\langle \nabla_8\right\rangle,$

\end{center}

which gives the following new algebras:

\begin{longtable}{llllllllllllllllll}
${\mathbf{N}}_{83}^{\alpha, \beta ,\gamma}$ & $:$ &
$e_1e_1=e_2$ & $e_1e_2=e_5$ & $e_1e_3=e_5$  &$e_2e_2=\alpha e_5$ \\ &  & $e_2e_3=e_4$ & $e_2e_4=\beta e_5$ & $e_3e_3=\gamma e_5$ & $e_4e_4= e_5$
\\

${\mathbf{N}}_{84}$ & $:$ &
$e_1e_1=e_2$ & $e_1e_2=e_5$ & $e_1e_4=e_5$ & $e_2e_3=e_4$
\\

${\mathbf{N}}_{85}^{\alpha, \beta}$ & $:$ &
$e_1e_1=e_2$ & $e_1e_2=e_5$ & $e_1e_4=e_5$  &$e_2e_2=\alpha e_5$  \\ & & $e_2e_3=e_4$ & $e_2e_4=\beta e_5$  & $e_3e_4=e_5$ 

\\
${\mathbf{N}}_{86}^{\alpha}$ & $:$ &
$e_1e_1=e_2$ & $e_1e_2=e_5$ & $e_1e_4=e_5$ \\
&&$e_2e_2=e_5$ & $e_2e_3=e_4$ & $e_3e_3=\alpha e_5$ \\

${\mathbf{N}}_{87}$ & $:$ &
$e_1e_1=e_2$ & $e_1e_2=e_5$ & $e_1e_4=e_5$  & $e_2e_3=e_4$ & $e_3e_3= e_5$ 
\\

${\mathbf{N}}_{88}^{\alpha, \beta}$ & $:$ &
$e_1e_1=e_2$ & $e_1e_2=e_5$  &$e_2e_2= e_5$ & $e_2e_3=e_4$  \\ & & $e_2e_4=\alpha e_5$ & $e_3e_3=\beta e_5$  & $e_4e_4=e_5$ 
\\

${\mathbf{N}}_{89}^{\alpha}$ & $:$ &
$e_1e_1=e_2$ & $e_1e_2=e_5$  &$e_2e_2= e_5$ \\
&& $e_2e_3=e_4$ & $e_2e_4=\alpha e_5$ & $e_3e_4= e_5$
\\

${\mathbf{N}}_{90}^{\alpha}$ & $:$ &
$e_1e_1=e_2$ & $e_1e_2=e_5$  & $e_2e_3=e_4$ \\
&& $e_2e_4=e_5$ & $e_3e_3=\alpha e_5$ & $e_4e_4= e_5$
\\

${\mathbf{N}}_{91}$ & $:$ &
$e_1e_1=e_2$ & $e_1e_2=e_5$  & $e_2e_3=e_4$ & $e_2e_4= e_5$ & $e_3e_4= e_5$
\\

${\mathbf{N}}_{92}$ & $:$ &
$e_1e_1=e_2$ & $e_1e_2=e_5$  & $e_2e_3=e_4$ & $e_3e_3= e_5$ & $e_4e_4= e_5$
\\

${\mathbf{N}}_{93}$ & $:$ &
$e_1e_1=e_2$ & $e_1e_2=e_5$  & $e_2e_3=e_4$ & $e_3e_4= e_5$ 
\\

${\mathbf{N}}_{94}$ & $:$ &
$e_1e_1=e_2$ & $e_1e_2=e_5$  & $e_2e_3=e_4$ & $e_4e_4= e_5$ 
\\

${\mathbf{N}}_{95}^{\alpha}$ & $:$ &
$e_1e_1=e_2$ & $e_1e_3=e_5$ & $e_2e_2=e_5$ \\&& $e_2e_3=e_4$ & $e_2e_4= e_5$ & $e_3e_3= \alpha e_5$ 
\\

${\mathbf{N}}_{96}^{\alpha, \beta}$ & $:$ &
$e_1e_1=e_2$ & $e_1e_3=e_5$ & $e_2e_2=e_5$\\
& & $e_2e_3=e_4$ & $e_2e_4= \alpha e_5$ & $e_3e_3= \beta e_5$ & $e_4e_4=e_5$
\\

${\mathbf{N}}_{97}$ & $:$ &
$e_1e_1=e_2$ & $e_1e_3=e_5$  & $e_2e_3=e_4$ & $e_2e_4=  e_5$ 
\\

${\mathbf{N}}_{98}$ & $:$ &
$e_1e_1=e_2$ & $e_1e_3=e_5$  & $e_2e_3=e_4$ & $e_2e_4=  e_5$ & $e_3e_3=  e_5$ 
\\

${\mathbf{N}}_{99}^{\alpha}$ & $:$ &
$e_1e_1=e_2$ & $e_1e_3=e_5$  & $e_2e_3=e_4$ \\&& $e_2e_4=  e_5$ & $e_3e_3= \alpha e_5$ & $e_4e_4=e_5$
\\

${\mathbf{N}}_{100}$ & $:$ &
$e_1e_1=e_2$ & $e_1e_3=e_5$  & $e_2e_3=e_4$ & $e_3e_3=  e_5$  & $e_4e_4=e_5$
\\

${\mathbf{N}}_{101}$ & $:$ &
$e_1e_1=e_2$ & $e_1e_3=e_5$  & $e_2e_3=e_4$ & $e_4e_4=e_5$
\\

${\mathbf{N}}_{102}$ & $:$ &
$e_1e_1=e_2$ & $e_1e_4=e_5$  & $e_2e_3=e_4$ 
\\

${\mathbf{N}}_{103}$ & $:$ &
$e_1e_1=e_2$ & $e_1e_4=e_5$ & $e_2e_2=e_5$ & $e_2e_3=e_4$ 
\\

${\mathbf{N}}_{104}^{\alpha}$ & $:$ &
$e_1e_1=e_2$ & $e_1e_4=e_5$ & $e_2e_2=e_5$ \\
&& $e_2e_3=e_4$ & $e_2e_4=e_5$ & $e_3e_3= \alpha e_5$ 
\\

${\mathbf{N}}_{105}^{\alpha}$ & $:$ &
$e_1e_1=e_2$ & $e_1e_4=e_5$ & $e_2e_2=\alpha e_5$ \\
&& $e_2e_3=e_4$ & $e_2e_4=e_5$ & $e_3e_4= e_5$ 
\\

${\mathbf{N}}_{106}$ & $:$ &
$e_1e_1=e_2$ & $e_1e_4=e_5$ & $e_2e_2=e_5$ & $e_2e_3=e_4$ & $e_3e_3=e_5$  
\\

${\mathbf{N}}_{107}$ & $:$ &
$e_1e_1=e_2$ & $e_1e_4=e_5$  & $e_2e_3=e_4$ & $e_2e_4=e_5$  
\\

${\mathbf{N}}_{108}$ & $:$ &
$e_1e_1=e_2$ & $e_1e_4=e_5$  & $e_2e_3=e_4$ & $e_2e_4=e_5$
& $e_3e_3=e_5$
\\

${\mathbf{N}}_{109}$ & $:$ &
$e_1e_1=e_2$ & $e_1e_4=e_5$  & $e_2e_3=e_4$ & $e_3e_3=e_5$
\\

${\mathbf{N}}_{110}$ & $:$ &
$e_1e_1=e_2$ & $e_2e_2=e_5$  & $e_2e_3=e_4$ & $e_2e_4=e_5$
\\

${\mathbf{N}}_{111}$ & $:$ &
$e_1e_1=e_2$ & $e_2e_2=e_5$  & $e_2e_3=e_4$ & $e_2e_4=e_5$
& $e_3e_3=e_5$
\\

${\mathbf{N}}_{112}^{\alpha}$ & $:$ &
$e_1e_1=e_2$ & $e_2e_2=e_5$  & $e_2e_3=e_4$ \\
&& $e_2e_4=\alpha e_5$
& $e_3e_3=e_5$ & $e_4e_4=e_5$
\\

${\mathbf{N}}_{113}$ & $:$ &
$e_1e_1=e_2$ & $e_2e_2=e_5$  & $e_2e_3=e_4$ & $e_2e_4= e_5$
& $e_3e_4=e_5$
\\

${\mathbf{N}}_{114}$ & $:$ &
$e_1e_1=e_2$ & $e_2e_2=e_5$  & $e_2e_3=e_4$ & $e_2e_4= e_5$
& $e_4e_4=e_5$
\\

${\mathbf{N}}_{115}$ & $:$ &
$e_1e_1=e_2$ & $e_2e_2=e_5$  & $e_2e_3=e_4$ & $e_3e_4=e_5$
\\

${\mathbf{N}}_{116}$ & $:$ &
$e_1e_1=e_2$ & $e_2e_2=e_5$  & $e_2e_3=e_4$ & $e_4e_4=e_5$
\\

${\mathbf{N}}_{117}$ & $:$ &
$e_1e_1=e_2$ & $e_2e_3=e_4$ & $e_2e_4=e_5$
\\

${\mathbf{N}}_{118}$ & $:$ &
$e_1e_1=e_2$ & $e_2e_3=e_4$ & $e_2e_4=e_5$ & $e_3e_3=e_5$
\\

${\mathbf{N}}_{119}$ & $:$ &
$e_1e_1=e_2$ & $e_2e_3=e_4$ & $e_2e_4=e_5$ & $e_3e_3=e_5$
& $e_4e_4=e_5$
\\

${\mathbf{N}}_{120}$ & $:$ &
$e_1e_1=e_2$ & $e_2e_3=e_4$ & $e_2e_4=e_5$ & $e_3e_4=e_5$
\\

${\mathbf{N}}_{121}$ & $:$ &
$e_1e_1=e_2$ & $e_2e_3=e_4$ & $e_2e_4=e_5$ & $e_4e_4=e_5$
\\

${\mathbf{N}}_{122}$ & $:$ &
$e_1e_1=e_2$ & $e_2e_3=e_4$ & $e_3e_3=e_5$ & $e_4e_4=e_5$
\\

${\mathbf{N}}_{123}$ & $:$ &
$e_1e_1=e_2$ & $e_2e_3=e_4$ & $e_3e_4=e_5$
\\

${\mathbf{N}}_{124}$ & $:$ &
$e_1e_1=e_2$ & $e_2e_3=e_4$ & $e_4e_4=e_5$
\\

\end{longtable}

\subsection{$ 1 $-dimensional central extensions of $ {\mathbf N}_{10}^{4*} $.} Here we will collect all information about $ {\mathbf N}_{10}^{4*}: $
$$
\begin{array}{|l|l|l|l|}
%\hline
%\text{ }  & \text{ } & \text{Cohomology} & \text{Automorphisms}\\
\hline
{\mathbf{N}}^{4*}_{10} &  
\begin{array}{l}
e_1e_1=e_2 \\ 
e_1e_2=e_4 \\  
e_3e_3=e_4
\end{array}
&
\begin{array}{l}
\mathrm{H}^2_{\mathfrak{D}}(\mathbf{N}^{4*}_{10})=\\  

\Big\langle 
\begin{array}{c} \relax
[\Delta_{13}],[\Delta_{14}],[\Delta_{22}], \\ \relax
 [\Delta_{23}],[\Delta_{33}]
\end{array}
\Big\rangle\\

\mathrm{H}^2_{\mathfrak{C}}(\mathbf{N}^{4*}_{10})=\mathrm{H}^2_{\mathfrak{D}}(\mathbf{N}^{4*}_{10})\oplus\\
\multicolumn{1}{r}{\langle [\Delta_{24}], [\Delta_{34}], [\Delta_{44}] \rangle}
\end{array} & 
\begin{array}{l} 
{\tiny \phi=\begin{pmatrix}
x&0&0&0\\
y&x^2&-\frac{zr}{x}&0\\
z&0&r&0\\
t&z^2+2xy&s&x^3
\end{pmatrix}} , \\
  \multicolumn{1}{r}{r^2=x^3}
  \end{array}\\
  
\hline

\end{array}$$

Let us use the following notations:
\begin{longtable}{llll}
$\nabla_1=[\Delta_{13}], $&$ \nabla_2=[\Delta_{14}], $&$ \nabla_3=[\Delta_{22}],$&$  \nabla_4=[\Delta_{23}],$\\
$\nabla_5=[\Delta_{24}], $&$ \nabla_6=[\Delta_{33}],  $&$  \nabla_7=[\Delta_{34}], $&$ \nabla_8=[\Delta_{44}]. $
\end{longtable}

Take $ \theta=\sum\limits_{i=1}^{8}\alpha_i\nabla_i\in\mathrm{H}^2_{\mathfrak{C}}(\mathbf{N}^{4*}_{10}) .$  Since
$$\phi^T\begin{pmatrix}
0&0&\alpha_1&\alpha_2\\
0&\alpha_3&\alpha_4&\alpha_5\\
\alpha_1&\alpha_4&\alpha_6&\alpha_7\\
\alpha_2&\alpha_5&\alpha_7&\alpha_8

\end{pmatrix}\phi=
\begin{pmatrix}
\alpha^*&\alpha^{**}&\alpha^{*}_1&\alpha_2^*\\
\alpha^{**}&\alpha_3^{*}&\alpha^*_4&\alpha^*_5\\
\alpha^{*}_1&\alpha^*_4&\alpha^*_6+\alpha^{**}&\alpha^*_7\\
\alpha^*_2&\alpha^*_5&\alpha^*_7&\alpha^*_8
\end{pmatrix},$$

we have

\begin{longtable}{lcl}
$\alpha_1^*$ & $=$& $-(\alpha_3y+\alpha_4z+\alpha_5t)\frac{zr}{x}+(\alpha_1x+\alpha_4y+\alpha_6z+\alpha_7t)r+$\\
&&$(\alpha_2x+\alpha_5y+\alpha_7z+\alpha_8t)s,$\\

$\alpha_2^*$& $=$& $(\alpha_2x+\alpha_5y+\alpha_7z+\alpha_8t)x^3,$ \\
$\alpha_3^*$& $=$& $\alpha_3x^4+2\alpha_5x^2(z^2+2xy)+\alpha_8(z^2+2xy)^2,$\\
$\alpha_4^*$& $=$& $-(\alpha_3x^2+\alpha_5(z^2+2xy))\frac{zr}{x}+(\alpha_4x^2+\alpha_7(z^2+2xy))r+$\\&&$(\alpha_5x^2+\alpha_8(z^2+2xy))s,$\\

$\alpha_5^*$& $=$& $(\alpha_5x^2+\alpha_8(z^2+2xy))x^3,$\\
$\alpha_6^*$& $=$& $-(\alpha_4r-\alpha_3\frac{zr}{x}+\alpha_5s)\frac{zr}{x}+(\alpha_6r-\alpha_4\frac{zr}{x}+\alpha_7s)r+$\\
&&$(\alpha_7r-\alpha_5\frac{zr}{x}+\alpha_8s)s-(\alpha_3y+\alpha_4z+\alpha_5t)x^2-$\\
&& $(\alpha_2x+\alpha_5y+\alpha_7z+\alpha_8t)(z^2+2xy),$\\
$\alpha_7^*$&$=$&$(\alpha_7r-\alpha_5\frac{zr}{x}+\alpha_8s)x^3,$\\
$\alpha_8^*$&$=$&$\alpha_8x^6.$
\end{longtable}

We are interested in $ (\alpha_5,\alpha_7,\alpha_8)\neq(0,0,0) .$ Let us consider the following cases:

\begin{enumerate}
	
	\item $ \alpha_8=0, \alpha_5=0, $ then $ \alpha_7\neq0.$ Now by choosing 
	\begin{center}
	 $  y=-\frac{ \alpha_2^2+\alpha_2\alpha_3+\alpha_4\alpha_7}{2\alpha_7^2}x,$ $z=-\frac{\alpha_2}{\alpha_7}x,$
	$s=-\frac{3\alpha^2_2\alpha_3+\alpha_2(\alpha_3^2+6\alpha_4 \alpha_7)+\alpha_7(\alpha_3\alpha_4+2\alpha_6\alpha_7)}{4\alpha_7^3}\sqrt{x^3},$ 
	$t=\frac{\alpha_7^2(\alpha_4^2-2\alpha_1\alpha_7)+\alpha_2^3\alpha_3+\alpha_2^2(\alpha_3^2+3\alpha_4\alpha_7)+2\alpha_2\alpha_7 (\alpha_3\alpha_4+\alpha_6\alpha_7)}{2\alpha_7^4}x,$
	\end{center}we have $\alpha^*_1=0, \alpha^*_2=0, \alpha^*_4=0, \alpha^*_6=0.$
	Then we have the representatives $ \left\langle \nabla_7 \right\rangle $ or $ \left\langle \nabla_3+\nabla_7 \right\rangle$ depending on whether $\alpha_3=0$ or not.

	\item  $ \alpha_8=0, \alpha_5\neq0,$ then by choosing 
	\begin{center}$  y=-\frac{\alpha_2\alpha_5+\alpha_7^2}{\alpha_5^2}x,$ $z=\frac{\alpha_7}{\alpha_5}x,$  $s=\frac{\alpha_3\alpha_7-\alpha_4\alpha_5}{\alpha^2_5}\sqrt{x^{3}},$ $t=\frac{\alpha_2\alpha_3\alpha_5+\alpha_5^2\alpha_6+3\alpha_4\alpha_5\alpha_7-2\alpha_3\alpha^2_7}{\alpha^3_5}x, $ \end{center} 
	we have $ \alpha^*_2=\alpha^*_4=\alpha^*_6=0 $ and $ \alpha_7^*=0. $ Therefore, we can suppose that $ \alpha_2=0, \alpha_4=0, \alpha_6=0, \alpha_7=0, $ and have the following cases:

\begin{enumerate}
	\item if $ \alpha_1=0, \alpha_3=0, $ then we have the representative $ \left\langle \nabla_5 \right\rangle; $
	
	\item if $ \alpha_1=0, \alpha_3\neq0, $ then by choosing $ x=\frac{\alpha_3}{\alpha_5}, r^2=x^3, $ we have the representative $ \left\langle \nabla_3+\nabla_5 \right\rangle; $
	
	\item if $ \alpha_1\neq0, $ then by choosing $ x=\sqrt[5]{\frac{\alpha^2_1}{\alpha^2_5}}, r^2=x^3, $ we have the family of representatives $ \left\langle \nabla_1+\alpha\nabla_3+\nabla_5 \right\rangle. $	
	
\end{enumerate}

	\item  $ \alpha_8\neq0,$ then by choosing $ y=-\frac{\alpha_5x^2+\alpha_8z^2}{2\alpha_8x},  s=\frac{\sqrt{x}(\alpha_5z-\alpha_7x)}{\alpha_8}, t=-\frac{\alpha_2x+\alpha_5y+\alpha_7z}{\alpha_8},$ we have $ \alpha^*_2=\alpha^*_5=\alpha^*_7=0 . $ Therefore, we can suppose that $ \alpha_2=0, \alpha_5=0, \alpha_7=0, $ and have the following cases:
	
			\begin{enumerate}
		\item if $ \alpha_3=0, \alpha_4=0, \alpha_6=0,$ then we have the representative $ \left\langle \nabla_8 \right\rangle$ and $ \left\langle \nabla_1+\nabla_8 \right\rangle $ depending on whether $\alpha_1=0$ or not.
		
		\item if $ \alpha_3=0, \alpha_4=0, \alpha_6\neq0, $ then by choosing $ x=\sqrt[3]{\frac{\alpha_6}{\alpha_8}}, r^2=x^3, z=-\frac{\alpha_1}{\sqrt[3]{\alpha^2_6\alpha_8}},$ we have the representative $ \left\langle \nabla_6+\nabla_8 \right\rangle; $	
		
		\item if $ \alpha_3=0, \alpha_4\neq0, $ then by choosing $ x=\sqrt[5]{\frac{\alpha^2_4}{\alpha^2_8}}, r^2=x^3, z=\frac{\alpha_6}{3\sqrt[5]{\alpha_4^3\alpha^2_8}},$ we have the representative $ \left\langle \alpha\nabla_1+\nabla_4+\nabla_8 \right\rangle; $
		
		\item if $ \alpha_3\neq0, $ then by choosing $ x=\sqrt{\frac{\alpha_3}{\alpha_8}}, r^2=x^3,  z=\frac{\alpha_4}{\sqrt{\alpha_3\alpha_8}},  $ we have the representative $ \left\langle \alpha\nabla_1+\nabla_3+\beta\nabla_6+\nabla_8 \right\rangle. $	
		
	\end{enumerate}	
	
\end{enumerate}

Summarizing, we have the following distinct orbits:
\begin{center}

$\left\langle \nabla_1+\alpha \nabla_3+\nabla_5  \right\rangle
^{O(\alpha)=O(\eta^4_5\alpha)=O(-\eta^3_5\alpha)=O(\eta^2_5\alpha)=O(-\eta_5\alpha)}, \,$

$\left\langle \alpha\nabla_1+ \nabla_3+\beta \nabla_6 + \nabla_8  \right\rangle^{{\tiny \begin{array}{l}
O(\alpha,\beta)=O(-\alpha,\beta)= O(\eta_3\alpha,\eta_3^2\beta)=\\
O(-\eta_3\alpha,\eta_3^2\beta)=O(-\eta^2_3\alpha,-\eta_3\beta)=O(\eta_3^2\alpha,-\eta_3\beta)
\end{array}}},$

$\left\langle \alpha\nabla_1+ \nabla_4+\nabla_8\right\rangle^
{{\tiny \begin{array}{l}
O(\alpha)=O(-\alpha)=O(\eta_5^4\alpha)=O(-\eta_5^4\alpha)=O(\eta_5^3\alpha)=\\O(-\eta_5^3\alpha)=O(\eta_5^2\alpha)=O(-\eta_5^2\alpha)=O(\eta_5\alpha)=O(-\eta_5\alpha) 
\end{array}}}, \, $

$\left\langle \nabla_1+ \nabla_8  \right\rangle,$
$\left\langle \nabla_3+ \nabla_5  \right\rangle, $
$\left\langle \nabla_3+\nabla_7  \right\rangle,$ 
$\left\langle \nabla_5 \right\rangle,$ 
$\left\langle \nabla_6+\nabla_8 \right\rangle,$ 
$\left\langle \nabla_7\right\rangle,$ 
$\left\langle \nabla_8\right\rangle,$

\end{center}
which gives the following new algebras:

\begin{longtable}{llllllllllllllllll}

${\mathbf{N}}^{\alpha}_{125}$ & $:$ &$e_1e_1=e_2$ & $e_1e_2=e_4$ & $e_1e_3= e_5$ \\ &  & $e_2e_2=\alpha e_5$ & $e_2e_4= e_5$ & $e_3e_3=e_4$\\

${\mathbf{N}}^{\alpha,\beta}_{126}$ & $:$ &$e_1e_1=e_2$ & $e_1e_2=e_4$ & $e_1e_3= \alpha e_5$ \\ && $e_2e_2= e_5$ &  $e_3e_3=e_4 + \beta e_5$ & $e_4e_4= e_5$ \\

${\mathbf{N}}^{\alpha}_{127}$ & $:$ &$e_1e_1=e_2$ & $e_1e_2=e_4$ & $e_1e_3= \alpha e_5$ \\ && $e_2e_3= e_5$ &  $e_3e_3=e_4$ & $e_4e_4= e_5$ \\

${\mathbf{N}}_{128}$ & $:$ &$e_1e_1=e_2$ & $e_1e_2=e_4$ & $e_1e_3=  e_5$ &  $e_3e_3=e_4$ & $e_4e_4= e_5$ \\

${\mathbf{N}}_{129}$ & $:$ &$e_1e_1=e_2$ & $e_1e_2=e_4$ & $e_2e_2=  e_5$ &  $e_2e_4=e_5$ & $e_3e_3= e_4$ \\

${\mathbf{N}}_{130}$ & $:$ &$e_1e_1=e_2$ & $e_1e_2=e_4$ & $e_2e_2=  e_5$  & $e_3e_3= e_4$ &  $e_3e_4=e_5$\\

${\mathbf{N}}_{131}$ & $:$ &$e_1e_1=e_2$ & $e_1e_2=e_4$ & $e_2e_4=e_5$ & $e_3e_3= e_4$ \\

${\mathbf{N}}_{132}$ & $:$ &$e_1e_1=e_2$ & $e_1e_2=e_4$ & $e_3e_3= e_4+e_5$ & $e_4e_4= e_5$\\

${\mathbf{N}}_{133}$ & $:$ &$e_1e_1=e_2$ & $e_1e_2=e_4$ & $e_3e_3= e_4$ & $e_3e_4= e_5$\\

${\mathbf{N}}_{134}$ & $:$ &$e_1e_1=e_2$ & $e_1e_2=e_4$ & $e_3e_3= e_4$ & $e_4e_4= e_5$\\
\end{longtable}

\subsection{$ 1 $-dimensional central extensions of $ {\mathbf N}_{11}^{4*} $.} Here we will collect all information about $ {\mathbf N}_{11}^{4*}: $
$$
\begin{array}{|l|l|l|l|}
%\hline
%\text{ }  & \text{ } & \text{Cohomology} & \text{Automorphisms}\\
\hline
{\mathbf{N}}^{4*}_{11} &  
\begin{array}{l}
e_1e_1=e_2 \\ 
e_1e_3=e_4 \\  
e_2e_2=e_4
\end{array}
&
\begin{array}{lcl}
\mathrm{H}^2_{\mathfrak{D}}(\mathbf{N}^{4*}_{11}) & =& \\
\multicolumn{3}{r}{\langle [\Delta_{12}],[\Delta_{22}],[\Delta_{23}],[\Delta_{33}]\rangle }\\
\mathrm{H}^2_{\mathfrak{C}}(\mathbf{N}^{4*}_{11})&=&\mathrm{H}^2_{\mathfrak{D}}(\mathbf{N}^{4*}_{11})\oplus\\
\multicolumn{3}{r}{\langle [\Delta_{14}], [\Delta_{24}], [\Delta_{34}], [\Delta_{44}]\rangle}
\end{array}
& \phi=\begin{pmatrix}
x&0&0&0\\
0&x^2&0&0\\
z&0&x^3&0\\
t&2xz&s&x^4
\end{pmatrix}\\
\hline
\end{array}$$

Let us use the following notations:
\begin{longtable}{llll}
$\nabla_1=[\Delta_{12}], $&$\nabla_2=[\Delta_{14}], $&$ \nabla_3=[\Delta_{22}], $&$  \nabla_4=[\Delta_{23}],$\\
$\nabla_5=[\Delta_{24}],  $&$ \nabla_6=[\Delta_{33}],  $&$  \nabla_7=[\Delta_{34}], $&$ \nabla_8=[\Delta_{44}].$
\end{longtable}
Take $ \theta=\sum\limits_{i=1}^{8}\alpha_i\nabla_i\in\mathrm{H}^2_{\mathfrak{C}}(\mathbf{N}^{4*}_{11}) .$  Since
$$\phi^T\begin{pmatrix}
0&\alpha_1&0&\alpha_2\\
\alpha_1&\alpha_3&\alpha_4&\alpha_5\\
0&\alpha_4&\alpha_6&\alpha_7\\
\alpha_2&\alpha_5&\alpha_7&\alpha_8

\end{pmatrix}\phi=
\begin{pmatrix}
\alpha^*&\alpha_1^{*}&\alpha^{**}&\alpha_2^*\\
\alpha_1^{*}&\alpha_3^{*}+\alpha^{**}&\alpha^*_4&\alpha^*_5\\
\alpha^{**}&\alpha^*_4&\alpha^*_6&\alpha^*_7\\
\alpha^*_2&\alpha^*_5&\alpha^*_7&\alpha^*_8
\end{pmatrix},$$
we have
\begin{longtable}{lcl}
$\alpha_1^*$&$=$&$(\alpha_1x+\alpha_4z+\alpha_5t)x^2+2(\alpha_2x+\alpha_7z+\alpha_8t)xz,$\\
$\alpha_2^*$&$=$&$(\alpha_2x+\alpha_7z+\alpha_8t)x^4,$\\

$\alpha_3^*$&$=$&$(\alpha_3x^2+4\alpha_5xz+4\alpha_8z^2)x^{2}-(\alpha_6z+\alpha_7t)x^3-(\alpha_2x+\alpha_7z+\alpha_8t)s,$\\

$\alpha_4^*$&$=$&$(\alpha_4x+2\alpha_7z)x^4+(\alpha_5x+2\alpha_8z)xs,$\\
$\alpha_5^*$&$=$&$(\alpha_5x+2\alpha_8z)x^5,$\\
$\alpha_6^*$&$=$&$\alpha_6x^6+2\alpha_7x^3s+\alpha_8s^2,$\\
$\alpha_7^*$&$=$&$(\alpha_7x^3+\alpha_8s)x^4,$\\
$\alpha_8^*$&$=$&$\alpha_8x^8.$
\end{longtable}

We are interested in $ (\alpha_2,\alpha_5,\alpha_7,\alpha_8)\neq(0,0,0,0) .$ Let us consider the following cases:

\begin{enumerate}
	\item  $ \alpha_8=0, \alpha_7=0, \alpha_5=0 ,$ then $ \alpha_2\neq0 $ and we have

\begin{enumerate}
	\item\label{cs1} if $ \alpha_4=0, \alpha_6=0, $ then by choosing $ x=2\alpha_2, z=-\alpha_1, s=8\alpha_2^2\alpha_3, t=0,$ we have the representative $ \left\langle \nabla_2 \right\rangle; $
	
	\item\label{css1} if $ \alpha_4=0, \alpha_6\neq0, $ then by choosing $ x=\frac{\alpha_2}{\alpha_6}, z=-\frac{\alpha_1}{2\alpha_6}, s=\frac{\alpha_1\alpha_2\alpha_6+2\alpha_2^2\alpha_3}{2\alpha_6^3}, t=0,$ we have the representative $ \left\langle \nabla_2+\nabla_6 \right\rangle; $
	
	\item if $ \alpha_4\neq0, \alpha_4=-2\alpha_2, \alpha_1\neq0,$ then by choosing 
	\begin{center}$ x=\sqrt{\frac{\alpha_1}{\alpha_2}}, z=0, s=\frac{\alpha_1\alpha_3\sqrt{\alpha_1}}{\alpha_2^2\sqrt{\alpha_2}}, t=0,$
	\end{center} we have the family of representatives $ \left\langle \nabla_1+\nabla_2-2\nabla_4+\alpha\nabla_6 \right\rangle; $	
	
	\item\label{css2} if $ \alpha_4\neq0, \alpha_4=-2\alpha_2, \alpha_1=0, \alpha_6\neq 0,$ then
	by choosing 
	$x=\alpha_2\alpha_6^{-1},$
	$z=0,$
	$s=\alpha_2^2\alpha_3\alpha_6^{-3},$
	$t=0,$
	we have the representative $ \left\langle \nabla_2-2\nabla_4  +\nabla_6 \right\rangle;$ 
	
		\item\label{cs2} if $ \alpha_4\neq0, \alpha_4=-2\alpha_2, \alpha_1=0,\alpha_6=0,$ then by choosing 
			$x=\alpha_2,$
	$z=0,$
	$s=\alpha_2^2\alpha_3,$
	$t=0,$
		we have the representative $ \left\langle \nabla_2-2\nabla_4 \right\rangle;$

	\item if $ \alpha_4\neq0, \alpha_4\neq-2\alpha_2, \alpha_6=0,$ then by choosing \begin{center}$ x=\alpha_4+2\alpha_2, z=-\alpha_1, s=\frac{\alpha_3(\alpha_4+2\alpha_2)^3}{\alpha_2}, t=0,$ 
	\end{center}
	we have the the family of representatives $ \left\langle \nabla_2+\alpha\nabla_4 \right\rangle_{\alpha\neq0,-2},$
	which will be jointed with the cases (\ref{cs1}) and (\ref{cs2});
	
	\item if $ \alpha_4\neq0, \alpha_4\neq-2\alpha_2, \alpha_6\neq0,$ then by choosing 
	\begin{center}
	   $ x=\frac{\alpha_2}{\alpha_6}, z=-\frac{\alpha_1\alpha_2}{\alpha_6(\alpha_4+2\alpha_2)}, s=\frac{\alpha_1\alpha^2_2\alpha_6+2\alpha_2^3\alpha_3+\alpha_2^2\alpha_3\alpha_4}{\alpha_6^3(\alpha_4+2\alpha_2)}, t=0,$ 	   
	\end{center} we have the family of representatives $ \left\langle \nabla_2+\alpha\nabla_4+\nabla_6 \right\rangle_{\alpha\neq0,-2}, $	
				which will be jointed with the cases (\ref{css1}) and (\ref{css2}).

\end{enumerate}

	\item  $ \alpha_8=0, \alpha_7=0, \alpha_5\neq0 ,$ then we have
	
	\begin{enumerate}
		\item\label{11ccs1} if $ \alpha_6=0, \alpha_2=0, $ then by choosing 
		\begin{center}$ x=4\alpha_5, z=-\alpha_3, s=-64\alpha_4\alpha_5^2,  t=\frac{\alpha_3\alpha_4-4\alpha_1\alpha_5}{\alpha_5} $
		\end{center} we have the representative $ \left\langle \nabla_5 \right\rangle; $
		
		\item\label{11cs1} if $ \alpha_6=0, \alpha_2\neq0, $ then by choosing \begin{center}
		   $x=\frac{\alpha_2}{\alpha_5},$ $z=-\frac{\alpha_2^2\alpha_4+\alpha_2\alpha_3\alpha_5}{4\alpha_5^3},$ $s=-\frac{\alpha_2^3\alpha_4}{\alpha_5^4},$ $t=\frac{(\alpha_2\alpha_4+2\alpha^2_2)(\alpha_2\alpha_4+\alpha_3\alpha_5)-4\alpha_1\alpha_2\alpha_5^2}{4\alpha_5^4},$
		  		\end{center}we have the representative $ \left\langle \nabla_2+\nabla_5 \right\rangle; $
		
		\item\label{11ccs2} if $ \alpha_6\neq0, \alpha_6=4\alpha_5, \alpha_3=0, \alpha_2=0, $ then by choosing $ x=\alpha_5, z=0, s=-\alpha_4\alpha_5^2, t=-\alpha_1, $ we have the representative $ \left\langle \nabla_5+4\nabla_6 \right\rangle; $		
		
		\item\label{11cs2} if $ \alpha_6\neq0, \alpha_6=4\alpha_5, \alpha_2\alpha_4+\alpha_3\alpha_5=0, \alpha_2\neq0, $ then by choosing 
		\begin{center}$ x=\frac{\alpha_2}{\alpha_5}, z=0, s=-\frac{\alpha_2^3\alpha_4}{\alpha_5^4}, t=-\frac{\alpha_1\alpha_2}{\alpha_5^2}, $  \end{center} we have the representative $ \left\langle \nabla_2+\nabla_5+4\nabla_6 \right\rangle; $
		
		\item if $ \alpha_6\neq0, \alpha_6=4\alpha_5, \alpha_2\alpha_4+\alpha_3\alpha_5\neq0, $ then by choosing 
		\begin{center}
		    $ x=\frac{\sqrt{\alpha_2\alpha_4+\alpha_3\alpha_5}}{\alpha_5}, z=0,$ $s=-\frac{\alpha_4\sqrt{(\alpha_2\alpha_4+\alpha_3\alpha_5)^3}}{\alpha^4_5},$ $t=-\frac{\alpha_1\sqrt{\alpha_2\alpha_4+\alpha_3\alpha_5}}{\alpha^2_5}, $	    
		\end{center} we have the family of representatives $ \left\langle \alpha\nabla_2+\nabla_3+\nabla_5+4\nabla_6 \right\rangle;  $
		
		\item if $ \alpha_6\neq0, \alpha_6\neq4\alpha_5, \alpha_2=0, $ then by choosing 
		\begin{center}
		   $ x=\alpha_6-4\alpha_5, z=\alpha_3,$ $s=\frac{\alpha_4(4\alpha_5-\alpha_6)^3}{\alpha_5},$ $t=\frac{4\alpha_1\alpha_5-\alpha_1\alpha_6-\alpha_3\alpha_4}{\alpha_5},$	
		\end{center}we have the family of representatives $ \left\langle \nabla_5+\alpha\nabla_6 \right\rangle_{\alpha\neq0,4}, $
		which will be jointed with the cases (\ref{11ccs1}) and (\ref{11ccs2});
		
		\item if $ \alpha_6\neq0, \alpha_6\neq4\alpha_5, \alpha_2\neq0, $ then by choosing 
		\begin{center}
		    $ x=\frac{\alpha_2}{\alpha_5}, z=\frac{\alpha_2 (\alpha_2 \alpha_4+\alpha_3 \alpha_5)}{\alpha_5^2\alpha_6-4\alpha_5^2},$ $s=-\frac{\alpha_4\alpha_2^3}{\alpha^4_5}, $
		   $ t=\frac{\alpha_2 (2 \alpha_2^2 \alpha_4+\alpha_3 \alpha_4 \alpha_5+\alpha_2 (\alpha_4^2+2 \alpha_3 \alpha_5)-\alpha_1 \alpha_5 (4 \alpha_5-\alpha_6))}{\alpha_5^3 (4 \alpha_5-\alpha_6)}, $ 
		\end{center}
		we have the family of representatives $ \left\langle \nabla_2+\nabla_5+\alpha\nabla_6 \right\rangle_{(\alpha\neq0,4)},  $			which will be jointed with the cases (\ref{11cs1}) and (\ref{11cs2}).

	\end{enumerate}

	\item $ \alpha_8=0, \alpha_7\neq0 ,$ then by choosing $ z=-\frac{\alpha_2}{\alpha_7}x,$ 
	$s=-\frac{\alpha_6}{2\alpha_7}x^3, $
	$t=\frac{ \alpha_2 (\alpha_6-4 \alpha_5)+\alpha_3 \alpha_7}{\alpha^2_7}x,$ we have $ \alpha_2^*=\alpha_3^*=\alpha_6^*=0 .$ Now we can suppose that $ \alpha_2=0, \alpha_3=0, \alpha_6=0 $ and have the following subcases:
	
	\begin{enumerate}
		\item if $ \alpha_1=0, \alpha_4=0, \alpha_5=0, $ then we have the representative $ \left\langle \nabla_7 \right\rangle; $
		
		\item if $ \alpha_1=0, \alpha_4=0, \alpha_5\neq0, $ then by choosing $ x=\frac{\alpha_5}{\alpha_7}, z=0, s=0, t=0, $ we have the representative $ \left\langle \nabla_5+\nabla_7 \right\rangle;$

		\item if $ \alpha_1=0, \alpha_4\neq0, $ then by choosing $ x=\sqrt{\frac{\alpha_4}{\alpha_7}}, z=0, s=0, t=0,$ we have the family of representatives $ \left\langle \nabla_4+\alpha\nabla_5+\nabla_7 \right\rangle; $
		
        \item if $ \alpha_1\neq0, $ then by choosing $ x=\sqrt[4]{\frac{\alpha_1}{\alpha_7}}, z=0, s=0, t=0,$ we have the family of representatives $ \left\langle \nabla_1+\alpha\nabla_4+\beta\nabla_5+\nabla_7 \right\rangle. $		
						
	\end{enumerate}

	\item  $ \alpha_8\neq0 ,$ then by choosing $ z=-\frac{\alpha_5}{2\alpha_8}x, s=-\frac{\alpha_7}{\alpha_8}x^3, t=\frac{\alpha_5\alpha_7-2\alpha_2\alpha_8}{2\alpha^2_8}x,$ we have $ \alpha_2^*=\alpha_5^*=\alpha_7^*=0 .$ Now we can suppose that $ \alpha_2=0, \alpha_5=0, \alpha_7=0 $ and have the following subcases:

\begin{enumerate}
	\item if $ \alpha_1=0, \alpha_3=0, \alpha_4=0, \alpha_6=0, $ then we have the representative $ \left\langle \nabla_8 \right\rangle; $
	
	\item if $ \alpha_1=0, \alpha_3=0, \alpha_4=0, \alpha_6\neq0, $ then by choosing $ x=\sqrt{\frac{\alpha_6}{\alpha_8}}, z=0, s=0, t=0,$ we have the representative $ \left\langle \nabla_6+\nabla_8 \right\rangle; $

	\item if $ \alpha_1=0, \alpha_3=0, \alpha_4\neq0, $ then by choosing $ x=\sqrt[3]{\frac{\alpha_4}{\alpha_8}}, z=0, s=0, t=0,$ we have the family of representatives $ \left\langle \nabla_4+\alpha\nabla_6+\nabla_8 \right\rangle;$
	
	\item if $ \alpha_1=0, \alpha_3\neq0, $ then by choosing $ x=\sqrt[4]{\frac{\alpha_3}{\alpha_8}}, z=0, s=0, t=0, $ we have the family of representatives $ \left\langle \nabla_3+\alpha\nabla_4+\beta\nabla_6+\nabla_8 \right\rangle;$
	
	\item if $ \alpha_1\neq0, $ then by choosing $ x=\sqrt[5]{\frac{\alpha_1}{\alpha_8}}, z=0, s=0, t=0, $ we have the family of representatives $ \left\langle \nabla_1+\alpha\nabla_3+\beta\nabla_4+\gamma\nabla_6+\nabla_8 \right\rangle. $			
	
\end{enumerate}

\end{enumerate}

Summarizing, we have the following distinct orbits:
\begin{center}

$\left\langle \nabla_1+ \nabla_2 - 2 \nabla_4 + \alpha\nabla_6\right\rangle^{O(\alpha)=O(-\alpha)},$ 
$\left\langle \nabla_1+ \alpha\nabla_3 +\beta \nabla_4+\gamma \nabla_6 + \nabla_8  \right\rangle^{
{\tiny \begin{array}{l}
O(\alpha,\beta,\gamma)=O(-\eta_5\alpha,\eta_5^2\beta,-\eta_5^3\gamma)=
O(\eta_5^2\alpha,\eta_5^4\beta,-\eta_5\gamma)=\\
O(-\eta_5^3\alpha,-\eta_5\beta,\eta_5^4\gamma)=
O(\eta_5^4\alpha,-\eta_5^3\beta,\eta_5^2\gamma)\end{array}}},$
$\left\langle \nabla_1+\alpha\nabla_4+\beta\nabla_5+\nabla_7 \right\rangle
^{O(\alpha,\beta)=O(\alpha,-\beta)=O(-\alpha,-i\beta)=O(-\alpha,i\beta)},$   
$\left\langle \alpha\nabla_2+ \nabla_3+\nabla_5 + 4\nabla_6 \right\rangle^{O(\alpha)=O(-\alpha)},$
$\left\langle \nabla_2+ \alpha\nabla_4 \right\rangle,$ 
$\left\langle \nabla_2+ \alpha\nabla_4 + \nabla_6 \right\rangle,$ 
$\left\langle \nabla_2+ \nabla_5 + \alpha\nabla_6 \right\rangle,$ $\left\langle \nabla_3+ \alpha\nabla_4+\beta\nabla_6 + \nabla_8 \right\rangle,$
  $\left\langle \nabla_4+ \alpha\nabla_5+\nabla_7 \right\rangle,$
$\left\langle \nabla_4+ \alpha\nabla_6+\nabla_8 \right\rangle,$  $\left\langle \nabla_5+ \alpha\nabla_6 \right\rangle,$
$\left\langle \nabla_5+ \nabla_7 \right\rangle,$ 
$\left\langle \nabla_6+ \nabla_8 \right\rangle,$
$\left\langle \nabla_7 \right\rangle,$
$\left\langle \nabla_8 \right\rangle,$
\end{center}
which gives the following new algebras:

\begin{longtable}{llllllllllllllllll}

${\mathbf{N}}_{135}^{\alpha}$ & $:$ &$e_1e_1=e_2$ & $e_1e_2=e_5$ & $e_1e_3=e_4$ & $e_1e_4=e_5$ \\ & & $e_2e_2=e_4$ & $e_2e_3=-2e_5$
  & $e_3e_3=\alpha e_5$
\\

${\mathbf{N}}_{136}^{\alpha, \beta ,\gamma}$ & $:$ &$e_1e_1=e_2$ & $e_1e_2=e_5$ & $e_1e_3=e_4$ & $e_2e_2=e_4+\alpha e_5$ \\ & & $e_2e_3=\beta e_5$ & $e_3e_3=\gamma e_5$
  & $e_4e_4=e_5$
\\

${\mathbf{N}}_{137}^{\alpha, \beta}$ & $:$ &$e_1e_1=e_2$ & $e_1e_2=e_5$ & $e_1e_3=e_4$ & $e_2e_2=e_4$ 
\\ & & $e_2e_3=\alpha e_5$ & $e_2e_4=\beta e_5$ & $e_3e_4=e_5$
\\

${\mathbf{N}}_{138}^{\alpha}$ & $:$ &$e_1e_1=e_2$ & $e_1e_3=e_4$ & $e_1e_4=\alpha e_5$ \\ && $e_2e_2=e_4+e_5$ & $e_2e_4=e_5$ & $e_3e_3=4e_5$ \\

${\mathbf{N}}_{139}^{\alpha}$ & $:$ &$e_1e_1=e_2$ & $e_1e_3=e_4$ & $e_1e_4=e_5$\\ & & $e_2e_2=e_4$ & $e_2e_3=\alpha e_5$ \\

${\mathbf{N}}_{140}^{\alpha}$ & $:$ &$e_1e_1=e_2$ & $e_1e_3=e_4$ & $e_1e_4=e_5$ \\ && $e_2e_2=e_4$ & $e_2e_3=\alpha e_5$ & $e_3e_3=e_5$ \\

${\mathbf{N}}_{141}^{\alpha}$ & $:$ &$e_1e_1=e_2$ & $e_1e_3=e_4$ & $e_1e_4=e_5$ \\& & $e_2e_2=e_4$ & $e_2e_4=e_5$  & $e_3e_3=\alpha e_5$ \\

${\mathbf{N}}_{142}^{\alpha, \beta}$ & $:$ &$e_1e_1=e_2$ & $e_1e_3=e_4$ & $e_2e_2=e_4+e_5$ \\ & & $e_2e_3=\alpha e_5$ & $e_3e_3=\beta e_5$  & $e_4e_4=e_5$ \\

${\mathbf{N}}_{143}^{\alpha}$ & $:$ &$e_1e_1=e_2$ & $e_1e_3=e_4$ & $e_2e_2=e_4$ \\ & & $e_2e_3=e_5$ & $e_2e_4=\alpha e_5$  & $e_3e_4=e_5$ \\

${\mathbf{N}}_{144}^{\alpha}$ & $:$ &$e_1e_1=e_2$ & $e_1e_3=e_4$ & $e_2e_2=e_4$ \\ & & $e_2e_3=e_5$ & $e_3e_3=\alpha e_5$  & $e_4e_4=e_5$ \\

${\mathbf{N}}_{145}^{\alpha}$ & $:$ &$e_1e_1=e_2$ & $e_1e_3=e_4$ & $e_2e_2=e_4$ \\ & & $e_2e_4=e_5$ & $e_3e_3=\alpha e_5$ \\

${\mathbf{N}}_{146}$ & $:$ &$e_1e_1=e_2$ & $e_1e_3=e_4$ & $e_2e_2=e_4$ \\ & & $e_2e_4=e_5$ & $e_3e_4=e_5$ \\

${\mathbf{N}}_{147}$ & $:$ &$e_1e_1=e_2$ & $e_1e_3=e_4$ & $e_2e_2=e_4$ \\ &  & $e_3e_3=e_5$ & $e_4e_4=e_5$ \\

${\mathbf{N}}_{148}$ & $:$ &$e_1e_1=e_2$ & $e_1e_3=e_4$ & $e_2e_2=e_4$ & $e_3e_4=e_5$ \\

${\mathbf{N}}_{149}$ & $:$ &$e_1e_1=e_2$ & $e_1e_3=e_4$ & $e_2e_2=e_4$ & $e_4e_4=e_5$ \\

\end{longtable}

\subsection{$ 1 $-dimensional central extensions of $ {\mathbf N}_{12}^{4*} $.} Here we will collect all information about $ {\mathbf N}_{12}^{4*}: $
$$
\begin{array}{|l|l|l|l|}
%\hline
%\text{ }  & \text{ } & \text{Cohomology} &\text{Automorphisms} \\
\hline
{\mathbf{N}}^{4*}_{12} & 
\begin{array}{l}
e_1e_1=e_2 \\ 
e_2e_2=e_4 \\  
e_3e_3=e_4
\end{array}
&
\begin{array}{lcl}
\mathrm{H}^2_{\mathfrak{D}}(\mathbf{N}^{4*}_{12})&=&\\
\multicolumn{3}{r}{\langle [\Delta_{12}],[\Delta_{13}],[\Delta_{23}],[\Delta_{33}]\rangle}\\
\mathrm{H}^2_{\mathfrak{C}}(\mathbf{N}^{4*}_{12})&=&\mathrm{H}^2_{\mathfrak{D}}(\mathbf{N}^{4*}_{12})\oplus\\
\multicolumn{3}{r}{\langle [\Delta_{14}], [\Delta_{24}], [\Delta_{34}], [\Delta_{44}] \rangle }
\end{array} & \phi_{\pm}=\begin{pmatrix}
x&0&0&0\\
0&x^2&0&0\\
0&0& \pm x^2&0\\
t&0&s&x^4
\end{pmatrix}\\
\hline
\end{array}$$

Let us use the following notations:
\begin{longtable}{llll}
$\nabla_1=[\Delta_{12}], $&$ \nabla_2=[\Delta_{13}], 
$&$ \nabla_3=[\Delta_{14}], $&$  \nabla_4=[\Delta_{23}],$\\
$\nabla_5=[\Delta_{24}], $&$ \nabla_6=[\Delta_{33}],  $&$  
\nabla_7=[\Delta_{34}], $&$ \nabla_8=[\Delta_{44}]. $
\end{longtable}
Take $ \theta=\sum\limits_{i=1}^{8}\alpha_i\nabla_i\in\mathrm{H}^2_{\mathfrak{C}}(\mathbf{N}^{4*}_{12}) .$  Since

$$\phi_{\pm}^T\begin{pmatrix}
0&\alpha_1&\alpha_2&\alpha_3\\
\alpha_1&0&\alpha_4&\alpha_5\\
\alpha_2&\alpha_4&\alpha_6&\alpha_7\\
\alpha_3&\alpha_5&\alpha_7&\alpha_8

\end{pmatrix}\phi_{\pm}=
\begin{pmatrix}
\alpha^*&\alpha_1^{*}&\alpha^{*}_2&\alpha_3^*\\
\alpha_1^{*}&0&\alpha^*_4&\alpha^*_5\\
\alpha^{*}_2&\alpha^*_4&\alpha^*_6&\alpha^*_7\\
\alpha^*_3&\alpha^*_5&\alpha^*_7&\alpha^*_8
\end{pmatrix},$$

we have
\begin{longtable}{ll}
$\alpha_1^*=(\alpha_1x+\alpha_5t)x^2,$&
$\alpha_2^*=(\alpha_3x+\alpha_8t)s\pm(\alpha_2x+\alpha_7t)x^2,$\\
$\alpha_3^*=(\alpha_3x+\alpha_8t)x^4,$&
$\alpha_4^*=(\alpha_5s\pm\alpha_4x^2)x^2,$\\
$\alpha_5^*=\alpha_5x^6,$&
$\alpha_6^*=\alpha_6x^4\pm2\alpha_7sx^2+\alpha_8s^2,$\\
$\alpha_7^*=(\alpha_8s\pm\alpha_7x^2)x^4,$&
$\alpha_8^*=\alpha_8x^8.$
\end{longtable}
We will consider only the action of $\phi_+$ for find representatives and after that we will see that the set of our representatives gives distinct orbits under action of $\phi_+$ and $\phi_-.$ 
We are interested in $ (\alpha_3,\alpha_5,\alpha_7,\alpha_8)\neq(0,0,0,0) .$ Let us consider the following cases:

\begin{enumerate}
	\item  $ \alpha_8=0, \alpha_5=0, \alpha_7=0, $ then $ \alpha_3\neq0 $ and we have the following subcases:
	
	\begin{enumerate}
		\item if $ \alpha_1=0, \alpha_4=0, \alpha_6=0, $ then by choosing $ x=\alpha_3, s=-\alpha_2\alpha_3, t=0, $ we have the representative $ \left\langle \nabla_3 \right\rangle; $
		
		\item if $ \alpha_1=0, \alpha_4=0, \alpha_6\neq0, $ then by choosing $ x=\frac{\alpha_6}{\alpha_3}, s=-\frac{\alpha_2\alpha_6^2}{\alpha_3^3}, t=0, $ we have the representative $ \left\langle \nabla_3+\nabla_6 \right\rangle; $
		
		\item if $ \alpha_1=0, \alpha_4\neq0, $ then by choosing $ x=\frac{\alpha_4}{\alpha_3}, s=-\frac{\alpha_2\alpha^2_4}{\alpha_3^3}, t=0, $ we have the representative $ \left\langle \nabla_3+\nabla_4+\alpha\nabla_6
		 \right\rangle; $				
		
		\item if $ \alpha_1\neq0, $ then by choosing $ x=\sqrt{\frac{\alpha_1}{\alpha_3}}, s=-\frac{\alpha_1\alpha_2}{\alpha_3^2}, t=0, $ we have the representative $ \left\langle \nabla_1+\nabla_3+\alpha\nabla_4+\beta\nabla_6 \right\rangle. $

	\end{enumerate}

	\item  $ \alpha_8=0, \alpha_5=0, \alpha_7\neq0, $ then we have the following subcases:
	
	\begin{enumerate}
		\item if $ \alpha_1=0, \alpha_3=0, \alpha_4=0, $ then by choosing $ x=1, s=-\frac{\alpha_6}{2\alpha_7}, t=\frac{\alpha_3\alpha_6-2\alpha_2\alpha_7}{2\alpha_7^2},$ we have the representative $ \left\langle \nabla_7 \right\rangle; $
		
		\item if $ \alpha_1=0, \alpha_3=0, \alpha_4\neq0, $ then by choosing $ x=\sqrt{\frac{\alpha_4}{\alpha_7}}, s=-\frac{\alpha_4\alpha_6}{2\alpha^2_7}, t=\frac{\sqrt{\alpha_4}(\alpha_3\alpha_6-2\alpha_2\alpha_4)}{2\alpha_7^2\sqrt{\alpha_7}},$ we have the representative $ \left\langle \nabla_4+\nabla_7 \right\rangle;$		
		
		\item if $ \alpha_1=0, \alpha_3\neq0, $ then by choosing $ x=\frac{\alpha_3}{\alpha_7}, s=-\frac{\alpha^2_3\alpha_6}{2\alpha^3_7}, t=\frac{\alpha_3^2\alpha_6-2\alpha_2\alpha_3\alpha_7}{2\alpha_7^3},$ we have the representative $ \left\langle \nabla_3+\alpha\nabla_4+\nabla_7 \right\rangle;$	
		
		\item if $ \alpha_1\neq0, $ then by choosing $ x=\sqrt[3]{\frac{\alpha_1}{\alpha_7}}, s=-\sqrt[3]{\frac{\alpha_1^2\alpha^3_6}{8\alpha^5_7}}, t=\frac{\sqrt[3]{\alpha_1}(\alpha_3\alpha_6-2\alpha_2\alpha_7)}{2\alpha_7^2\sqrt[3]{\alpha_7}},$ we have the representative $ \left\langle \nabla_1+\alpha\nabla_3+\beta\nabla_4+\nabla_7 \right\rangle.$				
		
	\end{enumerate}
	
	\item  $ \alpha_8=0, \alpha_5\neq0, $ then by choosing $ t=-\frac{\alpha_1}{\alpha_5}x, s=-\frac{\alpha_4}{\alpha_5}x^2,$ we have $ \alpha_1^*=\alpha_4^*=0. $ Now we can suppose that $ \alpha_1=0, \alpha_4=0 $ and have the following subcases:
	
	\begin{enumerate}
		\item if $\alpha_2=0, \alpha_3=0, \alpha_6=0, $ then we have the family of representatives $ \left\langle \nabla_5+\alpha\nabla_7 \right\rangle; $
		
		\item if $\alpha_2=0, \alpha_3=0, \alpha_6\neq0, $ then by choosing $ x=\sqrt{\frac{\alpha_6}{\alpha_5}}, s=0, t=0, $ we have the family of representatives $ \left\langle \nabla_5+\nabla_6+\alpha\nabla_7 \right\rangle; $ 		
		
		\item if $\alpha_2=0, \alpha_3\neq0, $ then by choosing $ x=\frac{\alpha_3}{\alpha_5}, s=0, t=0, $ we have the family of representatives $ \left\langle \nabla_3+\nabla_5+\alpha\nabla_6+\beta\nabla_7 \right\rangle; $	
		
		\item if $\alpha_2\neq0, $ then by choosing $ x=\sqrt[3]{\frac{\alpha_2}{\alpha_5}}, s=0, t=0, $ we have the family of representatives $ \left\langle \nabla_2+\alpha\nabla_3+\nabla_5+\beta\nabla_6+\gamma\nabla_7 \right\rangle. $	
		
	\end{enumerate}

	\item  $ \alpha_8\neq0, $ then by choosing $ t=-\frac{\alpha_3}{\alpha_8}x, s=-\frac{\alpha_7}{\alpha_8}x^2,$ we have $ \alpha_3^*=\alpha_7^*=0. $ Now we can suppose that $ \alpha_3=0, \alpha_7=0 $ and have the following subcases:
	
	\begin{enumerate}
		\item if $ \alpha_1=0, \alpha_2=0, \alpha_4=0, \alpha_5=0, \alpha_6=0, $ then we have the representative $ \left\langle \nabla_8 \right\rangle; $
		
		\item if $ \alpha_1=0, \alpha_2=0, \alpha_4=0, \alpha_5=0, \alpha_6\neq0, $ then by choosing $ x=\sqrt[4]{\frac{\alpha_6}{\alpha_8}}, s=0, t=0, $ we have the representative $ \left\langle \nabla_6+\nabla_8 \right\rangle; $
		
		\item if $ \alpha_1=0, \alpha_2=0, \alpha_4=0, \alpha_5\neq0, $ then by choosing $ x=\sqrt{\frac{\alpha_5}{\alpha_8}}, s=0, t=0, $ we have the family of representatives $ \left\langle \nabla_5+\alpha\nabla_6+\nabla_8 \right\rangle; $
		
		\item if $ \alpha_1=0, \alpha_2=0, \alpha_4\neq0, $ then by choosing $ x=\sqrt[4]{\frac{\alpha_4}{\alpha_8}}, s=0, t=0, $ we have the family of representatives $ \left\langle \nabla_4+\alpha\nabla_5+\beta\nabla_6+\nabla_8 \right\rangle; $	
		
		\item if $ \alpha_1=0, \alpha_2\neq0, $ then by choosing $ x=\sqrt[5]{\frac{\alpha_2}{\alpha_8}}, s=0, t=0, $ we have the family of representatives \begin{center}$ \left\langle \nabla_2+\alpha\nabla_4+\beta\nabla_5+\gamma\nabla_6+\nabla_8 \right\rangle; $
		\end{center}		
		
		\item if $ \alpha_1\neq0, $ then by choosing $ x=\sqrt[5]{\frac{\alpha_1}{\alpha_8}}, s=0, t=0, $ we have the family of representatives 
		\begin{center}
		  $ \left\langle \nabla_1+\alpha\nabla_2+\beta\nabla_4+\gamma\nabla_5+\mu\nabla_6+\nabla_8 \right\rangle. $
		\end{center}

	\end{enumerate}
		
\end{enumerate}

Summarizing all cases we have the following distinct orbits:
\begin{center}

$\left\langle \nabla_1+ \alpha\nabla_2 +\beta\nabla_4+\gamma\nabla_5+\mu\nabla_6+\nabla_8  \right\rangle^{
{\tiny \begin{array}{l}O(\alpha,\beta,\gamma,\mu)= 
O(\pm \alpha,\pm \eta_5^4 \beta,\eta_5^2\gamma,\eta_5^4\mu)= \\
O(\pm \alpha,\mp \eta_5^3 \beta,\eta_5^4\gamma,-\eta_5^3\mu)= 
O(\pm \alpha,\pm \eta_5^2 \beta,-\eta_5\gamma,\eta_5^2\mu)= \\
O(\pm \alpha,\mp \eta_5 \beta,-\eta_5^3\gamma,-\eta_5\mu)
\end{array}}}, $ 
$\left\langle \nabla_1+\nabla_3 +\alpha\nabla_4 + \beta\nabla_6\right\rangle
^{{\tiny \begin{array}{l}
O(\alpha,\beta)=O(-\alpha,\beta)=\\
O(\alpha,-\beta)=O(-\alpha,-\beta)
\end{array}}},$
$\left\langle \nabla_1+ \alpha\nabla_3 +\beta \nabla_4+\nabla_7  \right\rangle
^{{\tiny 
\begin{array}{l}
O(\alpha,\beta)=O(-\eta_3\alpha,\eta_3^2\beta)=\\
O(\eta_3^2\alpha,-\eta_3\beta)
\end{array}}},$
$\left\langle \nabla_2+ \alpha\nabla_3+\nabla_5 + \beta\nabla_6+\gamma\nabla_7 \right\rangle
^{{\tiny 
\begin{array}{l}
O(\alpha,\beta,\gamma)=O(-\alpha,\beta,-\gamma)=
O(-\eta_3\alpha,\eta_3^2\beta,\gamma)=\\
O(\eta_3\alpha,\eta_3^2\beta,-\gamma)=
O(\eta_3^2\alpha,-\eta_3\beta,\gamma)=
O(-\eta_3^2\alpha,-\eta_3\beta,-\gamma)
\end{array}}},$
$\left\langle \nabla_2+ \alpha\nabla_4 +\beta\nabla_5 + \gamma\nabla_6+\nabla_8 \right\rangle
^{
{\tiny \begin{array}{l}
O(\alpha,\beta,\gamma)=O(-\alpha,\beta,\gamma)=
O(\eta_5^4   \alpha,\eta_5^2\beta,\eta_5^4\gamma)=\\
O(-\eta_5^4\alpha,\eta_5^2\beta,\eta_5^4\gamma)=
O(-\eta_5^3\alpha,\eta_5^4\beta,-\eta_5^3\gamma)=\\
O(\eta_5^3\alpha,\eta_5^4\beta,-\eta_5^3\gamma)=
O(\eta_5^2\alpha,-\eta_5\beta,\eta_5^2\gamma)=\\
O(-\eta_5^2\alpha,-\eta_5\beta,\eta_5^2\gamma)=
O(-\eta_5\alpha,-\eta_5^3\beta,-\eta_5\gamma)=\\
O(\eta_5\alpha,-\eta_5^3\beta,-\eta_5\gamma)
\end{array}}},$

$\left\langle \nabla_3 \right\rangle, \,  \left\langle \nabla_3+ \nabla_4 +\alpha\nabla_6 \right\rangle, \, 
\left\langle \nabla_3+ \alpha\nabla_4+\nabla_7 \right\rangle
^{O(\alpha)=O(-\alpha)},$  

$\left\langle \nabla_3+ \nabla_5+\alpha\nabla_6+\beta\nabla_7 \right\rangle^{O(\alpha,\beta)=O(\alpha,-\beta)}, \,  \left\langle \nabla_3+\nabla_6 \right\rangle,$ 

$\left\langle \nabla_4+ \alpha\nabla_5+\beta\nabla_6+\nabla_8 \right\rangle^{{\tiny \begin{array}{l}
O(\alpha,\beta)=O(-i\alpha,-\beta)=\\
O(i\alpha,-\beta)=
O(-\alpha,\beta) \end{array}}},$ 
$\left\langle \nabla_4+\nabla_7 \right\rangle, \, \left\langle \nabla_5+\nabla_6+ \alpha\nabla_7 \right\rangle^{O(\alpha,\beta)=O(\alpha,-\beta)},$

$\left\langle \nabla_5+\alpha\nabla_6+ \nabla_8 \right\rangle, \, \left\langle \nabla_5+ \alpha\nabla_7 \right\rangle^{O(\alpha)=O(-\alpha)}, \, 
\left\langle \nabla_6+ \nabla_8 \right\rangle, \,
\left\langle \nabla_7 \right\rangle, \, \left\langle \nabla_8 \right\rangle,$

\end{center}
which gives the following new algebras:

\begin{longtable}{llllllllllllllllll}

${\mathbf{N}}_{150}^{\alpha, \beta, \gamma, \mu }$ & $:$ &$e_1e_1=e_2$ & $e_1e_2=e_5$ & $e_1e_3=\alpha e_5$ & $e_2e_2=e_4$ \\ && $e_2e_3=\beta e_5$   & $e_2e_4=\gamma e_5$ 
 & $e_3e_3=e_4+\mu e_5$
 & $e_4e_4=e_5$
\\

${\mathbf{N}}_{151}^{\alpha, \beta}$ & $:$ &$e_1e_1=e_2$ & $e_1e_2=e_5$ & $e_1e_4=e_5$ \\ && $e_2e_2=e_4$ & $e_2e_3=\alpha e_5$   & $e_3e_3=e_4+\beta e_5$ 
\\

${\mathbf{N}}_{152}^{\alpha, \beta}$ & $:$ &$e_1e_1=e_2$ & $e_1e_2=e_5$ & $e_1e_4=\alpha e_5$ & $e_2e_2=e_4$ \\ && $e_2e_3=\beta e_5$  & $e_3e_3=e_4$ 
 & $e_3e_4=e_5$
\\

${\mathbf{N}}_{153}^{\alpha, \beta, \gamma}$ & $:$ &$e_1e_1=e_2$ & $e_1e_3=e_5$ & $e_1e_4=\alpha e_5$ & $e_2e_2=e_4$ \\ && $e_2e_4=e_5$   & $e_3e_3=e_4+\beta e_5$ 
& $e_3e_4=\gamma e_5$
\\

${\mathbf{N}}_{154}^{\alpha, \beta, \gamma}$ & $:$ &$e_1e_1=e_2$ & $e_1e_3=e_5$ & $e_2e_2=e_4$ & $e_2e_3=\alpha e_5$ \\ && $e_2e_4=\beta e_5$   & $e_3e_3=e_4+\gamma e_5$ 
 & $e_4e_4=e_5$
\\ 

${\mathbf{N}}_{155}$ & $:$ &$e_1e_1=e_2$ & $e_1e_4=e_5$ & $e_2e_2=e_4$ &  $e_3e_3=e_4$
\\

${\mathbf{N}}_{156}^{\alpha}$ & $:$ &$e_1e_1=e_2$ & $e_1e_4=e_5$ & $e_2e_2=e_4$ \\ && $ e_2e_3=e_5$ &  $e_3e_3=e_4+\alpha e_5$
\\

${\mathbf{N}}_{157}^{\alpha}$ & $:$ &$e_1e_1=e_2$ & $e_1e_4=e_5$ & $e_2e_2=e_4$ \\ && $ e_2e_3=\alpha e_5$ &  $e_3e_3=e_4$   & $e_3e_4=e_5$
\\

${\mathbf{N}}_{158}^{\alpha, \beta}$ & $:$ &$e_1e_1=e_2$ & $e_1e_4=e_5$ & $e_2e_2=e_4$ \\ && $ e_2e_4=e_5$ &  $e_3e_3=e_4+\alpha e_5$   & $e_3e_4=\beta e_5$
\\

${\mathbf{N}}_{159}$ & $:$ &$e_1e_1=e_2$ & $e_1e_4=e_5$ & $e_2e_2=e_4$ &  $e_3e_3=e_4+e_5$
\\

${\mathbf{N}}_{160}^{\alpha, \beta}$ & $:$ &$e_1e_1=e_2$ & $e_2e_2=e_4$ & $e_2e_3=e_5$ \\ && $ e_2e_4=\alpha e_5$ &  $e_3e_3=e_4+\beta e_5$  & $e_4e_4=e_5$
\\

${\mathbf{N}}_{161}$ & $:$ &$e_1e_1=e_2$ & $e_2e_2=e_4$ & $e_2e_3=e_5$ \\&&  $e_3e_3=e_4$ & $e_3e_4=e_5$
\\

${\mathbf{N}}_{162}^{\alpha}$ & $:$ &$e_1e_1=e_2$ & $e_2e_2=e_4$ & $e_2e_4=e_5$ \\& & $e_3e_3=e_4+e_5$ & $e_3e_4=\alpha e_5$
\\

${\mathbf{N}}_{163}^{\alpha}$ & $:$ &$e_1e_1=e_2$ & $e_2e_2=e_4$ & $e_2e_4=e_5$ \\&& $e_3e_3=e_4+\alpha e_5$ & $e_4e_4=e_5$
\\

${\mathbf{N}}_{164}^{\alpha}$ & $:$ &$e_1e_1=e_2$ & $e_2e_2=e_4$ & $e_2e_4=e_5$ \\&& $e_3e_3=e_4$ & $e_3e_4=\alpha e_5$
\\

${\mathbf{N}}_{165}$ & $:$ &$e_1e_1=e_2$ & $e_2e_2=e_4$ &  $e_3e_3=e_4+e_5$ & $e_4e_4=e_5$
\\

${\mathbf{N}}_{166}$ & $:$ &$e_1e_1=e_2$ & $e_2e_2=e_4$ &  $e_3e_3=e_4$ & $e_3e_4=e_5$
\\

${\mathbf{N}}_{167}$ & $:$ &$e_1e_1=e_2$ & $e_2e_2=e_4$ &  $e_3e_3=e_4$ & $e_4e_4=e_5$
\\
\end{longtable}

\subsection{$ 1 $-dimensional central extensions of $ {\mathbf N}_{13}^{4*}(\lambda) $.} Here we will collect all information about $ {\mathbf N}_{13}^{4*}(\lambda): $

\begin{longtable}{|l|llll|}
\hline
${\mathbf{N}}^{4*}_{13}(\lambda)$ &  
$e_1e_1=e_2$ 
&$e_1e_2=e_3$ &$e_1e_3=e_4$ &$e_2e_2=\lambda e_4$ \\
\hline 
&
\multicolumn{4}{l|}{
$\begin{array}{l}
\mathrm{H}^2_{\mathfrak{D}}(\mathbf{N}^{4*}_{13}(2))=\langle [\Delta_{22}],4[\Delta_{23}]+[\Delta_{14}],[\Delta_{24}]\rangle,\\

\mathrm{H}^2_{\mathfrak{C}}(\mathbf{N}^{4*}_{13}(2))=\mathrm{H}^2_{\mathfrak{D}}(\mathbf{N}^{4*}_{13}(2))\oplus \langle [\Delta_{23}], [\Delta_{33}], [\Delta_{34}], [\Delta_{44}]\rangle \\

\mathrm{H}^2_{\mathfrak{D}}(\mathbf{N}^{4*}_{13}(\lambda )_{\lambda \neq2})=\langle [\Delta_{22}],(3\lambda-2)[\Delta_{23}]+[\Delta_{14}]\rangle,\\

\mathrm{H}^2_{\mathfrak{C}}(\mathbf{N}^{4*}_{13}(\lambda)_{\lambda \neq2})=\mathrm{H}^2_{\mathfrak{D}}(\mathbf{N}^{4*}_{13}(\lambda)\oplus \langle [\Delta_{23}], [\Delta_{24}], [\Delta_{33}], [\Delta_{34}], [\Delta_{44}] \rangle
\end{array}$}\\
\hline

& \multicolumn{4}{l|}{
$\phi=\begin{pmatrix}
x&0&0&0\\
y&x^2&0&0\\
z&2xy&x^3&0\\
t&\lambda y^2+2xz&(\lambda+2)x^2y&x^4
\end{pmatrix}$}\\ 
\hline
\end{longtable}

Let us use the following notations:
\begin{longtable}{lllll}
$\nabla_1=[\Delta_{14}]+(3\lambda-2)[\Delta_{23}],$ & $  \nabla_2=[\Delta_{22}],$ & $ \nabla_3=[\Delta_{23}],$ & $ \nabla_4=[\Delta_{24}],$ \\ $ \nabla_5=[\Delta_{33}],$ & $  \nabla_6=[\Delta_{34}],$ & $ \nabla_7=[\Delta_{44}].$ 
\end{longtable}

Take $ \theta=\sum\limits_{i=1}^{7}\alpha_i\nabla_i\in\mathrm{H}^2_{\mathfrak{C}}({\mathbf N}_{13}^{4*}(\lambda)) .$  Since

\begin{center}$\phi^T\begin{pmatrix}
0&0&0&\alpha_1\\
0&\alpha_2&(3\lambda-2)\alpha_1+\alpha_3&\alpha_4\\
0&(3\lambda-2)\alpha_1+\alpha_3&\alpha_5&\alpha_6\\
\alpha_1&\alpha_4&\alpha_6&\alpha_7
\end{pmatrix}\phi=
\begin{pmatrix}
\alpha^{**}&\alpha^{***}&\alpha^{*}&\alpha_1^*\\
\alpha^{***}&\alpha_2^*+\lambda\alpha^{*}&(3\lambda-2)\alpha^*_1+\alpha_3^*&\alpha^*_4\\
\alpha^{*}&(3\lambda-2)\alpha^*_1+\alpha_3^*&\alpha^*_5&\alpha^*_6\\
\alpha^*_1&\alpha^*_4&\alpha^*_6&\alpha^*_7
\end{pmatrix},$
\end{center}

we have
\begin{longtable}{lcl}
$\alpha_1^*$&$=$&$(\alpha_1x+\alpha_4y+\alpha_6z+\alpha_7t)x^4,$\\

$\alpha_2^*$&$=$&$\alpha_2x^4+4\lambda(\alpha_6y+\alpha_7z)xy^2+\lambda^2\alpha_7y^4+4(\alpha_4z+(\alpha_3+(3\lambda-2)\alpha_1)y)x^3$\\
\multicolumn{3}{c}{$+2(4\alpha_6yz+2\alpha_7z^2+(2\alpha_5+\lambda\alpha_4)y^2)x^2$}\\
\multicolumn{3}{r}{$-\lambda((\lambda+2)(\alpha_4y+\alpha_6z+\alpha_7t)y+((\alpha_3+4\lambda\alpha_1)y+\alpha_5z+\alpha_6t)x)x^2,$}\\
$\alpha_3^*$&$=$&$[(\lambda+2)(\alpha_4x^2+2\alpha_6xy+2\alpha_7xz+\lambda\alpha_7y^2)y$\\
\multicolumn{3}{c}{$+((\alpha_3+(3\lambda-2)\alpha_1)x^2+2\alpha_5xy+2\alpha_6xz+\lambda\alpha_6y^2)x]x^2$}\\
\multicolumn{3}{r}{$-(3\lambda-2)(\alpha_1x+\alpha_4y+\alpha_6z+\alpha_7t)x^4,$}\\
$\alpha_4^*$&$=$&$(\alpha_4x^2+2\alpha_6xy+2\alpha_7xz+\lambda\alpha_7y^2)x^4,$\\
$\alpha_5^*$&$=$&$(\alpha_5x^2+2(\lambda+2)\alpha_6xy+(\lambda+2)^2\alpha_7y^2)x^4,$\\
$\alpha_6^*$&$=$&$(\alpha_6x+(\lambda+2)\alpha_7y)x^6,$\\
$\alpha_7^*$&$=$&$\alpha_7x^8.$
\end{longtable}

We are interested in 
\begin{center}$ (\alpha_3,\alpha_4,\alpha_5,\alpha_6,\alpha_7)\neq(0,0,0,0,0) $ and $(\alpha_1,\alpha_4,\alpha_6,\alpha_7)\neq(0,0,0,0).$
\end{center} Let us consider the following cases:

\begin{enumerate}
	\item  $ \alpha_7=0, \alpha_6=0, \alpha_5=0, \alpha_4=0, $ then $ \alpha_1\neq0,$ $\alpha_3\neq0 $ and 
	
	\begin{enumerate}
		\item\label{13cs1} if $\lambda\notin \{ 1,2,4\},$  $(\lambda-4)\alpha_3\neq4(1-\lambda)(\lambda-2)\alpha_1, $ then by choosing $ y=\frac{\alpha_2x}{(\lambda-4)\alpha_3+4(\lambda-1)(\lambda-2)\alpha_1},$ we have the family of representatives \begin{center}$ \left\langle \alpha\nabla_1+\nabla_3 \right\rangle_{\alpha\notin \Big\{ 0,\frac{(\lambda-4)}{4(1-\lambda)(\lambda-2) }\Big\};\, \lambda\neq 1,2,4}; $
		\end{center}

		\item if  $\lambda\notin \{ 1,2,4\},$ $ (\lambda-4)\alpha_3=4(1-\lambda)(\lambda-2)\alpha_1,$ $ \alpha_2=0, $ then we have the family of representatives $ \left\langle \frac{\lambda-4}{4(1-\lambda)(\lambda-2)}\nabla_1+\nabla_3 \right\rangle_{\lambda\neq 1,2,4}, $
		which we will be jointed with the family from the case (\ref{13cs1});

		\item if  $\lambda\notin \{ 1,2,4\},$ $ (\lambda-4)\alpha_3=4(1-\lambda)(\lambda-2)\alpha_1,$ $\alpha_2\neq0, $ then by choosing $ x=\frac{\alpha_2}{\alpha_3}, y=0, z=0, t=0, $ we have the family of representatives $ \left\langle  (\lambda-4)\nabla_1+4(1-\lambda)(\lambda-2) (\nabla_2+\nabla_3) \right\rangle_{\lambda\neq 1,2,4}; $

	    \item if   $\lambda\in \{ 1,2,4\},$ then by choosing some suitable $x$ and $y$ we have the family of representatives 
		$ \left\langle \alpha\nabla_1+\nabla_3 \right\rangle_{\alpha\neq 0, \lambda\in \{1,2,4\}}, $ which will be jointed with the family from the case (\ref{13cs1}).
	\end{enumerate}
	
	\item  $ \alpha_7=0, \alpha_6=0, \alpha_5=0, \alpha_4\neq0, $ then we have
	
	\begin{enumerate}
		\item if $\alpha_3= 2(2-\lambda)\alpha_1 ,$ 
		then by choosing 
	\begin{center}
	    
	$x=4 \alpha_4^2, 
	y=-4 \alpha_1 \alpha_4, 
	z=\alpha_1 \alpha_3 (4-\lambda)-\alpha_2 \alpha_4-\alpha_1^2 (8-12 \lambda+3 \lambda^2), t=0,$
		\end{center}we have the representative $ \left\langle \nabla_4 \right\rangle; $
		
		\item if $ \alpha_3\neq2(2-\lambda)\alpha_1 ,$ then by choosing \begin{center}$
		    x=\frac{\alpha_3+2(\lambda-2)\alpha_1}{\alpha_4},$ $y=-\frac{\alpha_1(\alpha_3+2(\lambda-2)\alpha_1)}{\alpha_4^2},$
		    
		    $z=\frac{(2(2-\lambda)\alpha_1-\alpha_3)(\alpha_2\alpha_4+(\lambda-4)\alpha_1\alpha_3+(3\lambda^2-12\lambda+8)\alpha_1^2)}{4\alpha^3_4},$ $t=0,$ 
		    \end{center} we have the representative $ \left\langle \nabla_3+\nabla_4 \right\rangle . $		
		
	\end{enumerate}	
	
	\item  $ \alpha_7=0, \alpha_6=0, \alpha_5\neq0, $ then	
	
	\begin{enumerate}
		\item if $\alpha_4=0, $ then $\alpha_1\neq0$ and
		
		\begin{enumerate}

			\item\label{13css3} if $ \lambda\neq0, $ then by choosing 
			\begin{center}$ x=\frac{\alpha_1}{\alpha_5}, y=-\frac{\alpha_1\alpha_3}{2\alpha_5^2}, z=\frac{\alpha_1(2\alpha_2\alpha_5+(\lambda-2)\alpha_3^2+4(\lambda^2-3\lambda+2)\alpha_1\alpha_3)}{2\lambda\alpha_5^3}, t=0, $ 
			\end{center} 
			we have the family of representatives $ \left\langle \nabla_1+\nabla_5 \right\rangle_{\lambda\neq0};$
			
			\item if $ \lambda=0, $ then by choosing $x=\frac{\alpha_1}{\alpha_5}, y=-\frac{\alpha_1\alpha_3}{2\alpha_5^2}, z=0, t=0,$ we have the family representative $\left\langle \nabla_1+\alpha\nabla_2+\nabla_5 \right\rangle_{\alpha\neq0,\lambda=0}$
			and
			the representative 
			$\left\langle \nabla_1+\nabla_5 \right\rangle_{\lambda=0},$
			which will be jointed with the family from the case
    		(\ref{13css3}).

		\end{enumerate}
		
		\item if $\alpha_4\neq0$ and $\lambda=0,$ then we have the followings:

		\begin{enumerate}
			\item if $   \alpha_3\alpha_4=2\alpha_1(2\alpha_4+\alpha_5),$ 
			then by choosing
			\begin{center}
			    $x=4 \alpha_4^3,$ $y=-4 \alpha_1 \alpha_4^2,$
			    $z=4 \alpha_1 \alpha_3 \alpha_4-\alpha_2 \alpha_4^2-4 \alpha_1^2 (2 \alpha_4+\alpha_5),$
			    $t=0,$
			\end{center}we have the family of representatives $ \left\langle \alpha\nabla_4+\nabla_5 \right\rangle_{\alpha\neq0,\lambda=0}; $
			
			\item if $  \alpha_3\alpha_4\neq 2\alpha_1(2\alpha_4+\alpha_5), $ then by choosing 
			\begin{center}
		$ x=\frac{\alpha_3\alpha_4-2\alpha_1(2\alpha_4+\alpha_5)}{\alpha_4\alpha_5)}, 
		y=\frac{\alpha_1(2\alpha_1(2\alpha_4+\alpha_5)-\alpha_3\alpha_4)}{\alpha^2_4\alpha_5)},$ $z=\frac{(2\alpha_1(2\alpha_4+\alpha_5)-\alpha_3\alpha_4)(\alpha_2\alpha_4^2-4\alpha_1\alpha_3\alpha_4+4\alpha_1^2(2\alpha_4+\alpha_5))}{4\alpha_4^4\alpha_5}, t=0,$
			\end{center} we have the family of representatives $ \left\langle \nabla_3+\alpha\nabla_4+\nabla_5 \right\rangle_{\alpha\neq0, \lambda=0} ; $	
	\end{enumerate}
	\item if $\alpha_4\neq0$ and $\lambda\neq 0,$ then we have the followings:

		\begin{enumerate}
		
\item\label{13css5} if $4\alpha_4=\lambda\alpha_5,$ $4\lambda(\lambda-4)\alpha_1\alpha_3+\lambda^2\alpha_2\alpha_5+4(3\lambda^3-12\lambda^2+8\lambda+16)\alpha_1^2=0,$ $\lambda\alpha_3+2(\lambda^2-2\lambda-4)\alpha_1=0,$ then by choosing $x=1, y=-\frac{\alpha_1}{\alpha_4}, z=0, t=0,$ we have the family of representatives 	$ \left\langle \frac{\lambda}{4}\nabla_4+\nabla_5 \right\rangle_{\lambda\neq0} ; $

\item \label{13css6} if $4\alpha_4=\lambda\alpha_5,$ $4\lambda(\lambda-4)\alpha_1\alpha_3+\lambda^2\alpha_2\alpha_5+4(3\lambda^3-12\lambda^2+8\lambda+16)\alpha_1^2=0,$ $\lambda\alpha_3+2(\lambda^2-2\lambda-4)\alpha_1\neq0,$ then by choosing \begin{center}$x=\frac{\lambda\alpha_3+2(\lambda^2-2\lambda-4)\alpha_1}{\lambda\alpha_5}, y=-\frac{4\alpha_1(\lambda\alpha_3+2(\lambda^2-2\lambda-4)\alpha_1)}{\lambda^2\alpha^2_5}, z=0, t=0,$
\end{center} we have the family of representatives 	$ \left\langle \nabla_3+\frac{\lambda}{4}\nabla_4+\nabla_5 \right\rangle_{\lambda\neq0} ; $

\item if $4\alpha_4=\lambda\alpha_5,$ $4\lambda(\lambda-4)\alpha_1\alpha_3+\lambda^2\alpha_2\alpha_5+4(3\lambda^3-12\lambda^2+8\lambda+16)\alpha_1^2\neq0,$ then by choosing \begin{center}$x=\frac{\sqrt{4\lambda(\lambda-4)\alpha_1\alpha_3+\lambda^2\alpha_2\alpha_5+4(3\lambda^3-12\lambda^2+8\lambda+16)\alpha_1^2}}{\lambda\alpha_5},$ $y=-\frac{4\alpha_1\sqrt{4\lambda(\lambda-4)\alpha_1\alpha_3+\lambda^2\alpha_2\alpha_5+4(3\lambda^3-12\lambda^2+8\lambda+16)\alpha_1^2}}{\lambda^2\alpha^2_5}, z=0, t=0,$ 
\end{center}
we have the family of representatives 	$ \left\langle \nabla_2+\alpha\nabla_3+\frac{\lambda}{4}\nabla_4+\nabla_5 \right\rangle_{\lambda\neq0} ; $

\item if $\lambda\neq0, 4\alpha_4\neq\lambda\alpha_5 , $ then by choosing 
\begin{center}$ y=-\frac{\alpha_1}{\alpha_4}x, z=-\frac{\alpha_2\alpha_4^2+(\lambda-4)\alpha_1\alpha_3\alpha_4+\alpha_1^2(4\alpha_5+(3\lambda^2-12\lambda+8)\alpha_4)}{\alpha^2_4(4\alpha_4-\lambda\alpha_5)}x, t=0,$ 
\end{center} we have two families of representatives \begin{center}
$ \left\langle \alpha\nabla_4+\nabla_5 \right\rangle_{\alpha\neq\frac{\lambda}{4}} $ and $ \left\langle \nabla_3+\alpha\nabla_4+\nabla_5 \right\rangle_{\alpha\neq\frac{\lambda}{4}}  $
\end{center} depending on $\alpha_3\alpha_4-2\alpha_1\alpha_5+2(\lambda-2)\alpha_1\alpha_4=0$ or not.
These families  will be jointed with representatives from cases
(\ref{13css5}) and (\ref{13css6}).

		\end{enumerate}

	\end{enumerate}
	
	\item  $ \alpha_7=0, \alpha_6\neq0, $ then by choosing $ y=-\frac{\alpha_4}{2\alpha_6}x, z=-\frac{\alpha_4^2-2\alpha_1\alpha_6}{2\alpha^2_6}x, $ we have $ \alpha_1^*=0, \alpha_4^*=0. $ Since we can suppose that $ \alpha_1=0, \alpha_4=0 $ and
	
	\begin{enumerate}
		\item\label{13css7} if $ \lambda\neq0, \alpha_3=0,$ then by choosing $t={\frac{\alpha_2}{\lambda\alpha_6}}x, $ we have the representatives $ \left\langle \nabla_6 \right\rangle _{\lambda\neq0} $ and $ \left\langle \nabla_5+\nabla_6 \right\rangle _{\lambda\neq0} $ depending on $ \alpha_5=0 $ or not;			
		
		\item\label{13css8} if $ \lambda\neq0, \alpha_3\neq0, $ then by choosing $ x=\sqrt{\frac{\alpha_3}{\alpha_6}},  t={\frac{\alpha_2\sqrt{\alpha_3}}{\lambda\alpha_6\sqrt{\alpha_6}}}, $ we have the family of representatives	$ \left\langle \nabla_3+\alpha\nabla_5+\nabla_6 \right\rangle _{\lambda\neq0}; $				
		
		\item if $ \lambda=0, \alpha_2=0, \alpha_3=0, $ then we have the representatives $ \left\langle \nabla_6 \right\rangle _{\lambda=0} $ and $ \left\langle \nabla_5+\nabla_6 \right\rangle _{\lambda=0} $ depending on $ \alpha_5=0 $ or not, which will be jointed with representatives   from the case (\ref{13css7});
		
		\item if $ \lambda=0, \alpha_2=0, \alpha_3\neq0, $ then by choosing $ x=\sqrt{\frac{\alpha_3}{\alpha_6}},  t=0 ,$ we have the family of representatives	$ \left\langle \nabla_3+\alpha\nabla_5+\nabla_6 \right\rangle _{(\lambda=0)},$
		which will be jointed with the family of representatives   from the case (\ref{13css8});
		
		\item if $ \lambda=0, \alpha_2\neq0, $ then by choosing $ x=\sqrt[3]{\frac{\alpha_2}{\alpha_6}},  t=0, $ we have the family of representatives	$ \left\langle \nabla_2+\alpha\nabla_3+\beta\nabla_5+\nabla_6 \right\rangle _{\lambda=0}.$

		\end{enumerate}

		\item  $ \alpha_7\neq0, \lambda\neq-2,$ then by choosing 
		\begin{center}
		    $ y=-\frac{\alpha_6}{\alpha_7(\lambda+2)}x,$ $z=\frac{2(\lambda+2)^2\alpha_4\alpha_7-(\lambda+4)\alpha_6^2}{2(\lambda+2)^2\alpha^2_7}x,$ $t=\frac{(\lambda^2+6\lambda+8)\alpha_4\alpha_6\alpha_7-2(\lambda+2)^2\alpha_1\alpha^2_7-(\lambda+4)\alpha^3_6}{2(\lambda+2)^2\alpha^3_7}x ,$
		\end{center} we have $ \alpha_1^*=0, \alpha_4^*=0, \alpha_6^*=0.$ Now we can suppose that $ \alpha_1=0, \alpha_4=0, \alpha_6=0 $ then we have
	
	\begin{enumerate}
		\item\label{13cs5a} if $\alpha_3=0, \alpha_5=0, \alpha_2=0, $ then we have the representative $ \left\langle \nabla_7 \right\rangle_{\lambda\neq-2} ; $
		
	\item\label{13cs5b} if $\alpha_3=0, \alpha_5=0, \alpha_2\neq0, $ then by choosing $ x=\sqrt[4]{\alpha_2\alpha_7^{-1}}, $ we have the representative $ \left\langle \nabla_2+\nabla_7 \right\rangle_{\lambda\neq-2} ; $
	
			\item\label{13cs5c} if $\alpha_5\neq0, $ 
		then by choosing 
		$ x=\sqrt{\alpha_5\alpha_7^{-1}},$ 
		$y=-\frac{\alpha_3}{2 \sqrt{\alpha_5 \alpha_7}},$
		$z= \frac{\alpha_3^2}{ 4 \sqrt{\alpha_5^3\alpha_7}},$
		$t=0,$
		we have the family of representatives $ \left\langle \alpha\nabla_2+\nabla_5+\nabla_7 \right\rangle_{\lambda\neq-2} . $

%		\item\label{13cs5c} if $\alpha_3=0, \alpha_5\neq0, $ then by choosing $ x=\sqrt{\alpha_5\alpha_7^{-1}}, $ we have the family of representatives $ \left\langle \alpha\nabla_2+\nabla_5+\nabla_7 \right\rangle_{\lambda\neq-2} ; $
		
		\item\label{13cs5d} if $\alpha_3\neq0, \alpha_5=0 $ then by choosing $ x=\sqrt[3]{\alpha_3\alpha_7^{-1}},$ we have the family of representatives $ \left\langle \alpha\nabla_2+\nabla_3+ \nabla_7 \right\rangle_{\lambda\neq-2} . $

	\end{enumerate}		
	
	\item  $ \alpha_7\neq0, \lambda=-2,$ then
	
	\begin{enumerate}
		\item if $ \alpha_6=0, \alpha_5=0, $ then by choosing 
		$ z=\frac{y^2}{x}-\frac{\alpha_4x}{2\alpha_7}, t=-\frac{x \alpha_1+y \alpha_4}{\alpha_7}, $ we have $ \alpha_1^*=0$ and $\alpha_4^*=0.$ Now consider the followings:
		
		\begin{enumerate}
			\item if $\alpha_3=8\alpha_1, \alpha_2\alpha_7-\alpha_4^2=0,$ then we have the representative $ \left\langle \nabla_7 \right\rangle_{\lambda=-2},$ which will be jointed with the representative from the case (\ref{13cs5a});
			
			\item if $\alpha_3=8\alpha_1, \alpha_2\alpha_7-\alpha_4^2\neq0,$ then by choosing $ x=\sqrt[4]{\frac{\alpha_2\alpha_7-\alpha_4^2}{\alpha_7^2}}, y=0, $ we have the representative $ \left\langle \nabla_2+\nabla_7 \right\rangle_{\lambda=-2}, $ which will be jointed with the representative from the case (\ref{13cs5b});
			
			\item if $\alpha_3\neq8\alpha_1,$ then by choosing $ x=\sqrt[3]{\frac{\alpha_3-8\alpha_1}{\alpha_7}}, y=\frac{\alpha_2\alpha_7-\alpha_4^2}{48\alpha_1\alpha_7-6\alpha_3\alpha_7}x, $ we have the representative $ \left\langle \nabla_3+\nabla_7 \right\rangle_{\lambda=-2} . $

		\end{enumerate}
		
		\item if $ \alpha_6=0, \alpha_5\neq0, $ then by choosing    \begin{center}
		    $x=\sqrt{\frac{\alpha_5}{\alpha_7}}, y=\frac{8\alpha_1-\alpha_3}{2\sqrt{\alpha_5\alpha_7}}, z=\frac{\alpha_7(\alpha_3-8\alpha_1)^2-2\alpha_4\alpha_5^2}{4\alpha_5\alpha_7\sqrt{\alpha_5\alpha_7}},  t=\frac{\alpha_3\alpha_4-2\alpha_1(4\alpha_4+\alpha_5)}{2\alpha_7\sqrt{\alpha_5\alpha_7}}, $

		\end{center}we have the family of representatives $ \left\langle \alpha\nabla_2+\nabla_5+\nabla_7 \right\rangle_{\lambda=-2}, $
 which will be jointed with the representative from the case (\ref{13cs5c}).

		\item if $ \alpha_6\neq0, $ then we have  the following cases:
		
		\begin{enumerate}
			\item\label{13cs5c1} if $ \alpha_5\alpha_7=\alpha_6^2, 8\alpha_1\alpha_7+\alpha_4\alpha_6=\alpha_3\alpha_7, $ then by choosing
			\begin{center}
	$ x=\frac{\alpha_6}{\alpha_7}, y=0, z=-\frac{\alpha_4\alpha_6}{2\alpha_7^2}, t=\frac{\alpha_6(\alpha_4\alpha_6-2\alpha_1\alpha_7)}{2\alpha_7^3}, $
			\end{center} we have the family of representatives $ \left\langle \alpha\nabla_2+\nabla_5+\nabla_6+\nabla_7  \right\rangle_{\lambda=-2};$
			
			\item if $ \alpha_5\alpha_7=\alpha_6^2, 8\alpha_1\alpha_7+\alpha_4\alpha_6\neq \alpha_3\alpha_7, $ then by choosing 
			\begin{center}
 $ x=\frac{\alpha_6}{\alpha_7}, 
 y=\frac{\alpha_6 (\alpha_2 \alpha_7-\alpha_4^2 - 
    2 \alpha_1 \alpha_6 )}{
 6 \alpha_7 (\alpha_4 \alpha_6 + 
    8 \alpha_1 \alpha_7 - \alpha_3 \alpha_7))},$
    $z=\frac{y^2}{x}-\frac{\alpha_4x}{2\alpha_7}-\frac{\alpha_6y}{\alpha_7},  t=-\frac{x \alpha_1 + y \alpha_4 + z \alpha_6}{\alpha_7},$ 
			\end{center}
			we have the family of representatives 
			\begin{center}$ \left\langle \alpha\nabla_3+\nabla_5+\nabla_6+\nabla_7  \right\rangle_{\alpha\neq0,\lambda=-2}; $
			\end{center}

			\item if $ \alpha_5\alpha_7-\alpha_6^2\neq0, $ then by choosing \begin{center}
			    $x=\frac{\alpha_6}{\alpha_7}, 
			    y=\frac{\alpha_6 (\alpha_4 \alpha_6+8 \alpha_1 \alpha_7-\alpha_3 \alpha_7)}{2 \alpha_7 (-\alpha_6^2+\alpha_5 \alpha_7)}, z=\frac{y^2}{x}-\frac{\alpha_4}{2\alpha_7}x-\frac{\alpha_6}{\alpha_7}y,$  
			    $t=\frac{(\alpha_4\alpha_6-2\alpha_1\alpha_7)x^2-2\alpha_6\alpha_7y^2+2(\alpha_6^2-\alpha_4\alpha_7)xy}{2\alpha_7^2x},$ 
			    
			\end{center}
			we have the family of representatives \begin{center}$ \left\langle \alpha\nabla_2+\beta\nabla_5+\nabla_6+\nabla_7  \right\rangle_{\beta\neq1,\lambda=-2},$\end{center}
			which will be jointed with the family from the case (\ref{13cs5c1}).

		\end{enumerate}

\end{enumerate}

\end{enumerate}

Summarizing all cases we have the following distinct orbits:
\begin{center}

$ \left\langle (\lambda-4)\nabla_1+4(1-\lambda)(\lambda-2)(\nabla_2+\nabla_3) \right\rangle_{\lambda \notin \{ 1; 2; 4 \}},$ 
$\left\langle \nabla_1+\alpha\nabla_2+\nabla_5 \right\rangle_{\lambda=0, \alpha \neq 0}, $
$ \left\langle \alpha\nabla_1+\nabla_3 \right\rangle_{\alpha\neq0}, \, \left\langle \nabla_1+\nabla_5 \right\rangle,$
$\left\langle \nabla_2+\alpha\nabla_3+\frac{\lambda}{4}\nabla_4+\nabla_5 \right\rangle_{\lambda\neq0}^{O(\alpha)=O(-\alpha)}, $ 
$ \left\langle \nabla_2+\alpha\nabla_3+\beta\nabla_5+\nabla_6 \right\rangle _{\lambda=0}^{O(\alpha,\beta)=O(\eta_3^2\alpha,-\eta\beta)=
O(-\eta_3\alpha,\eta_3^2\beta)}, $ 
$\left\langle \alpha\nabla_2+\nabla_3+ \nabla_7 \right\rangle_{\lambda\neq-2}^{O(\alpha)=O(-\eta_3\alpha)=O(\eta_3^2\alpha)},$
$ \left\langle \alpha\nabla_2+\beta\nabla_5+\nabla_6+\nabla_7  \right\rangle_{\lambda=-2},$ 
$\left\langle \alpha\nabla_2+\nabla_5+\nabla_7 \right\rangle,$   
$\left\langle \nabla_2 + \nabla_7 \right\rangle,$ 
$\left\langle \nabla_3+\nabla_4 \right\rangle,$ 
$\left\langle \nabla_3+\alpha\nabla_4+\nabla_5 \right\rangle_{\alpha\neq0},$
$\left\langle \nabla_3+\alpha\nabla_5+\nabla_6 \right\rangle,$ 
$\left\langle \alpha\nabla_3+\nabla_5+\nabla_6+\nabla_7  \right\rangle_{\alpha\neq0,\lambda=-2},$ 
$\left\langle \nabla_3+\nabla_7 \right\rangle_{\lambda=-2},$
$\left\langle \nabla_4 \right\rangle_{\lambda\neq 2},$
$\left\langle \alpha\nabla_4+\nabla_5 \right\rangle_{\alpha\neq0},$
$\left\langle \nabla_5+\nabla_6 \right\rangle,$
$\left\langle \nabla_6 \right\rangle,$ 
$\left\langle \nabla_7 \right\rangle.$

\end{center}

Now we have the following new algebras

\begin{longtable}{llllllllllllllllll}

${\mathbf{N}}_{168}^{\lambda  \neq 1; 2; 4}$ & $:$ & $e_1e_1=e_2$ & $e_1e_2=e_3$ & $e_1e_3=e_4$ & $e_1e_4=(\lambda-4)e_5$\\ 
 & & \multicolumn{2}{l}{$e_2e_2=\lambda e_4 + 4(1-\lambda)(\lambda-2)e_5$}  & 
\multicolumn{2}{l}{$e_2e_3= -  \lambda(\lambda+2)e_5$}\\

${\mathbf{N}}_{169}^{\alpha\neq0} $ & $:$ & 
$e_1e_1=e_2$ & $e_1e_2=e_3$ & $e_1e_3=e_4$ &  $e_1e_4=e_5$ \\
& &  
$e_2e_2= \alpha e_5$ &  $e_2e_3= -2 e_5$ &  $e_3e_3= e_5$
\\

  ${\mathbf{N}}_{170}^{\lambda, \alpha \neq 0}$ & $:$ & 
  $e_1e_1=e_2$ & $e_1e_2=e_3$ & $e_1e_3=e_4$ \\ 
  & &  $e_1e_4= \alpha e_5$ & $e_2e_2=\lambda e_4$ & \multicolumn{2}{l}{$e_2e_3=(1+\alpha(3\lambda-2)) e_5$}
\\

${\mathbf{N}}_{171}^{\lambda}$ & $:$ & $e_1e_1=e_2$ 
& $e_1e_2=e_3$ & $e_1e_3=e_4$ & $e_1e_4=e_5$ \\ 
&& $e_2e_2=\lambda e_4$ & \multicolumn{2}{l}{$e_2e_3=(3\lambda-2) e_5$} & $e_3e_3= e_5$
\\
 ${\mathbf{N}}_{172}^{\lambda \neq 0,\alpha}$ & $:$ & 
 $e_1e_1=e_2$  & $e_1e_2=e_3$ & $e_1e_3=e_4$ &  $e_2e_2=\lambda e_4 + e_5$ \\
 &&  $e_2e_3=\alpha e_5$ &  $e_2e_4=\frac{\lambda} {4} e_5$ & $e_3e_3= e_5$

\\
 ${\mathbf{N}}_{173}^{\alpha, \beta}$ & $:$ & 
 $e_1e_1=e_2$ & $e_1e_2=e_3$ & $e_1e_3=e_4$ &  $e_2e_2=e_5$ \\
 & &  $e_2e_3=\alpha e_5$ &  $e_3e_3=\beta e_5$ &  $e_3e_4= e_5$
\\

${\mathbf{N}}_{174}^{\lambda\neq-2, \alpha}$ & $:$ & 
$e_1e_1=e_2$  & $e_1e_2=e_3$ & $e_1e_3=e_4$  \\ &&  $e_2e_2=\lambda e_4 + \alpha e_5$  & $e_2e_3=e_5$ &  $e_4e_4= e_5$
\\

 ${\mathbf{N}}_{175}^{\alpha, \beta}$ & $:$ & $e_1e_1=e_2$ 
& $e_1e_2=e_3$ & $e_1e_3=e_4$ &  $e_2e_2=-2 e_4+\alpha e_5$ \\
&& $e_3e_3=\beta e_5$ & $e_3e_4= e_5$ & $e_4e_4= e_5$
\\

 ${\mathbf{N}}_{176}^{\lambda,\alpha}$ & $:$ & $e_1e_1=e_2$ 
& $e_1e_2=e_3$ & $e_1e_3=e_4$\\& &  $e_2e_2=\lambda e_4+\alpha e_5$ & $e_3e_3= e_5$ & $e_4e_4= e_5$
\\
 ${\mathbf{N}}_{177}^{\lambda}$ & $:$ & $e_1e_1=e_2$ 
& $e_1e_2=e_3$ & $e_1e_3=e_4$ \\& &  $e_2e_2=\lambda e_4+ e_5$ & $e_4e_4= e_5$
\\
 ${\mathbf{N}}_{178}^{\lambda}$ & $:$ & $e_1e_1=e_2$ 
& $e_1e_2=e_3$ & $e_1e_3=e_4$ &  $e_2e_2=\lambda e_4$ \\ && $e_2e_3=e_5$ & $e_2e_4=e_5$
\\

 ${\mathbf{N}}_{179}^{\lambda,\alpha\neq0}$ & $:$ & $e_1e_1=e_2$ 
& $e_1e_2=e_3$ & $e_1e_3=e_4$ &  $e_2e_2=\lambda e_4$ \\ 
&& $e_2e_3=e_5$ & $e_2e_4=\alpha e_5$ & $e_3e_3=e_5$

\\
 ${\mathbf{N}}_{180}^{\lambda,\alpha}$ & $:$ & $e_1e_1=e_2$ 
& $e_1e_2=e_3$ & $e_1e_3=e_4$ &  $e_2e_2=\lambda e_4$ \\ 
&& $e_2e_3=e_5$ & $e_3e_3=\alpha e_5$ & $e_3e_4=e_5$
\\

 ${\mathbf{N}}_{181}^{\alpha \neq 0 }$ & $:$ & $e_1e_1=e_2$ 
& $e_1e_2=e_3$ & $e_1e_3=e_4$ &  $e_2e_2=-2 e_4$ \\ 
&& $e_2e_3=\alpha e_5$ & $e_3e_3= e_5$ & $e_3e_4=e_5$ & $e_4e_4=e_5$
\\
 ${\mathbf{N}}_{182}$ & $:$ & $e_1e_1=e_2$ 
& $e_1e_2=e_3$ & $e_1e_3=e_4$ \\&&  $e_2e_2=-2 e_4$ & $e_2e_3=e_5$ & $e_4e_4=e_5$
\\
 ${\mathbf{N}}_{183}^{\lambda\neq 2}$ & $:$ & $e_1e_1=e_2$ 
& $e_1e_2=e_3$ & $e_1e_3=e_4$ \\&&  $e_2e_2=\lambda e_4$ & $e_2e_4=e_5$
\\

 ${\mathbf{N}}_{184}^{\lambda, \alpha\neq 0}$ & $:$ & $e_1e_1=e_2$ 
& $e_1e_2=e_3$ & $e_1e_3=e_4$\\& &  $e_2e_2=\lambda e_4$ & $e_2e_4=\alpha e_5$ & $e_3e_3= e_5$
\\
 ${\mathbf{N}}_{185}^{\lambda}$ & $:$ & $e_1e_1=e_2$ 
& $e_1e_2=e_3$ & $e_1e_3=e_4$ \\&&  $e_2e_2=\lambda e_4$ & $e_3e_3= e_5$ & $e_3e_4= e_5$
\\
 ${\mathbf{N}}_{186}^{\lambda}$ & $:$ & $e_1e_1=e_2$ 
& $e_1e_2=e_3$ & $e_1e_3=e_4$ \\& &  $e_2e_2=\lambda e_4$ &  $e_3e_4= e_5$
\\
${\mathbf{N}}_{187}^{\lambda}$ & $:$ & $e_1e_1=e_2$ 
& $e_1e_2=e_3$ & $e_1e_3=e_4$ \\& &  $e_2e_2=\lambda e_4$ &  $e_4e_4= e_5$

 \end{longtable}

\subsection{$ 1 $-dimensional central extensions of $ {\mathbf N}_{14}^{4*} $.} Here we will collect all information about $ {\mathbf N}_{14}^{4*}: $
$$
\begin{array}{|l|l|l|l|}
%\hline
%\text{ }  & \text{ } & \text{Cohomology} & \text{Automorphisms}\\
\hline
{\mathbf{N}}^{4*}_{14} &  
\begin{array}{l}
e_1e_2=e_3 \\ 
e_1e_3=e_4
\end{array}
&
\begin{array}{lcl}
\mathrm{H}^2_{\mathfrak{D}}(\mathbf{N}^{4*}_{14})&=&\\
\multicolumn{3}{r}{\langle [\Delta_{11}],[\Delta_{22}],[\Delta_{23}],[\Delta_{33}]\rangle}\\
\mathrm{H}^2_{\mathfrak{C}}(\mathbf{N}^{4*}_{14})&=&\mathrm{H}^2_{\mathfrak{D}}(\mathbf{N}^{4*}_{14})\oplus\\
\multicolumn{3}{r}{\langle [\Delta_{14}], [\Delta_{24}], [\Delta_{34}], [\Delta_{44}]\rangle }
\end{array} &  \phi=\begin{pmatrix}
x&0&0&0\\
0&q&0&0\\
0&r&xq&0\\
t&s&xr&x^2q
\end{pmatrix}\\
\hline
\end{array}$$
Let us use the following notations:
\begin{longtable}{llll}
$\nabla_1=[\Delta_{11}],$&
$ \nabla_2=[\Delta_{14}],$&
$ \nabla_3=[\Delta_{22}],$&
$ \nabla_4=[\Delta_{23}],$\\
$\nabla_5=[\Delta_{24}],$&
$\nabla_6=[\Delta_{33}],$&
$\nabla_7=[\Delta_{34}],$&
$\nabla_8=[\Delta_{44}]. $
\end{longtable}
Take $ \theta=\sum\limits_{i=1}^{8}\alpha_i\nabla_i\in\mathrm{H}^2_{\mathfrak{C}}(\mathbf{N}^{4*}_{14}) .$  Since

$$\phi^T\begin{pmatrix}
\alpha_1&0&0&\alpha_2\\
0&\alpha_3&\alpha_4&\alpha_5\\
0&\alpha_4&\alpha_6&\alpha_7\\
\alpha_2&\alpha_5&\alpha_7&\alpha_8

\end{pmatrix}\phi=
\begin{pmatrix}
\alpha_1^*&\alpha^{*}&\alpha^{**}&\alpha_2^*\\
\alpha^{*}&\alpha_3^*&\alpha^*_4&\alpha^*_5\\
\alpha^{**}&\alpha^*_4&\alpha^*_6&\alpha^*_7\\
\alpha^*_2&\alpha^*_5&\alpha^*_7&\alpha^*_8
\end{pmatrix},$$

we have
\begin{longtable}{lcl}
$\alpha_1^*$&$=$&$\alpha_1x^2+2\alpha_2xt+\alpha_8t^2,$\\
$\alpha_2^*$&$=$&$(\alpha_2x+\alpha_8t)x^2q,$\\

$\alpha_3^*$&$=$&$(\alpha_3q+\alpha_4r+\alpha_5s)q+(\alpha_4q+\alpha_6r+\alpha_7s)r+(\alpha_5q+\alpha_7r+\alpha_8s)s,$\\
$\alpha_4^*$&$=$&$(\alpha_4q+\alpha_6r+\alpha_7s)xq+(\alpha_5q+\alpha_7r+\alpha_8s)xr,$\\
$\alpha_5^*$&$=$&$(\alpha_5q+\alpha_7r+\alpha_8s)x^2q,$\\
$\alpha_6^*$&$=$&$(\alpha_6q^2+2\alpha_7qr+\alpha_8r^2)x^2,$\\
$\alpha_7^*$&$=$&$(\alpha_7q+\alpha_8r)x^3q,$\\
$\alpha_8^*$&$=$&$\alpha_8x^4q^2.$
\end{longtable}

We are interested in $ (\alpha_2,\alpha_5,\alpha_7,\alpha_8)\neq(0,0,0,0) .$ Let us consider the following cases:

\begin{enumerate}
	\item  $ \alpha_8=0, \alpha_7=0, \alpha_5=0, $ then $ \alpha_2\neq0 $ and we have
	
	\begin{enumerate}
		\item if $ \alpha_6=0, \alpha_4=0, \alpha_3=0, $ then by choosing $ x=2\alpha_2, q=1, r=0, s=0, t=-\alpha_1,  $ we have the representative $ \left\langle \nabla_2 \right\rangle; $
		
		\item if $ \alpha_6=0, \alpha_4=0, \alpha_3\neq0, $ then by choosing $ x=\alpha_3, q=\alpha_2\alpha_3^2, r=0, s=0, t=-\frac{\alpha_1\alpha_3}{2\alpha_2},  $ we have the representative $ \left\langle \nabla_2+\nabla_3 \right\rangle; $
		
		\item if $ \alpha_6=0, \alpha_4\neq0, $ then by choosing $ x=\alpha_4, q=\alpha_2\alpha_4, r=-\frac{\alpha_2\alpha_3}{2}, s=0, t=-\frac{\alpha_1\alpha_4}{2\alpha_2},  $ we have the representative $ \left\langle \nabla_2+\nabla_4 \right\rangle; $
		
		\item if $ \alpha_6\neq0, \alpha_3\alpha_6-\alpha_4^2=0, $ then by choosing $ x=\alpha_6, q=\alpha_2, r=-\frac{\alpha_2\alpha_4}{\alpha_6}, s=0, t=-\frac{\alpha_1\alpha_6}{2\alpha_2},  $ we have the representative $ \left\langle \nabla_2+\nabla_6 \right\rangle; $

		\item if $ \alpha_6\neq0, \alpha_3\alpha_6-\alpha_4^2\neq0, $ then by choosing \begin{center}$ x=\frac{\sqrt{\alpha_3\alpha_6-\alpha_4^2}}{\alpha_6},$ $q=\frac{\alpha_2\sqrt{\alpha_3\alpha_6-\alpha_4^2}}{\alpha^2_6},$ $r=-\frac{\alpha_2\alpha_4\sqrt{\alpha_3\alpha_6-\alpha_4^2}}{\alpha^3_6},$ $s=0,  t=-\frac{\alpha_1\sqrt{\alpha_3\alpha_6-\alpha_4^2}}{2\alpha_2\alpha_6},  $\end{center}
		we have the representative $ \left\langle \nabla_2+\nabla_3+\nabla_6 \right\rangle. $

	\end{enumerate}

	\item $ \alpha_8=0, \alpha_7=0, \alpha_5\neq0,$ then we have
	
	\begin{enumerate}
		\item\label{14cs2} if $ \alpha_6=0, \alpha_2=0, $ then by choosing
		\begin{center}$x=1,  r=-\frac{\alpha_4}{\alpha_5}q, s=\frac{2\alpha_4^2-\alpha_3\alpha_5}{2\alpha_5^2}q, t=0,  $
		\end{center} we have the representatives $ \left\langle \nabla_5 \right\rangle  $ and $ \left\langle \nabla_1+\nabla_5 \right\rangle  $ depending on whether $ \alpha_1=0 $ or not;
		
		\item\label{14cccs1} if $ \alpha_6=0, \alpha_2\neq0, $ then by choosing 
		\begin{center}$ x=\alpha_5, q=\alpha_2, r=-\frac{\alpha_2\alpha_4}{\alpha_5},  s=\frac{\alpha_2(2\alpha_4^2-\alpha_3\alpha_5)}{2\alpha_5^2}, t=-\frac{\alpha_1\alpha_5}{2\alpha_2},  $
		\end{center}
		we have the representatives $ \left\langle \nabla_2+\nabla_5 \right\rangle; $
		
		\item if $ \alpha_6\neq0, \alpha_5=-\alpha_6, $ then we have the following subcases:
		
		\begin{enumerate}
			\item\label{14cs1} if $ \alpha_2=0, \alpha_4=0, \alpha_1=0, $ then we have the representative $ \left\langle \nabla_5-\nabla_6 \right\rangle; $
			
			\item\label{14ccs1} if $ \alpha_2=0, \alpha_4=0, \alpha_1\neq0, $ then by choosing \begin{center}
			    $ x=1, q=\sqrt{\frac{\alpha_1}{\alpha_5}}, r=0, s=-\frac{\alpha_3\sqrt{\alpha_1}}{2\alpha_5\sqrt{\alpha_5}}, t=0,$
			\end{center} we have the representative $ \left\langle \nabla_1+\nabla_5-\nabla_6 \right\rangle; $
			
			\item if $ \alpha_2=0, \alpha_4\neq0, \alpha_1=0, $ then by choosing \begin{center}
			    $ x={\frac{\alpha_4}{\alpha_5}}, q=1, r=0, s=-\frac{\alpha_3}{2\alpha_5}, t=0, $ 
			\end{center} we have the representative $ \left\langle \nabla_4+\nabla_5-\nabla_6 \right\rangle; $
			
			\item if $ \alpha_2=0, \alpha_4\neq0, \alpha_1\neq0, $ then by choosing \begin{center}
			    $ x=\frac{\alpha_4}{\alpha_5}, q=\sqrt{\frac{\alpha_1}{\alpha_5}}, r=0, s=-\frac{\alpha_3\sqrt{\alpha_1}}{2\alpha_5\sqrt{\alpha_5}}, t=0, $
			\end{center} we have the representative $ \left\langle \nabla_1+\nabla_4+\nabla_5-\nabla_6 \right\rangle; $
			
			\item\label{14cccs2} if $ \alpha_2\neq0, \alpha_4=0, $ then by choosing \begin{center} $ x=\alpha_5, q=\alpha_2, r=0, s=-\frac{\alpha_2\alpha_3}{2\alpha_5}, t=-\frac{\alpha_1\alpha_5}{2\alpha_2},$ \end{center}  we have the representative $ \left\langle \nabla_2+\nabla_5-\nabla_6 \right\rangle; $
			
			\item if $ \alpha_2\neq0, \alpha_4\neq0, $ then by choosing 
			\begin{center}
			    $ x=\frac{\alpha_4}{\alpha_5}, q={\frac{\alpha_2\alpha_4}{\alpha^2_5}}, r=0, s=-\frac{\alpha_2\alpha_3\alpha_4}{2\alpha_5^3}, t=-\frac{\alpha_1\alpha_4}{2\alpha_2\alpha_5}, $ 			 
			\end{center} we have the representative $ \left\langle \nabla_2+\nabla_4+\nabla_5-\nabla_6 \right\rangle. $			
									
		\end{enumerate}		
				
		\item if $ \alpha_6\neq0, \alpha_5\neq-\alpha_6, $ then we have the following subcases:	
		
		\begin{enumerate}
			\item if $ \alpha_2=0, \alpha_1=0, $ then by choosing 
\begin{center}			$x=1,
			q=1, 
			s=\frac{   \alpha_4^2 (2 \alpha_5+\alpha_6)-\alpha_3 (\alpha_5+\alpha_6)^2}{2 \alpha_5 (\alpha_5+\alpha_6)^2},
			r=-\frac{\alpha_4}{\alpha_5+\alpha_6}, t=0,$ \end{center} 
			we have the family of representatives $ \left\langle \nabla_5+\alpha\nabla_6 \right\rangle_{\alpha\neq0,-1}, $ which will be jointed with  the cases (\ref{14cs2}) and (\ref{14cs1});
			
			\item if $ \alpha_2=0, \alpha_1\neq0, $ then by choosing 
	\begin{center}
	    $ x=1, q=\sqrt{\frac{\alpha_1}{\alpha_5}},$ $r=-\frac{\alpha_4\sqrt{\alpha_1}}{(\alpha_5+\alpha_6)\sqrt{\alpha_5}},$ $s=\frac{(\alpha_4^2(2\alpha_5+\alpha_6)-\alpha_3(\alpha_5+\alpha_6)^2)\sqrt{\alpha_1}}{2\alpha_5(\alpha_5+\alpha_6)^2\sqrt{\alpha_5}}, t=0, $
	\end{center}		 we have the family of representatives $ \left\langle \nabla_1+\nabla_5+\alpha\nabla_6 \right\rangle_{\alpha\neq0,-1}, $ which will be jointed with  the cases (\ref{14cs2}) and (\ref{14ccs1});
			
			\item if $ \alpha_2\neq0, $ then by choosing 
			\begin{center}
			    $ x=\alpha_5, q=\alpha_2,$ $r=-\frac{\alpha_2\alpha_4}{\alpha_5+\alpha_6},$ $s=\frac{\alpha_2(\alpha_4^2(2\alpha_5+\alpha_6)-\alpha_3(\alpha_5+\alpha_6)^2)}{2\alpha_5(\alpha_5+\alpha_6)^2}, t=-\frac{\alpha_1\alpha_5}{2\alpha_2},$ 
			\end{center} we have the family of representatives $ \left\langle \nabla_2+\nabla_5+\alpha\nabla_6 \right\rangle_{\alpha\neq0,-1}, $
			which will be jointed with  the cases (\ref{14cccs1}) and (\ref{14cccs2}).
		\end{enumerate}

	\end{enumerate}

	\item $ \alpha_8=0, \alpha_7\neq0 $ then by choosing $ r=-\frac{\alpha_5}{\alpha_7}q, s=\frac{\alpha_5\alpha_6-\alpha_4\alpha_7}{\alpha_7^2}q,$ we have $ \alpha_4^*=\alpha_5^*=0. $ Therefore, we can suppose that $ \alpha_4=0, \alpha_5=0, $ thus we have
	
	\begin{enumerate}
		\item if $ \alpha_2=0, \alpha_1=0, \alpha_3=0, $ then we have the representatives $ \left\langle \nabla_7 \right\rangle  $ and $ \left\langle \nabla_6+\nabla_7\right\rangle  $ depending on whether $ \alpha_6=0 $ or not;
		
		\item if $ \alpha_2=0, \alpha_1=0, \alpha_3\neq0, $ then by choosing $  x=\sqrt[3]{\frac{\alpha_3}{\alpha_7}}, q=1, r=0, s=0, t=0, $ we have the family of representatives $ \left\langle \nabla_3+\alpha\nabla_6+\nabla_7\right\rangle; $
		
		\item if $ \alpha_2=0, \alpha_1\neq0, \alpha_3=0, $ then we have the representatives $ \left\langle \nabla_1+\nabla_7\right\rangle  $ and $ \left\langle \nabla_1+\nabla_6+\nabla_7\right\rangle $ depending on whether $\alpha_6=0$ or not;
		
		\item if $ \alpha_2=0, \alpha_1\neq0, \alpha_3\neq0, $ then by choosing  \begin{center}$x=\sqrt[3]{\frac{\alpha_3}{\alpha_7}}, q=\sqrt[6]{\frac{\alpha^3_1}{\alpha_3\alpha^2_7}}, r=0, s=0, t=0,  $
		\end{center} we have the family of representatives $ \left\langle \nabla_1+\nabla_3+\alpha\nabla_6+\nabla_7\right\rangle; $				
		
		\item if $ \alpha_2\neq0, \alpha_3=0, $ then by choosing $ q=\frac{\alpha_2}{\alpha_7}, r=0, s=0, t=-\frac{\alpha_1}{2\alpha_2}x,$ we have the representatives $ \left\langle \nabla_2+\nabla_7 \right\rangle  $ and $ \left\langle \nabla_2+\nabla_6+\nabla_7\right\rangle $	depending on whether $ \alpha_6=0 $ or not;
		
		\item if $ \alpha_2\neq0, \alpha_3\neq0, $ then by choosing $  x=\sqrt[3]{\frac{\alpha_3}{\alpha_7}}, q=\frac{\alpha_2}{\alpha_7}, r=0, s=0, t=-\frac{\alpha_1\sqrt[3]{\alpha_3}}{2\alpha_2\sqrt[3]{\alpha_7}} $ we have the family of representatives $ \left\langle \nabla_2+\nabla_3+\alpha\nabla_6+\nabla_7\right\rangle. $		
				
	\end{enumerate}

	\item  $ \alpha_8\neq0 $ then by choosing $ r=-\frac{\alpha_7}{\alpha_8}q, s=\frac{\alpha_7^2-\alpha_5\alpha_8}{\alpha_8^2}q, t=-\frac{\alpha_2}{\alpha_8}x, $ we have $ \alpha_2^*=\alpha_5^*=\alpha_7^*=0. $ Therefore, we can suppose that $ \alpha_2=0, \alpha_5=0, \alpha_7=0,$ then we have
	
	\begin{enumerate}
		\item if $\alpha_1=0, \alpha_3=0, \alpha_4=0,$ then we have the representatives $ \left\langle \nabla_8 \right\rangle  $ and $ \left\langle \nabla_6+\nabla_8 \right\rangle  $ depending on whether $ \alpha_6=0 $ or not;
		
		\item if $\alpha_1=0, \alpha_3=0, \alpha_4\neq0,$ then by choosing $ x=\sqrt[3]{\frac{\alpha_4}{\alpha_8}}, q=1, r=0, s=0, t=0, $ we have the family of representatives  $ \left\langle \nabla_4+\alpha\nabla_6+\nabla_8 \right\rangle;$
		
		\item if $\alpha_1=0, \alpha_3\neq0,$ then by choosing $ x=\sqrt[4]{\frac{\alpha_3}{\alpha_8}}, q=1, r=0, s=0, t=0, $ we have the family of representatives $ \left\langle \nabla_3+\alpha\nabla_4+\beta\nabla_6+\nabla_8 \right\rangle; $			
		
		\item if $\alpha_1\neq0, \alpha_3=0, \alpha_4=0,$ then we have the representatives $ \left\langle \nabla_1+\nabla_8 \right\rangle  $ and $ \left\langle \nabla_1+\nabla_6+\nabla_8 \right\rangle  $ depending on whether $ \alpha_6=0 $ or not;
		
		\item if $\alpha_1\neq0, \alpha_3=0, \alpha_4\neq0,$ then by choosing $ x=\sqrt[3]{\frac{\alpha_4}{\alpha_8}}, q=\sqrt[6]{\frac{\alpha_1^3}{\alpha_4^2\alpha_8}}, r=0, s=0, t=0, $ we have the family of representative $ \left\langle \nabla_1+\nabla_4+\alpha\nabla_6+\nabla_8 \right\rangle; $		
		
		\item if $\alpha_1\neq0, \alpha_3\neq0,$ then by choosing $ x=\sqrt[4]{\frac{\alpha_3}{\alpha_8}}, q=\frac{\sqrt{\alpha_1}}{\sqrt[4]{\alpha_3\alpha_8}}, r=0, s=0, t=0, $ we have the family of representative $ \left\langle \nabla_1+\nabla_3+\alpha\nabla_4+\beta\nabla_6+\nabla_8 \right\rangle. $

	\end{enumerate}

\end{enumerate}

Summarizing, we have the following distinct orbits:
\begin{center}

$\left\langle \nabla_1+ \nabla_3+\alpha\nabla_4+\beta\nabla_6+\nabla_8 \right\rangle^{O(\alpha,\beta)=O(i\alpha,-\beta)=O(-i\alpha,-\beta)=O(-\alpha,\beta)},$
$\left\langle \nabla_1+\nabla_3+\alpha\nabla_6 +\nabla_7 \right\rangle^{O(\alpha)=O(-\eta_3\alpha)=O(\eta_3^2\alpha)},$
$\left\langle \nabla_1+ \nabla_4+\nabla_5-\nabla_6 \right\rangle,$
$\left\langle \nabla_1+ \nabla_4+\alpha\nabla_6+\nabla_8 \right\rangle^{O(\alpha)=O(-\eta_3 \alpha)=O(\eta_3^2 \alpha)},$
$\left\langle \nabla_1+\nabla_5+\alpha\nabla_6 \right\rangle,$  
$\left\langle \nabla_1+\nabla_6 +\nabla_7 \right\rangle,$
$\left\langle \nabla_1+\nabla_6+\nabla_8 \right\rangle,$ 
$\left\langle \nabla_1+\nabla_7 \right\rangle,$
$\left\langle \nabla_1+\nabla_8 \right\rangle,$
$\left\langle \nabla_2 \right\rangle,$ 
$\left\langle \nabla_2+\nabla_3 \right\rangle,$
$\left\langle \nabla_2+ \nabla_3+\nabla_6 \right\rangle,$
$\left\langle \nabla_2+\nabla_3+\alpha\nabla_6+\nabla_7 \right\rangle^{O(\alpha)=O(-\eta_3\alpha)=O(\eta_3^2\alpha)},$
$\left\langle \nabla_2+\nabla_4 \right\rangle,$
$\left\langle \nabla_2+\nabla_4+\nabla_5-\nabla_6 \right\rangle,$
$\left\langle \nabla_2+\nabla_5+\alpha\nabla_6 \right\rangle,$
$\left\langle \nabla_2+\nabla_6 \right\rangle,$
$\left\langle \nabla_2+ \nabla_6+\nabla_7 \right\rangle,$
$\left\langle \nabla_2+\nabla_7 \right\rangle,$
$\left\langle \nabla_3+\alpha\nabla_4+\beta\nabla_6+\nabla_8 \right\rangle^{O(\alpha,\beta)=O(i\alpha,-\beta)=O(-i\alpha,-\beta)=O(-\alpha,\beta)},$ 
$\left\langle \nabla_3+\alpha\nabla_6+\nabla_7 \right\rangle^{O(\alpha)=O(-\eta_3 \alpha)=O(\eta_3^2 \alpha)},$
$
\left\langle \nabla_4+\nabla_5-\nabla_6 \right\rangle,$
$\left\langle \nabla_4+\alpha\nabla_6+\nabla_8 \right\rangle^{O(\alpha)=O(-\eta_3 \alpha)=O(\eta_3^2 \alpha)}, $
$ \left\langle \nabla_5+\alpha\nabla_6 \right\rangle,$ 
$\left\langle \nabla_6+\nabla_7 \right\rangle,$ 
$\left\langle \nabla_6+\nabla_8 \right\rangle,$ 
$\left\langle \nabla_7 \right\rangle,$ 
$\left\langle \nabla_8 \right\rangle, $

\end{center}
which gives the following new algebras:

\begin{longtable}{llllllllllllllllll}

${\mathbf{N}}_{188}^{\alpha, \beta}$ & $:$ &$e_1e_1=e_5$ & $e_1e_2=e_3$ & $e_1e_3=e_4$ & $e_2e_2=e_5$ \\ && $e_2e_3=\alpha e_5$ & $e_3e_3=\beta e_5$ 
  & $e_4e_4=e_5$
\\

${\mathbf{N}}_{189}^{\alpha}$ & $:$ &$e_1e_1=e_5$ & $e_1e_2=e_3$ & $e_1e_3=e_4$ \\&& $e_2e_2=e_5$ & $e_3e_3=\alpha e_5$ & $e_3e_4=e_5$
\\

${\mathbf{N}}_{190}$ & $:$ &$e_1e_1=e_5$ & $e_1e_2=e_3$ & $e_1e_3=e_4$ \\&& $e_2e_3=e_5$ & $e_2e_4=e_5$ & $e_3e_3=-e_5$ 
\\

${\mathbf{N}}_{191}^{\alpha}$ & $:$ &$e_1e_1=e_5$ & $e_1e_2=e_3$ & $e_1e_3=e_4$ \\&& 
$e_2e_3=e_5$ & $e_3e_3=\alpha e_5$ & $e_4e_4=e_5$
\\

${\mathbf{N}}_{191}^{\alpha}$ & $:$ &$e_1e_1=e_5$ & $e_1e_2=e_3$ & $e_1e_3=e_4$ & $e_2e_4=e_5$ & $e_3e_3=\alpha e_5$ \\

${\mathbf{N}}_{192}$ & $:$ &$e_1e_1=e_5$ & $e_1e_2=e_3$ & $e_1e_3=e_4$ & $e_3e_3=e_5$ & $e_3e_4=e_5$
\\

${\mathbf{N}}_{193}$ & $:$ &$e_1e_1=e_5$ & $e_1e_2=e_3$ & $e_1e_3=e_4$ & $e_3e_3=e_5$ & $e_4e_4=e_5$
\\

${\mathbf{N}}_{194}$ & $:$ &$e_1e_1=e_5$ & $e_1e_2=e_3$ & $e_1e_3=e_4$ & $e_3e_4=e_5$
\\

${\mathbf{N}}_{195}$ & $:$ &$e_1e_1=e_5$ & $e_1e_2=e_3$ & $e_1e_3=e_4$ & $e_4e_4=e_5$
\\

${\mathbf{N}}_{196}$ & $:$ & $e_1e_2=e_3$ & $e_1e_3=e_4$ & $e_1e_4=e_5$
\\

${\mathbf{N}}_{197}$ & $:$ &$e_1e_2=e_3$ & $e_1e_3=e_4$ & $e_1e_4=e_5$ & $e_2e_2=e_5$
\\

${\mathbf{N}}_{198}$ & $:$ &$e_1e_2=e_3$ & $e_1e_3=e_4$ & $e_1e_4=e_5$ & $e_2e_2=e_5$ & $e_3e_3=e_5$
\\

${\mathbf{N}}_{199}^{\alpha}$ & $:$ &$e_1e_2=e_3$ & $e_1e_3=e_4$ & $e_1e_4=e_5$ \\&& $e_2e_2=e_5$ & $e_3e_3=\alpha e_5$ & $e_3e_4=e_5$\\

${\mathbf{N}}_{200}$ & $:$ &$e_1e_2=e_3$ & $e_1e_3=e_4$ & $e_1e_4=e_5$ & $e_2e_3=e_5$\\

${\mathbf{N}}_{201}$ & $:$ &$e_1e_2=e_3$ & $e_1e_3=e_4$ & $e_1e_4=e_5$\\&& $e_2e_3=e_5$ & $e_2e_4=e_5$ & $e_3e_3=-e_5$\\

${\mathbf{N}}_{202}^{\alpha}$ & $:$ &$e_1e_2=e_3$ & $e_1e_3=e_4$ & $e_1e_4=e_5$ & $e_2e_4=e_5$ & $e_3e_3=\alpha e_5$ \\

${\mathbf{N}}_{203}$ & $:$ &$e_1e_2=e_3$ & $e_1e_3=e_4$ & $e_1e_4=e_5$  & $e_3e_3=e_5$\\

${\mathbf{N}}_{204}$ & $:$ &$e_1e_2=e_3$ & $e_1e_3=e_4$ & $e_1e_4=e_5$  & $e_3e_3=e_5$ & $ e_3e_4=e_5$ \\

${\mathbf{N}}_{205}$ & $:$ &$e_1e_2=e_3$ & $e_1e_3=e_4$ & $e_1e_4=e_5$  & $e_3e_4=e_5$\\

${\mathbf{N}}_{206}^{\alpha, \beta}$ & $:$ &$e_1e_2=e_3$ & $e_1e_3=e_4$ & $e_2e_2=e_5$ \\&& $e_2e_3=\alpha e_5$ & $e_3e_3=\beta e_5$ & $e_4e_4=e_5$
\\

${\mathbf{N}}_{207}^{\alpha}$ & $:$ &$e_1e_2=e_3$ & $e_1e_3=e_4$ & $e_2e_2=e_5$ & $e_3e_3=\alpha e_5$ & $e_3e_4=e_5$
\\

${\mathbf{N}}_{208}$ & $:$ &$e_1e_2=e_3$ & $e_1e_3=e_4$ & $e_2e_3=e_5$ & $e_2e_4=e_5$ & $e_3e_3=-e_5$
\\

${\mathbf{N}}_{209}^{\alpha}$ & $:$ &$e_1e_2=e_3$ & $e_1e_3=e_4$ & $e_2e_3=e_5$ & $e_3e_3=\alpha e_5$ & $e_4e_4=e_5$
\\

${\mathbf{N}}_{210}^{\alpha}$ & $:$ &$e_1e_2=e_3$ & $e_1e_3=e_4$ & $e_2e_4=e_5$ & $e_3e_3=\alpha e_5$
\\

${\mathbf{N}}_{211}$ & $:$ &$e_1e_2=e_3$ & $e_1e_3=e_4$ & $e_3e_3=e_5$ & $e_3e_4=e_5$
\\

${\mathbf{N}}_{212}$ & $:$ &$e_1e_2=e_3$ & $e_1e_3=e_4$ & $e_3e_3=e_5$ & $e_4e_4=e_5$
\\

${\mathbf{N}}_{213}$ & $:$ &$e_1e_2=e_3$ & $e_1e_3=e_4$ & $e_3e_4=e_5$
\\

${\mathbf{N}}_{214}$ & $:$ &$e_1e_2=e_3$ & $e_1e_3=e_4$ & $e_4e_4=e_5$
\\

\end{longtable}

\subsection{$ 1 $-dimensional central extensions of $ {\mathbf N}_{15}^{4*} $.} Here we will collect all information about $ {\mathbf N}_{15}^{4*}: $
$$
\begin{array}{|l|l|l|l|}
%\hline
%\text{ }  & \text{ } & \text{Cohomology} &\text{Automorphisms} \\
\hline
{\mathbf{N}}^{4*}_{15} &  
\begin{array}{l}
e_1e_2=e_3 \\ 
e_1e_3=e_4 \\ 
e_2e_2=e_4
\end{array}
&
\begin{array}{lcl}
\mathrm{H}^2_{\mathfrak{D}}(\mathbf{N}^{4*}_{15})&=&\\
\multicolumn{3}{r}{\langle [\Delta_{11}],[\Delta_{22}],[\Delta_{23}],[\Delta_{33}]\rangle}\\
\mathrm{H}^2_{\mathfrak{C}}(\mathbf{N}^{4*}_{15})=\mathrm{H}^2_{\mathfrak{D}}(\mathbf{N}^{4*}_{15})\oplus\\
\multicolumn{3}{r}{\langle [\Delta_{14}], [\Delta_{24}], [\Delta_{34}], [\Delta_{44}] \rangle }
\end{array}  
& 
\phi=\begin{pmatrix}
x&0&0&0\\
0&x^2&0&0\\
0&r&x^3&0\\
t&s&xr&x^4
\end{pmatrix}\\
\hline
\end{array}$$

Let us use the following notations:
\begin{longtable}{llll}
$\nabla_1=[\Delta_{11}],$&$ \nabla_2=[\Delta_{14}], $&$ \nabla_3=[\Delta_{22}], $&$  \nabla_4=[\Delta_{23}],$\\
$\nabla_5=[\Delta_{24}],  $&$ \nabla_6=[\Delta_{33}],  $&$  \nabla_7=[\Delta_{34}], $&$ \nabla_8=[\Delta_{44}]. $
\end{longtable}
Take $ \theta=\sum\limits_{i=1}^{8}\alpha_i\nabla_i\in\mathrm{H}^2_{\mathfrak{C}}(\mathbf{N}^{4*}_{15}) .$  Since

$$\phi^T\begin{pmatrix}
\alpha_1&0&0&\alpha_2\\
0&\alpha_3&\alpha_4&\alpha_5\\
0&\alpha_4&\alpha_6&\alpha_7\\
\alpha_2&\alpha_5&\alpha_7&\alpha_8

\end{pmatrix}\phi=
\begin{pmatrix}
\alpha_1^*&\alpha^{*}&\alpha^{**}&\alpha_2^*\\
\alpha^{*}&\alpha_3^*+\alpha^{**}&\alpha^*_4&\alpha^*_5\\
\alpha^{**}&\alpha^*_4&\alpha^*_6&\alpha^*_7\\
\alpha^*_2&\alpha^*_5&\alpha^*_7&\alpha^*_8
\end{pmatrix},$$
we have
\begin{longtable}{lcl}
$\alpha_1^*$&$=$&$\alpha_1x^2+2\alpha_2xt+\alpha_8t^2,$\\
$\alpha_2^*$&$=$&$(\alpha_2x+\alpha_8t)x^4,$\\
$\alpha_3^*$&$=$&$x^4 \alpha_3+2 r x^2 \alpha_4+2 s x^2 \alpha_5+r^2 \alpha_6+2 r s \alpha_7+$\\
&&\multicolumn{1}{r}{$s^2 \alpha_8-x (r x \alpha_2+t x^2 \alpha_7+r t \alpha_8),$}\\
$\alpha_4^*$&$=$&$(\alpha_4x^2+\alpha_6r+\alpha_7s)x^3+(\alpha_5x^2+\alpha_7r+\alpha_8s)xr,$\\
$\alpha_5^*$&$=$&$(\alpha_5x^2+\alpha_7r+\alpha_8s)x^4,$\\
$\alpha_6^*$&$=$&$(\alpha_6x^4+2\alpha_7x^2r+\alpha_8r^2)x^2,$\\
$\alpha_7^*$&$=$&$(\alpha_7x^2+\alpha_8r)x^5,$\\
$\alpha_8^*$&$=$&$\alpha_8x^8.$
\end{longtable}

We are interested in $ (\alpha_2,\alpha_5,\alpha_7,\alpha_8)\neq(0,0,0,0) .$ Let us consider the following cases:

\begin{enumerate}
	\item  $ \alpha_8=0, \alpha_7=0, \alpha_5=0, $ then $ \alpha_2\neq0 $ and we have
	
	\begin{enumerate}
		\item\label{15cs1a} if $ \alpha_6=0, \alpha_4=0, $ then by choosing $ x=2\alpha_2, r=4\alpha_2\alpha_3, s=0,  t=-\alpha_1, $ we have the representative $ \left\langle \nabla_2 \right\rangle; $
		
		\item\label{15cs1b} if $ \alpha_6=0, \alpha_4\neq0, \alpha_2=2\alpha_4,$ then  we have the representatives $ \left\langle 2\nabla_2+\nabla_4 \right\rangle $ and  $ \left\langle 2\nabla_2+\nabla_3+\nabla_4 \right\rangle $ depending on whether $ \alpha_3=0 $ or not;

		\item if $ \alpha_6=0, \alpha_4\neq0, \alpha_2\neq2\alpha_4,$ then by choosing $ x=\alpha_2-2\alpha_4, r=\alpha_3(\alpha_2-2\alpha_4), s=0, t=\frac{\alpha_1(2\alpha_4-\alpha_2)}{2\alpha_2}, $ we have the family of representatives $ \left\langle \nabla_2+\alpha\nabla_4 \right\rangle_{\alpha\neq0,\frac{1}{2}}, $ which will be jointed with representatives from the cases 	(\ref{15cs1a}) and (\ref{15cs1b});

		\item  if $ \alpha_6\neq0,$ then by choosing $ x=\frac{\alpha_2}{\alpha_6}, r=-\frac{\alpha_2^2\alpha_4}{\alpha_6^3}, s=0, t=-\frac{\alpha_1}{2\alpha_6}, $ we have the representative $ \left\langle \nabla_2+\alpha\nabla_3+\nabla_6 \right\rangle. $

		\end{enumerate}	
	
	\item  $ \alpha_8=0, \alpha_7=0, \alpha_5\neq0 $ then we have	
	
	\begin{enumerate}

		\item if $ \alpha_5\neq-\alpha_6, $ then we have the following subcases:

\begin{enumerate}
	\item\label{152ai} if $ \alpha_2=0, \alpha_1=0, $ then 
	by choosing
	\begin{center}
	    $x=2 \alpha_5 (\alpha_5+\alpha_6),$
	    $s=2 \alpha_5 (\alpha_4^2 (2 \alpha_5+\alpha_6)-\alpha_3 (\alpha_5+\alpha_6)^2),$
	    $r=-4 \alpha_4 \alpha_5^2 (\alpha_5+\alpha_6),$
	\end{center}
	we have the family of representatives $ \left\langle \nabla_5+\alpha\nabla_6 \right\rangle_{\alpha\neq-1} ;$
	
	\item\label{152aii} if $ \alpha_2=0, \alpha_1\neq0, $ then by choosing 
	\begin{center}$ x=\sqrt[4]{\frac{\alpha_1}{\alpha_5}},$ $r=-\frac{\alpha_4\sqrt{\alpha_1}}{(\alpha_5+\alpha_6)\sqrt{\alpha_5}},$ $s=\frac{((\alpha_5+\alpha_6)(2\alpha_4^2-\alpha_2\alpha_4)-\alpha_3(\alpha_5+\alpha_6)^2-\alpha_4^2\alpha_6)\sqrt{\alpha_1}}{2\alpha_5(\alpha_5+\alpha_6)^2\sqrt{\alpha_5}}, t=0, $
	\end{center} 
	we have the family of representatives $ \left\langle \nabla_1+\nabla_5+\alpha\nabla_6 \right\rangle_{\alpha\neq-1} ; $ 			
	
	\item\label{152aiii} if $ \alpha_2\neq0, $ then by choosing 
	\begin{center}
	    $ x=\frac{\alpha_2}{\alpha_5},$ $r=-\frac{\alpha_2^2\alpha_4}{\alpha_5^2(\alpha_5+\alpha_6)},$ $s=\frac{\alpha_2^2((\alpha_5+\alpha_6)(2\alpha_4^2-\alpha_2\alpha_4)-\alpha_3(\alpha_5+\alpha_6)^2-\alpha_4^2\alpha_6)}{2\alpha^3_5(\alpha_5+\alpha_6)^2}, t=-\frac{\alpha_1}{2\alpha_5}, $
	\end{center} we have the family of representatives $ \left\langle \nabla_2+\nabla_5+\alpha\nabla_6 \right\rangle_{\alpha\neq-1} . $			
\end{enumerate}			
		\item if $  \alpha_6=-\alpha_5, $ then we have the following subcases:
		
		\begin{enumerate}
			\item if $ \alpha_4=0, \alpha_2=0, \alpha_1=0, $ then we have the representative $ \left\langle \nabla_5-\nabla_6 \right\rangle,$
			which will be jointed with the family from the case (\ref{152ai});
			
			\item if $ \alpha_4=0, \alpha_2=0, \alpha_1\neq0, $ then by choosing $ x=\sqrt[4]{\frac{\alpha_1}{\alpha_5}}, r=0, s=-\frac{\alpha_3\sqrt{\alpha_1}}{2\alpha_5\sqrt{\alpha_4}}, t=0,$ we have the representative  $ \left\langle \nabla_1+\nabla_5-\nabla_6 \right\rangle,$
			which will be jointed with the family from the case (\ref{152aii});
			
			\item if $ \alpha_4=0, \alpha_2\neq0, $ then by choosing $ x=\frac{\alpha_2}{\alpha_5}, r=0, s=-\frac{\alpha^2_2\alpha_3}{2\alpha^3_5}, t=-\frac{\alpha_1}{2\alpha_5},$ we have the representative $ \left\langle \nabla_2+\nabla_5-\nabla_6 \right\rangle,$
			which will be jointed with the family from the case (\ref{152aiii});
			
		    \item if $ \alpha_4\neq0, $ then by choosing $ x=\frac{\alpha_4}{\alpha_5}, $
		    $s=-\frac{\alpha_3 \alpha_4^2}{ 2 \alpha_5^3},$
		    $r=0,$ we have the families of representatives 
		    \begin{center}$ \left\langle \alpha\nabla_1+\nabla_4+\nabla_5-\nabla_6 \right\rangle $	and $ \left\langle \alpha\nabla_2+\nabla_4+\nabla_5-\nabla_6 \right\rangle_{\alpha\neq0} $
		    \end{center} depending on $\alpha_2=0$ or not.		
		\end{enumerate}		
				
	\end{enumerate}
	
	\item  $ \alpha_8=0, \alpha_7\neq0, $ then by choosing 
	\begin{center}
	    $ r=-\frac{\alpha_5}{\alpha_7}x^2, s=\frac{\alpha_5\alpha_6-\alpha_4\alpha_7}{\alpha_7^2}x^2, t=\frac{\alpha_3\alpha_7^2-2\alpha_4\alpha_5\alpha_7+\alpha_5^2\alpha_6+\alpha_2\alpha_5\alpha_7}{\alpha_7^3}x,$ 
	\end{center} we have $ \alpha_3^*=\alpha_4^*=\alpha_5^*=0. $ Therefore, we can suppose that $ \alpha_3=0, \alpha_4=0, \alpha_5=0, $ and we have
	
	\begin{enumerate}
		\item if $ \alpha_1=0, \alpha_2=0, $ then we have the representatives $ \left\langle \nabla_7 \right\rangle  $ and $ \left\langle \nabla_6+\nabla_7 \right\rangle  $ depending on whether $ \alpha_6=0 $ or not;
		
		\item if $ \alpha_1=0, \alpha_2\neq0, $ then by choosing $ x=\sqrt{\alpha_2\alpha_7^{-1}}, r=0, s=0, t=0, $ we have the family of representative $ \left\langle \nabla_2+\alpha\nabla_6+\nabla_7 \right\rangle; $		
		
		\item if $ \alpha_1\neq0, $ then by choosing $ x=\sqrt[5]{\frac{\alpha_1}{\alpha_7}}, r=0, s=0, t=0, $ we have the family of  representative $ \left\langle \nabla_1+\alpha\nabla_2+\beta\nabla_6+\nabla_7 \right\rangle. $

	\end{enumerate}	

	\item $ \alpha_8\neq0,$ then by choosing $ r=-\frac{\alpha_7}{\alpha_8}x^2, t=-\frac{\alpha_2}{\alpha_8}x, s=-\frac{\alpha_5x^2+\alpha_7r}{\alpha_8}$ we have $ \alpha_2^*=\alpha_5^*=\alpha_7^*=0. $ Therefore, we can suppose that $ \alpha_2=0, \alpha_5=0, \alpha_7=0,$ then we have
	
	\begin{enumerate}
		\item if $ \alpha_1=0, \alpha_3=0, \alpha_4=0, $ then we have the representatives $ \left\langle \nabla_8 \right\rangle  $ and $ \left\langle \nabla_6+\nabla_8 \right\rangle  $ depending on whether $ \alpha_6=0 $ or not;
		
		\item if $ \alpha_1=0, \alpha_3=0, \alpha_4\neq0, $ then by choosing $ x=\sqrt[3]{{\alpha_4}{\alpha_8^{-1}}}, r=0, s=0, t=0, $ we have the family of representative $ \left\langle \nabla_4+\alpha\nabla_6+\nabla_8 \right\rangle; $
		
		\item if $ \alpha_1=0, \alpha_3\neq0, $ then by choosing $ x=\sqrt[4]{{\alpha_3}{\alpha_8^{-1}}}, r=0, s=0, t=0, $ we have the family of representative $ \left\langle \nabla_3+\alpha\nabla_4+\beta\nabla_6+\nabla_8 \right\rangle; $				
		
		\item if $ \alpha_1\neq0, $ then by choosing $ x=\sqrt[6]{{\alpha_1}{\alpha_8^{-1}}}, r=0, s=0, t=0, $ we have the family of representative $ \left\langle \nabla_1+\alpha\nabla_3+\beta\nabla_4+\gamma\nabla_6+\nabla_8 \right\rangle. $		
		
	\end{enumerate}

\end{enumerate}

Summarizing all cases we have the following distinct orbits:
\begin{center}

$\left\langle \nabla_1+\alpha\nabla_2+\beta\nabla_6+\nabla_7 \right\rangle
^{O(\alpha,\beta)=O(\eta_5^2\alpha,-\eta_5\beta)=O(\eta_5^4\alpha,\eta_5^2\beta)=
O(-\eta_5\alpha,-\eta_5^3\beta)=O(-\eta_5^3\alpha,\eta_5^4\beta)},$
$\left\langle \nabla_1+\alpha\nabla_3+\beta\nabla_4 +\gamma\nabla_6+\nabla_8 \right\rangle^{
{\tiny \begin{array}{l}O(\alpha,\beta,\gamma)=O(-\eta_3\alpha,\beta,\eta_3^2\gamma)=
O(-\eta_3\alpha,-\beta,\eta_3^2\gamma)=\\
O(\eta_3^2\alpha,-\beta,-\eta_3\gamma)=
O(\eta_3^2\alpha,\beta,-\eta_3\gamma)=
O(\alpha,-\beta,\gamma)
\end{array}}},$
$\left\langle \alpha\nabla_1+ \nabla_4+\nabla_5-\nabla_6 \right\rangle,$
$\left\langle \nabla_1+\nabla_5+\alpha\nabla_6 \right\rangle,$  
$\left\langle 2\nabla_2+\nabla_3 +\nabla_4 \right\rangle,$  
$ \left\langle \nabla_2+\alpha\nabla_3+\nabla_6 \right\rangle,$
$\left\langle \nabla_2+\alpha\nabla_4 \right\rangle,$ 
$\left\langle \alpha\nabla_2+\nabla_4+\nabla_5-\nabla_6 \right\rangle_{\alpha\neq0},$
$\left\langle \nabla_2+\nabla_5+\alpha\nabla_6 \right\rangle,$ 
$\left\langle \nabla_2+\alpha\nabla_6+\nabla_7 \right\rangle^{O(\alpha)=O(-\alpha)},$
$\left\langle \nabla_3+\alpha\nabla_4+\beta\nabla_6+\nabla_8 \right\rangle
^{O(\alpha,\beta)=O(i\alpha,-\beta)=
O(-i\alpha,-\beta)=
O(-\alpha,\beta)},$
$ \left\langle \nabla_4+\alpha\nabla_6+\nabla_8 \right\rangle^{O(\alpha)=O(-\eta_3\alpha)=
O(\eta_3^2\alpha)},$
$\left\langle \nabla_5+\alpha\nabla_6 \right\rangle, $
$\left\langle \nabla_6+ \nabla_7 \right\rangle,$ 
$\left\langle \nabla_6+\nabla_8 \right\rangle,$ 
$\left\langle \nabla_7 \right\rangle,$ 
$\left\langle \nabla_8 \right\rangle,  $

\end{center}
which gives the following new algebras:

\begin{longtable}{llllllllllllllllll}

${\mathbf{N}}_{215}^{\alpha, \beta}$ & $:$ &$e_1e_1=e_5$ & $e_1e_2=e_3$ & $e_1e_3=e_4$ & $e_1e_4=\alpha e_5$  \\ && $e_2e_2=e_4$ & $e_3e_3=\beta e_5$ 
  & $e_3e_4=e_5$
\\

${\mathbf{N}}_{216}^{\alpha, \beta ,\gamma}$ & $:$ &$e_1e_1=e_5$ & $e_1e_2=e_3$ & $e_1e_3=e_4$ & $e_2e_2=e_4+\alpha e_5$  \\ && $e_2e_3=\beta e_5$ 
 & $e_3e_3=\gamma e_5$  & $e_4e_4=e_5$
\\

${\mathbf{N}}_{217}^{\alpha}$ & $:$ &$e_1e_1=\alpha e_5$ & $e_1e_2=e_3$ & $e_1e_3=e_4$ & $e_2e_2=e_4$ \\ & & $e_2e_3=e_5$ & $e_2e_4=e_5$ 
  & $e_3e_3=-e_5$
\\

${\mathbf{N}}_{218}^\alpha$ & $:$ &$e_1e_1=e_5$ & $e_1e_2=e_3$ & $e_1e_3=e_4$\\& & $e_2e_2=e_4$ & $e_2e_4=e_5$ & $e_3e_3=\alpha e_5$ 
\\

${\mathbf{N}}_{219}$ & $:$ &$e_1e_2=e_3$ & $e_1e_3=e_4$ & $e_1e_4=2e_5$ & $e_2e_2=e_4+e_5$ & $ e_2e_3=e_5$ 
\\

${\mathbf{N}}_{220}^{\alpha}$ & $:$ &$e_1e_2=e_3$ & $e_1e_3=e_4$ & $e_1e_4=e_5$  & $e_2e_2=e_4+\alpha e_5$ & $e_3e_3=e_5$ \\

${\mathbf{N}}_{221}^{\alpha}$ & $:$ &$e_1e_2=e_3$ & $e_1e_3=e_4$ & $e_1e_4=e_5$ & $e_2e_2=e_4$ & $e_2e_3=\alpha e_5$ \\

${\mathbf{N}}_{222}^{\alpha\neq 0}$ & $:$ &$e_1e_2=e_3$ & $e_1e_3=e_4$ & $e_1e_4=\alpha e_5$ & $e_2e_2=e_4$ \\ & & $e_2e_3=e_5$ & $e_2e_4=e_5 $ 
  & $ e_3e_3=-e_5$ \\

${\mathbf{N}}_{223}^{\alpha}$ & $:$ &$e_1e_2=e_3$ & $e_1e_3=e_4$ & $e_1e_4=e_5$\\&& $e_2e_2=e_4$ & $e_2e_4=e_5$ & $ e_3e_3=\alpha e_5$ \\

${\mathbf{N}}_{224}^{\alpha}$ & $:$ &$e_1e_2=e_3$ & $e_1e_3=e_4$ & $e_1e_4=e_5$ \\&& $e_2e_2=e_4$ & $e_3e_3=\alpha e_5$ & $ e_3e_4=e_5$\\

${\mathbf{N}}_{225}^{\alpha, \beta}$ & $:$ &$e_1e_2=e_3$ & $e_1e_3=e_4$ & \multicolumn{2}{l}{$e_2e_2=e_4+e_5$} \\&& $e_2e_3=\alpha e_5$ & $e_3e_3=\beta e_5$ & $ e_4e_4=e_5$\\

${\mathbf{N}}_{226}^{\alpha}$ & $:$ &$e_1e_2=e_3$ & $e_1e_3=e_4$ & $e_2e_2=e_4$ \\&& $e_2e_3=e_5$ & $e_3e_3=\alpha e_5$ & $ e_4e_4=e_5$\\

${\mathbf{N}}_{227}^{\alpha}$ & $:$ &$e_1e_2=e_3$ & $e_1e_3=e_4$ & $e_2e_2=e_4$ & $e_2e_4=e_5$ & $ e_3e_3=\alpha e_5$\\

${\mathbf{N}}_{228}$ & $:$ &$e_1e_2=e_3$ & $e_1e_3=e_4$  & $e_2e_2=e_4$ & $e_3e_3=e_5$ & $ e_3e_4=e_5$\\

${\mathbf{N}}_{229}$ & $:$ &$e_1e_2=e_3$ & $e_1e_3=e_4$  & $e_2e_2=e_4$ & $e_3e_3=e_5$ & $ e_4e_4=e_5$\\

${\mathbf{N}}_{230}$ & $:$ &$e_1e_2=e_3$ & $e_1e_3=e_4$  & $e_2e_2=e_4$ & $ e_3e_4=e_5$\\

${\mathbf{N}}_{231}$ & $:$ &$e_1e_2=e_3$ & $e_1e_3=e_4$  & $e_2e_2=e_4$ & $ e_4e_4=e_5$\\

\end{longtable}

\subsection{$ 1 $-dimensional central extensions of $ {\mathbf N}_{16}^{4*} $.} Here we will collect all information about $ {\mathbf N}_{16}^{4*}: $
\begin{longtable}{|l|l|l|}
%\hline
%\text{ }  & \text{ } & \text{Cohomology} \\
\hline
${\mathbf{N}}^{4*}_{16}$ &  
$\begin{array}{l}
e_1e_2=e_3 \\ 
e_1e_3=e_4 \\ 
e_2e_3=e_4
\end{array}$
&
$\begin{array}{l}
\mathrm{H}^2_{\mathfrak{D}}(\mathbf{N}^{4*}_{16})=\langle [\Delta_{11}],[\Delta_{22}],[\Delta_{23}],[\Delta_{33}]\rangle \\
\mathrm{H}^2_{\mathfrak{C}}(\mathbf{N}^{4*}_{16})=\mathrm{H}^2_{\mathfrak{D}}(\mathbf{N}^{4*}_{16})\oplus
\langle [\Delta_{14}], [\Delta_{24}], [\Delta_{34}], [\Delta_{44}]\rangle 
\end{array}$\\
\hline
\multicolumn{3}{|c|}{
 $ \phi_1=\begin{pmatrix}
x&0&0&0\\
0&x&0&0\\
0&0&x^2&0\\
t&s&0&x^3
\end{pmatrix}, \phi_2=\begin{pmatrix}
0&y&0&0\\
y&0&0&0\\
0&0&y^2&0\\
t&s&0&y^3
\end{pmatrix}$}\\
\hline

\end{longtable}

Let us use the following notations:
\begin{longtable}{llll}
$\nabla_1=[\Delta_{11}], $&$ \nabla_2=[\Delta_{14}], $&$ \nabla_3=[\Delta_{22}], $&$  \nabla_4=[\Delta_{23}],$\\
$\nabla_5=[\Delta_{24}],  $&$ \nabla_6=[\Delta_{33}],  $&$  \nabla_7=[\Delta_{34}], $&$ \nabla_8=[\Delta_{44}]. $
\end{longtable}
Take $ \theta=\sum\limits_{i=1}^{8}\alpha_i\nabla_i\in\mathrm{H}^2_{\mathfrak{C}}(\mathbf{N}^{4*}_{16}) .$  Since

$$\phi^T\begin{pmatrix}
\alpha_1&0&0&\alpha_2\\
0&\alpha_3&\alpha_4&\alpha_5\\
0&\alpha_4&\alpha_6&\alpha_7\\
\alpha_2&\alpha_5&\alpha_7&\alpha_8

\end{pmatrix}\phi=
\begin{pmatrix}
\alpha_1^*&\alpha^{*}&\alpha^{**}&\alpha_2^*\\
\alpha^{*}&\alpha^*_3&\alpha^*_4+\alpha^{**}&\alpha_5^*\\
\alpha^{**}&\alpha^*_4+\alpha^{**}&\alpha_6^*&\alpha^*_7\\
\alpha^*_2&\alpha^*_5&\alpha^*_7&\alpha^*_8
\end{pmatrix},$$

in the case $\phi=\phi_1$, we have
\begin{longtable}{ll}
$\alpha_1^*=\alpha_1x^2+2\alpha_2xt+\alpha_8t^2,$&
$\alpha_2^*=(\alpha_2x+\alpha_8t)x^3,$\\

$\alpha_3^*=\alpha_3x^2+2\alpha_5xs+\alpha_8s^2,$&
$\alpha_4^*=(\alpha_4x+\alpha_7s)x^2-\alpha_7x^2t,$\\

$\alpha_5^*=(\alpha_5x+\alpha_8s)x^3,$&
$\alpha_6^*=\alpha_6x^4,$\\

$\alpha_7^*=\alpha_7x^5,$&
$\alpha_8^*=\alpha_8x^6;$
\end{longtable}
and on the opposite case, for $\phi=\phi_2,$ we have

\begin{longtable}{ll}
$\alpha_1^*=   \alpha_3y^2+2   \alpha_5 t y+  \alpha_8 t^2,$&
$\alpha_2^*=  (\alpha_5 y +  \alpha_8t )y^3,$\\

$\alpha_3^*=  \alpha_1y^2+2  \alpha_2s y +  \alpha_8s^2,$&
$\alpha_4^*= ((s-t)  \alpha_7-y  \alpha_4)y^2,$\\

$\alpha_5^*= (y  \alpha_2+s  \alpha_8)y^3,$&
$\alpha_6^*=  \alpha_6 y^4,$\\

$\alpha_7^*=  \alpha_7y^5,$&
$\alpha_8^*=\alpha_8 y^6.$
\end{longtable}

We are interested in $ (\alpha_2,\alpha_5,\alpha_7,\alpha_8)\neq(0,0,0,0) .$ Let us consider the following cases:

\begin{enumerate}
	\item \label{caso1} $ \alpha_8=0, \alpha_7=0, \alpha_5=0, $ then $ \alpha_2\neq0$ and
	\begin{enumerate}
	    \item if $\alpha_4\neq 0,$ then by choosing
	    $\phi=\phi_1,$ $x= \alpha_4 \alpha_2^{-1},$ $t=-\frac{ \alpha_1  \alpha_4 }{2 \alpha_2^{2}},$
	    we have the family of representatives 
	     $ \left\langle \nabla_2+\alpha\nabla_3+\nabla_4 +\beta \nabla_6 \right\rangle;$
	     
	     \item if $\alpha_4=0,\alpha_3\neq0,$ then by choosing
	    $\phi=\phi_1,$ $x=\sqrt{ \alpha_3  \alpha_2^{-1}},$ $t=-\frac{ \alpha_1 \sqrt{ \alpha_3}}{2  \sqrt{\alpha_2^3}},$
	     we have the family of representatives 
	     $ \left\langle \nabla_2+\nabla_3+\alpha \nabla_6 \right\rangle;$
	   
	      \item if  $\alpha_4=0,\alpha_3=0,$ then by choosing
	    $\phi=\phi_1,$ $x=2  \alpha_2,$ $t=- \alpha_1,$ $s=0,$
	     we have the family of representatives 
	     $ \left\langle \nabla_2+ \alpha \nabla_6 \right\rangle.$
	     
	\end{enumerate}

	\item  $ \alpha_8=0, \alpha_7=0, \alpha_5\neq0 $ and
	
	\begin{enumerate}
	    \item if $\alpha_2\neq 0,\alpha_4\neq 0,$ then by choosing 
\begin{center}
 $\phi=\phi_1,$   $x= \frac{\alpha_4}{\alpha_5},$ 
	    $t=-\frac{\alpha_1  \alpha_4}{ 2  \alpha_2  \alpha_5},$
	    $s=-\frac{ \alpha_3  \alpha_4}{ 2  \alpha_5^2},$	    
\end{center}	    we have the following family of representatives
\begin{center}    $ \left\langle \alpha \nabla_2+\nabla_4+\nabla_5+ \beta \nabla_6 \right\rangle_{\alpha\neq0};$
\end{center}

	    \item if $\alpha_2\neq 0,\alpha_4= 0,$ then by choosing 
\begin{center}
 $\phi=\phi_1,$   $x=2  \alpha_2  \alpha_5,$ 
	    $t=- \alpha_1  \alpha_5,$
	    $s=- \alpha_2  \alpha_3,$	    
\end{center}	    we have the following family of representatives
    $ \left\langle \alpha \nabla_2+ \nabla_5+ \beta \nabla_6 \right\rangle_{\alpha\neq0};$

\item if $\alpha_2= 0,$ then by choosing  $\phi=\phi_2,$ $y=1,$ $t=0,$ $s=0,$ we have the representative with $\alpha_5^*=0$ and $\alpha_2^*\neq0,$ which was considered above.

	\end{enumerate}

	\item $ \alpha_8=0, \alpha_7\neq0,$ then we have
	
	\begin{enumerate}
		\item if $ \alpha_2=0, \alpha_5=0, \alpha_1=0, \alpha_3=0, $ then we have the representatives $ \left\langle \nabla_7 \right\rangle  $ and $ \left\langle \nabla_6+\nabla_7\right\rangle  $ depending on whether $ \alpha_6=0 $ or not;
		
	       \item if $ \alpha_2=0, \alpha_5=0, \alpha_1\neq0, $ then by choosing 
	\begin{center}
	           $\phi=\phi_1,$ $ x=\sqrt[3]{{\alpha_1}{\alpha_7^{-1}}}, s=0, t={\alpha_4\sqrt[3]{\alpha_1} \alpha_7^{-1}}, $
	\end{center}we have the family of representatives	$ \left\langle \nabla_1+\alpha\nabla_3+\beta\nabla_6+\nabla_7\right\rangle;$

        \item if $ \alpha_2\neq0, $ then by choosing
        \begin{center}
        $\phi=\phi_1,$  $ x={\alpha_2}{\alpha_7^{-1}}, s=-({\alpha_1\alpha_7+2\alpha_2\alpha_4}) /(2\alpha_7^2), t=-{\alpha_1}/ ({2\alpha_7}), $ 
        \end{center} we have the family of representatives	$ \left\langle \nabla_2+\alpha\nabla_3+\beta\nabla_5+\gamma\nabla_6+\nabla_7\right\rangle. $

	\end{enumerate}

	\item $ \alpha_8\neq0,$ then by choosing $\phi=\phi_1,$  $ t=-\frac{\alpha_2}{\alpha_8}x, s=-\frac{\alpha_5}{\alpha_8}x,$ we have $ \alpha_2^*=\alpha_5^*=0. $ Now we can suppose that $ \alpha_2=0, \alpha_5=0 $ and we have
	
	\begin{enumerate}
		\item if $ \alpha_1=0, \alpha_3=0, \alpha_4=0, \alpha_6=0, $ then we have the representatives $ \left\langle \nabla_8\right\rangle  $ and $ \left\langle \nabla_7+\nabla_8\right\rangle  $ depending on whether $ \alpha_7=0 $ or not;
		
		\item if $ \alpha_1=0, \alpha_3=0, \alpha_4=0, \alpha_6\neq0, $ then by choosing  $\phi=\phi_1,$  $ x=\sqrt{{\alpha_6}{\alpha_8^{-1}}},$ $ s=0,$ $ t=0, $ we have the family of representative $ \left\langle \nabla_6+\alpha\nabla_7+\nabla_8\right\rangle;  $
		
		\item if $ \alpha_1=0, \alpha_3=0, \alpha_4\neq0, $ then by choosing  $\phi=\phi_1,$  $ x=\sqrt[3]{{\alpha_4}{\alpha_8^{-1}}}, s=0, t=0, $ we have the family of representatives $ \left\langle \nabla_4+\alpha\nabla_6+\beta\nabla_7+\nabla_8\right\rangle;  $

		\item if $ \alpha_1\neq0, $ then by choosing  $\phi=\phi_1,$  $ x=\sqrt[4]{{\alpha_1}{\alpha_8^{-1}}}, s=0, t=0, $ we have the family of representatives $ \left\langle \nabla_1+\alpha\nabla_3+\beta\nabla_4+\gamma\nabla_6+\mu\nabla_7+\nabla_8\right\rangle.  $		
				
	\end{enumerate}

\end{enumerate}

Summarizing, we have the following distinct orbits:
\begin{center}

$ \left\langle \nabla_1+\alpha\nabla_3+\beta\nabla_4+\gamma\nabla_6+\mu\nabla_7+\nabla_8\right\rangle
^{
{\tiny 
\begin{array}{l}
O(\alpha,\beta,\gamma,\mu)=
O(\alpha,i\beta,-\gamma,-i\mu)=\\
O(\alpha,-i\beta,-\gamma,i\mu)=
O(\alpha,-\beta,\gamma,-\mu)=\\
O(\frac{1}{\alpha},-\frac{\beta}{\sqrt[4]{\alpha^{3}}},
\frac{\gamma}{\sqrt{\alpha}},\frac{\mu}{\sqrt[4]{\alpha}})=\\
O(\frac{1}{\alpha},-\frac{i\beta}{ \sqrt[4]{\alpha^{3}}},-\frac{\gamma}{\sqrt{\alpha}},-\frac{i\mu}{\sqrt[4]{\alpha}})=\\
O(\frac{1}{\alpha},\frac{\beta}{\sqrt[4]{\alpha^{3}}},
\frac{\gamma}{\sqrt{\alpha}},-\frac{\mu}{\sqrt[4]{\alpha}})=
O(\frac{1}{\alpha},-\frac{\beta}{\sqrt[4]{\alpha^{3}}},
\frac{\gamma}{\sqrt{\alpha}},\frac{\mu}{\sqrt[4]{\alpha}})\end{array}}},$
$ \left\langle \nabla_1+\alpha\nabla_3+\beta\nabla_6+\nabla_7\right\rangle
^{{\tiny \begin{array}{l}
O(\alpha,\beta)=
O(\alpha,-\eta_3\beta)=
O(\alpha,\eta_3^2\beta)=\\
O(\alpha^{-1},-\eta_3\beta\sqrt[3]{\alpha^{-1}})=
O(\alpha^{-1},\eta_3^2\beta\sqrt[3]{\alpha^{-1}})=
O(\alpha^{-1},\beta\sqrt[3]{\alpha^{-1}})\end{array}}},$
$ \left\langle \nabla_2+\alpha\nabla_3+\nabla_4 +\beta \nabla_6 \right\rangle,$
$ \left\langle \nabla_2+\alpha\nabla_3+\beta\nabla_5+\gamma\nabla_6+\nabla_7\right\rangle^{O(\alpha,\beta,\gamma)=
O(-\frac{\alpha}{\beta^4},\frac{1}{\beta},\frac{\gamma}{\beta})},$
$ \left\langle \nabla_2+\nabla_3+\alpha \nabla_6 \right\rangle,$
$ \left\langle \alpha \nabla_2+\nabla_4+\nabla_5+ \beta \nabla_6 \right\rangle_{\alpha\neq0}^{O(\alpha,\beta)=O(\alpha^{-1},\beta\alpha^{-1})},$
$ \left\langle \alpha \nabla_2+ \nabla_5+ \beta \nabla_6 \right\rangle_{\alpha\neq0}^{O(\alpha,\beta)=O(\alpha^{-1},\beta\alpha^{-1})},$
$ \left\langle \nabla_2+ \alpha \nabla_6 \right\rangle,$
$ \left\langle \nabla_4+\alpha\nabla_6+\beta\nabla_7+\nabla_8\right\rangle
^{{\tiny \begin{array}{l}
O(\alpha,\beta)=O(\eta_3^2\alpha,-\eta_3\beta)=O(-\eta_3\alpha,\eta_3^2\beta)
=\\
O(\eta_3^2\alpha,\eta_3\beta)=O(-\eta_3\alpha,-\eta_3^2\beta)
=O(\alpha,-\beta)\end{array} }},$
$ \left\langle \nabla_6+\nabla_7\right\rangle,$
$ \left\langle \nabla_6+\alpha\nabla_7+\nabla_8\right\rangle^{O(\alpha)=O(-\alpha)},$
$ \left\langle \nabla_7 \right\rangle, $
$ \left\langle \nabla_7+\nabla_8\right\rangle,$
$ \left\langle \nabla_8\right\rangle, 
$

\end{center}
which gives the following new algebras:

\begin{longtable}{llllllllllllllllll}

${\mathbf{N}}_{232}^{\alpha, \beta, \gamma, \mu}$ & $:$ &$e_1e_1=e_5$ & $e_1e_2=e_3$ & $e_1e_3=e_4$ & $e_2e_2= \alpha e_5$ \\ && $e_2e_3=e_4+\beta e_5$   & $e_3e_3=\gamma e_5$ 
  & $e_3e_4=\mu e_5$ & $e_4e_4= e_5$
\\

${\mathbf{N}}_{233}^{\alpha, \beta}$ & $:$ &$e_1e_1=e_5$ & $e_1e_2=e_3$ & $e_1e_3=e_4$ & $e_2e_2= \alpha e_5$ \\ & & $e_2e_3=e_4$ & $e_3e_3=\beta e_5$ & $e_3e_4= e_5$ 
\\

${\mathbf{N}}_{234}^{\alpha, \beta}$ & $:$ & $e_1e_2=e_3$ & $e_1e_3=e_4$ & $e_1e_4=e_5$ \\ && $e_2e_2= \alpha e_5$  & $e_2e_3=e_4+e_5$ & $e_3e_3=\beta e_5$  
\\

${\mathbf{N}}_{235}^{\alpha, \beta, \gamma}$ & $:$ & $e_1e_2=e_3$ & $e_1e_3=e_4$ & $e_1e_4= e_5$  & $e_2e_2=\alpha e_5$ \\ & & $e_2e_3=e_4$ & $e_2e_4= \beta e_5$ & $e_3e_3= \gamma e_5$ & $e_3e_4= e_5$  
\\

${\mathbf{N}}_{236}^{\alpha}$ & $:$ & $e_1e_2=e_3$ & $e_1e_3=e_4$ & $e_1e_4=e_5$ \\ && $e_2e_2= e_5$  & $e_2e_3=e_4$ & $e_3e_3=\alpha e_5$  
\\

${\mathbf{N}}_{237}^{\alpha\neq0, \beta}$ & $:$ & $e_1e_2=e_3$ & $e_1e_3=e_4$ & $e_1e_4=\alpha e_5$  \\ && $e_2e_3=e_4+e_5$ & $e_2e_4=e_5$ & $e_3e_3=\beta e_5$  
\\

${\mathbf{N}}_{238}^{\alpha\neq0, \beta}$ & $:$ & $e_1e_2=e_3$ & $e_1e_3=e_4$ & $e_1e_4=\alpha e_5$ \\ & & $e_2e_3=e_4$  & $e_2e_4=e_5$ & $e_3e_3=\beta e_5$  
\\

${\mathbf{N}}_{239}^{\alpha}$ & $:$ & $e_1e_2=e_3$ & $e_1e_3=e_4$ & $e_1e_4= e_5$  \\ && $e_2e_3=e_4$  & $e_3e_3= \alpha e_5$  
\\

${\mathbf{N}}_{240}^{\alpha, \beta}$ & $:$ & $e_1e_2=e_3$ & $e_1e_3=e_4$ & $e_2e_3=e_4+e_5$  \\ && $e_3e_3= \alpha e_5$  & $e_3e_4= \beta e_5$ & $e_4e_4= e_5$
\\

${\mathbf{N}}_{241}$ & $:$ & $e_1e_2=e_3$ & $e_1e_3=e_4$ & $e_2e_3=e_4$  \\ && $e_3e_3= e_5$  & $e_3e_4= e_5$ 
\\

${\mathbf{N}}_{242}^{\alpha}$ & $:$ & $e_1e_2=e_3$ & $e_1e_3=e_4$ & $e_2e_3=e_4$  \\ && $e_3e_3= e_5$  & $e_3e_4= \alpha e_5$ & $e_4e_4= e_5$
\\

${\mathbf{N}}_{243}$ & $:$ & $e_1e_2=e_3$ & $e_1e_3=e_4$ & $e_2e_3=e_4$  & $e_3e_4= e_5$ 
\\

${\mathbf{N}}_{244}$ & $:$ & $e_1e_2=e_3$ & $e_1e_3=e_4$ & $e_2e_3=e_4$ \\ & & $e_3e_4= e_5$  & $e_4e_4= e_5$ 
\\
${\mathbf{N}}_{245}$ & $:$ & $e_1e_2=e_3$ & $e_1e_3=e_4$ & $e_2e_3=e_4$    & $e_4e_4= e_5$ 
\\

\end{longtable}

\subsection{$ 1 $-dimensional central extensions of $ {\mathbf N}_{17}^{4*} $.} Here we will collect all information about $ {\mathbf N}_{17}^{4*}: $
\begin{longtable}{|l|l|l|}
%\hline
%\text{ }  & \text{ } & \text{Cohomology} \\
\hline
${\mathbf{N}}^{4*}_{17}$ & 
$\begin{array}{l}
e_1e_2=e_3 \\ 
e_3e_3=e_4
\end{array}$
&
$\begin{array}{l}
\mathrm{H}^2_{\mathfrak{D}}(\mathbf{N}^{4*}_{17})=\langle [\Delta_{11}],[\Delta_{13}],[\Delta_{22}],[\Delta_{23}]\rangle\\
\mathrm{H}^2_{\mathfrak{C}}(\mathbf{N}^{4*}_{17})=\mathrm{H}^2_{\mathfrak{D}}(\mathbf{N}^{4*}_{17})\oplus
\langle [\Delta_{14}], [\Delta_{24}] , [\Delta_{34}], [\Delta_{44}] \rangle 
\end{array}$\\
\hline
\multicolumn{3}{|c|}{ 
$\phi_1=\begin{pmatrix}
x&0&0&0\\
0&q&0&0\\
0&0&xq&0\\
t&s&0&x^2q^2
\end{pmatrix}, \phi_2=\begin{pmatrix}
0&p&0&0\\
y&0&0&0\\
0&0&yp&0\\
t&s&0&y^2p^2
\end{pmatrix}$}\\
\hline

\end{longtable}

Let us use the following notations:
\begin{longtable}{llll}
$\nabla_1=[\Delta_{11}],$ & $\nabla_2=[\Delta_{13}],$ & $ \nabla_3=[\Delta_{14}],$ & $  \nabla_4=[\Delta_{22}],$ \\
$\nabla_5=[\Delta_{23}],$ & $ \nabla_6=[\Delta_{24}],$ & $  \nabla_7=[\Delta_{34}],$ & $ \nabla_8=[\Delta_{44}].$ 
\end{longtable}

Take $ \theta=\sum\limits_{i=1}^{8}\alpha_i\nabla_i\in\mathrm{H}^2_{\mathfrak{C}}(\mathbf{N}^{4*}_{17}) .$  Since

$$\phi^T\begin{pmatrix}
\alpha_1&0&\alpha_2&\alpha_3\\
0&\alpha_4&\alpha_5&\alpha_6\\
\alpha_2&\alpha_5&0&\alpha_7\\
\alpha_3&\alpha_6&\alpha_7&\alpha_8

\end{pmatrix}\phi=
\begin{pmatrix}
\alpha_1^*&\alpha^{*}&\alpha^{*}_2&\alpha_3^*\\
\alpha^{*}&\alpha^*_4&\alpha^*_5&\alpha_6^*\\
\alpha^{*}_2&\alpha^*_5& 0 &\alpha^*_7\\
\alpha^*_3&\alpha^*_6&\alpha^*_7&\alpha^*_8
\end{pmatrix},$$

then in the case $\phi=\phi_1,$ we have
\begin{longtable}{ll}
$\alpha_1^*=\alpha_1x^2+2\alpha_3xt+\alpha_8t^2,$&
$\alpha_2^*=(\alpha_2x+\alpha_7t)xq,$\\

$\alpha_3^*=(\alpha_3x+\alpha_8t)x^2q^2,$&
$\alpha_4^*=\alpha_4q^2+2\alpha_6qs+\alpha_8s^2,$\\

$\alpha_5^*=(\alpha_5q+\alpha_7s)xq,$&
$\alpha_6^*=(\alpha_6q+\alpha_8s)x^2q^2,$\\

$\alpha_7^*=\alpha_7x^3q^3,$&
$\alpha_8^*=\alpha_8x^4q^4;$
\end{longtable}

and in the opposite case $\phi=\phi_2,$ we have
\begin{longtable}{ll}
$\alpha_1^*= \alpha_4p^2+2   \alpha_6p t+  \alpha_8t^2,$&
$\alpha_2^*= ( \alpha_5p + \alpha_7t)p y,$\\

$\alpha_3^*= ( \alpha_6p+ \alpha_8t)p^2 y^2 ,$&
$\alpha_4^*= \alpha_1y^2+2  \alpha_3s y+ \alpha_8s^2,$\\

$\alpha_5^*=  ( \alpha_2y+ \alpha_7s)p y,$&
$\alpha_6^*=( \alpha_3y+ \alpha_8s)p^2 y^2 ,$\\

$\alpha_7^*=  \alpha_7p^3 y^3,$&
$\alpha_8^*=  \alpha_8p^4 y^4.$
\end{longtable}

We are interested in $ (\alpha_3,\alpha_6,\alpha_7,\alpha_8)\neq(0,0,0,0) .$ Let us consider the following cases:

\begin{enumerate}
	\item $ \alpha_8=0, \alpha_7=0, \alpha_6=0, $ then $ \alpha_3\neq0 $ and choosing $\phi=\phi_1,$ $ t=-\frac{\alpha_1}{2\alpha_3}x,$ we get $ \alpha_1^*=0. $ Now consider the following subcases:
	
	\begin{enumerate}
		\item if $ \alpha_2=0, \alpha_4=0, \alpha_5=0, $ then we have the representative $\left\langle \nabla_3 \right\rangle;$
		
		\item if $ \alpha_2=0, \alpha_4=0, \alpha_5\neq0, $ then by choosing $\phi=\phi_1,$ $ x=\sqrt{\frac{\alpha_5}{\alpha_3}}, q=1, s=0, t=-\frac{\alpha_1\sqrt{\alpha_5}}{2\alpha_3\sqrt{\alpha_3}}, $ we have the representative $\left\langle \nabla_3+\nabla_5 \right\rangle;$
		
		\item if $ \alpha_2=0, \alpha_4\neq0, $ then by choosing $\phi=\phi_1,$ $ x=\sqrt[3]{\frac{\alpha_4}{\alpha_3}}, q=1, s=0, t=-\frac{\alpha_1\sqrt[3]{\alpha_5}}{2\alpha_3\sqrt[3]{\alpha_3}}, 
		 $ we have the representative $\left\langle \nabla_3+\nabla_4+\alpha\nabla_5 \right\rangle;$
			
		\item if $ \alpha_2\neq0, \alpha_4=0, \alpha_5=0, $ then by choosing  $\phi=\phi_1,$ $ x=\alpha_2, q=\frac{1}{\alpha_3}, s=0, t=-\frac{\alpha_1\alpha_2}{2\alpha_3}, $ we have the representative $\left\langle \nabla_2+\nabla_3 \right\rangle;$	
		
		\item if $ \alpha_2\neq0, \alpha_4=0, \alpha_5\neq0, $ then by choosing $\phi=\phi_1,$ $ x=\sqrt{\frac{\alpha_5}{\alpha_3}}, q=\sqrt{\frac{\alpha_2^2}{\alpha_3\alpha_5}}, s=0, t=-\frac{\alpha_1\sqrt{\alpha_5}}{2\alpha_3\sqrt{\alpha_3}}, $ we have the representative $\left\langle \nabla_2+\nabla_3+\nabla_5 \right\rangle;$
		
		\item if $ \alpha_2\neq0, \alpha_4\neq0, $ then by choosing $\phi=\phi_1,$ $ x=\sqrt[3]{{\alpha_4}{\alpha_3^{-3}}}, q=\sqrt[3]{{\alpha_2^3}{\alpha^{-2}_3\alpha_4^{-1}}},$ we have the family of   representative $\left\langle \nabla_2+\nabla_3+\nabla_4+\alpha\nabla_5 \right\rangle.$

	\end{enumerate}

	\item  $ \alpha_8=0, \alpha_7=0, \alpha_6\neq0, $ and $ \alpha_3=0 ,$ then by choosing some suitable automorphism $ \phi_2 $ we have $ \alpha_3^*\neq0 $ which is the case considered above. Now we can suppose that $ \alpha_3\neq0, $ and choosing $ t=-\frac{\alpha_1}{2\alpha_3}x, s=-\frac{\alpha_4}{2\alpha_6}x, $ we have $ \alpha_1^*=0, \alpha_4^*=0. $ Therefore, we can suppose that $\alpha_1=0, \alpha_4=0.$ Consider the following subcases:
	
	\begin{enumerate}
		\item $\alpha_2=0, \alpha_5=0,$ then by choosing $\phi=\phi_1,$ $ x=\alpha_6, q=\alpha_3, s=0, t=0, $ we have the representative $ \left\langle \nabla_3+\nabla_6 \right\rangle; $

		\item $\alpha_2\neq0,$ then by choosing $\phi=\phi_1,$ $  x={\alpha_3^{-1}}{\sqrt{\alpha_2\alpha_6}}, q=\sqrt{\alpha_2\alpha_6^{-1}}, s=0, t=0, $ we have the family of representatives $ \left\langle \nabla_2+\nabla_3+\alpha\nabla_5+\nabla_6 \right\rangle. $

	\end{enumerate}

 	\item  $ \alpha_8=0, \alpha_7\neq0, $ then by choosing $\phi=\phi_1,$ $ t=-{\alpha_2}{\alpha_7^{-1}}x, s=-{\alpha_5}{\alpha_7^{-1}}q, $ we have $ \alpha_2^*=0, \alpha_5^*=0. $ Therefore, we can suppose that $\alpha_2=0, \alpha_5=0.$  Consider the following subcases:
 	
 	\begin{enumerate}
 		\item if $ \alpha_1=0, \alpha_4=0, \alpha_3=0, \alpha_6=0, $ then we have the representative $ \left\langle \nabla_7 \right\rangle ;$
 		
 		\item if $ \alpha_1=0, \alpha_4=0, \alpha_3=0, \alpha_6\neq0, $ then by choosing  $\phi=\phi_1,$ $ x={\alpha_6}{\alpha_7^{-1}}, q=1, s=0, t=0, $ we have the representative $ \left\langle \nabla_6+\nabla_7 \right\rangle ;$
 		
 	    \item if $ \alpha_1=0, \alpha_4=0, \alpha_3\neq0, $ and $ \alpha_6=0, $ then by choosing some suitable automorphism $ \phi_2,$ we have $ \alpha_6^*\neq0 .$ Thus we can consider the case $ \alpha_6\neq0 $ and choosing  $\phi=\phi_1,$ $ x={\alpha_6}{\alpha_7^{-1}}, q={\alpha_3}{\alpha_7^{-1}}, s=0, t=0, $ we have the representative $ \left\langle \nabla_3+\nabla_6+\nabla_7 \right\rangle ;$
 	
 	    \item if $ \alpha_1=0, \alpha_4\neq0, \alpha_3=0, \alpha_6=0,$ then by choosing  $\phi=\phi_1,$ $ x=1, q={\alpha_4}{\alpha_7^{-1}}, s=0, t=0, $ we have the representative $ \left\langle \nabla_4+\nabla_7 \right\rangle;$
 	
  	    \item if $ \alpha_1=0, \alpha_4\neq0, \alpha_3=0, \alpha_6\neq0,$ then by choosing  $\phi=\phi_1,$ $ x={\alpha_6}{\alpha_7}^{-1}, q={\alpha_4\alpha_7^2}{\alpha_6^{-3}},$ we have the representative $ \left\langle \nabla_4+\nabla_6+\nabla_7 \right\rangle;$
  	
  	    \item if $ \alpha_1=0, \alpha_4\neq0, \alpha_3\neq0,$ then by choosing 
  	    \begin{center}$\phi=\phi_1,$ $  x=\sqrt[3]{{\alpha_4}{\alpha_3^{-1}}}, q={\alpha_3}{\alpha_7^{-1}}, s=0, t=0, $ \end{center} 
  	    we have the family of representatives $ \left\langle \nabla_3+\nabla_4+\alpha\nabla_6+\nabla_7 \right\rangle ;$  	    	     	
 	
 	    \item if $ \alpha_1\neq0.$ In case of $ \alpha_4=0,$ choosing some suitable automorphism $ \phi_2,$ we have $ \alpha_4^*\neq0.$ Thus, we can suppose $ \alpha_4\neq0 ,$ and choosing 
 	    \begin{center}
 	        $\phi=\phi_1,$ $   x=\sqrt[8]{{\alpha^3_4}{\alpha_1^{-1}\alpha_7^{-2}}}, q=\sqrt[8]{{\alpha_1^3}{\alpha_4^{-1}\alpha^{-2}_7}}, s=0, t=0, $ 
 	    \end{center} we have the family of representatives $ \left\langle \nabla_1+\alpha\nabla_3+\nabla_4+\beta\nabla_6+\nabla_7 \right\rangle  .$
 	
 	   \end{enumerate}	

 	\item  $ \alpha_8\neq0, $ then by choosing    $\phi=\phi_1,$     $ t=-\frac{\alpha_3}{\alpha_8}x, s=-\frac{\alpha_6}{\alpha_8}q, $ we have $ \alpha_3^*=0, \alpha_6^*=0. $ Consider the following cases:
	
	\begin{enumerate}
		\item if $ \alpha_1=0, \alpha_4=0, \alpha_2=0, \alpha_5=0, $ then we have the representatives $ \left\langle \nabla_8 \right\rangle $ and $ \left\langle \nabla_7+\nabla_8 \right\rangle $ depending on whether $ \alpha_7=0 $ or not;
		
		\item if $ \alpha_1=0, \alpha_4=0, \alpha_2=0, \alpha_5\neq0, $ then we have the representatives $ \left\langle \nabla_5+\nabla_8 \right\rangle $ and $ \left\langle \nabla_5+\nabla_7+\nabla_8 \right\rangle $ depending on whether $ \alpha_7=0 $ or not;
		
 	    \item if $ \alpha_1=0, \alpha_4=0, \alpha_2\neq0.$ In case of $ \alpha_5=0,$ choosing some suitable automorphism $ \phi_2,$ we have $ \alpha_5^*\neq0 .$ Thus, we can suppose $ \alpha_5\neq0 ,$ and choosing
 	    \begin{center}$\phi=\phi_1,$ $  x=\sqrt[5]{{\alpha^3_5}{\alpha_2^{-2}\alpha_8^{-1}}}, q=\sqrt[5]{{\alpha^3_2}{\alpha_5^{-2}\alpha_8^{-1}}}, s=0, t=0,$ \end{center} 
 	    we have the family of representatives $ \left\langle \nabla_2+\nabla_5+\alpha\nabla_7+\nabla_8 \right\rangle ;$		
		
		\item if $ \alpha_1=0, \alpha_4\neq0, \alpha_2=0, \alpha_5=0, $ then we have the representatives $ \left\langle \nabla_4+\nabla_8 \right\rangle $ and $ \left\langle \nabla_4+\nabla_7+\nabla_8 \right\rangle $ depending on whether $ \alpha_7=0 $ or not;		
		
 	    \item if $ \alpha_1=0, \alpha_4\neq0, \alpha_2=0, \alpha_5\neq0 ,$ then by choosing
 	    \begin{center}
 	        $\phi=\phi_1,$ 
 	        $x={\alpha_4}{\alpha_5^{-1}}, q={\alpha^2_5}{\alpha_4^{-1}\sqrt{\alpha_4^{-1}\alpha_8^{-1}}}, s=0, t=0, $ 
 	        \end{center} we have the family of  representatives $ \left\langle \nabla_4+\nabla_5+\alpha\nabla_7+\nabla_8 \right\rangle ;$		
		
 	    \item if $ \alpha_1=0, \alpha_4\neq0, \alpha_2\neq0 ,$ then by choosing
 	    \begin{center}
 	        $\phi=\phi_1,$
 	        $  x=\sqrt[8]{{\alpha_4^3}{\alpha_2^{-2}\alpha_8^{-1}}}, q=\sqrt[4]{{\alpha^2_2}{\alpha_4^{-1}\alpha_8^{-1}}}, s=0, t=0, $
 	        \end{center} we have the family of representatives $ \left\langle \nabla_2+\nabla_4+\alpha\nabla_5+\beta\nabla_7+\nabla_8 \right\rangle ;$
		
 	    \item if $ \alpha_1\neq0$ then by choosing some suitable automorphism $ \phi_2, $ we have $ \alpha_4^*\neq0 $. Thus, we can suppose $\alpha_4\neq0 ,$ and choosing 
 	    \begin{center}
 	        $\phi=\phi_1,$ $x=\sqrt[6]{{\alpha^2_4}{\alpha_1^{-1}\alpha_8^{-1}}}, q=\sqrt[6]{{\alpha_1^2}{\alpha_4^{-1}\alpha_8^{-1}}}, s=0, t=0, $
 	        \end{center} we have the family of representatives 
 	        \begin{center}$ \left\langle \nabla_1+\alpha\nabla_2+\nabla_4+\beta\nabla_5+\gamma\nabla_7+\nabla_8 \right\rangle.$
 	        \end{center}

	\end{enumerate}

\end{enumerate}

Summarizing, we have the following distinct orbits:
\begin{center}

$\left\langle \nabla_1+\alpha\nabla_2+\nabla_4+\beta\nabla_5+\gamma\nabla_7+\nabla_8 \right\rangle^{
{\tiny 
\begin{array}{l}
O(\alpha,\beta,\gamma)=
O(\eta_3^2\alpha, \eta_3^2\beta, \eta_3^2\gamma)=
O(-\eta_3^2\alpha, \eta_3^2\beta, -\eta_3^2\gamma)=\\
O(\eta_3^2\alpha, -\eta_3^2\beta, -\eta_3^2\gamma)=
O(-\eta_3^2\alpha, \eta_3^2\beta, \eta_3^2\gamma)=\\
O(\eta_3\alpha, \eta_3\beta, -\eta_3\gamma)=
O(-\eta_3\alpha, \eta_3\beta, \eta_3\gamma)=\\
O(\eta_3\alpha, -\eta_3\beta, \eta_3\gamma)=
O(-\eta_3\alpha, -\eta_3\beta, -\eta_3\gamma)=\\
O(-\alpha, \beta, - \gamma)=
O(\alpha, -\beta, - \gamma)=\\
O(-\alpha, -\beta, \gamma)=
O(\beta,\alpha,\gamma)=\\
O(\eta_3^2\beta, \eta_3^2\alpha, \eta_3^2\gamma)=
O(-\eta_3^2\beta, \eta_3^2\alpha, -\eta_3^2\gamma)=\\
O(\eta_3^2\beta, -\eta_3^2\alpha, -\eta_3^2\gamma)=
O(-\eta_3^2\beta, \eta_3^2\alpha, \eta_3^2\gamma)=\\
O(\eta_3\beta, \eta_3\alpha, -\eta_3\gamma)=
O(-\eta_3\beta, \eta_3\alpha, \eta_3\gamma)=\\
O(\eta_3\beta, -\eta_3\alpha, \eta_3\gamma)=
O(-\eta_3\beta, -\eta_3\alpha, -\eta_3\gamma)=\\
O(-\beta, \alpha, - \gamma)=
O(\beta, -\alpha, - \gamma)=
O(-\beta, -\alpha, \gamma)
\end{array}}}, $ $
\left\langle \nabla_1+\alpha\nabla_3+\nabla_4+\beta\nabla_6 +\nabla_7 \right\rangle^{
{\tiny\begin{array}{l}
O(\alpha,\beta)=
O(\eta_4\alpha,-\eta_4\beta)=
O(-\eta_4\alpha,\eta_4\beta)=
O(\eta_4^3\alpha,-\eta_4^3\beta)=\\
O(-\eta_4^3\alpha,\eta_4^3\beta)=
O(-i\alpha,-i\beta)=
O(i\alpha,i\beta)=
O(-\alpha,-\beta)=\\
O(\beta,\alpha)=
O(\eta_4\beta,-\eta_4\alpha)=
O(-\eta_4\beta,\eta_4\alpha)=
O(\eta_4^3\beta,-\eta_4^3\alpha)=\\
O(-\eta_4^3\beta,\eta_4^3\alpha)=
O(-i\beta,-i\alpha)=
O(i\beta,i\alpha)=
O(-\beta,-\alpha)
\end{array}}}, $ 
$ \left\langle \nabla_2+\nabla_3 \right\rangle, $ 
$\left\langle \nabla_2+\nabla_3+\nabla_4+\alpha\nabla_5 \right\rangle^{O(\alpha)=O(-\eta_3 \alpha)=O(\eta_3^2 \alpha)}, $ 
$ \left\langle \nabla_2+\nabla_3+\nabla_5 \right\rangle, $ 
$ \left\langle \nabla_2+\nabla_3 +\alpha\nabla_5+\nabla_6 \right\rangle^{O(\alpha)=O(\alpha^{-1})}, $  
$\left\langle \nabla_2+\nabla_4+\alpha\nabla_5+\beta\nabla_7+\nabla_8 \right\rangle
^{{\tiny 
\begin{array}{l}
O(\alpha,\beta)=
O(\eta_4^3\alpha,-\eta_4^3\beta)=
O(-\eta_4^3\alpha,\eta_4^3\beta)=
O(\eta_4\alpha,-\eta_4\beta)=\\
O(-\eta_4\alpha,\eta_4\beta)=
O(i\alpha,i\beta)=
O(-i\alpha,-i\beta)=
O(-\alpha,-\beta)
\end{array}}}, $ 
$ \left\langle \nabla_2+\nabla_5+\alpha\nabla_7+\nabla_8 \right\rangle
^{
{\tiny 
\begin{array}{l}
O(\alpha)=O(\eta_5^2\alpha)=O(\eta_5^4\alpha)=\\
O(-\eta_5\alpha)=O(-\eta_5^3\alpha) \end{array}}}, $ 
$ \left\langle \nabla_3 \right\rangle, $ 
$ \left\langle \nabla_3+\nabla_4+\alpha\nabla_5 \right\rangle
^{O(\alpha)=O(-\eta_3\alpha)=O(\eta_3^2\alpha)}, $  
$\left\langle \nabla_3+\nabla_4+\alpha\nabla_6+\nabla_7 \right\rangle^{O(\alpha)=O(-\eta_3\alpha)=O(\eta_3^2\alpha)}, $ 
$ \left\langle \nabla_3+\nabla_5 \right\rangle, $ 
$ \left\langle \nabla_3+\nabla_6 \right\rangle, $ 
$ \left\langle \nabla_3+\nabla_6+\nabla_7 \right\rangle, $
$\left\langle \nabla_4+\nabla_5+\alpha\nabla_7+\nabla_8 \right\rangle^{O(\alpha)=O(-\alpha)}, $ 
$ \left\langle \nabla_4+\nabla_6+\nabla_7 \right\rangle, $ 
$ \left\langle \nabla_4+\nabla_7 \right\rangle, $ 
$ \left\langle \nabla_4+\nabla_7+\nabla_8 \right\rangle, $ 
$ \left\langle \nabla_4+\nabla_8 \right\rangle, $
$\left\langle \nabla_5+\nabla_7+\nabla_8 \right\rangle, $ 
$ \left\langle \nabla_5+\nabla_8 \right\rangle, $ 
$ \left\langle \nabla_6+\nabla_7 \right\rangle, $ 
$ \left\langle \nabla_7 \right\rangle, $ 
$ \left\langle \nabla_7+\nabla_8 \right\rangle, $ 
$ \left\langle \nabla_8 \right\rangle, $

\end{center}
which gives the following new algebras:

\begin{longtable}{llllllllllllllllll}

${\mathbf{N}}_{246}^{\alpha, \beta ,\gamma}$ & $:$ &$e_1e_1=e_5$ & $e_1e_2=e_3$ & $e_1e_3=\alpha e_5$ & $e_2e_2=e_5$  \\ && $e_2e_3=\beta e_5$ & $e_3e_3=e_4$ 
 & $e_3e_4=\gamma e_5$ & $ e_4e_4=e_5$
\\

${\mathbf{N}}_{247}^{\alpha, \beta}$ & $:$ &$e_1e_1=e_5$ & $e_1e_2=e_3$ & $e_1e_4=\alpha e_5$ & $e_2e_2=e_5$ \\ & & $e_2e_4=\beta e_5$ & $e_3e_3=e_4$ 
  & $e_3e_4=e_5$ \\

${\mathbf{N}}_{248}$ & $:$ &$e_1e_2=e_3$ & $e_1e_3=e_5$ & $e_1e_4=e_5$ & $e_3e_3=e_4$ \\

${\mathbf{N}}_{249}^{\alpha}$ & $:$ &$e_1e_2=e_3$ & $e_1e_3=e_5$ & $e_1e_4=e_5$ \\ && $e_2e_2=e_5$ & $e_2e_3=\alpha e_5$ & $e_3e_3=e_4$ 
 \\

${\mathbf{N}}_{250}$ & $:$ &$e_1e_2=e_3$ & $e_1e_3=e_5$ & $e_1e_4=e_5$ & $e_2e_3=e_5$ & $e_3e_3=e_4$ 
 \\

${\mathbf{N}}_{251}^{\alpha}$ & $:$ &$e_1e_2=e_3$ & $e_1e_3=e_5$ & $e_1e_4=e_5$ \\ && $e_2e_3=\alpha e_5$ & $e_2e_4=e_5$ & $e_3e_3=e_4$ 
 \\

${\mathbf{N}}_{252}^{\alpha, \beta}$ & $:$ &$e_1e_2=e_3$ & $e_1e_3=e_5$ & $e_2e_2=e_5$ & $e_2e_3=\alpha e_5$ \\ & & $e_3e_3=e_4$ & $e_3e_4=\beta e_5$   & $ e_4e_4=e_5$
 \\

${\mathbf{N}}_{253}^{\alpha}$ & $:$ &$e_1e_2=e_3$ & $e_1e_3=e_5$ & $e_2e_3=e_5$ \\ && $e_3e_3=e_4$ & $e_3e_4= \alpha e_5$ & $e_4e_4=e_5$ 
 \\

${\mathbf{N}}_{254}$ & $:$ &$e_1e_2=e_3$ & $e_1e_4=e_5$ & $e_3e_3=e_4$ 
 \\

${\mathbf{N}}_{255}^{\alpha}$ & $:$ &$e_1e_2=e_3$ & $e_1e_4=e_5$ & $e_2e_2=e_5$ & $e_2e_3=\alpha e_5$ & $e_3e_3=e_4$
 \\

${\mathbf{N}}_{256}^{\alpha}$ & $:$ &$e_1e_2=e_3$ & $e_1e_4=e_5$ & $e_2e_2=e_5$ \\ && $e_2e_4=\alpha e_5$ & $e_3e_3=e_4$ & $e_3e_4=e_5$ 
 \\

${\mathbf{N}}_{257}$ & $:$ &$e_1e_2=e_3$ & $e_1e_4=e_5$ & $ e_2e_3=e_5 $ & $e_3e_3=e_4$ \\

${\mathbf{N}}_{258}$ & $:$ &$e_1e_2=e_3$ & $e_1e_4=e_5$ & $ e_2e_4=e_5 $ & $e_3e_3=e_4$ \\

${\mathbf{N}}_{259}$ & $:$ &$e_1e_2=e_3$ & $e_1e_4=e_5$ & $ e_2e_4=e_5 $ & $e_3e_3=e_4$ & $ e_3e_4=e_5$\\

${\mathbf{N}}_{260}^{\alpha}$ & $:$ &$e_1e_2=e_3$ & $e_2e_2=e_5$ & $e_2e_3=e_5$ \\ && $e_3e_3=e_4$ & $e_3e_4=\alpha e_5$ & $e_4e_4=e_5$ 
 \\

${\mathbf{N}}_{261}$ & $:$ &$e_1e_2=e_3$ & $e_2e_2=e_5$ & $ e_2e_4=e_5 $ & $e_3e_3=e_4$ & $ e_3e_4=e_5$\\

${\mathbf{N}}_{262}$ & $:$ &$e_1e_2=e_3$ & $e_2e_2=e_5$ & $e_3e_3=e_4$ & $ e_3e_4=e_5$\\

${\mathbf{N}}_{263}$ & $:$ &$e_1e_2=e_3$ & $e_2e_2=e_5$ & $ e_3e_3=e_4 $ & $e_3e_4=e_5$ & $ e_4e_4=e_5$\\

${\mathbf{N}}_{264}$ & $:$ &$e_1e_2=e_3$ & $e_2e_2=e_5$ & $e_3e_3=e_4$ & $ e_4e_4=e_5$\\

${\mathbf{N}}_{265}$ & $:$ &$e_1e_2=e_3$ & $e_2e_3=e_5$ & $ e_3e_3=e_4 $ & $e_3e_4=e_5$ & $ e_4e_4=e_5$\\

${\mathbf{N}}_{266}$ & $:$ &$e_1e_2=e_3$ & $e_2e_3=e_5$ & $e_3e_3=e_4$ & $ e_4e_4=e_5$\\

${\mathbf{N}}_{267}$ & $:$ &$e_1e_2=e_3$ & $e_2e_4=e_5$ & $e_3e_3=e_4$ & $ e_3e_4=e_5$\\

${\mathbf{N}}_{268}$ & $:$ &$e_1e_2=e_3$ & $e_3e_3=e_4$ & $ e_3e_4=e_5$\\

${\mathbf{N}}_{269}$ & $:$ &$e_1e_2=e_3$ & $ e_3e_3=e_4 $ & $e_3e_4=e_5$ & $ e_4e_4=e_5$\\

${\mathbf{N}}_{270}$ & $:$ &$e_1e_2=e_3$ & $e_3e_3=e_4$ & $ e_4e_4=e_5$\\

\end{longtable}

\subsection{$ 1 $-dimensional central extensions of $ {\mathbf N}_{18}^{4*} $.} Here we will collect all information about $ {\mathbf N}_{18}^{4*}: $
$$
\begin{array}{|l|l|l|l|}
%\hline
%\text{ }  & \text{ } & \text{Cohomology} & \text{Automorphisms} \\
\hline
{\mathbf{N}}^{4*}_{18} & 
 \begin{array}{l}
e_1e_1=e_4 \\ 
e_1e_2=e_3 \\ 
e_3e_3=e_4
\end{array}
& 

\begin{array}{lcl}
\mathrm{H}^2_{\mathfrak{D}}(\mathbf{N}^{4*}_{18})&=&\\ 
\multicolumn{3}{r}{\langle [\Delta_{11}],[\Delta_{13}],[\Delta_{22}],[\Delta_{23}]\rangle}\\
\mathrm{H}^2_{\mathfrak{C}}(\mathbf{N}^{4*}_{18})&=&\mathrm{H}^2_{\mathfrak{D}}(\mathbf{N}^{4*}_{18})\oplus\\
\multicolumn{3}{r}{\langle [\Delta_{14}], [\Delta_{24}], [\Delta_{34}], [\Delta_{44}] \rangle}
\end{array}  
&   
   \phi_{\pm}=\begin{pmatrix}
x&0&0&0\\
0&\pm 1&0&0\\
0&0&\pm x &0\\
t&s&0&x^2
\end{pmatrix}\\
\hline
\end{array}$$

Let us use the following notations:
\begin{longtable}{llll}
$\nabla_1=[\Delta_{11}], $&$ \nabla_2=[\Delta_{13}], $&
$\nabla_3=[\Delta_{14}], $&$ \nabla_4=[\Delta_{22}], $\\
$\nabla_5=[\Delta_{23}], $&$ \nabla_6=[\Delta_{24}], $&
$\nabla_7=[\Delta_{34}], $&$ \nabla_8=[\Delta_{44}]. $
\end{longtable}
Take $ \theta=\sum\limits_{i=1}^{8}\alpha_i\nabla_i\in\mathrm{H}^2_{\mathfrak{C}}(\mathbf{N}^{4*}_{18}) .$  Since

$$\phi_{\pm}^T\begin{pmatrix}
\alpha_1&0&\alpha_2&\alpha_3\\
0&\alpha_4&\alpha_5&\alpha_6\\
\alpha_2&\alpha_5&0&\alpha_7\\
\alpha_3&\alpha_6&\alpha_7&\alpha_8

\end{pmatrix}\phi_{\pm}=
\begin{pmatrix}
\alpha_1^*&\alpha{*}&\alpha^{*}_2&\alpha_3^*\\
\alpha{*}&\alpha^*_4&\alpha^*_5&\alpha_6^*\\
\alpha^{*}_2&\alpha^*_5&0&\alpha^*_7\\
\alpha^*_3&\alpha^*_6&\alpha^*_7&\alpha^*_8
\end{pmatrix},$$
we have
\begin{longtable}{ll}
$\alpha_1^*=\alpha_1x^2+2\alpha_3xt+\alpha_8t^2,$&
$\alpha_2^*=\pm (\alpha_2x+\alpha_7t) x,$\\

$\alpha_3^*=(\alpha_3x+\alpha_8t)x^2,$&
$\alpha_4^*=\alpha_4\pm2\alpha_6s+\alpha_8s^2,$\\

$\alpha_5^*=(\alpha_5\pm\alpha_7s)x,$&
$\alpha_6^*=(\pm\alpha_6+\alpha_8s)x^2,$\\

$\alpha_7^*=\pm\alpha_7x^3,$&
$\alpha_8^*=\alpha_8x^4.$
\end{longtable}

We are interested in $ (\alpha_3,\alpha_6,\alpha_7,\alpha_8)\neq(0,0,0,0) .$ Let us consider $\phi=\phi_+$ and  the following cases:

\begin{enumerate}
	\item $ \alpha_8=0, \alpha_7=0, \alpha_6=0, $ then $ \alpha_3\neq0 $ and choosing $ t=-\frac{\alpha_1}{2\alpha_3}x, $ we get $ \alpha_1^*=0. $ Now consider the following subcases:
	
	\begin{enumerate}
		\item if $ \alpha_2=0, \alpha_4=0, \alpha_5=0, $ then we have the representative $\left\langle \nabla_3 \right\rangle ;$
		
		\item if $ \alpha_2=0, \alpha_4=0, \alpha_5\neq0, $ then by choosing $ x=\sqrt{\frac{\alpha_5}{\alpha_3}}, s=0,  t=-\frac{\alpha_1\sqrt{\alpha_5}}{2\alpha_3\sqrt{\alpha_3}}, $ we have the representative $\left\langle \nabla_3+\nabla_5 \right\rangle;$
		
		\item if $ \alpha_2=0, \alpha_4\neq0, $ then by choosing $ x=\sqrt[3]{\frac{\alpha_4}{\alpha_3}}, s=0,  t=-\frac{\alpha_1\sqrt[3]{\alpha_4}}{2\alpha_3\sqrt[3]{\alpha_3}},
		$ we have the family of representatives $\left\langle \nabla_3+\nabla_4+\alpha\nabla_5 \right\rangle;$
		
		\item if $ \alpha_2\neq0, $ then by choosing $ x={\alpha_2}{\alpha_3^{-1}}, s=0,  t=-\frac{\alpha_1\alpha_2}{2\alpha^2_3},$ we have the family of representatives $\left\langle \nabla_2+\nabla_3+\alpha\nabla_4+\beta\nabla_5 \right\rangle.$

	\end{enumerate}

	\item  $ \alpha_8=0, \alpha_7=0, \alpha_6\neq0, $ then by choosing $ s=-\frac{\alpha_4}{2\alpha_6}x, $ we have $ \alpha_4^*=0. $ Consider the following cases:

\begin{enumerate}
	\item $\alpha_3=0,$ then we have two families of representatives $ \left\langle \alpha\nabla_1+\beta\nabla_2+\nabla_6 \right\rangle $ and $ \left\langle \alpha\nabla_1+\beta\nabla_2+\nabla_5+\nabla_6 \right\rangle $ depending on whether $ \alpha_5=0 $ or not;
	
	\item $\alpha_3\neq0$ then by choosing $ x=\frac{\alpha_6}{\alpha_3}, s=-\frac{\alpha_4}{2\alpha_6}, t=-\frac{\alpha_1\alpha_6}{2\alpha_3^2},$ we have the family of representatives $ \left\langle \alpha\nabla_2+\nabla_3+\beta\nabla_5+\nabla_6 \right\rangle. $
		
\end{enumerate}

 	\item  $ \alpha_8=0, \alpha_7\neq0, $ then by choosing $ t=-{\alpha_2}{\alpha_7^{-1}}x, s=-{\alpha_5}{\alpha_7^{-1}}, $ we have $ \alpha_2^*=0, \alpha_5^*=0. $ Thus, we can suppose that $ \alpha_2=0, \alpha_5=0 $ and now consider the following cases:

\begin{enumerate}
	\item if $ \alpha_1=0, \alpha_4=0, \alpha_6=0, $ then we have the family of representatives $ \left\langle \alpha\nabla_3+\nabla_7 \right\rangle;$
	
	\item if $ \alpha_1=0, \alpha_4=0, \alpha_6\neq0, $ then by choosing $ x={\alpha_6}{\alpha_7^{-1}}, s=0, t=0, $ we have the family of representatives $ \left\langle \alpha\nabla_3+\nabla_6+\nabla_7 \right\rangle;$
	
	\item if $ \alpha_1=0, \alpha_4\neq0, $ then by choosing $ x=\sqrt[3]{{\alpha_4}{\alpha_7^{-1}}}, s=0, t=0, $ we have the family of representatives $ \left\langle \alpha\nabla_3+\nabla_4+\beta\nabla_6+\nabla_7 \right\rangle;$
	
	\item if $ \alpha_1\neq0,$ then by choosing $ x={\alpha_1}{\alpha_7^{-1}}, s=0, t=0, $ we have the family of representatives $ \left\langle \nabla_1+\alpha\nabla_3+\beta\nabla_4+\gamma\nabla_6+\nabla_7 \right\rangle.$

\end{enumerate}

 	\item $ \alpha_8\neq0, $ then by choosing $ t=-\frac{\alpha_3}{\alpha_8}x, s=-\frac{\alpha_6}{\alpha_8}, $ we have $ \alpha_3^*=0, \alpha_6^*=0. $ Thus, we can suppose that $ \alpha_3=0, \alpha_6=0. $ Consider the following cases:

\begin{enumerate}
	\item if $ \alpha_1=0, \alpha_2=0, \alpha_4=0, \alpha_5=0, $ then we have the representatives $ \left\langle \nabla_8 \right\rangle $ and $ \left\langle \nabla_7+\nabla_8 \right\rangle $ depending on whether $ \alpha_7=0 $ or not;
	
	\item if $ \alpha_1=0, \alpha_2=0, \alpha_4=0, \alpha_5\neq0, $ then by choosing  $x=\sqrt[3]{{\alpha_5}{\alpha_8^{-1}}}, s=0, t=0,$ we have the family of representatives $ \left\langle \nabla_5+\alpha\nabla_7+\nabla_8 \right\rangle;$
	
	\item if $ \alpha_1=0, \alpha_2=0, \alpha_4\neq0$ then by choosing $  x=\sqrt[4]{{\alpha_4}{\alpha_8^{-1}}}, s=0, t=0, $ we have the family of representatives $ \left\langle \nabla_4+\alpha\nabla_5+\beta\nabla_7+\nabla_8 \right\rangle;$
					
	\item if $ \alpha_1=0, \alpha_2\neq0, $ then by choosing $  x=\sqrt{{\alpha_2}{\alpha_8^{-1}}}, s=0, t=0,$ we have the family of representatives $ \left\langle \nabla_2+\alpha\nabla_4+\beta\nabla_5+\gamma\nabla_7+\nabla_8 \right\rangle;$

	\item if $ \alpha_1\neq0,$ then by choosing $  x=\sqrt{{\alpha_1}{\alpha_8}^{-1}}, s=0, t=0, $ we have the family of representatives $ \left\langle \nabla_1+\alpha\nabla_2+\beta\nabla_4+\gamma\nabla_5+\mu\nabla_7+\nabla_8 \right\rangle  .$

\end{enumerate}

\end{enumerate}

Summarizing all cases, we have the following distinct orbits:
\begin{center}

$\left\langle \nabla_1+\alpha\nabla_2+\beta\nabla_4+\gamma\nabla_5+\mu\nabla_7+\nabla_8 \right\rangle
^{
{\tiny
\begin{array}{l}O(\alpha,\beta,\gamma,\mu)=
O(-\alpha,-\beta,-\gamma,\mu)=\\
O(-\alpha,\beta,\gamma,-\mu)=
O(\alpha,-\beta,-\gamma,-\mu)
\end{array}}}, $ 
$\left\langle \alpha\nabla_1+\beta\nabla_2+\nabla_5+\nabla_6 \right\rangle^{O(\alpha,\beta)=O(-\alpha,\beta)}, $ 
$\left\langle \alpha\nabla_1+\beta\nabla_2+\nabla_6 \right\rangle^{O(\alpha,\beta)=O(-\alpha,\beta)},$
$\left\langle \nabla_1+\alpha\nabla_3+\beta\nabla_4+\gamma\nabla_6+\nabla_7 \right\rangle^{O(\alpha,\beta,\gamma)=
O(-\alpha,\beta,-\gamma)}, $ 
$ \left\langle \nabla_2+\nabla_3+\alpha\nabla_4+\beta\nabla_5 \right\rangle^{O(\alpha,\beta)=
O(-\alpha,\beta)}, $ $ \left\langle \alpha\nabla_2+\nabla_3 +\beta\nabla_5+\nabla_6 \right\rangle, $
$\left\langle \nabla_2+\alpha\nabla_4+\beta\nabla_5+\gamma\nabla_7+\nabla_8 \right\rangle^{O(\alpha,\beta,\gamma)=
O(\alpha,i\beta, i\gamma)=O(\alpha,-i\beta, -i\gamma)}, $ $ \left\langle \nabla_3 \right\rangle, $ 
$ \left\langle \nabla_3+\nabla_4+\alpha\nabla_5 \right\rangle
^{O(\alpha)=O(-\eta_3\alpha)=O(\eta_3^2\alpha)}, $ 
$ \left\langle \alpha\nabla_3+\nabla_4+\beta\nabla_6+\nabla_7 \right\rangle^{{\tiny 
\begin{array}{l}
O(\alpha,\beta)=O(-\alpha,-\beta)=O(-\alpha,\eta_3\beta)
=\\O(-\alpha,-\eta_3^2\beta)=O(\alpha,-\eta_3\beta)=O(\alpha,\eta_3^2\beta)\end{array}}}, $
$\left\langle \nabla_3+\nabla_5 \right\rangle, $ 
$ \left\langle \alpha\nabla_3+\nabla_6+\nabla_7 \right\rangle
^{O(\alpha,\beta)=O(-\alpha,\beta)}, $ 
$ \left\langle \alpha\nabla_3+\nabla_7 \right\rangle^{O(\alpha)=O(-\alpha)}, $ 
$ \left\langle \nabla_4+\alpha\nabla_5+\beta\nabla_7+\nabla_8 \right\rangle
^{
{\tiny 
\begin{array}{l}O(\alpha,\beta)=O(i\alpha,-i\beta)=
O(-i\alpha,i\beta)=O(-\alpha,-\beta)=\\
O(\alpha,-\beta)=O(i\alpha,i\beta)=O(-i\alpha,-i\beta)=O(-\alpha,\beta)
\end{array}}},  $
$\left\langle \nabla_5+\alpha\nabla_7+\nabla_8 \right\rangle^{
O(\alpha)=O(\eta_3\alpha)=O(-\eta_3^2\alpha)=O(-\alpha)=O(-\eta_3\alpha)=O(\eta_3^2\alpha)}, $ 
$ \left\langle \nabla_7+\nabla_8 \right\rangle, $ $ \left\langle \nabla_8\right\rangle, $

\end{center}
which gives the following new algebras:

\begin{longtable}{llllllllllllllllll}

${\mathbf{N}}_{271}^{\alpha, \beta, \gamma,\mu}$ & $:$ &$e_1e_1=e_4+e_5$ & $e_1e_2=e_3$ & $e_1e_3=\alpha e_5$ & $e_2e_2=\beta e_5$  \\ && $e_2e_3=\gamma e_5$ & $e_3e_3=e_4$ 
  & $e_3e_4=\mu e_5$ & $ e_4e_4=e_5$
\\

${\mathbf{N}}_{272}^{\alpha, \beta}$ & $:$ &$e_1e_1=e_4+\alpha e_5$ & $e_1e_2=e_3$ & $e_1e_3=\beta e_5$\\ & & $e_2e_3=e_5$ & $e_2e_4=e_5$ & $e_3e_3=e_4$
\\

${\mathbf{N}}_{273}^{\alpha, \beta}$ & $:$ &$e_1e_1=e_4+\alpha e_5$ & $e_1e_2=e_3$ & $e_1e_3=\beta e_5$  & $e_2e_4=e_5$ & $e_3e_3=e_4$
\\

${\mathbf{N}}_{274}^{\alpha, \beta ,\gamma}$ & $:$ &$e_1e_1=e_4+e_5$ & $e_1e_2=e_3$ & $e_1e_4=\alpha e_5$ & $e_2e_2=\beta e_5$ \\ & & $e_2e_4=\gamma e_5$ & $e_3e_3=e_4$ 
  & $e_3e_4=e_5$
\\

${\mathbf{N}}_{275}^{\alpha, \beta}$ & $:$ &$e_1e_1=e_4$ & $e_1e_2=e_3$ & $e_1e_3=e_5$  & $e_1e_4=e_5$  \\ && $ e_2e_2=\alpha e_5$ & $ e_2e_3=\beta e_5$  & $e_3e_3=e_4$
\\

${\mathbf{N}}_{276}^{\alpha, \beta}$ & $:$ &$e_1e_1=e_4$ & $e_1e_2=e_3$ & $e_1e_3=\alpha e_5$  & $e_1e_4=e_5$  \\ && $ e_2e_3=\beta e_5$ & $ e_2e_4=e_5$  & $e_3e_3=e_4$
\\

${\mathbf{N}}_{277}^{\alpha, \beta ,\gamma}$ & $:$ &$e_1e_1=e_4$ & $e_1e_2=e_3$ & $e_1e_3=e_5$ & $e_2e_2=\alpha e_5$  \\ && $e_2e_3=\beta e_5$ & $e_3e_3=e_4$ 
  & $e_3e_4=\gamma e_5$ & $ e_4e_4=e_5$
\\

${\mathbf{N}}_{278}$ & $:$ &$e_1e_1=e_4$ & $e_1e_2=e_3$ & $e_1e_4=e_5$ & $e_3e_3=e_4$ \\

${\mathbf{N}}_{279}^{\alpha}$ & $:$ &$e_1e_1=e_4$ & $e_1e_2=e_3$ & $e_1e_4=e_5$ \\ && $e_2e_2=e_5$ & $e_2e_3=\alpha e_5$ & $e_3e_3=e_4$ 
\\

${\mathbf{N}}_{280}^{\alpha, \beta}$ & $:$ &$e_1e_1=e_4$ & $e_1e_2=e_3$ & $e_1e_4=\alpha e_5$ & $e_2e_2=e_5$  \\ && $e_2e_4=\beta e_5$ & $e_3e_3=e_4$ 
  & $e_3e_4=e_5$
\\

${\mathbf{N}}_{281}$ & $:$ &$e_1e_1=e_4$ & $e_1e_2=e_3$ & $e_1e_4=e_5$ & $e_2e_3=e_5$ & $e_3e_3=e_4$ 
\\

${\mathbf{N}}_{282}^{\alpha}$ & $:$ &$e_1e_1=e_4$ & $e_1e_2=e_3$ & $e_1e_4=\alpha e_5$ \\ && $e_2e_4=e_5$ & $e_3e_3=e_4$  & $e_3e_4=e_5$
\\

${\mathbf{N}}_{283}^{\alpha}$ & $:$ &$e_1e_1=e_4$ & $e_1e_2=e_3$ & $e_1e_4=\alpha e_5$ & $e_3e_3=e_4$  & $e_3e_4=e_5$
\\

${\mathbf{N}}_{284}^{\alpha, \beta}$ & $:$ &$e_1e_1=e_4$ & $e_1e_2=e_3$ & $e_2e_2=e_5$ & $e_2e_3=\alpha e_5$ \\ & & $e_3e_3=e_4$ & $e_3e_4=\beta e_5$ 
  & $e_4e_4=e_5$
\\

${\mathbf{N}}_{285}^{\alpha}$ & $:$ &$e_1e_1=e_4$ & $e_1e_2=e_3$ & $e_2e_3=e_5$ \\ && $e_3e_3=e_4$  & $e_3e_4=\alpha e_5$ & $ e_4e_4=e_5$
\\

${\mathbf{N}}_{286}$ & $:$ &$e_1e_1=e_4$ & $e_1e_2=e_3$ & $e_3e_3=e_4$ & $e_3e_4=e_5$ & $e_4e_4=e_5$ 
\\

${\mathbf{N}}_{287}$ & $:$ &$e_1e_1=e_4$ & $e_1e_2=e_3$ & $e_3e_3=e_4$  & $e_4e_4=e_5$ 
\\

\end{longtable}

\subsection{$ 1 $-dimensional central extensions of $ {\mathbf N}_{19}^{4*} $.} Here we will collect all information about $ {\mathbf N}_{19}^{4*}: $
\begin{longtable}{|l|l|l|}
%\hline
%\text{ }  & \text{ } & \text{Cohomology} \\
\hline
${\mathbf{N}}^{4*}_{19}$ & 
$\begin{array}{ll}
e_1e_1=e_4 & e_1e_2=e_3 \\  
e_2e_2=e_4 & e_3e_3=e_4
\end{array}$
&
$\begin{array}{l}
\mathrm{H}^2_{\mathfrak{D}}(\mathbf{N}^{4*}_{19})=\langle [\Delta_{11}],[\Delta_{13}],[\Delta_{22}],[\Delta_{23}]\rangle\\
\mathrm{H}^2_{\mathfrak{C}}(\mathbf{N}^{4*}_{19})=\mathrm{H}^2_{\mathfrak{D}}(\mathbf{N}^{4*}_{19})\oplus
\langle [\Delta_{14}], [\Delta_{24}], [\Delta_{34}], [\Delta_{44}] \rangle 
\end{array}$\\
\hline

%\multicolumn{3}{|c|}{\text{Automorphisms}} \\
%\hline
\multicolumn{3}{|c|}{
$\begin{array}{l}
\begin{array}{l}
 \phi_1=\begin{pmatrix}
x&0&0&0\\
0&q&0&0\\
0&0&xq&0\\
t&s&0&1
\end{pmatrix}, \\ 
\multicolumn{1}{r}{x^2=1,q^2=1}
\end{array}
\begin{array}{l} 
\phi_2=\begin{pmatrix}
0&p&0&0\\
y&0&0&0\\
0&0&yp&0\\
t&s&0&1
\end{pmatrix}, \\
\multicolumn{1}{r}{y^2=1, p^2=1}
\end{array}
\end{array}$}\\
\hline

\end{longtable}

Let us use the following notations:
\begin{longtable}{llll}
$\nabla_1=[\Delta_{11}], $&$ \nabla_2=[\Delta_{13}], $&$ \nabla_3=[\Delta_{14}], $&$  \nabla_4=[\Delta_{22}],$\\
$\nabla_5=[\Delta_{23}],  $&$ \nabla_6=[\Delta_{24}],  $&$  \nabla_7=[\Delta_{34}], $&$ \nabla_8=[\Delta_{44}]. $
\end{longtable}

Take $ \theta=\sum\limits_{i=1}^{8}\alpha_i\nabla_i\in\mathrm{H}^2_{\mathfrak{C}}(\mathbf{N}^{4*}_{19}) .$  Since

$$\phi^T\begin{pmatrix}
\alpha_1&0&\alpha_2&\alpha_3\\
0&\alpha_4&\alpha_5&\alpha_6\\
\alpha_2&\alpha_5&0&\alpha_7\\
\alpha_3&\alpha_6&\alpha_7&\alpha_8

\end{pmatrix}\phi=
\begin{pmatrix}
\alpha_1^*&\alpha^{*}&\alpha^{*}_2&\alpha_3^*\\
\alpha^{*}&\alpha^*_4&\alpha^*_5&\alpha_6^*\\
\alpha^{*}_2&\alpha^*_5&0&\alpha^*_7\\
\alpha^*_3&\alpha^*_6&\alpha^*_7&\alpha^*_8
\end{pmatrix},$$

then, in the case $\phi=\phi_1^{x=1,q=1},$ 
we have
\begin{longtable}{llll}
$\alpha_1^*=\alpha_1+2\alpha_3t+\alpha_8t^2,$&
$\alpha_2^*=\alpha_2+\alpha_7t,$&
$\alpha_3^*=\alpha_3+\alpha_8t,$&
$\alpha_4^*=\alpha_4+2\alpha_6s+\alpha_8s^2,$\\

$\alpha_5^*=\alpha_5+\alpha_7s,$&
$\alpha_6^*=\alpha_6+\alpha_8s,$&
$\alpha_7^*=\alpha_7,$&
$\alpha_8^*=\alpha_8.$
\end{longtable}

For define the main families of representatives, we will use  $\phi=\phi_1^{x=1,q=1}$ and for find equal orbits we will use other automorphisms. 
We are interested in 
\begin{center}$ (\alpha_3,\alpha_6,\alpha_7,\alpha_8)\neq(0,0,0,0) .$
\end{center} Let us consider the following cases:

\begin{enumerate}
	\item if $ \alpha_8=0, \alpha_7=0, \alpha_6=0, $ then $ \alpha_3\neq0 $ and choosing $ t=-\frac{\alpha_1}{2\alpha_3},$ we have the family of representatives $ \left\langle \alpha\nabla_2+\nabla_3+\beta\nabla_4+\gamma\nabla_5 \right\rangle; $
	
	\item if $ \alpha_8=0, \alpha_7=0, \alpha_6\neq0 $ and $ \alpha_3=0,$ then by choosing some suitable automorphism $ \phi_2 $ we have $ \alpha_3^*\neq0 ,$ $ \alpha_6^* = 0,$ which is the case considered above;
	
	\item if $ \alpha_8=0, \alpha_7=0, \alpha_6\neq0, \alpha_3\neq0, $ then by choosing $ t=-\frac{\alpha_1}{2\alpha_3}, s=\frac{\alpha_4}{2\alpha_6},$ we have the family of representatives $ \left\langle \alpha\nabla_2+\beta\nabla_3+\gamma\nabla_5+\nabla_6 \right\rangle_{\beta\neq0} ;$	
	
	\item if $ \alpha_8=0, \alpha_7\neq0, $ then by choosing $ t=-{\alpha_2}{\alpha_7}^{-1}, s=-{\alpha_5}{\alpha_7}^{-1},$ we have the family of representatives $ \left\langle \alpha\nabla_1+\beta\nabla_3+\gamma\nabla_4+\mu\nabla_6+\nabla_7 \right\rangle; $	
	
	\item if $ \alpha_8\neq0, $ then by choosing $ t=-{\alpha_3}{\alpha_8}^{-1}, s=-{\alpha_6}{\alpha_8}^{-1},$ we have the family of representatives $ \left\langle \alpha\nabla_1+\beta\nabla_2+\gamma\nabla_4+\mu\nabla_5+\nu\nabla_7+\nabla_8 \right\rangle. $
		
\end{enumerate}

Summarizing, we have the following distinct orbits:
\begin{center}

$\left\langle \alpha\nabla_1+\beta\nabla_2+\gamma\nabla_4+\mu\nabla_5+\nu\nabla_7+\nabla_8 \right\rangle
^{{\tiny 
\begin{array}{l}
O(\alpha,\beta,\gamma,\mu,\nu)=
O(\alpha,-\beta,\gamma,\mu,-\nu)=\\
O(\alpha,\beta,\gamma,-\mu,-\nu)=
O(\alpha,-\beta,\gamma,-\mu,\nu)=\\
O(\gamma,\mu,\alpha,\beta,\nu)=
O(\gamma,-\mu,\alpha,\beta,-\nu)=\\
O(\gamma,\mu,\alpha,-\beta,-\nu)=
O(\gamma,-\mu,\alpha,-\beta,\nu)
\end{array}}},$ 
$\left\langle \alpha\nabla_1+\beta\nabla_3+\gamma\nabla_4+\mu\nabla_6+\nabla_7 \right\rangle
^{{\tiny 
\begin{array}{l}
O(\alpha,\beta,\gamma,\mu)=
O(-\alpha,-\beta,-\gamma,\mu)=\\
O(-\alpha,\beta,-\gamma,-\mu)=
O(\alpha,-\beta,\gamma,-\mu)=\\
O(\gamma,\mu,\alpha,\beta)=
O(-\gamma,-\mu,-\alpha,\beta)=\\
O(-\gamma,\mu,-\alpha,-\beta)=
O(\gamma,-\mu,\alpha,-\beta)
\end{array}}},$
$\left\langle \alpha\nabla_2+\nabla_3+\beta\nabla_4+\gamma\nabla_5 \right\rangle^{{\tiny 
\begin{array}{l}
O(\alpha,\beta,\gamma)=
O(-\alpha,\beta,\gamma)=
O(-\alpha,-\beta,\gamma)=
O(\alpha,-\beta,\gamma)=
\end{array}}},$ 
$\left\langle \alpha\nabla_2+\beta\nabla_3+\gamma\nabla_5+\nabla_6 \right\rangle_{\beta \neq0}^{{\tiny 
\begin{array}{l}
O(\alpha,\beta,\gamma)=
O(\alpha,-\beta,-\gamma)=
O(\frac{\gamma}{\beta},\frac{1}{\beta},\frac{\alpha}{\beta})=
O(\frac{\gamma}{\beta},-\frac{1}{\beta},-\frac{\alpha}{\beta})
 
\end{array}}},$

\end{center}
which gives the following new algebras:
\begin{longtable}{llllllllllllllllll}

${\mathbf{N}}_{288}^{\alpha, \beta, \gamma,\mu,\nu}$ & $:$ &$e_1e_1=e_4+\alpha e_5$ & $e_1e_2=e_3$ & $e_1e_3=\beta e_5$ & $e_2e_2=e_4+\gamma e_5$ \\ & & $e_2e_3=\mu e_5$   & $e_3e_3=e_4$ 
 & $e_3e_4=\nu e_5$ & $ e_4e_4=e_5$
\\

${\mathbf{N}}_{289}^{\alpha, \beta, \gamma,\mu}$ & $:$ &$e_1e_1=e_4+\alpha e_5$ & $e_1e_2=e_3$ & $e_1e_4=\beta e_5$ & $e_2e_2=e_4+\gamma e_5$\\ &  & $e_2e_4=\mu e_5$   & $e_3e_3=e_4$ 
 & $e_3e_4=e_5$
\\

${\mathbf{N}}_{290}^{\alpha, \beta}$ & $:$ &$e_1e_1=e_4$ & $e_1e_2=e_3$ & $e_1e_3=\alpha e_5$ & $e_1e_4=e_5$ \\ & & $e_2e_2=e_4+\beta e_5$    & $e_2e_3=\gamma e_5$ 
 & $e_3e_3=e_4$
\\

${\mathbf{N}}_{291}^{\alpha, \beta\neq0 ,\gamma} $ & $:$ &$e_1e_1=e_4$ & $e_1e_2=e_3$ & $e_1e_3=\alpha e_5$ & $e_1e_4=\beta e_5$ \\ & & $e_2e_3=\gamma e_5 $   & $e_2e_4=e_5$   & $e_3e_3=e_4$ \\

\end{longtable}

\section{Central extensions of nilpotent non-$\mathfrak{CCD}$-algebras.}

\subsection{$ 1 $-dimensional central extensions of $ {\mathbf N}_{01}^{4} $.} Here we will collect all information about $ {\mathbf N}_{01}^{4}: $

$$
\begin{array}{|l|l|l|l|}
\hline
\text{ }  & \text{ } & \text{Cohomology} & \text{Automorphisms}  \\
\hline
{\mathbf{N}}^{4}_{01} & 
\begin{array}{l}
e_1e_1=e_2 \\ 
e_1e_2=e_3 \\ 
e_2e_3=e_4
\end{array}
&\begin{array}{l}\mathrm{H}^2_{\mathfrak{C}}(\mathbf{N}^{4}_{01})=\Big \langle [\Delta_{ij}] \Big\rangle\\
 (i,j) \notin \{ (1,1),(1,2),(2,3)\}
\end{array} 
&

 \phi=\begin{pmatrix}
x&0&0&0\\
0&x^2&0&0\\
z&0&x^3&0\\
t&0&x^2z&x^5
\end{pmatrix}\\

\hline
\end{array}$$

Let us use the following notations:
\begin{longtable}{llll}
$\nabla_1=[\Delta_{13}], $&
$\nabla_2=[\Delta_{14}], $&
$\nabla_3=[\Delta_{22}], $&
$\nabla_4=[\Delta_{24}],$\\
$\nabla_5=[\Delta_{33}],$&
$\nabla_6=[\Delta_{34}],$&
$\nabla_7=[\Delta_{44}]. $
\end{longtable}
	
Take $ \theta=\sum\limits_{i=1}^{7}\alpha_i\nabla_i\in\mathrm{H}^2_{\mathfrak{C}}(\mathbf{N}^{4}_{01}) .$ Since
$$\phi^T\begin{pmatrix}
0&0&\alpha_1&\alpha_2\\
0&\alpha_3&0&\alpha_4\\
\alpha_1&0&\alpha_5&\alpha_6\\
\alpha_2&\alpha_4&\alpha_6&\alpha_7

\end{pmatrix}\phi=
\begin{pmatrix}
\alpha^*&\alpha^{**}&\alpha^{*}_1&\alpha^*_2\\
\alpha^{**}&\alpha^*_3&\alpha^{***}&\alpha^*_4\\
\alpha^{*}_1&\alpha^{***}&\alpha^*_5&\alpha^*_6\\
\alpha^*_2&\alpha^*_4&\alpha^*_6&\alpha^*_7
\end{pmatrix}$$	
we have	
\begin{longtable}{lll}
\multicolumn{3}{l}{$\alpha_1^*=\big((\alpha_1x+\alpha_5z+\alpha_6t)x+(\alpha_2x+\alpha_6z+\alpha_7t)z\big)x^2,$}\\ 
$\alpha_2^*=(\alpha_2x+\alpha_6z+\alpha_7t)x^5,$&
$\alpha_3^*=\alpha_3x^4,$&
$\alpha_4^*=\alpha_4x^7,$\\
$\alpha_5^*=(\alpha_5x^{2}+2\alpha_6xz+\alpha_7z^2)x^4,$&
$\alpha_6^*=(\alpha_6x+\alpha_7z)x^{7},$&
$\alpha_7^*=\alpha_7x^{10}.$
\end{longtable}	

We are interested in $ (\alpha_2,\alpha_4,\alpha_6,\alpha_7)\neq(0,0,0,0)$ and consider following cases:

\begin{enumerate}
\item  $ \alpha_7=\alpha_6=\alpha_4=0, $ then $ \alpha_2\neq0, $ and we have the following subcases:

\begin{enumerate}
	\item  if $ \alpha_5=-\alpha_2, $ then we have
	\begin{enumerate}
		\item\label{19-cs1a1} if $ \alpha_1=0, \alpha_3=0, $ then we have the representative $ \langle \nabla_2-\nabla_5 \rangle;$
		
		\item\label{19-cs1a2} if $ \alpha_1=0, \alpha_3\neq0, $ then by choosing $ x=\sqrt{{\alpha_3}{\alpha_2}^{-1}}, z=0, t=0, $ we have the representative $ \langle \nabla_2+\nabla_3-\nabla_5 \rangle; $
		
		\item  if $ \alpha_1\neq0, $ then by choosing $ x=\sqrt{{\alpha_1}{\alpha_2}^{-1}}, z=0, t=0, $ we have the family of representatives $ \langle \nabla_1+\nabla_2+\alpha\nabla_3-\nabla_5 \rangle; $	
	\end{enumerate}
	
	\item if $ \alpha_5\neq-\alpha_2, $ then by choosing $ z=-\frac{\alpha_1}{\alpha_5+\alpha_2}x, t=0, $ we have the family of representatives $ \langle \nabla_2+\alpha\nabla_5 \rangle_{\alpha\neq-1}$ and $ \langle \nabla_2+\nabla_3+\alpha\nabla_5 \rangle_{\alpha\neq-1} $ depending on whether $ \alpha_3=0 $ or not, which will be jointed with the cases (\ref{19-cs1a1}) and (\ref{19-cs1a2}).
\end{enumerate}

\item  $ \alpha_7=0, \alpha_6=0, \alpha_4\neq0, $ then we have the following subcases:

\begin{enumerate}
	
	\item if $ \alpha_5=-\alpha_2, \alpha_1=0,$ \begin{enumerate}
	\item\label{19-cs2a1} if $ \alpha_3=0,$ then we have the representatives  $ \langle \nabla_4 \rangle$ and $ \langle \nabla_2+\nabla_4-\nabla_5 \rangle$ depending on whether $ \alpha_2=0 $ or not;
	\item\label{19-cs2a2} if $ \alpha_3\neq0,$ then by choosing $ x=\sqrt[3]{{\alpha_3}{\alpha_4}^{-1} },$ we have the family of representatives $ \langle \alpha\nabla_2+\nabla_3+ \nabla_4-\alpha\nabla_5 \rangle;$
\end{enumerate}

	\item if $ \alpha_5=-\alpha_2, \alpha_1\neq0,$ then by choosing $ x=\sqrt[3]{{\alpha_1}{\alpha_4}^{-1} },$ we have the family of representatives  $ \langle \nabla_1+\alpha\nabla_2+\beta\nabla_3+ \nabla_4-\alpha\nabla_5 \rangle;$

	\item if $ \alpha_5\neq-\alpha_2, $ then we have
\begin{enumerate}
	\item if $ \alpha_3=0, \alpha_2=0,$ then by  choosing $ x=\frac{\alpha_5}{\alpha_4}, z=-\frac{\alpha_1\alpha_5}{\alpha_4(\alpha_2+\alpha_5)}, t=0, $ we have the representative $ \langle \nabla_4+\nabla_5 \rangle;$
	\item if $ \alpha_3=0, \alpha_2\neq0,$ then by choosing $ x=\frac{\alpha_2}{\alpha_4}, z=-\frac{\alpha_1\alpha_2}{\alpha_4(\alpha_2+\alpha_5)}, t=0,$ we have the family of representatives $ \langle \nabla_2+\nabla_4+\alpha\nabla_5 \rangle_{\alpha\neq-1},$ which will be jointed with a representative from the case (\ref{19-cs2a1});
	\item if $ \alpha_3\neq0,$ then by choosing $ x=\sqrt[3]{{\alpha_3}{\alpha_4}^{-1}}, z=-\frac{\alpha_1\sqrt[3]{\alpha_3}}{(\alpha_2+\alpha_5)\sqrt[3]{\alpha_4}}, t=0,$ we have the family of representatives $ \langle \alpha\nabla_2+\nabla_3+\nabla_4+\beta\nabla_5 \rangle_{\beta\neq-\alpha},$ which will be jointed with the family from the case (\ref{19-cs2a1}).	
\end{enumerate}

\end{enumerate}

\item $ \alpha_7=0, \alpha_6\neq0, $ then we consider the following subcases:
\begin{enumerate}
	\item if $  \alpha_3=0, \alpha_4=0,  $	 then  choosing $ z=-\frac{\alpha_2}{\alpha_6}x, t=-\frac{\alpha_1x+\alpha_5z}{\alpha_6}, $ 
	we have  representatives $ \langle \nabla_6 \rangle $ and $ \langle \nabla_5+\nabla_6 \rangle$ depending on whether $ \alpha_5=2\alpha_2 $ or not;
	
	\item if $ \alpha_3=0, \alpha_4\neq0, $ then by choosing $ x=\frac{\alpha_4}{\alpha_6}, z=-\frac{\alpha_2\alpha_4}{\alpha^2_6}, t=\frac{\alpha_4(\alpha_2\alpha_5-\alpha_1\alpha_6)}{\alpha^3_6}, $ we have the family of representatives $ \langle \nabla_4+\alpha\nabla_5+\nabla_6 \rangle;$
	
	\item if $ \alpha_3\neq0, $ then by choosing $ x=\sqrt[4]{\frac{\alpha_3}{\alpha_6}}, z=-\frac{\alpha_2\sqrt[4]{\alpha_3}}{\alpha_6\sqrt[4]{\alpha_6}}, t=\frac{(\alpha_2\alpha_5-\alpha_1\alpha_6)\sqrt[4]{\alpha_3}}{\alpha_6^2\sqrt[4]{\alpha_6}},$ we have the family of representatives $ \langle \nabla_3+\alpha\nabla_4+\beta\nabla_5+\nabla_6 \rangle.$
	
\end{enumerate}	
\item  $ \alpha_7\neq0, $ then by choosing $ z=-\frac{\alpha_6}{\alpha_7}x, t=\frac{\alpha_6^2+\alpha_2\alpha_7}{\alpha^2_7}x,$ we have $\alpha_2^*=0, \alpha_6^*=0.$ Thus, we can suppose that $\alpha_2=0, \alpha_6=0$ and now consider following subcases:

\begin{enumerate}
	\item $ \alpha_1=0, \alpha_3=0, \alpha_4=0, $ then we have  representatives $ \langle \nabla_7 \rangle $ and $ \langle \nabla_5+\nabla_7 \rangle$ depending on whether $ \alpha_5\alpha_7-\alpha^2_6=0 $ or not;
	
	\item $ \alpha_1=0, \alpha_3=0, \alpha_4\neq0, $ then by choosing $ x=\sqrt[3]{{\alpha_4}{\alpha_7}^{-1}}, $ we have  the family of representatives $ \langle \nabla_4+\alpha\nabla_5+\nabla_7 \rangle;$

	\item $ \alpha_1=0, \alpha_3\neq0, $ then by choosing $ x=\sqrt[6]{{\alpha_3}{\alpha_7}^{-1}}, $ we have  the family of representatives
$ \langle \nabla_3+\alpha\nabla_4+\beta\nabla_5+\nabla_7 \rangle;$

	\item $ \alpha_1\neq0, $ then by choosing $ x=\sqrt[6]{{\alpha_1}{\alpha_7^{-3}}}, $ we have  the family of representatives
	    $ \langle \nabla_1+\alpha\nabla_3+\beta\nabla_4+\gamma\nabla_5+\nabla_7 \rangle.$

\end{enumerate}

\end{enumerate}	

Summarizing all cases, we have the following distinct orbits:
\begin{center}

$\left\langle \nabla_1+ \alpha \nabla_2 + \beta \nabla_3 + \nabla_4 -\alpha \nabla_5 \right\rangle^
{O(\alpha, \beta)=O(-\eta_3\alpha, \eta_3 \beta)=O(\eta_3^2\alpha,-\eta_3^2\beta)},$ 
$\left\langle \nabla_1+  \nabla_2 + \alpha \nabla_3 - \nabla_5 \right\rangle, $
$ \left\langle \nabla_1+ \alpha \nabla_3 + \beta \nabla_4 + \gamma \nabla_5 + \nabla_7  \right\rangle^
{{\tiny \begin{array}{l}
O(\alpha, \beta, \gamma)=O(\alpha, \beta, -\eta_3\gamma)=
O(\alpha, -\beta, -\eta_3\gamma)=\\
O(\alpha, -\beta, \eta_3^2 \gamma)=
O(\alpha, \beta, \eta_3^2\gamma)=O(\alpha, -\beta, \gamma) \end{array} }},
$
$\left\langle  \alpha \nabla_2 + \nabla_3 + \nabla_4 +\beta \nabla_5 \right\rangle
^{O(\alpha, \beta)=O(-\eta_3\alpha, -\eta_3 \beta)=O(\eta_3^2\alpha, \eta_3^2\beta)},$
$\left\langle \nabla_2+ \nabla_3 + \alpha \nabla_5  \right\rangle,$
$\left\langle \nabla_2+ \nabla_4 + \alpha \nabla_5  \right\rangle,$ $\left\langle  \nabla_2 + \alpha \nabla_5 \right\rangle, $
$\left\langle \nabla_3+\alpha\nabla_4+\beta\nabla_5+\nabla_6 \right\rangle
^{{\tiny \begin{array}{l}
O(\alpha, \beta)=O(-i\alpha, -\beta)=\\
O(i\alpha, -\beta)=O(-\alpha, \beta) \end{array}}},$
$\left\langle  \nabla_3 + \alpha\nabla_4 + \beta \nabla_5 + \nabla_7 \right\rangle^
{{\tiny \begin{array}{l}
O(\alpha, \beta)=O(\alpha, -\eta_3 \beta)=O(-\alpha, -\eta_3\beta)=\\
O(-\alpha, \eta_3^2 \beta)=O(\alpha, \eta_3^2 \beta)=O(-\alpha,  \beta)
\end{array}}},$ 
$ \langle \nabla_4 \rangle,$
$\left\langle \nabla_4 + \nabla_5  \right\rangle,$ 
$ \langle \nabla_4+\alpha\nabla_5+\nabla_6 \rangle,$
$\left\langle \nabla_4 + \alpha \nabla_5 + \nabla_7 \right\rangle^{O(\alpha)=O(-\eta_3\alpha)=O(\eta_3^2\alpha)},$
$ \langle \nabla_5+\nabla_6 \rangle,$
$\left\langle \nabla_5 + \nabla_7 \right\rangle,$
$ \langle \nabla_6 \rangle,$
$\left\langle \nabla_7 \right\rangle.
$

\end{center}

Hence, we have the following new algebras:

\begin{longtable}{llllllllllllllllll}

${\mathbf{N}}_{292}^{\alpha, \beta}$ & $:$ &$e_1e_1=e_2$ & $e_1e_2=e_3$ & $e_1e_3=e_5$ & $e_1e_4=\alpha e_5$ \\ && $e_2e_2=\beta e_5$  & $e_2e_3=e_4$ 
   & $e_2e_4=e_5$ & $ e_3e_3=-\alpha e_5$
\\

${\mathbf{N}}_{293}^{\alpha}$ & $:$ &$e_1e_1=e_2$ & $e_1e_2=e_3$ & $e_1e_3=e_5$ & $e_1e_4=e_5$ \\ && $e_2e_2=\alpha e_5$  & $e_2e_3=e_4$ 
  & $e_3e_3=-e_5$
\\

${\mathbf{N}}_{294}^{\alpha, \beta ,\gamma}$ & $:$ &$e_1e_1=e_2$ & $e_1e_2=e_3$ & $e_1e_3=e_5$ & $e_2e_2=\alpha e_5$ \\ && $e_2e_3=e_4$  & $e_2e_4=\beta e_5$ 
  & $e_3e_3=\gamma e_5$ & $ e_4e_4=e_5$
\\

${\mathbf{N}}_{295}^{\alpha, \beta}$ & $:$ &$e_1e_1=e_2$ & $e_1e_2=e_3$ & $e_1e_4=\alpha e_5$ & $e_2e_2=e_5$ \\ && $e_2e_3=e_4$  & $e_2e_4=e_5$ 
   & $e_3e_3=\beta e_5$
\\

${\mathbf{N}}_{296}^{\alpha}$ & $:$ &$e_1e_1=e_2$ & $e_1e_2=e_3$ & $e_1e_4=e_5$ \\ && $e_2e_2=e_5$  & $e_2e_3=e_4$ & $e_3e_3=\alpha e_5$
\\

${\mathbf{N}}_{297}^{\alpha}$ & $:$ &$e_1e_1=e_2$ & $e_1e_2=e_3$ & $e_1e_4=e_5$ \\ && $e_2e_3=e_4$  & $e_2e_4=e_5$ & $e_3e_3=\alpha e_5$
\\

${\mathbf{N}}_{298}^{\alpha}$ & $:$ &$e_1e_1=e_2$ & $e_1e_2=e_3$ & $e_1e_4=e_5$  & $e_2e_3=e_4$ & $e_3e_3=\alpha e_5$
\\

${\mathbf{N}}_{299}^{\alpha, \beta}$ & $:$ &$e_1e_1=e_2$ & $e_1e_2=e_3$ & $e_2e_2=e_5$ & $e_2e_3=e_4$ \\ && $e_2e_4=\alpha e_4$  & $e_3e_3=\beta e_5$ 
    & $e_3e_4=e_5$
\\

${\mathbf{N}}_{300}^{\alpha, \beta}$ & $:$ &$e_1e_1=e_2$ & $e_1e_2=e_3$ & $e_2e_2=e_5$ & $e_2e_3=e_4$ \\ && $e_2e_4=\alpha e_4$  & $e_3e_3=\beta e_5$ 
  & $e_4e_4=e_5$
\\

${\mathbf{N}}_{301}$ & $:$ &$e_1e_1=e_2$ & $e_1e_2=e_3$ & $e_2e_3=e_4$ & $e_2e_4=e_5$ 
\\

${\mathbf{N}}_{302}$ & $:$ &$e_1e_1=e_2$ & $e_1e_2=e_3$ & $e_2e_3=e_4$ & $e_2e_4=e_5$ & $e_3e_3=e_5$ 
\\

${\mathbf{N}}_{303}^{\alpha}$ & $:$ &$e_1e_1=e_2$ & $e_1e_2=e_3$ & $e_2e_3=e_4$ \\ & & $e_2e_4=e_5$ & $e_3e_3=\alpha e_5$ & $ e_3e_4=e_5$
\\

${\mathbf{N}}_{304}^{\alpha}$ & $:$ &$e_1e_1=e_2$ & $e_1e_2=e_3$ & $e_2e_3=e_4$ \\ & & $e_2e_4=e_5$ & $e_3e_3=\alpha e_5$ & $ e_4e_4=e_5$
\\

${\mathbf{N}}_{305}$ & $:$ &$e_1e_1=e_2$ & $e_1e_2=e_3$ & $e_2e_3=e_4$ & $e_3e_3=e_5$ & $e_3e_4=e_5$ 
\\

${\mathbf{N}}_{306}$ & $:$ &$e_1e_1=e_2$ & $e_1e_2=e_3$ & $e_2e_3=e_4$ & $e_3e_3=e_5$ & $e_4e_4=e_5$ 
\\

${\mathbf{N}}_{307}$ & $:$ &$e_1e_1=e_2$ & $e_1e_2=e_3$ & $e_2e_3=e_4$  & $e_4e_4=e_5$ 
\\

${\mathbf{N}}_{308}$ & $:$ &$e_1e_1=e_2$ & $e_1e_2=e_3$ & $e_2e_3=e_4$  & $e_4e_4=e_5$ 
\\

\end{longtable}

\subsection{$ 1 $-dimensional central extensions of $ {\mathbf N}_{02}^{4} $.} Here we will collect all information about $ {\mathbf N}_{02}^{4}: $

$$
\begin{array}{|l|l|l|l|}
\hline
\text{ }  & \text{ } & \text{Cohomology} & \text{Automorphisms} \\
\hline
{\mathbf{N}}^{4}_{02} & 
\begin{array}{l}
e_1e_1=e_2 \\ 
e_1e_2=e_3 \\ 
e_1e_3=e_4 \\ 
e_2e_3=e_4
\end{array}
&
\begin{array}{l}
\mathrm{H}^2_{\mathfrak{C}}(\mathbf{N}^{4}_{02})=\Big \langle [\Delta_{ij}] \Big\rangle\\
(i,j) \notin \{ (1,1),(1,2),(1,3)\}
\end{array}

& 
 
 \phi=\begin{pmatrix}
1&0&0&0\\
0&1&0&0\\
z&0&1&0\\
t&2z&z&1
\end{pmatrix}\\
\hline
\end{array}$$

Let us use the following notations:
\begin{longtable}{llll}
$\nabla_1=[\Delta_{14}],$&
$\nabla_2=[\Delta_{22}],$&
$\nabla_3=[\Delta_{23}],$&
$\nabla_4=[\Delta_{24}],$\\
$\nabla_5=[\Delta_{33}],$&
$\nabla_6=[\Delta_{34}],$&
$\nabla_7=[\Delta_{44}]. $
\end{longtable}
Take $ \theta=\sum\limits_{i=1}^{7}\alpha_i\nabla_i\in\mathrm{H}^2_{\mathfrak{C}}(\mathbf{N}^{4}_{02}) .$
Since
$$\phi^T\begin{pmatrix}
0&0&0&\alpha_1\\
0&\alpha_2&\alpha_3&\alpha_4\\
0&\alpha_3&\alpha_5&\alpha_6\\
\alpha_1&\alpha_4&\alpha_6&\alpha_7

\end{pmatrix}\phi=
\begin{pmatrix}
\alpha^*&\alpha^{**}&\alpha^{***}&\alpha^*_1\\
\alpha^{**}&\alpha^*_2&\alpha^*_3+\alpha^{***}&\alpha^*_4\\
\alpha^{***}&\alpha^*_3+\alpha^{***}&\alpha^*_5&\alpha^*_6\\
\alpha^*_1&\alpha^*_4&\alpha^*_6&\alpha^*_7
\end{pmatrix}$$	
we have	

\begin{longtable}{lcl}
$\alpha_1^*$&$=$&$\alpha_1+\alpha_6z+\alpha_7t,$\\
$\alpha_2^*$&$=$&$\alpha_2+4\alpha_4z+4\alpha_7z^2,$\\
$\alpha_3^*$&$=$&$\alpha_3+2\alpha_6z+(\alpha_4+2\alpha_7z)z-(\alpha_5z+\alpha_6t)-(\alpha_1+\alpha_6z+\alpha_7t)z,$\\
$\alpha_4^*$&$=$&$\alpha_4+2\alpha_7z,$\\ 
$\alpha_5^*$&$=$&$\alpha_5+2\alpha_6z+\alpha_7z^2,$\\
$\alpha_6^*$&$=$&$\alpha_6+\alpha_7z,$\\
$\alpha_7^*$&$=$&$\alpha_7.$
\end{longtable}	

We are interested in $ (\alpha_1,\alpha_4,\alpha_6,\alpha_7)\neq(0,0,0,0) $ and consider following cases:

\begin{enumerate}
\item if $ \alpha_7=\alpha_6=\alpha_4=0, $ then $ \alpha_1\neq0, $ and we have
\begin{enumerate}
	\item if $ \alpha_5=-\alpha_1 ,$ then we have the family of representatives \begin{center}$ \langle \nabla_1+\alpha\nabla_2+\beta\nabla_3-\nabla_5 \rangle;$
	\end{center}
	\item if $ \alpha_5\neq-\alpha_1 ,$ then by choosing $ z=-\frac{\alpha_3}{\alpha_1+\alpha_5}, t=0, $ we have the family of representatives
	$ \langle \nabla_1+\alpha\nabla_2+\beta\nabla_5 \rangle_{\beta\neq-1};$
	\end{enumerate}

\item if $ \alpha_7=0, \alpha_6=0, \alpha_4\neq0, $ then by choosing $ z=-\frac{\alpha_2}{4\alpha_4}, t=0, $ we have the family of representatives 
$ \langle \alpha\nabla_1+\beta\nabla_3+\nabla_4+\gamma\nabla_5 \rangle;$

\item if $ \alpha_7=0, \alpha_6\neq0, $ then by choosing \begin{center} $ z=-{\alpha_1}{\alpha_6}^{-1}, t=({\alpha_3\alpha_6-\alpha_1(2\alpha_6+\alpha_4-\alpha_5)}){\alpha_6^{-1}},$\end{center} we have the family of  representatives $ \langle \alpha\nabla_2+\beta\nabla_4+\gamma\nabla_5+\nabla_6 \rangle;$

\item if $ \alpha_7\neq0, $ then by choosing $ z=-{\alpha_6}{\alpha_7}^{-1}, t=({\alpha^2_6-\alpha_1\alpha_7}){\alpha^{-2}_7},$ we have the family of representatives $ \langle \alpha\nabla_2+\beta\nabla_3+\gamma\nabla_4+\mu\nabla_5+\nabla_7 \rangle.$

\end{enumerate}		
	
Summarizing, we have the following distinct orbits:
\begin{center}

$\left\langle \nabla_1+ \alpha \nabla_2 + \beta \nabla_3 - \nabla_5 \right\rangle,$ 
$\left\langle \nabla_1+ \alpha \nabla_2 + \beta \nabla_5 \right\rangle_{\beta\neq -1},$ $\left\langle \alpha \nabla_1+ \beta \nabla_3 + \nabla_4 + \gamma\nabla_5 \right\rangle,$
$\left\langle \alpha \nabla_2 + \beta \nabla_3 + \gamma\nabla_4 + \mu \nabla_5 + \nabla_7\right\rangle,$ 
$\left\langle \alpha \nabla_2 + \beta \nabla_4 +\gamma\nabla_5+ \nabla_6\right\rangle,$

\end{center}

which gives the following new algebras:

\begin{longtable}{llllllllllllllllll}

${\mathbf{N}}_{309}^{\alpha, \beta}$ & $:$ &$e_1e_1=e_2$ & $e_1e_2=e_3$ & $e_1e_3=e_4$ & $e_1e_4=e_5$ \\ & & $e_2e_2=\alpha e_5$  & $e_2e_3=e_4+\beta e_5$ 
   & $e_3e_3=-e_5$ \\

${\mathbf{N}}_{310}^{\alpha, \beta\neq-1}  $ & $:$ &$e_1e_1=e_2$ & $e_1e_2=e_3$ & $e_1e_3=e_4$ & $e_1e_4=e_5$  \\ && $e_2e_2=\alpha e_5$  & $e_2e_3=e_4$ 
   & $e_3e_3=\beta e_5$ \\

${\mathbf{N}}_{311}^{\alpha, \beta,\gamma}$ & $:$ &$e_1e_1=e_2$ & $e_1e_2=e_3$ & $e_1e_3=e_4$ & $e_1e_4=\alpha e_5$ \\ & & $e_2e_3=e_4+\beta e_5$  & $e_2e_4=e_5$ 
   & $e_3e_3=\gamma e_5$ \\

${\mathbf{N}}_{312}^{\alpha, \beta,\gamma, \mu}$ & $:$ &$e_1e_1=e_2$ & $e_1e_2=e_3$ & $e_1e_3=e_4$ & $e_2e_2=\alpha e_5$  \\ && $e_2e_3=e_4+\beta e_5$  & $e_2e_4=\gamma e_5$ 
  & $e_3e_3=\mu e_5$ & $ e_4e_4=e_5$\\

${\mathbf{N}}_{313}^{\alpha, \beta,\gamma}$ & $:$ &$e_1e_1=e_2$ & $e_1e_2=e_3$ & $e_1e_3=e_4$ & $e_2e_2=\alpha e_5$ \\ & & $e_2e_3=e_4$  & $e_2e_4=\beta e_5$ 
 & $e_3e_3=\gamma e_5$ & $ e_3e_4=e_5$\\

\end{longtable}

\subsection{$ 1 $-dimensional central extensions of $ {\mathbf N}_{03}^{4} $.} Here we will collect all information about $ {\mathbf N}_{03}^{4}: $

$$
\begin{array}{|l|l|l|l|}
\hline
\text{ }  & \text{ } & \text{Cohomology} & \text{Automorphisms} \\
\hline
{\mathbf{N}}^{4}_{03} & 
\begin{array}{l}
e_1e_1=e_2 \\ 
e_1e_2=e_3 \\ 
e_3e_3=e_4
\end{array}
&\begin{array}{l}
\mathrm{H}^2_{\mathfrak{C}}(\mathbf{N}^{4}_{03})= 
\Big \langle [\Delta_{ij}] \Big\rangle\\
{(i,j) \notin \{ (1,1),(1,2),(3,3)}\}
\end{array}
&
 \phi=\begin{pmatrix}
x&0&0&0\\
0&x^2&0&0\\
0&0&x^3&0\\
t&0&0&x^6
\end{pmatrix}\\
\hline
\end{array}$$	

Let us use the following notations:
\begin{longtable}{llll}
$\nabla_1=[\Delta_{13}],$&
$\nabla_2=[\Delta_{14}],$&
$\nabla_3=[\Delta_{22}],$&
$\nabla_4=[\Delta_{23}],$\\
$\nabla_5=[\Delta_{24}],$&
$\nabla_6=[\Delta_{34}],$&
$\nabla_7=[\Delta_{44}]. $
\end{longtable}
	
Take $ \theta=\sum\limits_{i=1}^{7}\alpha_i\nabla_i\in\mathrm{H}^2_{\mathfrak{C}}(\mathbf{N}^{4}_{03}) .$
Since
$$\phi^T\begin{pmatrix}
0&0&\alpha_1&\alpha_2\\
0&\alpha_3&\alpha_4&\alpha_5\\
\alpha_1&\alpha_4&0&\alpha_6\\
\alpha_2&\alpha_5&\alpha_6&\alpha_7

\end{pmatrix}\phi=
\begin{pmatrix}
\alpha^*&\alpha^{**}&\alpha^{*}_1&\alpha^*_2\\
\alpha^{**}&\alpha^*_3&\alpha^{*}_4&\alpha^*_5\\
\alpha^{*}_1&\alpha^{*}_4&0&\alpha^*_6\\
\alpha^*_2&\alpha^*_5&\alpha^*_6&\alpha^*_7
\end{pmatrix}$$	
we have

\begin{longtable}{llll}
$\alpha_1^*=(\alpha_1x+\alpha_6t)x^3,$&
$\alpha_2^*=(\alpha_2x+\alpha_7t)x^6,$&
$\alpha_3^*=\alpha_3x^4,$&
$\alpha_4^*=\alpha_4x^5,$\\
$\alpha_5^*=\alpha_5x^8,$&
$\alpha_6^*=\alpha_6x^{9},$&
$\alpha_7^*=\alpha_7x^{12}.$
\end{longtable}	

We are interested in $ (\alpha_2,\alpha_5,\alpha_6,\alpha_7)\neq(0,0,0,0) $ and consider following cases:

\begin{enumerate}
\item  $ \alpha_7=\alpha_6=\alpha_5=0, $ then $ \alpha_2\neq0, $ and we have the following subcases:

    \begin{enumerate}
		\item if $ \alpha_1=0, \alpha_3=0, $ then we have the representatives $ \langle \nabla_2 \rangle $ and $ \langle \nabla_2+\nabla_4 \rangle $ depending on whether $ \alpha_4=0 $ or not;
		\item if $ \alpha_1=0, \alpha_3\neq0, $ then by choosing $ x=\sqrt[3]{{\alpha_3}{\alpha_2}^{-1}}, t=0, $ we have the family of representatives 
		$ \langle \nabla_2+\nabla_3+\alpha\nabla_4 \rangle; $
		\item if $ \alpha_1\neq0, $ then by choosing $ x=\sqrt[3]{{\alpha_1}{\alpha_2^{-1}}}, t=0, $ we have the family of representatives $ \langle \nabla_1+\nabla_2+\alpha\nabla_3+\beta\nabla_4 \rangle. $
		
	\end{enumerate}
	
\item  $ \alpha_7=0, \alpha_6=0, \alpha_5\neq0, $ then we have the following subcases:

    \begin{enumerate}
	\item if $ \alpha_1=0, \alpha_2=0, \alpha_3=0, $ then we have the representatives $ \langle \nabla_5 \rangle $ and $ \langle \nabla_4+\nabla_5 \rangle $ depending on whether $ \alpha_5=0 $ or not;
	\item if $ \alpha_1=0, \alpha_2=0, \alpha_3\neq0, $ then by choosing $ x=\sqrt[4]{{\alpha_3}{\alpha_5^{-1}}}, t=0, $ we have the family of  representatives $ \langle \nabla_3+\alpha\nabla_4+\nabla_5 \rangle;$
	\item if $ \alpha_1=0, \alpha_2\neq0, $ then by choosing $ x={\alpha_2}{\alpha_5^{-1}}, t=0, $ we have the family of representatives $ \langle \nabla_2+\alpha\nabla_3+\beta\nabla_4+\nabla_5 \rangle;$	
	
	\item if $ \alpha_1\neq0, $ then by choosing $ x=\sqrt[4]{{\alpha_1}{\alpha_5^{-1}}}, t=0, $ we have the family of representatives  $ \langle \nabla_1+\alpha\nabla_2+\beta\nabla_3+\gamma\nabla_4+\nabla_5 \rangle. $
    \end{enumerate}

\item  $ \alpha_7=0, \alpha_6\neq0, $ then we have the following subcases:
\begin{enumerate}
	\item if $  \alpha_2=0, \alpha_3=0, \alpha_4=0, $	 then we have  representatives $ \langle \nabla_6 \rangle $ and $ \langle \nabla_5+\nabla_6 \rangle$ depending on whether $ \alpha_5=0 $ or not;
	
	\item if $ \alpha_2=0, \alpha_3=0, \alpha_4\neq0, $ then by choosing $ x=\sqrt[4]{{\alpha_4}{\alpha_6}^{-1}}, t=-\alpha_1\sqrt[4]{\alpha_4  \alpha_6^{-5}} , $ we have the family of representatives $ \langle \nabla_4+\alpha\nabla_5+\nabla_6 \rangle;$
	
	\item if $ \alpha_2=0, \alpha_3\neq0, $ then by choosing $ x=\sqrt[5]{{\alpha_3}{\alpha_6^{-1}}}, t=- \alpha_1\sqrt[5]{\alpha_3 \alpha_6^{-6}}, $ we have the family of representatives $ \langle \nabla_3+\alpha\nabla_4+\beta\nabla_5+\nabla_6 \rangle;$
	
	\item if $ \alpha_2\neq0, $ then by choosing $ x={{\alpha_2}{\alpha_6^{-1}}}, t=-{\alpha_1\sqrt{\alpha_2 \alpha_6^{-3}}}, $ we have the family of representatives $ \langle \nabla_2+\alpha\nabla_3+\beta\nabla_4+\gamma\nabla_5+\nabla_6 \rangle.$
	
\end{enumerate}	
\item  $ \alpha_7\neq0, $ then we have the following subcases:

\begin{enumerate}
	\item $ \alpha_1\alpha_7-\alpha_2\alpha_6=0, \alpha_3=0, \alpha_4=0, \alpha_5=0,$ then we have  representatives $ \langle \nabla_7 \rangle $ and $ \langle \nabla_6+\nabla_7 \rangle$ depending on whether $ \alpha_6=0 $ or not;
	
	\item $ \alpha_1\alpha_7-\alpha_2\alpha_6=0, \alpha_3=0, \alpha_4=0, \alpha_5\neq0,$ then by choosing $ x=\sqrt[4]{{\alpha_5}{\alpha_7^{-1}}},$ $t=-{\alpha_2\sqrt[4]{\alpha_5 \alpha_7^{-5}}}, $ we have  the family of representatives $ \langle \nabla_5+\alpha\nabla_6+\nabla_7 \rangle;$
	
	\item $ \alpha_1\alpha_7-\alpha_2\alpha_6=0, \alpha_3=0, \alpha_4\neq0, $ then by choosing $ x=\sqrt[7]{{\alpha_4\alpha_7^{-1}}},$ $t=-{\alpha_2\sqrt[7]{\alpha_4 \alpha_7^{-8}}}, $ we have the family of  representatives $ \langle \nabla_4+\alpha\nabla_5+\beta\nabla_6+\nabla_7 \rangle;$
	
	\item $\alpha_1\alpha_7-\alpha_2\alpha_6=0, \alpha_3\neq0,$ then by choosing $ x=\sqrt[8]{{\alpha_3}{\alpha_7^{-1}}},$ $t=-{\alpha_2\sqrt[8]{\alpha_3 \alpha_7^{-9}}}, $ we have the family of representatives $ \langle \nabla_3+\alpha\nabla_4+\beta\nabla_5+\gamma\nabla_6+\nabla_7 \rangle;$	
	
	\item $\alpha_1\alpha_7-\alpha_2\alpha_6\neq0, $ then by choosing \begin{center}$ x=\sqrt[8]{{(\alpha_1\alpha_7-\alpha_2\alpha_6)}{\alpha^{-2}_7}}, t=-{\alpha_2\sqrt[8]{(\alpha_1\alpha_7-\alpha_2\alpha_6)\alpha_7^{-10}}}, $\end{center} we have the family of representatives  $ \langle \nabla_1+\alpha\nabla_3+\beta\nabla_4+\gamma\nabla_5+\mu\nabla_6+\nabla_7 \rangle.$		
	
\end{enumerate}

\end{enumerate}		
	
Summarizing, we have the following distinct orbits:
\begin{center}

$\left\langle \nabla_1+\nabla_2 + \alpha \nabla_3 + \beta \nabla_4 \right\rangle
^{O(\alpha, \beta)=O(\alpha, -\eta_3 \beta)=O(\alpha, \eta_3^2\beta)}, $ 
$\left\langle \nabla_1+ \alpha \nabla_2 + \beta \nabla_3 + \gamma \nabla_4 + \nabla_5\right\rangle
^{O(\alpha, \beta, \gamma)=O(-i\alpha, \beta, i\gamma)=O(i\alpha, \beta,-i \gamma)=O(-\alpha, \beta, -\gamma)},$
$\left\langle \nabla_1+\alpha \nabla_3+ \beta \nabla_4 + \gamma\nabla_5 + \mu\nabla_6 + \nabla_7\right\rangle^{
{\tiny\begin{array}{l}
O(\alpha,\beta,\gamma, \mu)=O(\alpha,\eta_4^3\beta,-\gamma, -\eta_4^3\mu)=\\
O(\alpha,-\eta_4^3\beta,-\gamma, \eta_4^3\mu)=
O(\alpha,\eta_4 \beta,-\gamma, -\eta_4\mu)=\\
O(\alpha,-\eta_4\beta,-\gamma, \eta_4\mu)= O(\alpha,i\beta,\gamma, i\mu)=\\
O(\alpha,-i\beta,\gamma, -i\mu)=O(\alpha,-\beta, \gamma, -\mu)
\end{array}} },$ 
$  \left\langle \nabla_2\right\rangle,$ 
$\left\langle \nabla_2+\nabla_3 +\alpha\nabla_4\right\rangle^{O(\alpha)=O(-\eta_3\alpha)=O(\eta_3^2\alpha)},$
$\left\langle \nabla_2+\alpha\nabla_3 +\beta\nabla_4 + \nabla_5\right\rangle,$
$ \left\langle \nabla_2+\alpha \nabla_3+ \beta \nabla_4 + \gamma\nabla_5 + \nabla_6\right\rangle^{O(\alpha,\beta,\gamma)=O(-\alpha,\beta,-\gamma)},$ 
$\left\langle \nabla_2+ \nabla_4\right\rangle,$ 
$ \left\langle \nabla_3+\alpha \nabla_4+\nabla_5\right\rangle^{
{\tiny \begin{array}{l}
O(\alpha)=O(-\alpha)=\\
O(i\alpha)=O(-i\alpha)
\end{array}}},$
$ \left\langle \nabla_3+\alpha \nabla_4+ \beta \nabla_5+ \nabla_6\right\rangle
^{{\tiny \begin{array}{l}
O(\alpha, \beta)=O(\eta_5^4\alpha, -\eta_5\beta)=
O(-\eta_5^3\alpha, \eta_5^2\beta)=\\
O(\eta_5^2\alpha, -\eta_5^3\beta)=O(-\eta_5\alpha, \eta_5^4\beta)
\end{array}}},$ 
$\left\langle \nabla_3+\alpha \nabla_4+ \beta \nabla_5 +\gamma\nabla_6 + \nabla_7\right\rangle^{
{\tiny\begin{array}{l}
O(\alpha,\beta,\gamma)=O(\eta_4^3\alpha,-\beta,-\eta_4^3\gamma)=
O(-\eta_4^3\alpha,-\beta,\eta_4^3\gamma)=\\
O(\eta_4\alpha,-\beta,-\eta_4\gamma)=
O(-\eta_4\alpha,-\beta,\eta_4\gamma)=\\ O(i\alpha,\beta,i\gamma)=O(-i\alpha,\beta,-i\gamma)=O(-\alpha,\beta,-\gamma)
\end{array}} },$  
$ \left\langle \nabla_4+ \nabla_5\right\rangle,$ 
$\left\langle \nabla_4+\alpha \nabla_5+ \nabla_6\right\rangle^{O(\alpha)=O(i\alpha)=O(-\alpha)=O(-i\alpha)}, $
$ \left\langle \nabla_4+\alpha \nabla_5+ \beta\nabla_6 + \nabla_7\right\rangle^{
{\tiny\begin{array}{l}
O(\alpha,\beta)=
O(\eta^4_7\alpha,-\eta^3_7\beta)=
O(-\eta_7\alpha,\eta_7^6\beta)=O(-\eta_7^5\alpha,\eta^2_7\beta)=\\
O(\eta^2_7\alpha,-\eta^5_7\beta)= O(\eta_7^6\alpha,-\eta_7\beta)=O(-\eta^3_7\alpha,\eta^4_7\beta)
\end{array}}},$
$\left\langle \nabla_5\right\rangle,$
$\left\langle \nabla_5 + \nabla_6\right\rangle,$
$ \left\langle \nabla_5+\alpha \nabla_6+ \nabla_7\right\rangle^{O(\alpha)=O(i\alpha)=O(-\alpha)=O(-i\alpha)},$
$\left\langle \nabla_6\right\rangle,$ 
$\left\langle \nabla_6 +\nabla_7\right\rangle,$ 
$\left\langle \nabla_7\right\rangle,$
\end{center}
which gives the following new algebras:

\begin{longtable}{llllllllllllllllll}

${\mathbf{N}}_{314}^{\alpha, \beta}$ & $:$ &$e_1e_1=e_2$ & $e_1e_2=e_3$ & $e_1e_3=e_5$ & $e_1e_4=e_5$  \\ & & $e_2e_2=\alpha e_5$  & $e_2e_3=\beta e_5$ 
  & $e_3e_3=e_4$ \\

${\mathbf{N}}_{315}^{\alpha, \beta ,\gamma}$ & $:$ &$e_1e_1=e_2$ & $e_1e_2=e_3$ & $e_1e_3=e_5$ & $e_1e_4=\alpha e_5$   \\ && $e_2e_2=\beta e_5$  & $e_2e_3=\gamma e_5$ 
   &  $e_2e_4=e_5 $ & $e_3e_3=e_4$ \\  
 
${\mathbf{N}}_{316}^{\alpha, \beta, \gamma, \mu}$ & $:$ &$e_1e_1=e_2$ & $e_1e_2=e_3$ & $e_1e_3=e_5$ & $e_2e_2=\alpha e_5$ & $e_2e_3=\beta e_5$   \\ & & $e_2e_4=\gamma e_5$ 
    &  $e_3e_3= e_4 $ & $e_3e_4=\mu e_5$ & $ e_4e_4=e_5$ \\ 

${\mathbf{N}}_{317}$ & $:$ &$e_1e_1=e_2$ & $e_1e_2=e_3$ & $e_1e_4=e_5$ & $e_3e_3=e_4$ \\

${\mathbf{N}}_{318}^{\alpha}$ & $:$ &$e_1e_1=e_2$ & $e_1e_2=e_3$ & $e_1e_4=e_5$ \\ && $e_2e_2=e_5$ & $e_2e_3=\alpha e_5$ & $e_3e_3=e_4$ \\  

${\mathbf{N}}_{319}^{\alpha, \beta}$ & $:$ &$e_1e_1=e_2$ & $e_1e_2=e_3$ & $e_1e_4=e_5$ & $e_2e_2=\alpha e_5$   \\ && $e_2e_3=\beta e_5$  & $e_2e_4=e_5$ 
   & $e_3e_3=e_4$ \\  

${\mathbf{N}}_{320}^{\alpha, \beta ,\gamma}$ & $:$ &$e_1e_1=e_2$ & $e_1e_2=e_3$ & $e_1e_4=e_5$ & $e_2e_2=\alpha e_5$   \\ && $e_2e_3=\beta e_5$  & $e_2e_4=\gamma e_5$ 
   &  $e_3e_3=e_4 $ & $e_3e_4=e_5$ \\

${\mathbf{N}}_{321}$ & $:$ &$e_1e_1=e_2$ & $e_1e_2=e_3$ & $e_1e_4=e_5$ & $e_2e_3=e_5$ & $e_3e_3=e_4$ \\  

${\mathbf{N}}_{322}^{\alpha}$ & $:$ &$e_1e_1=e_2$ & $e_1e_2=e_3$ & $e_2e_2=e_5$ \\ && $e_2e_3=\alpha e_5$ & $e_2e_4=e_5$  & $e_3e_3=e_4$ 
  \\  

${\mathbf{N}}_{323}^{\alpha, \beta}$ & $:$ &$e_1e_1=e_2$ & $e_1e_2=e_3$ & $e_2e_2=e_5$ & $e_2e_3=\alpha e_5$   \\ && $e_2e_4=\beta e_5$  & $e_3e_3=e_4$ 
  &  $e_3e_4=e_5$ \\  

${\mathbf{N}}_{324}^{\alpha, \beta ,\gamma}$ & $:$ &$e_1e_1=e_2$ & $e_1e_2=e_3$ & $e_2e_2=e_5$ & $e_2e_3=\alpha e_5$   \\ && $e_2e_4=\beta e_5$  & $e_3e_3=e_4$ 
  &  $e_3e_4=\gamma e_5 $ & $e_4e_4=e_5$ \\

${\mathbf{N}}_{325}$ & $:$ &$e_1e_1=e_2$ & $e_1e_2=e_3$ & $e_2e_3=e_5$ \\ && $e_2e_4=e_5$ & $e_3e_3=e_4$ \\  
 
${\mathbf{N}}_{326}^{\alpha}$ & $:$ &$e_1e_1=e_2$ & $e_1e_2=e_3$ & $e_2e_3=e_5$ \\ && $e_2e_4=\alpha e_5$ & $e_3e_3=e_4$  & $e_3e_4=e_5$ \\ 

${\mathbf{N}}_{327}^{\alpha, \beta}$ & $:$ &$e_1e_1=e_2$ & $e_1e_2=e_3$ & $e_2e_3=e_5$ & $e_2e_4=\alpha e_5$  \\ & & $e_3e_3=e_4$  & $e_3e_4=\beta e_5$ 
  &  $e_4e_4=e_5 $  \\

${\mathbf{N}}_{328}$ & $:$ &$e_1e_1=e_2$ & $e_1e_2=e_3$  & $e_2e_4=e_5$   & $e_3e_3=e_4$ \\  

${\mathbf{N}}_{329}$ & $:$ &$e_1e_1=e_2$ & $e_1e_2=e_3$  & $e_2e_4=e_5$   & $e_3e_3=e_4$ & $e_3e_4=e_5 $ \\

${\mathbf{N}}^\alpha_{330}$ & $:$ &$e_1e_1=e_2$ & $e_1e_2=e_3$  & $e_2e_4=e_5$  \\ & & $e_3e_3=e_4$ & $e_3e_4=\alpha e_5 $  & $ e_4e_4=e_5$\\

${\mathbf{N}}_{331}$ & $:$ &$e_1e_1=e_2$ & $e_1e_2=e_3$   & $e_3e_3=e_4$ & $e_3e_4=e_5 $ \\
 
${\mathbf{N}}_{332}$ & $:$ &$e_1e_1=e_2$ & $e_1e_2=e_3$  & $e_3e_3=e_4$   & $e_3e_4=e_5$ & $e_4e_4=e_5 $ \\ 

${\mathbf{N}}_{333}$ & $:$ &$e_1e_1=e_2$ & $e_1e_2=e_3$    & $e_3e_3=e_4$ & $e_4e_4=e_5 $ \\
\end{longtable}

\subsection{$ 1 $-dimensional central extensions of $ {\mathbf N}_{04}^{4} $.} Here we will collect all information about $ {\mathbf N}_{04}^{4}: $

$$
\begin{array}{|l|l|l|l|}
\hline
\text{ }  & \text{ } & \text{Cohomology}  & \text{Automorphisms} \\
\hline
{\mathbf{N}}^{4}_{04} & 
\begin{array}{l}
e_1e_1=e_2 \\ 
e_1e_2=e_3 \\ 
e_2e_2=e_4 \\ 
e_3e_3=e_4
\end{array}
&
\begin{array}{l}
\mathrm{H}^2_{\mathfrak{C}}(\mathbf{N}^{4}_{04})=
\Big \langle [\Delta_{ij}] \Big\rangle\\
(i,j) \notin \{(1,1),(1,2),(3,3)\}
\end{array}
 &

\phi_{\pm}=\begin{pmatrix}
\pm1&0&0&0\\
0&1&0&0\\
0&0&\pm1&0\\
t&0&0&1
\end{pmatrix}\\
\hline
\end{array}$$	

Let us use the following notations:
\begin{longtable}{llll}
$\nabla_1=[\Delta_{13}], $&$ \nabla_2=[\Delta_{14}], $&$ \nabla_3=[\Delta_{22}], $&$  \nabla_4=[\Delta_{23}], $\\$
\nabla_5=[\Delta_{24}],  $&$ 
\nabla_6=[\Delta_{34}],  $&$  \nabla_7=[\Delta_{44}]. $
\end{longtable}

Take $ \theta=\sum\limits_{i=1}^{7}\alpha_i\nabla_i\in\mathrm{H}^2_{\mathfrak{C}}(\mathbf{N}^{4}_{04}) .$
Since
$$\phi^T\begin{pmatrix}
0&0&\alpha_1&\alpha_2\\
0&\alpha_3&\alpha_4&\alpha_5\\
\alpha_1&\alpha_4&0&\alpha_6\\
\alpha_2&\alpha_5&\alpha_6&\alpha_7

\end{pmatrix}\phi=
\begin{pmatrix}
\alpha^*&\alpha^{**}&\alpha^{*}_1&\alpha^*_2\\
\alpha^{**}&\alpha^*_3&\alpha^{*}_4&\alpha^*_5\\
\alpha^{*}_1&\alpha^{*}_4&0&\alpha^*_6\\
\alpha^*_2&\alpha^*_5&\alpha^*_6&\alpha^*_7
\end{pmatrix}$$		
we have	

\begin{longtable}{llll}
$\alpha_1^*=\alpha_1\pm\alpha_6t,$&
$\alpha_2^*=\pm\alpha_2+\alpha_7t,$&
$\alpha_3^*=\alpha_3,$&
$\alpha_4^*=\pm\alpha_4,$\\
$\alpha_5^*=\alpha_5,$&
$\alpha_6^*=\pm \alpha_6,$&
$\alpha_7^*=\alpha_7.$
\end{longtable}	

We are interested in $ (\alpha_2,\alpha_5,\alpha_6,\alpha_7)\neq(0,0,0,0) $ and consider following cases:

\begin{enumerate}
\item if $ \alpha_7=\alpha_6=\alpha_5=0, $ then $ \alpha_2\neq0, $ and we have the family of representatives 
\begin{center}
    $ \langle \alpha\nabla_1+\nabla_2+\beta\nabla_3+\gamma\nabla_4 \rangle;$ 
	\end{center}

\item if $ \alpha_7=0, \alpha_6=0, \alpha_5\neq0, $ then  we have the family of representatives 
\begin{center}
    $ \langle \alpha\nabla_1+\beta\nabla_2+\gamma\nabla_3+\mu\nabla_4+\nabla_5 \rangle;$  
\end{center}

\item if $ \alpha_7=0, \alpha_6\neq0, $ then by choosing $\phi=\phi_+, t=-{\alpha_1}{\alpha_6^{-1}}, $ we have the family of  representatives
\begin{center}$ \langle \alpha\nabla_2+\beta\nabla_3+\gamma\nabla_4+\mu\nabla_5+\nabla_6 \rangle;$  
\end{center}

\item if $ \alpha_7\neq0, $ then by choosing $\phi=\phi_+, t=-{\alpha_2}{\alpha_7^{-1}}$ we have the family of representatives
\begin{center}
$\langle \alpha\nabla_1+\beta\nabla_3+\gamma\nabla_4+\mu\nabla_5+\nu\nabla_6+\nabla_7 \rangle.$
\end{center}

\end{enumerate}

Summarizing, we have the following distinct orbits:

\begin{center}
$ \langle \alpha\nabla_1+\nabla_2+\beta\nabla_3+\gamma\nabla_4 \rangle^{O(\alpha, \beta, \gamma)= O(-\alpha, -\beta, \gamma)},$  
$\langle \alpha\nabla_1+\beta\nabla_2+\gamma\nabla_3+\mu\nabla_4+\nabla_5 \rangle
^{O(\alpha, \beta, \gamma,\mu)= O(\alpha, -\beta, \gamma, -\mu)}$
 $\langle \alpha\nabla_1+\beta\nabla_3+\gamma\nabla_4+\mu\nabla_5+\nu\nabla_6+\nabla_7 \rangle^{O(\alpha, \beta, \gamma,\mu,\nu)=O(\alpha, \beta, -\gamma,\mu,-\nu)},$
$ \langle \alpha\nabla_2+\beta\nabla_3+\gamma\nabla_4+\mu\nabla_5+\nabla_6 \rangle ^{O(\alpha, \beta, \gamma, \mu)= O(\alpha, -\beta, \gamma, -\mu)},$

\end{center}
which gives the following new algebras:

\begin{longtable}{llllllllllllllllll}

${\mathbf{N}}_{334}^{\alpha, \beta ,\gamma}$ & $:$ &$e_1e_1=e_2$ & $e_1e_2=e_3$ & $e_1e_3=\alpha e_5$ & $e_1e_4=e_5$ \\ & & $e_2e_2=e_4+\beta e_5$    & $e_2e_3=\gamma e_5$ 
   & $e_3e_3=e_4$ \\

${\mathbf{N}}_{335}^{\alpha, \beta,\gamma, \mu}$ & $:$ &$e_1e_1=e_2$ & $e_1e_2=e_3$ & $e_1e_3=\alpha e_5$  & $e_1e_4=\beta e_5$   \\ && $e_2e_2=e_4+\gamma e_5$ & $e_2e_3=\mu e_5$  & $e_2e_4= e_5$
 & $e_3e_3=e_4$ \\

${\mathbf{N}}_{336}^{\alpha, \beta,\gamma, \mu, \nu}$ & $:$ &$e_1e_1=e_2$ & $e_1e_2=e_3$ & $e_1e_3=\alpha e_5$  \\ & & $e_2e_2=e_4+\beta e_5$  & $e_2e_3=\gamma e_5$   & $e_2e_4= \mu e_5$ \\ && $e_3e_3=e_4$ & $e_3e_4=\nu e_5$ & $e_4e_4= e_5$ \\

${\mathbf{N}}_{337}^{\alpha, \beta,\gamma, \mu}$ & $:$ &$e_1e_1=e_2$ & $e_1e_2=e_3$ & $e_1e_4=\alpha e_5$   & $e_2e_2=e_4+\beta e_5$  \\ && $e_2e_3=\gamma e_5$   & $e_2e_4= \mu e_5$ & $e_3e_3=e_4$ & $e_3e_4= e_5$ \\
\end{longtable}

\subsection{$ 1 $-dimensional central extensions of $ {\mathbf N}_{05}^{4} $.} Here we will collect all information about $ {\mathbf N}_{05}^{4}: $

$$
\begin{array}{|l|l|l|l|}
\hline
\text{ }  & \text{ } & \text{Cohomology} & \text{Automorphisms}\\
\hline
{\mathbf{N}}^{4}_{05} & 
\begin{array}{l}
e_1e_1=e_2 \\ 
e_1e_3=e_4 \\ 
e_2e_2=e_3
\end{array}
&

 \begin{array}{l}
\mathrm{H}^2_{\mathfrak{C}}(\mathbf{N}^{4}_{05})=
\Big \langle [\Delta_{ij}] \Big\rangle\\
{(i,j) \notin \{ (1,1),(1,3),(2,2)} \}
\end{array}
 
 & 
 
 \phi=\begin{pmatrix}
x&0&0&0\\
0&x^2&0&0\\
z&0&x^4&0\\
t&2xz&0&x^5
\end{pmatrix}\\

\hline
\end{array}$$

Let us use the following notations:
 \begin{longtable}{llll}
 $\nabla_1=[\Delta_{12}]$&$  \nabla_2=[\Delta_{14}] $&$  \nabla_3=[\Delta_{23}] $&$   \nabla_4=[\Delta_{24}] $\\
 $ 
  \nabla_5=[\Delta_{33}] $&$  \nabla_6=[\Delta_{34}]  $&$   \nabla_7=[\Delta_{44}]. $
  \end{longtable}
	
Take $ \theta=\sum\limits_{i=1}^{7}\alpha_i\nabla_i\in\mathrm{H}^2_{\mathfrak{C}}(\mathbf{N}^{4}_{05}) .$  Since
$$\phi^T\begin{pmatrix}
0&\alpha_1&0&\alpha_2\\
\alpha_1&0&\alpha_3&\alpha_4\\
0&\alpha_3&\alpha_5&\alpha_6\\
\alpha_2&\alpha_4&\alpha_6&\alpha_7

\end{pmatrix}\phi=
\begin{pmatrix}
\alpha^*&\alpha^*_1&\alpha^{***}&\alpha^*_2\\
\alpha^*_1&\alpha^{**}&\alpha^*_3&\alpha^*_4\\
\alpha^{***}&\alpha^*_3&\alpha^*_5&\alpha^*_6\\
\alpha^*_2&\alpha^*_4&\alpha^*_6&\alpha^*_7
\end{pmatrix}$$	

we have
	
\begin{longtable}{lll}
\multicolumn{3}{l}{$\alpha_1^*=(\alpha_1x+\alpha_3z+\alpha_4t)x^2+2(\alpha_2x+\alpha_6z+\alpha_7t)xz,$}\\

$\alpha_2^*=(\alpha_2x+\alpha_6z+\alpha_7t)x^5,$&
$\alpha_3^*=(\alpha_3x+2\alpha_6z)x^5,$&
$\alpha_4^*=(\alpha_4x+2\alpha_7z)x^6,$\\

$\alpha_5^*=\alpha_5x^8,$&
$\alpha_6^*=\alpha_6x^9,$&
$\alpha_7^*=\alpha_7x^{10}.$
\end{longtable}

We are interested in $ (\alpha_2,\alpha_4,\alpha_6,\alpha_7)\neq(0,0,0,0) $ and consider following cases:

\begin{enumerate}
	\item  $ \alpha_7=\alpha_6=\alpha_4=0, $ then $ \alpha_2\neq0 $ and we have the following subcases:
	
	\begin{enumerate}
		\item $ \alpha_3=-2\alpha_2,$
		\begin{enumerate}
			\item\label{05cs1a1} if $ \alpha_1=0, \alpha_5=0, $ then we have the representative $ \langle \nabla_2-2\nabla_3 \rangle;$
			\item\label{05cs1a2} if $ \alpha_1=0, \alpha_5\neq0, $ then by choosing $ x=\sqrt{{\alpha_2}{\alpha_5^{-1}}}, z=0, t=0, $ we have the representative $ \langle \nabla_2-2\nabla_3+\nabla_5 \rangle; $
			\item if $ \alpha_1\neq0, $ then by choosing $ x=\sqrt[3]{{\alpha_1}{\alpha_2^{-1}}}, z=0, t=0, $ we have the family of representatives $ \langle \nabla_1+\nabla_2-2\nabla_3+\alpha\nabla_5 \rangle.$
		\end{enumerate}
		\item $ \alpha_3\neq -2\alpha_2,$ 
		\begin{enumerate}
			\item if $ \alpha_5=0,$ then  choosing $x=1, z=-\frac{\alpha_1}{\alpha_3+2\alpha_2},$ we have the family of representatives $ \langle \nabla_2+\alpha\nabla_3 \rangle_{\alpha\neq-2},$ which will be jointed with the case (\ref{05cs1a1});
			\item if $ \alpha_5\neq0, $ then by choosing $ x=\sqrt{\frac{\alpha_2}{\alpha_5}}, z=-\frac{\alpha_1\sqrt{\alpha_2}}{(\alpha_3+2\alpha_2)\sqrt{\alpha_5}}, $ we have the family of representatives  $ \langle\nabla_2+\alpha\nabla_3+\nabla_5 \rangle_{\alpha\neq-2},$
			 which will be jointed with the case (\ref{05cs1a2}).
		\end{enumerate}
		
	\end{enumerate}
	\item  $ \alpha_7=\alpha_6=0, \alpha_4\neq0, $ then we have the following subcases:
	\begin{enumerate}
		\item if $ \alpha_2=0, \alpha_3=0,$ then we have representatives $ \langle \nabla_4 \rangle $ and $ \langle \nabla_4+\nabla_5 \rangle$ depending on whether $ \alpha_5=0 $ or not;
		\item if $ \alpha_2=0, \alpha_3\neq0 ,$ then by choosing $ x={\alpha_3}{\alpha_4^{-1}}, z=0, t=-{\alpha_1\alpha_3}{\alpha_4^{-2}}, $ we have the family of representatives $ \langle \nabla_3+ \nabla_4+\alpha\nabla_5 \rangle;$
		\item if $ \alpha_2\neq0 ,$ then by choosing $ x={\alpha_2}{\alpha_4^{-1}}, z=0, t=-{\alpha_1\alpha_2}{\alpha_4^{-2}}, $ we have the family of representatives $ \langle \nabla_2+\alpha\nabla_3+ \nabla_4+\beta\nabla_5 \rangle.$
	\end{enumerate}

	\item  $ \alpha_7=0, \alpha_6\neq0, $ then by choosing $ z=-{\alpha_2}x{\alpha_6^{-1}} ,$ we have $\alpha_2^*=0.$ Thus, we can suppose that $\alpha_2=0$ and consider following subcases:
	\begin{enumerate}
		\item $ \alpha_4=0,$ 
		\begin{enumerate}
			\item if $ \alpha_1=0, \alpha_3=0, $	 then we have  representatives $ \langle \nabla_6 \rangle $ and $ \langle \nabla_5+\nabla_6 \rangle$ depending on whether $ \alpha_5=0 $ or not;
			
			\item if $ \alpha_1=0, \alpha_3\neq0 $ then by choosing $ x=\sqrt[3]{{\alpha_3}{\alpha_6^{-1}}}, z=0, t=0, $ we have the family of  representatives $ \langle \nabla_3+\alpha\nabla_5+\nabla_6 \rangle;$
			
			\item if $ \alpha_1\neq0, $ then by choosing $ x=\sqrt[6]{{\alpha_1}{\alpha_6^{-1}}}, z=0, t=0, $ we have the family of representatives $ \langle \nabla_1+\alpha\nabla_3+\beta\nabla_5+\nabla_6 \rangle.$
			
		\end{enumerate}
		
		\item $ \alpha_4\neq0 ,$ then by choosing $ x=\sqrt{{\alpha_4}{\alpha_6^{-1}}}, t=-{\alpha_1\sqrt{\alpha_4 \alpha_6^{-3}}},  $ we have the family of  representatives $ \langle \alpha\nabla_3+\nabla_4+\beta\nabla_5+\nabla_6 \rangle.$
	\end{enumerate}	
	
	\item $ \alpha_7\neq0, $ 
	
	\begin{enumerate}
		\item if $ \alpha_1\alpha_7-\alpha_2\alpha_4=0, \alpha_3=0, \alpha_6=0, $ then 
		then by choosing $z=-\frac{\alpha_4}{2\alpha_7}x, t=\frac{\alpha_4\alpha_6-2\alpha_2\alpha_7}{2\alpha^2_7}x, $ 
		we have representatives $ \langle \nabla_7 \rangle $ and $ \langle \nabla_5+\nabla_7 \rangle$ depending on whether $ \alpha_5=0 $ or not;
		
		\item if $\alpha_1\alpha_7-\alpha_2\alpha_4=0, \alpha_3\alpha_7-\alpha_4\alpha_6=0, \alpha_6\neq0, $ then by choosing 
		\begin{center}$ x=\frac{\alpha_6}{\alpha_7}, z=-\frac{\alpha_4\alpha_6}{2\alpha_7^2}, t=\frac{\alpha_6(\alpha_4\alpha_6-2\alpha_2\alpha_7)}{2\alpha^3_7},$ 
		\end{center} we have the family of representatives $ \langle \alpha\nabla_5+\nabla_6+\nabla_7 \rangle;$
		
		\item if $ 2\alpha_1\alpha_7^2-\alpha_3\alpha_4\alpha_7+\alpha_4^2\alpha_6-2\alpha_2\alpha_4\alpha_7=0, \alpha_3\alpha_7-\alpha_4\alpha_6\neq0, $ then by choosing 
		\begin{center}$ x=\sqrt[4]{\frac{\alpha_3\alpha_7-\alpha_4\alpha_7}{\alpha^2_7}}, z=-\frac{\alpha_4\sqrt[4]{\alpha_3\alpha_7-\alpha_4\alpha_7}}{2\alpha_7\sqrt[4]{2\alpha_7^2}}, t=\frac{(\alpha_4\alpha_6-2\alpha_2\alpha_7)\sqrt[4]{\alpha_3\alpha_7-\alpha_4\alpha_6}}{2\alpha^2_7\sqrt[4]{2\alpha_7^2}}, $
		\end{center} we have the family of representatives $ \langle \nabla_3+\alpha\nabla_5+\beta\nabla_6+\nabla_7 \rangle;$
		
		\item if $ 2\alpha_1\alpha_7^2-\alpha_3\alpha_4\alpha_7+\alpha_4^2\alpha_6-2\alpha_2\alpha_4\alpha_7\neq0, $ then by choosing
		\begin{center}$ x=\sqrt[7]{\frac{2\alpha_1\alpha_7^2-\alpha_3\alpha_4\alpha_7+\alpha_4^2\alpha_6-2\alpha_2\alpha_4\alpha_7}{2\alpha^3_7}},$ $z=-\frac{\alpha_4\sqrt[7]{2\alpha_1\alpha_7^2-\alpha_3\alpha_4\alpha_7+\alpha_4^2\alpha_6-2\alpha_2\alpha_4\alpha_7}}{2\alpha_7\sqrt[7]{2\alpha_7^3}},$ 
		$t=\frac{(\alpha_4\alpha_6-2\alpha_2\alpha_7)\sqrt[7]{2\alpha_1\alpha_7^2-\alpha_3\alpha_4\alpha_7+\alpha_4^2\alpha_6-2\alpha_2\alpha_4\alpha_7}}{2\alpha^3_7\sqrt[7]{2\alpha_7^3}}, $
		\end{center} we have the family of representatives $ \langle \nabla_1+\alpha\nabla_3+\beta\nabla_5+\gamma\nabla_6+\nabla_7 \rangle.$	 	 	 	
		
	\end{enumerate}
	
\end{enumerate}	
Summarizing, we have the following distinct orbits:

\begin{center}
$ \langle \nabla_1+\nabla_2-2\nabla_3+\alpha\nabla_5 \rangle^{O(\alpha)=O(-\eta_3\alpha)=O(\eta^2_3\alpha)},$ 
$\langle \nabla_1+\alpha\nabla_3+\beta\nabla_5+\nabla_6 \rangle^{\tiny{\begin{array}{l}
O(\alpha,\beta)=O(\alpha,-\eta_3\beta)=O(-\alpha,\eta_3\beta)=\\
O(-\alpha,-\eta_3^2\beta)=O(\alpha,\eta_3^2\beta)=O(-\alpha,-\beta)
\end{array}}},$
$\langle \nabla_1+\alpha\nabla_3+\beta\nabla_5+\gamma\nabla_6+\nabla_7 \rangle^{\tiny{\begin{array}{l}
O(\alpha,\beta,\gamma)=O(\eta_7^4\alpha,\eta^2_7\beta,-\eta_7\gamma)=
O(-\eta_7\alpha,\eta^4_7\beta,\eta^2_7\gamma)=\\
O(-\eta_7^5\alpha,\eta^6_{7}\beta,-\eta^3_7\gamma)=
O(\eta^2_7\alpha,-\eta_7\beta,\eta^4_7\gamma)=\\
O(\eta_7^6\alpha,-\eta^3_7\beta,-\eta^5_7\gamma)=
O(-\eta^3_7\alpha,-\eta^5_7\beta,\eta^6_7\gamma)
\end{array}}},$
$ \langle \nabla_2+\alpha\nabla_3 \rangle,$
$\langle \nabla_2+\alpha\nabla_3+ \nabla_4+\beta\nabla_5 \rangle,$ $\langle\nabla_2+\alpha\nabla_3+\nabla_5 \rangle,$
$\langle \nabla_3+ \nabla_4+\alpha\nabla_5 \rangle,$

$ \langle \alpha\nabla_3+\nabla_4+\beta\nabla_5+\nabla_6 \rangle^{O(\alpha,\beta)=O(-\alpha,-\beta)},$ 
$\langle \nabla_3+\alpha\nabla_5+\nabla_6 \rangle^{O(\alpha)=O(-\eta_3\alpha)=O(\eta^2_3\alpha)},$
$ \langle \nabla_3+\alpha\nabla_5+\beta\nabla_6+\nabla_7 \rangle^{O(\alpha,\beta)=O(-\alpha,-i\beta)=
O(-\alpha,i \beta)=O(\alpha,-\beta)},$
$\langle \nabla_4 \rangle,$ 
$\langle \nabla_4+\nabla_5 \rangle,$ 
$\langle \nabla_5+\nabla_6 \rangle,$
$\langle \alpha\nabla_5+\nabla_6+\nabla_7 \rangle,$ 
$\langle \nabla_5+\nabla_7 \rangle,$ 
$\langle\nabla_6 \rangle,$ 
$\langle\nabla_7 \rangle,$
\end{center}

which gives the following new algebras:

\begin{longtable}{llllllllllllllllll}

${\mathbf{N}}_{338}^{\alpha}$ & $:$ &$e_1e_1=e_2$ & $e_1e_2=e_5$ & $e_1e_3=e_4$ & $e_1e_4=e_5$ \\ & & $e_2e_2=e_3$  & $e_2e_3=-2e_5$ 
   & $e_3e_3=\alpha e_5$ \\

${\mathbf{N}}_{339}^{\alpha, \beta}$ & $:$ &$e_1e_1=e_2$ & $e_1e_2=e_5$ & $e_1e_3=e_4$ & $e_2e_2=e_3$  \\ && $e_2e_3=\alpha e_5$  & $e_3e_3=\beta e_5$ 
  & $e_3e_4=e_5$ \\  

${\mathbf{N}}_{340}^{\alpha, \beta ,\gamma}$ & $:$ &$e_1e_1=e_2$ & $e_1e_2=e_5$ & $e_1e_3=e_4$ & $e_2e_2=e_3$ \\ & & $e_2e_3=\alpha e_5$  & $e_3e_3=\beta e_5$ 
   & $e_3e_4=\gamma e_5$ & $ e_4e_4=e_5$ \\ 

${\mathbf{N}}_{341}^{\alpha}$ & $:$ &$e_1e_1=e_2$ & $e_1e_3=e_4$ & $e_1e_4=e_5$ & $e_2e_2=e_3$ & $e_2e_3=\alpha e_5$   \\

${\mathbf{N}}_{342}^{\alpha, \beta}$ & $:$ &$e_1e_1=e_2$ & $e_1e_3=e_4$ & $e_1e_4=e_5$ & $e_2e_2=e_3$  \\ && $e_2e_3=\alpha e_5$ & $ e_2e_4=e_5$    &
$ e_3e_3=\beta e_5$
\\ 

${\mathbf{N}}_{343}^{\alpha}$ & $:$ &$e_1e_1=e_2$ & $e_1e_3=e_4$ & $e_1e_4=e_5$ \\ && $e_2e_2=e_3$ & $e_2e_3=\alpha e_5$ &
$ e_3e_3=e_5$
\\ 

${\mathbf{N}}_{344}^{\alpha}$ & $:$ &$e_1e_1=e_2$ & $e_1e_3=e_4$ & $e_2e_2=e_3$ \\ && $ e_2e_3=e_5$ & $e_2e_4=e_5$ & $e_3e_3=\alpha e_5$ 
\\ 

${\mathbf{N}}_{345}^{\alpha, \beta}$ & $:$ &$e_1e_1=e_2$ & $e_1e_3=e_4$ & $e_2e_2=e_3$ & $ e_2e_3=\alpha e_5$  \\ && $e_2e_4=e_5$ & $e_3e_3=\beta e_5$  & $ e_3e_4=e_5$
\\

${\mathbf{N}}_{346}^{\alpha}$ & $:$ &$e_1e_1=e_2$ & $e_1e_3=e_4$ & $e_2e_2=e_3$ \\ && $ e_2e_3=e_5$ & $e_3e_3=\alpha e_5$  &
$ e_3e_4=e_5$
\\

${\mathbf{N}}_{347}^{\alpha, \beta}$ & $:$ &$e_1e_1=e_2$ & $e_1e_3=e_4$ & $e_2e_2=e_3$ & $ e_2e_3=e_5$  \\ && $e_3e_3=\alpha e_5$ & $e_3e_4=\beta e_5$   &
$ e_4e_4=e_5$
\\

${\mathbf{N}}_{348}$ & $:$ &$e_1e_1=e_2$ & $e_1e_3=e_4$ & $e_2e_2=e_3$ & $ e_2e_4=e_5$ 
\\

${\mathbf{N}}_{349}$ & $:$ &$e_1e_1=e_2$ & $e_1e_3=e_4$ & $e_2e_2=e_3$ & $ e_2e_4=e_5$ & $ e_3e_3=e_5$ 
\\

${\mathbf{N}}_{350}$ & $:$ &$e_1e_1=e_2$ & $e_1e_3=e_4$ & $e_2e_2=e_3$ & $ e_3e_3=e_5$ & $ e_3e_4=e_5$ 
\\

${\mathbf{N}}_{351}^{\alpha}$ & $:$ &$e_1e_1=e_2$ & $e_1e_3=e_4$ & $e_2e_2=e_3$ \\ && $ e_3e_3=\alpha e_5$ & $ e_3e_4=e_5$ & $ e_4e_4=e_5$ 
\\

${\mathbf{N}}_{352}$ & $:$ &$e_1e_1=e_2$ & $e_1e_3=e_4$ & $e_2e_2=e_3$ & $ e_3e_3=e_5$ & $ e_4e_4=e_5$ 
\\

${\mathbf{N}}_{353}$ & $:$ &$e_1e_1=e_2$ & $e_1e_3=e_4$ & $e_2e_2=e_3$  & $ e_3e_4=e_5$ 
\\

${\mathbf{N}}_{354}$ & $:$ &$e_1e_1=e_2$ & $e_1e_3=e_4$ & $e_2e_2=e_3$  & $ e_4e_4=e_5$ 
\\

\end{longtable}

\subsection{$ 1 $-dimensional central extensions of $ {\mathbf N}_{06}^{4} $.} Here we will collect all information about $ {\mathbf N}_{06}^{4}: $

$$
\begin{array}{|l|l|l|l|}
\hline
\text{ }  & \text{ } & \text{Cohomology} & \text{Automorphisms} \\
\hline
{\mathbf{N}}^{4}_{06} & 
\begin{array}{l}
e_1e_1=e_2 \\ 
e_1e_2=e_4 \\ 
e_1e_3=e_4 \\ 
e_2e_2=e_3
\end{array}
&\begin{array}{l}
\mathrm{H}^2_{\mathfrak{C}}(\mathbf{N}^{4}_{06})=
\Big \langle [\Delta_{ij}] \Big\rangle\\
(i,j) \notin \{(1,1),(1,2),(2,2)\}
\end{array}
&
 \phi_{\pm}= \begin{pmatrix}
\pm1&0&0&0\\
0&1&0&0\\
z&0&1&0\\
t&\pm2z&0&\pm1
\end{pmatrix}\\
\hline
\end{array}$$
	
	Let us use the following notations:
\begin{longtable}{llll}
$\nabla_1=[\Delta_{13}],$&
$\nabla_2=[\Delta_{14}], $&
$\nabla_3=[\Delta_{23}], $&
$\nabla_4=[\Delta_{24}], $\\
$\nabla_5=[\Delta_{33}], $&
$\nabla_6=[\Delta_{34}], $&
$\nabla_7=[\Delta_{44}]. $
\end{longtable}
Take $ \theta=\sum\limits_{i=1}^{7}\alpha_i\nabla_i\in\mathrm{H}^2_{\mathfrak{C}}(\mathbf{N}^{4}_{06}).$ Since
	$$\phi_{\pm}^T\begin{pmatrix}
	0&0&\alpha_1&\alpha_2\\
	0&0&\alpha_3&\alpha_4\\
	\alpha_1&\alpha_3&\alpha_5&\alpha_6\\
	\alpha_2&\alpha_4&\alpha_6&\alpha_7
	
\end{pmatrix}\phi_{\pm}=
\begin{pmatrix}
\alpha^*&\alpha^{**}&\alpha^{*}_1+\alpha^{**}&\alpha^*_2\\
\alpha^{**}&\alpha^{***}&\alpha^*_3&\alpha^*_4\\
\alpha^{*}_1+\alpha^{**}&\alpha^*_3&\alpha^*_5&\alpha^*_6\\
\alpha^*_2&\alpha^*_4&\alpha^*_6&\alpha^*_7
\end{pmatrix}$$	
we have	

\begin{longtable}{lll}
\multicolumn{3}{l}{
$\alpha_1^*=\pm \alpha_1- \alpha_3z- \alpha_4t+ \alpha_5z+ \alpha_6t-2    ( \alpha_2\pm \alpha_6z\pm \alpha_7t)z,$}\\
$\alpha_2^*=\alpha_2\pm\alpha_6z \pm\alpha_7t,$&
$\alpha_3^*=\alpha_3 \pm 2\alpha_6z,$&
$\alpha_4^*=2\alpha_7z \pm \alpha_4,$\\
$\alpha_5^*=\alpha_5,$&
$\alpha_6^*=\pm \alpha_6,$&
$\alpha_7^*=\alpha_7.$
\end{longtable}	

Since $ (\alpha_2,\alpha_4,\alpha_6,\alpha_7)\neq(0,0,0,0)$
and for $\phi=\phi_+,$ we have the following cases:

\begin{enumerate}
	\item if $ \alpha_7=\alpha_6=\alpha_4=0, $ then $ \alpha_2\neq0, $ and we have the following subcase:
	
	\begin{enumerate}
		\item if $ \alpha_5=\alpha_3+2\alpha_2, $ then we have the  family of representatives
		\begin{center}$ \langle \alpha\nabla_1+\nabla_2+\beta\nabla_3+(\beta+2)\nabla_5 \rangle;$
		\end{center}
		\item if $ \alpha_5\neq\alpha_3+2\alpha_2, $ then by choosing $ z=0, t=\frac{\alpha_1}{\alpha_3+2\alpha_2-\alpha_5} ,$ we have the family of representative $ \langle \nabla_2+\alpha\nabla_3+\beta\nabla_5 \rangle_{\beta\neq\alpha+2};$
	\end{enumerate}
		\item if $ \alpha_7=0, \alpha_6=0, \alpha_4\neq0, $ then by choosing $ z=0,  t=\frac{\alpha_1}{\alpha_4}, $ we have the family of representatives $\langle \alpha\nabla_2+\beta\nabla_3+\nabla_4+\gamma\nabla_5 \rangle;$

	\item if $ \alpha_7=0, \alpha_6\neq0, $ then we have the following subcases:

\begin{enumerate}
	\item if $ \alpha_6=\alpha_4, $ then by choosing $z=-{\alpha_2}{\alpha_6^{-1}}, t=0,$ we have the family of representatives 	$\langle \alpha\nabla_1+\beta\nabla_3+\nabla_4+\gamma\nabla_5+\nabla_6 \rangle;$
	\item if $ \alpha_6\neq\alpha_4, $ then by choosing $ z=-\frac{\alpha_2}{\alpha_6}, t=\frac{\alpha_1\alpha_6-\alpha_2\alpha_5+\alpha_2\alpha_3}{\alpha_6(\alpha_4-\alpha_6)},$ we have the family of representatives 
	$ \langle \alpha\nabla_3+\beta\nabla_4+\gamma\nabla_5+\nabla_6 \rangle_{\beta\neq1};$ 
\end{enumerate}

	\item if $ \alpha_7\neq0, $ then by choosing $ z=-\frac{\alpha_4}{2\alpha_7}, t=\frac{\alpha_4\alpha_6-2\alpha_2\alpha_7}{2\alpha^2_7}, $ we have the family of representatives 
	\begin{center}$ \langle \alpha\nabla_1+\beta\nabla_3+\gamma\nabla_5+\mu\nabla_6+\nabla_7 \rangle.$
	\end{center} 
\end{enumerate}	
	
Summarizing, we have the following distinct orbits:

\begin{center}
 $ \langle \alpha\nabla_1+\nabla_2+\beta\nabla_3+(\beta+2)\nabla_5 \rangle_{\alpha \neq 0}^{O(\alpha, \beta)= O(-\alpha, \beta)},$  
 $\langle \alpha\nabla_1+\beta\nabla_3+\nabla_4+\gamma\nabla_5+\nabla_6 \rangle^{O(\alpha, \beta, \gamma)= O(\alpha, -\beta,-\gamma)}_{\alpha\neq 0},$
  $ \langle \alpha\nabla_1+\beta\nabla_3+\gamma\nabla_5+\mu\nabla_6+\nabla_7 \rangle^{O(\alpha, \beta, \gamma, \mu)= O(-\alpha, \beta, \gamma, -\mu)},$ $\langle \nabla_2+\alpha\nabla_3+\beta\nabla_5 \rangle,$
  $ \langle \alpha\nabla_2+\beta\nabla_3+\nabla_4+\gamma\nabla_5 \rangle^{O(\alpha, \beta, \gamma)= O(-\alpha, -\beta,- \gamma)},$ 
  $\langle \alpha \nabla_3+\beta \nabla_4+\gamma\nabla_5+\nabla_6 \rangle^{O(\alpha, \beta, \gamma)= O(-\alpha, \beta, -\gamma)},$
\end{center}
which gives the following new algebras:

\begin{longtable}{llllllllllllllllll}

${\mathbf{N}}_{355}^{\alpha\neq0, \beta} $ & $:$ &$e_1e_1=e_2$ & $e_1e_3=e_4+\alpha e_5$ & $e_1e_4=e_5$  \\ &&   $e_2e_2=e_3$  & $e_2e_3=\beta e_5$  & \multicolumn{2}{l}{$e_3e_3=(\beta+2)e_5$}  \\
 
${\mathbf{N}}_{356}^{\alpha\neq0, \beta,\gamma}$ & $:$ &$e_1e_1=e_2$ & $e_1e_3=e_4+\alpha e_5$ & $e_2e_2=e_3$ & $e_2e_3=\beta e_5$  \\ & &  $e_2e_4=e_5$  & $e_3e_3=\gamma e_5$ & $e_3e_4=e_5$& \\ 
 
${\mathbf{N}}_{357}^{\alpha, \beta,\gamma,\mu}$ & $:$ &$e_1e_1=e_2$ & $e_1e_3=e_4+\alpha e_5$ & $e_2e_2=e_3$ & $e_2e_3=\beta e_5$  \\ & &  $e_3e_3=\gamma e_5$  & $e_3e_4=\mu e_5$   & $e_4e_4=e_5$& \\ 

${\mathbf{N}}_{358}^{\alpha,\beta}$ & $:$ &$e_1e_1=e_2$ & $e_1e_3=e_4$ & $e_1e_4=e_5$   \\ && $e_2e_2=e_3$ & $e_2e_3=\alpha e_5$  & $e_3e_3=\beta e_5$  \\

${\mathbf{N}}_{359}^{\alpha, \beta,\gamma}$ & $:$ &$e_1e_1=e_2$ & $e_1e_3=e_4$ & $e_1e_4=\alpha e_5$ & $e_2e_2=e_3$ \\ & &  $e_2e_3=\beta e_5$  & $e_2e_4=e_5$  & $e_3e_3=\gamma e_5$& \\

${\mathbf{N}}_{360}^{\alpha,\beta,\gamma}$ & $:$ &$e_1e_1=e_2$ & $e_1e_3=e_4$ & $e_2e_2=e_3$ & $e_2e_3=\alpha e_5$  \\ & &  $e_2e_4=\beta e_5$  & $e_3e_3=\gamma e_5$  & $e_3e_4=e_5$& \\

\end{longtable}

\subsection{$ 1 $-dimensional central extensions of $ {\mathbf N}_{07}^{4} $.} Here we will collect all information about $ {\mathbf N}_{07}^{4}: $

$$
\begin{array}{|l|l|l|l|}
\hline
\text{ }  & \text{ } & \text{Cohomology} & \text{Automorphisms} \\
\hline
{\mathbf{N}}^{4}_{07} & 
\begin{array}{l}
e_1e_1=e_2 \\ 
e_2e_2=e_3 \\ 
e_2e_3=e_4
\end{array}
&\begin{array}{l}
\mathrm{H}^2_{\mathfrak{C}}(\mathbf{N}^{4}_{07})=
\Big \langle [\Delta_{ij}] \Big\rangle\\
(i,j) \notin \{ (1,1),(2,2),(2,3) \} 
\end{array}

& 
\phi=\begin{pmatrix}
x&0&0&0\\
0&x^2&0&0\\
0&0&x^4&0\\
t&0&0&x^6
\end{pmatrix}\\
\hline
\end{array}$$

	Let us use the following notations:
\begin{longtable}{llll}	
$\nabla_1=[\Delta_{12}],$&
$\nabla_2=[\Delta_{13}],$&
$\nabla_3=[\Delta_{14}],$&
$\nabla_4=[\Delta_{24}],$\\
$\nabla_5=[\Delta_{33}],$&
$\nabla_6=[\Delta_{34}],$&
$\nabla_7=[\Delta_{44}].$
\end{longtable}
	Take $ \theta=\sum\limits_{i=1}^{7}\alpha_i\nabla_i\in\mathrm{H}^2_{\mathfrak{C}}(\mathbf{N}^{4}_{07}) .$	Since
	$$\phi^T\begin{pmatrix}
	0&\alpha_1&\alpha_2&\alpha_3\\
	\alpha_1&0&0&\alpha_4\\
	\alpha_2&0&\alpha_5&\alpha_6\\
	\alpha_3&\alpha_4&\alpha_6&\alpha_7
	
\end{pmatrix}\phi=
\begin{pmatrix}
\alpha^*&\alpha^*_1&\alpha^{*}_2&\alpha^*_3\\
\alpha^*_1&0&0&\alpha^*_4\\
\alpha^{*}_2&0&\alpha^*_5&\alpha^*_6\\
\alpha^*_3&\alpha^*_4&\alpha^*_6&\alpha^*_7
\end{pmatrix}$$	
we have	

\begin{longtable}{llll}
$\alpha_1^*=(\alpha_1x+\alpha_4t)x^2,$&
$\alpha_2^*=(\alpha_2x+\alpha_6t)x^4,$&
$\alpha_3^*=(\alpha_3x+\alpha_7t)x^6,$&
$\alpha_4^*=\alpha_4x^8,$\\
$\alpha_5^*=\alpha_5x^8,$&
$\alpha_6^*=\alpha_6x^{10},$&
$\alpha_7^*=\alpha_7x^{12}.$
\end{longtable}	

We are interested in $ (\alpha_3,\alpha_4,\alpha_6,\alpha_7)\neq(0,0,0,0) $ and consider following cases:

\begin{enumerate}
\item $ \alpha_7=\alpha_6=\alpha_4=0, $ then $ \alpha_3\neq0, $ and we have the following subcases:

	\begin{enumerate}
		\item if $ \alpha_1=0, \alpha_2=0, $ then we have  representatives $ \langle \nabla_3 \rangle $ and $ \langle \nabla_3+\nabla_5 \rangle $ depending on whether $ \alpha_5=0 $ or not;
		\item if $ \alpha_1=0, \alpha_2\neq0, $ then by choosing $ x=\sqrt{{\alpha_2}{\alpha_3^{-1}}}, t=0, $ we have the family of representatives $ \langle \nabla_2+\nabla_3+\alpha\nabla_5 \rangle;$
		\item if $ \alpha_1\neq0, $ then by choosing $ x=\sqrt[4]{{\alpha_1}{\alpha_3^{-1}}}, t=0, $ we have the family of representatives $ \langle \nabla_1+\alpha\nabla_2+\nabla_3+\beta\nabla_5 \rangle.$
	\end{enumerate}

\item  $ \alpha_7=0, \alpha_6=0, \alpha_4\neq0, $ then we have the following subcases:
\begin{enumerate}
	\item if $ \alpha_2=0, \alpha_3=0 ,$ then we have the familty of representative $ \langle \nabla_4+\alpha\nabla_5 \rangle;$
	\item if $ \alpha_2=0, \alpha_3\neq0 ,$ then by choosing $ x={\alpha_3}{\alpha_4^{-1}}, t=-{\alpha_1\alpha_3}{\alpha_4^{-2}}, $ we have the family of representatives $ \langle \nabla_3+ \nabla_4+\alpha\nabla_5 \rangle;$
	\item if $ \alpha_2\neq0,$ then by choosing $ x=\sqrt[3]{{\alpha_2}{\alpha_4^{-1}}}, t=-{\alpha_1\sqrt[3]{\alpha_2 \alpha_4^{-4}}}, $ we have the family of representatives $ \langle \nabla_2+\alpha\nabla_3+ \nabla_4+\beta\nabla_5 \rangle.$
\end{enumerate}

\item $ \alpha_7=0, \alpha_6\neq0, $ then we have the following subcases:
\begin{enumerate}
		\item if $ \alpha_1=0, \alpha_3=0, \alpha_4=0, $ then we have  representatives $ \langle \nabla_6 \rangle $ and $ \langle \nabla_5+\nabla_6 \rangle$ depending on whether $ \alpha_5=0 $ or not;
		
		\item if $ \alpha_1\alpha_6=\alpha_2\alpha_4, \alpha_3=0, \alpha_4\neq0, $ then  choosing $ x=\sqrt{{\alpha_4}{\alpha_6^{-1}}}, t=-{\alpha_2\sqrt{\alpha_4} \alpha_6^{-3}}, $ we have the family of representatives $ \langle \nabla_4+\alpha\nabla_5+\nabla_6 \rangle;$
		
		\item if $ \alpha_1\alpha_6=\alpha_2\alpha_4, \alpha_3\neq0, $ then by choosing $ x=\sqrt[3]{\frac{\alpha_3}{\alpha_6}},  t=-\frac{\alpha_2\sqrt[3]{\alpha_3}}{\alpha_6\sqrt[3]{\alpha_6}},$ we have the family of representatives $ \langle \nabla_3+\alpha\nabla_4+\beta\nabla_5+\nabla_6 \rangle;$

	\item if $ \alpha_1\alpha_6\neq\alpha_2\alpha_4  ,$ then  choosing $ x=\sqrt[7]{\frac{\alpha_1\alpha_6-\alpha_2\alpha_4}{\alpha^2_6}},  t=-\frac{\alpha_2\sqrt[7]{\alpha_1\alpha_6-\alpha_2\alpha_4}}{\alpha_6\sqrt[7]{\alpha^2_6}},$ we have the family of representatives $ \langle \nabla_1+\alpha\nabla_3+\beta\nabla_4+\gamma\nabla_5+\nabla_6 \rangle.$
\end{enumerate}	
\item  $ \alpha_7\neq0, $ then we have the following subcases:

\begin{enumerate}
	\item $ \alpha_1=0, \alpha_2\alpha_7=\alpha_3\alpha_6, \alpha_4=0, \alpha_5=0, $ then we have  representatives $ \langle \nabla_7 \rangle $ and $ \langle \nabla_6+\nabla_7 \rangle$ depending on whether $ \alpha_6=0 $ or not;
	
	\item $ \alpha_1=0, \alpha_2\alpha_7=\alpha_3\alpha_6, \alpha_4=0, \alpha_5\neq0, $ then by choosing $ x=\sqrt[4]{{\alpha_5}{\alpha_7^{-1}}}, t=-{\alpha_3\sqrt[4]{\alpha_5 \alpha_7^{-5}}}, $ we have  family of representatives $ \langle \nabla_5+\alpha\nabla_6+\nabla_7 \rangle;$
	
	\item $ \alpha_1\alpha_7=\alpha_3\alpha_4, \alpha_2\alpha_7=\alpha_3\alpha_6, \alpha_4\neq0, $ then by choosing $ x=\sqrt[4]{{\alpha_4}{\alpha_7^{-1}}},
	t=-{\alpha_3\sqrt[4]{\alpha_4 \alpha_7^{-5}}}, $ we have family of representatives $ \langle \nabla_4+\alpha\nabla_5+\beta\nabla_6+\nabla_7 \rangle;$
	
\item $ \alpha_1\alpha_7=\alpha_3\alpha_4, \alpha_2\alpha_7\neq\alpha_3\alpha_6, $ then by choosing $ x=\sqrt[7]{(\alpha_2\alpha_7-\alpha_3\alpha_6)\alpha^{-2}_7},$ $t=-{\alpha_3\sqrt[7]{(\alpha_2\alpha_7-\alpha_3\alpha_6)\alpha_7^{-9}}}, $ we have the family of representatives \begin{center}$ \langle \nabla_2+\alpha\nabla_4+\beta\nabla_5+\gamma\nabla_6+\nabla_7 \rangle;$\end{center}	

\item $ \alpha_1\alpha_7\neq \alpha_3\alpha_4, $ then by choosing 
\begin{center}$ x=\sqrt[9]{(\alpha_1\alpha_7-\alpha_3\alpha_4)\alpha^{-2}_7}, t=-\alpha_3\sqrt[9]{(\alpha_1\alpha_7-\alpha_3\alpha_4)\alpha_7^{-11} }, $ 
\end{center} we have  family of representatives 
$ \langle \nabla_1+\alpha\nabla_2+\beta\nabla_4+\gamma\nabla_5+\mu\nabla_6+\nabla_7 \rangle.$		 	 	
	
\end{enumerate}

\end{enumerate}

Summarizing, we have the following distinct orbits:

\begin{center}
  $ \langle \nabla_1+\alpha\nabla_2+\nabla_3+\beta\nabla_5 \rangle
  ^{O(\alpha,\beta)=O(-\alpha,i\beta)=
  O(-\alpha,-i\beta)=O(\alpha,-\beta)},$
$ \langle \nabla_1+\alpha\nabla_2+\beta\nabla_4+\gamma\nabla_5+\mu\nabla_6+\nabla_7 \rangle
^{\tiny{\begin{array}{l}
O(\alpha,\beta,\gamma,\mu)=O(-\eta_9^7\alpha,\eta^4_9\beta,\eta^4_9\gamma,\eta^2_9\mu)=\\
O(-\eta^5_9\alpha,\eta_9^8\beta,\eta_9^8\gamma,\eta^4_9\mu)=
O(-\eta_3\alpha,-\eta_3\beta,-\eta_3\gamma,\eta^2_3\mu)=\\
O(-\eta_9\alpha,-\eta^7_9\beta,-\eta^7_9\gamma,\eta^8_9\mu)=
O(\eta_9^8\alpha,\eta^2_9\beta,\eta^2_9\gamma,-\eta_9\mu)=\\
O(\eta^2_3\alpha,\eta^2_3\beta,\eta^2_3\gamma,-\eta_3\mu)=
O(\eta^4_9\alpha,-\eta_9\beta,-\eta_9\gamma,-\eta^5_9\mu)=\\
O(\eta^2_9\alpha,-\eta^5_9\beta,-\eta^5_9\gamma,-\eta^7_9\mu)
\end{array}}},$
   $ \langle \nabla_1+\alpha\nabla_3+\beta\nabla_4+\gamma\nabla_5+\nabla_6 \rangle
 ^{\tiny{\begin{array}{l}
 O(\alpha,\beta,\gamma)=O(-\eta_7^3\alpha,\eta^2_7\beta,\eta^2_7\gamma)=\\
 O(\eta_7^6\alpha,\eta^4_7\beta,\eta^4_7\gamma)=
 O(\eta^2_7\alpha,\eta^6_{7}\beta,\eta^6_7\gamma)=\\
O(-\eta^5_7\alpha,-\eta_7\beta,-\eta_7\gamma)=
O(-\eta_7\alpha,-\eta^3_7\beta,-\eta^3_7\gamma)=\\
O(\eta^4_7\alpha,-\eta^5_7\beta,-\eta^5_7\gamma)
\end{array}}}, $
$ \langle \nabla_2+\alpha\nabla_3+ \nabla_4+\beta\nabla_5 \rangle^{O(\alpha,\beta)=O(-\eta_3\alpha,\beta)=O(\eta^2_3\alpha,\beta)},$ 
$\langle \nabla_2+\nabla_3+\alpha\nabla_5 \rangle^{O(\alpha)=O(-\alpha)}, $ 
  $ \langle \nabla_2+\alpha\nabla_4+\beta\nabla_5+\gamma\nabla_6+\nabla_7 \rangle^{\tiny{\begin{array}{l}
 O(\alpha,\beta,\gamma)=O(\eta_7^4\alpha,\eta^4_7\beta,\eta^2_7\gamma)=
 O(-\eta_7\alpha,-\eta_7\beta,\eta^4_7\gamma)=\\
 O(-\eta^5_7\alpha,-\eta^5_{7}\beta,\eta^6_7\gamma)=
O(\eta^2_7\alpha,\eta^2_7\beta,-\eta_7\gamma)=\\
O(\eta_7^6\alpha,\eta_7^6\beta,-\eta^3_7\gamma)=
O(-\eta^3_7\alpha,-\eta^3_7\beta,-\eta^5_7\gamma)
\end{array}}}, $
$ \langle \nabla_3 \rangle,$ 
$\langle \nabla_3+ \nabla_4+\alpha\nabla_5 \rangle,$ 
$\langle \nabla_3+\alpha\nabla_4+\beta\nabla_5+\nabla_6 \rangle^{O(\alpha,\beta)=O(-\eta_3\alpha,-\eta_3\beta)=O(\eta^2_3\alpha,\eta^2_3\beta)},$
 $ \langle \nabla_3+\nabla_5 \rangle,$ 
 $\langle \nabla_4+\alpha\nabla_5 \rangle,$ 
 $\langle \nabla_4+\alpha\nabla_5+\nabla_6 \rangle,$ 
 $\langle \nabla_4+\alpha\nabla_5+\beta\nabla_6+\nabla_7 \rangle^{O(\alpha,\beta)=O(\alpha,-\beta)},$
   $ \langle \nabla_5+\nabla_6 \rangle,$ 
   $\langle \nabla_5+\alpha\nabla_6+\nabla_7 \rangle^{O(\alpha)= O(-\alpha) },$
   $\langle \nabla_6 \rangle,$ 
   $\langle \nabla_6 +\nabla_7 \rangle,$ 
   $\langle \nabla_7 \rangle,$
 \end{center}
which gives the following new algebras:

\begin{longtable}{llllllllllllllllll}

${\mathbf{N}}_{361}^{\alpha, \beta}$ & $:$ &$e_1e_1=e_2$ & $e_1e_2=e_5$ & $e_1e_3=\alpha e_5$ & $e_1e_4=e_5$  \\ & & $e_2e_2=e_3$  & $e_2e_3=e_4$   & $e_3e_3=\beta e_5$ \\

${\mathbf{N}}_{362}^{\alpha, \beta,\gamma,\mu}$ & $:$ &$e_1e_1=e_2$ & $e_1e_2=e_5$ & $e_1e_3=\alpha e_5$  & $e_2e_2=e_3$  & $e_2e_3=e_4$   \\ & & $e_2e_4=\beta e_5$  & $ e_3e_3=\gamma e_5$ & $e_3e_4=\mu e_5 $ &$ e_4e_4=e_5$\\

${\mathbf{N}}_{363}^{\alpha, \beta,\gamma}$ & $:$ &$e_1e_1=e_2$ & $e_1e_2=e_5$ & $e_1e_4=\alpha e_5$  & $e_2e_2=e_3$  \\ &  & $e_2e_3=e_4$  & $e_2e_4=\beta e_5$   & $ e_3e_3=\gamma e_5$ & $e_3e_4=e_5 $ \\

${\mathbf{N}}_{364}^{\alpha, \beta}$ & $:$ &$e_1e_1=e_2$ & $e_1e_3=e_5$ & $e_1e_4=\alpha e_5$  & $e_2e_2=e_3$   \\ & & $e_2e_3=e_4$ & $e_2e_4=e_5$   & $e_3e_3=\beta e_5$ \\

${\mathbf{N}}_{365}^{\alpha}$ & $:$ &$e_1e_1=e_2$ & $e_1e_3=e_5$ & $e_1e_4= e_5$  \\ && $e_2e_2=e_3$  & $e_2e_3=e_4$  & $e_3e_3=\alpha e_5$ \\

${\mathbf{N}}_{366}^{\alpha, \beta,\gamma}$ & $:$ &$e_1e_1=e_2$ & $e_1e_3=e_5$  & $e_2e_2=e_3$  & $e_2e_3=e_4$   \\ & & $e_1e_4=\alpha e_5$ & $ e_3e_3=\beta e_5$  & $e_3e_4=\gamma e_5 $ & $e_4e_4=e_5$ \\

${\mathbf{N}}_{367}$ & $:$ &$e_1e_1=e_2$  & $e_1e_4=e_5$  & $e_2e_2=e_3$  & $e_2e_3=e_4$  \\

${\mathbf{N}}_{368}^{\alpha}$ & $:$ &$e_1e_1=e_2$ & $e_1e_4=e_5$  & $e_2e_2=e_3$  \\ && $e_2e_3=e_4$ & $e_2e_4=e_5$ & $e_3e_3=\alpha e_5$ \\

${\mathbf{N}}_{369}^{\alpha,\beta}$ & $:$ &$e_1e_1=e_2$ & $e_1e_4=e_5$  & $e_2e_2=e_3$  & $e_2e_3=e_4$  \\ & & $e_2e_4=\alpha e_5$ & $e_3e_3=\beta e_5$  & $e_3e_4=e_5$& \\

${\mathbf{N}}_{370}$ & $:$ &$e_1e_1=e_2$ & $e_1e_4=e_5$  & $e_2e_2=e_3$  & $e_2e_3=e_4$ & $e_3e_3=e_5$ \\

${\mathbf{N}}_{371}^{\alpha}$ & $:$ &$e_1e_1=e_2$  & $e_2e_2=e_3$  & $e_2e_3=e_4$ & $e_2e_4=e_5 $ & $e_3e_3=\alpha e_5$ \\

${\mathbf{N}}_{372}^{\alpha}$ & $:$ &$e_1e_1=e_2$  & $e_2e_2=e_3$  & $e_2e_3=e_4$ \\ && $e_2e_4=e_5 $ & $e_3e_3=\alpha e_5$ & $ e_3e_4=e_5$ \\

${\mathbf{N}}_{373}^{\alpha, \beta}$ & $:$ &$e_1e_1=e_2$  & $e_2e_2=e_3$  & $e_2e_3=e_4$ & $e_2e_4=e_5 $  \\ & & $e_3e_3=\alpha e_5$ & $ e_3e_4=\beta e_5$  & $e_4e_4=e_5$\\

${\mathbf{N}}_{374}$ & $:$ &$e_1e_1=e_2$   & $e_2e_2=e_3$  & $e_2e_3=e_4$ & $e_3e_3=e_5$ & $e_3e_4=e_5$\\

${\mathbf{N}}_{375}^{\alpha}$ & $:$ &$e_1e_1=e_2$   & $e_2e_2=e_3$  & $e_2e_3=e_4$\\ & & $e_3e_3=e_5$ & $e_3e_4=\alpha e_5$& $e_4e_4=e_5$\\

${\mathbf{N}}_{376}$ & $:$ &$e_1e_1=e_2$   & $e_2e_2=e_3$  & $e_2e_3=e_4$  & $e_3e_4=e_5$\\

${\mathbf{N}}_{377}$ & $:$ &$e_1e_1=e_2$   & $e_2e_2=e_3$  & $e_2e_3=e_4$ & $e_3e_4=e_5$ & $e_4e_4=e_5$\\

${\mathbf{N}}_{378}$ & $:$ &$e_1e_1=e_2$   & $e_2e_2=e_3$  & $e_2e_3=e_4$  & $e_4e_4=e_5$\\
\end{longtable}

\subsection{$ 1 $-dimensional central extensions of $ {\mathbf N}_{08}^{4} $.} Here we will collect all information about $ {\mathbf N}_{08}^{4}: $

$$
\begin{array}{|l|l|l|l|}
\hline
\text{ }  & \text{ } & \text{Cohomology} & \text{Automorphisms}\\
\hline
{\mathbf{N}}^{4}_{08} & 
\begin{array}{l}
e_1e_1=e_2 \\ 
e_1e_3=e_4 \\ 
e_2e_2=e_3 \\ 
e_2e_3=e_4
\end{array}
&
\begin{array}{l}
\mathrm{H}^2_{\mathfrak{C}}(\mathbf{N}^{4}_{08})=
\Big \langle [\Delta_{ij}] \Big\rangle\\
{(i,j) \notin \{(1,1),(1,3),(2,2)}\}
\end{array}

& 
  \phi=\begin{pmatrix}
1&0&0&0\\
0&1&0&0\\
0&0&1&0\\
t&0&0&1
\end{pmatrix}\\
\hline
\end{array}$$

Let us use the following notations:
\begin{longtable}{llll}
$\nabla_1=[\Delta_{12}],$& 
$\nabla_2=[\Delta_{14}],$& 
$\nabla_3=[\Delta_{23}],$&
$\nabla_4=[\Delta_{24}],$\\
$\nabla_5=[\Delta_{33}],$&
$\nabla_6=[\Delta_{34}],$&
$\nabla_7=[\Delta_{44}]. $
\end{longtable}
Take $ \theta=\sum\limits_{i=1}^{7}\alpha_i\nabla_i\in\mathrm{H}^2_{\mathfrak{C}}(\mathbf{N}^{4}_{08}) .$ Since
$$\phi^T\begin{pmatrix}
0&\alpha_1&0&\alpha_2\\
\alpha_1&0&\alpha_3&\alpha_4\\
0&\alpha_3&\alpha_5&\alpha_6\\
\alpha_2&\alpha_4&\alpha_6&\alpha_7

\end{pmatrix}\phi=
\begin{pmatrix}
\alpha^*&\alpha^{*}_1&\alpha^{**}&\alpha^*_2\\
\alpha^{*}_1&0&\alpha^*_3+\alpha^{**}&\alpha^*_4\\
\alpha^{**}&\alpha^*_3+\alpha^{**}&\alpha^*_5&\alpha^*_6\\
\alpha^*_2&\alpha^*_4&\alpha^*_6&\alpha^*_7
\end{pmatrix}$$	
we have	

\begin{longtable}{llll}
$\alpha_1^*=\alpha_1+\alpha_4t,$&
$\alpha_2^*=\alpha_2+\alpha_7t,$&
$\alpha_3^*=\alpha_3-\alpha_6t,$&
$\alpha_4^*=\alpha_4,$\\
$\alpha_5^*=\alpha_5,$&
$\alpha_6^*=\alpha_6,$&
$\alpha_7^*=\alpha_7.$
\end{longtable}	

Since $ (\alpha_2,\alpha_4,\alpha_6,\alpha_7)\neq(0,0,0,0),$ we have the following cases:

\begin{enumerate}
\item if $ \alpha_7=\alpha_6=\alpha_4=0, $ then $ \alpha_2\neq0, $ and we have the family of representatives 
\begin{center}$ \langle \alpha\nabla_1+\nabla_2+\beta\nabla_3+\gamma\nabla_5 \rangle;$
\end{center}

\item if $ \alpha_7=0, \alpha_6=0, \alpha_4\neq0, $ then by choosing $ t=-{\alpha_1}{\alpha_4^{-1}},$ we have the family of representatives 
\begin{center}$ \langle \alpha\nabla_2+\beta\nabla_3+\nabla_4+\gamma\nabla_5 \rangle;$
\end{center}

\item if $ \alpha_7=0, \alpha_6\neq0, $ then by choosing $ t={\alpha_3}{\alpha_6^{-1}},$ we have the family of representatives 
\begin{center}$ \langle \alpha\nabla_1+\beta\nabla_2+\gamma\nabla_4+\mu\nabla_5+\nabla_6 \rangle;$\end{center}

\item if $ \alpha_7\neq0, $ then by choosing $ t=-{\alpha_2}{\alpha_7^{-1}},$ we have the family of  representatives 
\begin{center}$ \langle \alpha\nabla_1+\beta\nabla_3+\gamma\nabla_4+\mu\nabla_5+\nu\nabla_6+\nabla_7 \rangle.$\end{center}

\end{enumerate}		
	
Summarizing, we have the following distinct orbits:

\begin{center}
$ \langle \alpha\nabla_1+\nabla_2+\beta\nabla_3+\gamma\nabla_5 \rangle,$ 
$\langle \alpha\nabla_1+\beta\nabla_2+\gamma\nabla_4+\mu\nabla_5+\nabla_6 \rangle,$
$ \langle \alpha\nabla_1+\beta\nabla_3+\gamma\nabla_4+\mu\nabla_5+\nu\nabla_6+\nabla_7 \rangle,$  
$\langle \alpha\nabla_2+\beta\nabla_3+\nabla_4+\gamma\nabla_5 \rangle,$
\end{center}

which gives the following new algebras:

\begin{longtable}{llllllllllllllllll}

${\mathbf{N}}_{379}^{\alpha, \beta,\gamma}$ & $:$ &$e_1e_1=e_2$ & $e_1e_2=\alpha e_5$ & $e_1e_3=e_4$ & $e_1e_4=e_5$  \\ & & $e_2e_2=e_3$    & $e_2e_3=e_4+\beta e_5$  & $e_3e_3=\gamma e_5$ \\

${\mathbf{N}}_{380}^{\alpha, \beta,\gamma,\mu}$ & $:$ &$e_1e_1=e_2$ & $e_1e_2=\alpha e_5$ & $e_1e_3=e_4$ \\ & & $e_1e_4=\beta e_5$& $e_2e_2=e_3$    & $e_2e_3=e_4$ \\ & & $ e_2e_4=\gamma e_5$  & $e_3e_3=\mu e_5$ & $e_3e_4= e_5$\\

${\mathbf{N}}_{381}^{\alpha, \beta,\gamma,\mu, \nu}$ & $:$ &$e_1e_1=e_2$ & $e_1e_2=\alpha e_5$ & $e_1e_3=e_4$  \\ && $e_2e_2=e_3$  & $e_2e_3=e_4+\beta e_5$  & $ e_2e_4=\gamma e_5$  \\ && $e_3e_3=\mu e_5$ & $e_3e_4=\nu e_5 $ & $ e_4e_4=e_5$\\

${\mathbf{N}}_{382}^{\alpha, \beta,\gamma}$ & $:$ &$e_1e_1=e_2$  & $e_1e_3=e_4$  & $e_1e_4=\alpha e_5$& $e_2e_2=e_3$   \\ & & $e_2e_3=e_4+\beta e_5$   & $ e_2e_4=e_5$  & $e_3e_3=\gamma e_5$ \\

\end{longtable}

\subsection{$ 1 $-dimensional central extensions of $ {\mathbf N}_{09}^{4} $.} Here we will collect all information about $ {\mathbf N}_{09}^{4}: $

$$
\begin{array}{|l|l|l|l|}
\hline
\text{ }  & \text{ } & \text{Cohomology} & \text{Automorphisms} \\
\hline
{\mathbf{N}}^{4}_{09} &
\begin{array}{l}
e_1e_1=e_2 \\ 
e_2e_2=e_3 \\ 
e_3e_3=e_4
\end{array}
&\begin{array}{l}
\mathrm{H}^2_{\mathfrak{C}}(\mathbf{N}^{4}_{09})=
\Big \langle [\Delta_{ij}] \Big\rangle\\
(i,j) \notin\{ (1,1),(2,2),(3,3)\}
\end{array}
&
 \phi=\begin{pmatrix}
x&0&0&0\\
0&x^2&0&0\\
0&0&x^4&0\\
t&0&0&x^8
\end{pmatrix}\\
\hline
\end{array}$$

Let us use the following notations:
\begin{longtable}{llll} $
\nabla_1=[\Delta_{12}],$ & $  \nabla_2=[\Delta_{13}],$ & $ \nabla_3=[\Delta_{14}],$ & $  \nabla_4=[\Delta_{23}],$ \\ $
\nabla_5=[\Delta_{24}],$ & $  \nabla_6=[\Delta_{34}],$ & $  \nabla_7=[\Delta_{44}].$ 
\end{longtable}

Take $ \theta=\sum\limits_{i=1}^{7}\alpha_i\nabla_i\in\mathrm{H}^2_{\mathfrak{C}}(\mathbf{N}^{4}_{09}) .$ Since
$$\phi^T\begin{pmatrix}
0&\alpha_1&\alpha_2&\alpha_3\\
\alpha_1&0&\alpha_4&\alpha_5\\
\alpha_2&\alpha_4&0&\alpha_6\\
\alpha_3&\alpha_5&\alpha_6&\alpha_7

\end{pmatrix}\phi=
\begin{pmatrix}
\alpha^*&\alpha^*_1&\alpha^{*}_2&\alpha^*_3\\
\alpha^*_1&0&\alpha^*_4&\alpha^*_5\\
\alpha^{*}_2&\alpha^*_4&0&\alpha^*_6\\
\alpha^*_3&\alpha^*_5&\alpha^*_6&\alpha^*_7
\end{pmatrix}$$	
we have	

\begin{longtable}{llll}
$\alpha_1^*=(\alpha_1x+\alpha_5t)x^2,$&
$\alpha_2^*=(\alpha_2x+\alpha_6t)x^4,$&
$\alpha_3^*=(\alpha_3x+\alpha_7t)x^8,$&
$\alpha_4^*=\alpha_4x^6,$\\
$\alpha_5^*=\alpha_5x^{10},$&
$\alpha_6^*=\alpha_6x^{12},$&
$\alpha_7^*=\alpha_7x^{16}.$
\end{longtable}	

Since $ (\alpha_3,\alpha_5,\alpha_6,\alpha_7)\neq(0,0,0,0),$ we have the following cases:

\begin{enumerate}
\item $ \alpha_7=\alpha_6=\alpha_5=0, $ then $ \alpha_3\neq0, $ and we have the following subcases:

\begin{enumerate}
	\item if $ \alpha_1=0, \alpha_2=0, $ then we have representatives $ \langle \nabla_3 \rangle $ and $ \langle \nabla_3+\nabla_4 \rangle $ depending on whether $ \alpha_4=0 $ or not;
	\item if $ \alpha_1=0, \alpha_2\neq0, $ then by choosing $ x=\sqrt[4]{{\alpha_2}{\alpha_3^{-1}}}, t=0, $ we have the family of representatives $ \langle \nabla_2+\nabla_3+\alpha\nabla_4 \rangle;$
	\item if $ \alpha_1\neq0, $ then by choosing $ x=\sqrt[6]{{\alpha_1}{\alpha_3^{-1}}}, t=0, $ we have the family of  representatives 
$ \langle \nabla_1+\alpha\nabla_2+\nabla_3+\beta\nabla_4 \rangle.$
	
\end{enumerate}

\item  $ \alpha_7=0, \alpha_6=0, \alpha_5\neq0, $ then we have the following cases:
\begin{enumerate}
	\item if $ \alpha_2=0, \alpha_3=0,$ then we have representatives  $ \langle \nabla_5 \rangle$ and $ \langle \nabla_4+\nabla_5 \rangle$ depending on whether $ \alpha_4=0 $ or not;
	\item if $ \alpha_2=0, \alpha_3\neq0,$ then  choosing $ x={\alpha_3}{\alpha_5^{-1}}, t=-{\alpha_1\alpha_3}{\alpha_5^{-2}}, $ we have the
	family of representatives $ \langle \nabla_3+ \alpha\nabla_4+\nabla_5 \rangle;$
	\item if $ \alpha_2\neq0,$ then by choosing $ x=\sqrt[5]{{\alpha_2}{\alpha_5^{-1}}}, t=-{\alpha_1\sqrt[5]{\alpha_2 \alpha_5^{-6} }}, $ we have the family of representatives $ \langle \nabla_2+\alpha\nabla_3+ \beta\nabla_4+\nabla_5 \rangle.$
\end{enumerate}

\item  $ \alpha_7=0, \alpha_6\neq0, $ then we have the following cases:
\begin{enumerate}
	\item if $  \alpha_1\alpha_6=\alpha_2\alpha_5, \alpha_3=0, \alpha_4=0,$ then we have  representatives $ \langle \nabla_6 \rangle $ and $ \langle \nabla_5+\nabla_6 \rangle$ depending on whether $ \alpha_5=0 $ or not;
	
	\item if $ \alpha_1\alpha_6=\alpha_2\alpha_5, \alpha_3=0, \alpha_4\neq0, $ then by choosing $ x=\sqrt[6]{{\alpha_4}{\alpha_6^{-1}}}, t=-{\alpha_2\sqrt[6]{\alpha_4 \alpha_6^{-7} }}, $ we have the family of representatives 
	$ \langle \nabla_4+\alpha\nabla_5+\nabla_6 \rangle;$
	
	\item if $ \alpha_1\alpha_6=\alpha_2\alpha_5, \alpha_3\neq0, $ then by choosing $ x=\sqrt[3]{{\alpha_3}{\alpha_6^{-1}}}, t=-{\alpha_2\sqrt[3]{\alpha_3 \alpha_6^{-4}}}, $ we have the family of representatives $ \langle \nabla_3+\alpha\nabla_4+\beta\nabla_5+\nabla_6 \rangle;$

	\item if $ \alpha_1\alpha_6\neq \alpha_2\alpha_5 ,$ then by choosing 
	\begin{center}$ x=\sqrt[9]{(\alpha_1\alpha_6-\alpha_2\alpha_5)\alpha^{-2}_6},$ $t=-{\alpha_2\sqrt[9]{(\alpha_1\alpha_6-\alpha_2\alpha_5)\alpha_6^{-11}}},$ 
	\end{center} we have the family of representatives 
\begin{center}	$ \langle \nabla_1+\alpha\nabla_3+\beta\nabla_4+\gamma\nabla_5+\nabla_6 \rangle.$
\end{center}
\end{enumerate}	

\item $ \alpha_7\neq0, $ then we have the following cases:
\begin{enumerate}
	\item if $ \alpha_1=0, \alpha_2\alpha_7=\alpha_3\alpha_6, \alpha_4=0, \alpha_5=0, $ then  choosing $x=1, t=-\frac{\alpha_3}{\alpha_7},$ we have  representatives $ \langle \nabla_7 \rangle $ and $ \langle \nabla_6+\nabla_7 \rangle$ depending on whether $ \alpha_6=0 $ or not;
	
	\item if $ \alpha_1\alpha_7=\alpha_3\alpha_5, \alpha_2\alpha_7=\alpha_3\alpha_6, \alpha_4=0, \alpha_5\neq0, $ then by choosing $ x=\sqrt[6]{{\alpha_5}{\alpha_7^{-1}}}, t=-{\alpha_3\sqrt[6]{\alpha_5 \alpha_7^{-7} }}, $ we have the family of  representatives $ \langle \nabla_5+\alpha\nabla_6+\nabla_7 \rangle;$
	
	\item if $ \alpha_1\alpha_7=\alpha_3\alpha_5, \alpha_2\alpha_7=\alpha_3\alpha_6, \alpha_4\neq0, $ 
	then by choosing \begin{center}$ x=\sqrt[10]{{\alpha_4}{\alpha_7^{-1}}},$ $t=-{\alpha_3\sqrt[10]{\alpha_4 \alpha_7^{-11}}},$ 
	\end{center} we have the family of  representatives $ \langle \nabla_4+\alpha\nabla_5+\beta\nabla_6+\nabla_7 \rangle;$
	
	\item $ \alpha_1\alpha_7=\alpha_3\alpha_5, \alpha_2\alpha_7\neq\alpha_3\alpha_6, $ then by choosing 
	\begin{center}$ x=\sqrt[11]{(\alpha_2\alpha_7-\alpha_3\alpha_6)\alpha^{-2}_7
	},$ 
	$t=-{\alpha_3\sqrt[11]{(\alpha_2\alpha_7-\alpha_3\alpha_6) \alpha_7^{-13} }}, $
	\end{center} we have the family of  representatives 
\begin{center}	$\langle \nabla_2+\alpha\nabla_4+\beta\nabla_5+\gamma\nabla_6+\nabla_7 \rangle;$	
\end{center}
	
	\item $ \alpha_1\alpha_7\neq\alpha_3\alpha_5, $ then by choosing 
	\begin{center}$ x=\sqrt[13]{({\alpha_1\alpha_7-\alpha_3\alpha_5}){\alpha^{-2}_7}},$ $t=-{\alpha_3\sqrt[13]{(\alpha_1\alpha_7-\alpha_3\alpha_5)\alpha_7^{-15} }}, $
	\end{center} we have  the family of representatives 
	\begin{center}
	    $ \langle \nabla_1+\alpha\nabla_2+\beta\nabla_4+\gamma\nabla_5+\mu\nabla_6+\nabla_7 \rangle.$		 	 	
	
	\end{center}
\end{enumerate}

\end{enumerate}		
	
Summarizing, we have the following distinct orbits:

\begin{center}
$ \langle \nabla_1+\alpha\nabla_2+\nabla_3+\beta\nabla_4 \rangle
^{O(\alpha,\beta)=O(-\eta_3\alpha,\beta)=
O(-\eta_3\alpha,-\beta)=O(\eta_3^2\alpha,-\beta)=
O(\eta^2_3\alpha,\beta)=O(\alpha,-\beta)},$
$\langle \nabla_1+\alpha\nabla_2+\beta\nabla_4+\gamma\nabla_5+\mu\nabla_6+\nabla_7 \rangle^{\tiny{\begin{array}{l}
O(\alpha,\beta,\gamma,\mu)=
O(-\eta_{13}^{11}\alpha,\eta^{10}_{13}\beta,\eta^6_{13}\gamma,\eta^4_{13}\mu)=\\
O(-\eta^9_{13}\alpha,-\eta^7_{13}\beta,\eta^{12}_{13}\gamma,\eta^8_{13}\mu)=
O(-\eta^7_{13}\alpha,\eta^4_{13}\beta,-\eta^5_{13}\gamma,\eta^{12}_{13}\mu)=\\
O(-\eta^5_{13}\alpha,-\eta^{1}_{13}\beta,-\eta^{11}_{13}\gamma,-\eta^{3}_{13}\mu)=
O(-\eta^{3}_{13}\alpha,-\eta^{11}_{13}\beta,\eta^4_{13}\gamma,-\eta^7_{13}\mu)= \\ O(-\eta^{1}_{13}\alpha,\eta^8_{13}\beta,\eta^{10}_{13}\gamma,-\eta^{11}_{13}\mu)=
O(\eta^{12}_{13}\alpha,-\eta^5_{13}\beta,-\eta^{3}_{13}\gamma,\eta^{2}_{13}\mu)=\\
O(\eta^{10}_{13}\alpha,\eta^{2}_{13}\beta,-\eta^{9}_{13}\gamma,\eta^6_{13}\mu)=
O(\eta^8_{13}\alpha,\eta^{12}_{13}\beta,\eta^{2}_{13}\gamma,\eta^{10}_{13}\mu)=\\
O(\eta^6_{13}\alpha,-\eta^{9}_{13}\beta,\eta^{8}_{13}\gamma,-\eta^{1}_{13}\mu)=
O(\eta^4_{13}\alpha,\eta^{6}_{13}\beta,-\eta^{1}_{13}\gamma,-\eta^5_{13}\mu)= \\ O(\eta^{2}_{13}\alpha,-\eta^{3}_{13}\beta,-\eta^{7}_{13}\gamma,-\eta^9_{13}\mu)
\end{array}}}, $
$ \langle \nabla_1+\alpha\nabla_3+\beta\nabla_4+\gamma\nabla_5+\nabla_6 \rangle^{\tiny{\begin{array}{l}
O(\alpha,\beta,\gamma,\mu)=
O(-\eta_3\alpha,\eta^2_3\beta,\eta^2_9\gamma)=
O(\eta^2_3\alpha,-\eta_3\beta,\eta^4_9\gamma)=\\
O(\alpha,\beta,\eta^2_3\gamma)=
O(-\eta_3\alpha,\eta^2_3\beta,\eta^8_9\gamma)=
O(\eta^2_3\alpha,-\eta_3\beta,-\eta_9\gamma)=\\
O(\alpha,\beta,-\eta_3\gamma)=
O(-\eta_3\alpha,\eta^2_3\beta,-\eta^5_9\gamma)=
O(\eta^2_3\alpha,-\eta_3\beta,-\eta^7_9\gamma)
\end{array}}},$
$\langle \nabla_2+\nabla_3+\alpha\nabla_4 \rangle^{O(\alpha)=O(i\alpha)=O(-\alpha)=(-i\alpha)}, $
$\langle \nabla_2+\alpha\nabla_3+ \beta\nabla_4+\nabla_5 \rangle^{
O(\alpha,\beta)=
O(-\eta_5\alpha,\eta^4_5\beta)=
O(\eta^2_5\alpha,-\eta^3_5\beta)=
O(-\eta^3_5\alpha,\eta^2_5\beta)=
O(\eta_5^4\alpha,-\eta_5\beta)},$
$ \langle \nabla_2+\alpha\nabla_4+\beta\nabla_5+\gamma\nabla_6+\nabla_7 \rangle^{\tiny{\begin{array}{l}
O(\alpha,\beta,\gamma)=
O(\eta_{11}^{10}\alpha,\eta^6_{11}\beta,\eta^4_{11}\gamma=\\
O(-\eta^9_{11}\alpha,-\eta_{11}\beta,\eta^38_{11}\gamma)=
O(\eta^8_{11}\alpha,-\eta^7_{11}\beta,-\eta_{11}\gamma)=\\
O(-\eta^7_{11}\alpha,\eta^2_{11}\beta,-\eta^5_{11}\gamma)= O(\eta^{6}_{11}\alpha,\eta^8_{11}\beta,-\eta^9_{11}\gamma)=\\  O(-\eta^{5}_{11}\alpha,-\eta^3_{11}\beta,\eta^2_{11}\gamma)=
O(\eta^4_{11}\alpha,-\eta^9_{13}\beta,\eta^{6}_{11}\gamma)=\\
O(-\eta^3_{11}\alpha,\eta^{4}_{11}\beta,\eta^{10}_{11}\gamma)=
O(\eta^2_{11}\alpha,\eta^{10}_{11}\beta,-\eta^{3}_{11}\gamma)=\\
O(-\eta_{11}\alpha,-\eta^{5}_{11}\beta,-\eta^{7}_{11}\gamma)
\end{array}}},$
$\langle \nabla_3 \rangle,$ 
$\langle \nabla_3+\nabla_4 \rangle,$ 
$\langle \nabla_3+ \alpha\nabla_4+\nabla_5 \rangle,$
$ \langle \nabla_3+\alpha\nabla_4+\beta\nabla_5+\nabla_6 \rangle^{O(\alpha,\beta)=O(\alpha,-\eta_3\beta)=O(\alpha,\eta^2_3\beta)},$ 
$\langle \nabla_4+\nabla_5 \rangle $
$ \langle \nabla_4+\alpha\nabla_5+\nabla_6 \rangle^{{\tiny 
\begin{array}{l}O(\alpha)=O(-\eta_3\alpha)=\\O(\eta^2_3\alpha)
\end{array}}}, $
$\langle \nabla_4+\alpha\nabla_5+\beta\nabla_6+\nabla_7 \rangle^{ 
{\tiny 
\begin{array}{l}
O(\alpha,\beta)=O(-\eta_5\alpha,\eta^4_5\beta)=O(\eta^2_5\alpha,-\eta^3_5\beta)=\\
O(-\eta^3_5\alpha,\eta^2_5\beta)=O(\eta^4_5\alpha,-\eta_5 \beta)
\end{array}}},$ 
$ \langle \nabla_5 \rangle,$
$\langle \nabla_5+\nabla_6 \rangle,$ 
$\langle \nabla_5+\alpha\nabla_6+\nabla_7 \rangle^{O(\alpha)=O(-\eta_3\alpha)=O(\eta^2_3\alpha)},$
$\langle \nabla_6 \rangle,$ 
$\langle \nabla_6+\nabla_7 \rangle,$ 
$\langle \nabla_7 \rangle,$ 

\end{center}

which gives the following new algebras:

\begin{longtable}{llllllllllllllllll}

${\mathbf{N}}_{383}^{\alpha, \beta}$ & $:$ &$e_1e_1=e_2$ & $e_1e_2=e_5$ & $e_1e_3=\alpha e_5$ & $e_1e_4=e_5$  \\ && $e_2e_2=e_3$   & $e_2e_3=\beta e_5$  & $e_3e_3=e_4$ \\

${\mathbf{N}}_{384}^{\alpha, \beta,\gamma,\mu}$ & $:$ &$e_1e_1=e_2$ & $e_1e_2=e_5$ & $e_1e_3=\alpha e_5$ & $e_2e_2=e_3$ & $e_2e_3=\beta e_5$   \\ & & $e_2e_4=\gamma e_5$  & $e_3e_3=e_4$& $e_3e_4=\mu e_5$& $e_4e_4=e_5$ \\

${\mathbf{N}}_{385}^{\alpha, \beta,\gamma}$ & $:$ &$e_1e_1=e_2$ & $e_1e_2=e_5$ & $e_1e_4=\alpha e_5$ & $e_2e_2=e_3$  \\ && $e_2e_3=\beta e_5$   & $e_2e_4=\gamma e_5$   & $e_3e_3=e_4$& $e_3e_4=e_5$ \\

${\mathbf{N}}_{386}^{\alpha}$ & $:$ &$e_1e_1=e_2$ & $e_1e_3=e_5$ & $e_1e_4=e_5$ \\ && $e_2e_2=e_3$ & $e_2e_3=\alpha e_5$& $e_3e_3=e_4$ \\

${\mathbf{N}}_{387}^{\alpha, \beta}$ & $:$ &$e_1e_1=e_2$ & $e_1e_3=e_5$ & $e_1e_4=\alpha e_5$ & $e_2e_2=e_3$  \\ && $e_2e_3=\beta e_5$   & $e_2e_4=e_5$   & $e_3e_3=e_4$ \\

${\mathbf{N}}_{388}^{\alpha, \beta,\gamma}$ & $:$ &$e_1e_1=e_2$ & $e_1e_3=e_5$ & $e_2e_2=e_3$ & $e_2e_3=\alpha e_5$  \\ && $e_2e_4=\beta e_5$& $e_3e_3=e_4$ & $ e_3e_4=\gamma e_5$& $e_4e_4=e_5$ \\

${\mathbf{N}}_{389}$ & $:$ &$e_1e_1=e_2$ & $e_1e_4=e_5$ & $e_2e_2=e_3$ & $e_3e_3=e_4$ \\

${\mathbf{N}}_{390}$ & $:$ &$e_1e_1=e_2$ &  $e_1e_4=e_5$ & $e_2e_2=e_3$& $ e_2e_3=e_5$ & $e_3e_3=e_4$ \\

${\mathbf{N}}_{391}^{\alpha}$ & $:$ &$e_1e_1=e_2$ & $e_1e_4=e_5$ & $e_2e_2=e_3$\\ & & $e_2e_3=\alpha e_5$& $ e_2e_4=e_5$ & $e_3e_3=e_4$ \\

${\mathbf{N}}_{392}^{\alpha, \beta}$ & $:$ &$e_1e_1=e_2$ & $e_1e_4=e_5$ & $e_2e_2=e_3$ & $e_2e_3=\alpha e_5$ \\ && $ e_2e_4=\beta e_5$ & $e_3e_3=e_4$  &  $ e_3e_4=e_5$\\

${\mathbf{N}}_{393}$ & $:$ &$e_1e_1=e_2$ &  $e_2e_2=e_3$ & $e_2e_3=e_5$& $ e_2e_4=e_5$ & $e_3e_3=e_4$ \\

${\mathbf{N}}_{394}^{\alpha}$ & $:$ &$e_1e_1=e_2$ & $e_2e_2=e_3$ & $e_2e_3=e_5$ \\ && $e_2e_4=\alpha e_5$ & $e_3e_3=e_4$ & $ e_3e_4=e_5$  \\

${\mathbf{N}}_{395}^{\alpha, \beta}$ & $:$ &$e_1e_1=e_2$ & $e_2e_2=e_3$ & $e_2e_3=e_5$ & $e_2e_4=\alpha e_5$ \\ & & $e_3e_3=e_4$ & $ e_3e_4=\beta e_5$     & $ e_4e_4=e_5$  \\

${\mathbf{N}}_{396}$ & $:$ &$e_1e_1=e_2$ &  $e_2e_2=e_3$ & $ e_2e_4=e_5$ & $e_3e_3=e_4$ \\

${\mathbf{N}}_{397}$ & $:$ &$e_1e_1=e_2$ &  $e_2e_2=e_3$ & $ e_2e_4=e_5$ & $e_3e_3=e_4$ & $ e_3e_4=e_5$ \\

${\mathbf{N}}_{398}^{\alpha}$ & $:$ &$e_1e_1=e_2$ &  $e_2e_2=e_3$ & $ e_2e_4=e_5$ \\ && $e_3e_3=e_4$ & $ e_3e_4=\alpha e_5$ & $ e_4e_4=e_5$\\

${\mathbf{N}}_{399}$ & $:$ &$e_1e_1=e_2$ &  $e_2e_2=e_3$  & $e_3e_3=e_4$ & $ e_3e_4=e_5$ \\

${\mathbf{N}}_{400}$ & $:$ &$e_1e_1=e_2$ &  $e_2e_2=e_3$ & $ e_3e_3=e_4$ & $e_3e_4=e_5$ & $ e_4e_4=e_5$ \\

${\mathbf{N}}_{401}$ & $:$ &$e_1e_1=e_2$ &  $e_2e_2=e_3$ & $ e_3e_3=e_4$  & $ e_4e_4=e_5$ \\

\end{longtable}

\subsection{$ 1 $-dimensional central extensions of $ {\mathbf N}_{10}^{4} $.} Here we will collect all information about $ {\mathbf N}_{10}^{4}: $

$$
\begin{array}{|l|l|l|l|}
\hline
\text{ }  & \text{ } & \text{Cohomology} & \text{Automorphisms} \\
\hline
{\mathbf{N}}^{4}_{10} & 
\begin{array}{l}
e_1e_1=e_2 \\ 
e_1e_2=e_4 \\ 
e_2e_2=e_3 \\ 
e_3e_3=e_4
\end{array}
&
\begin{array}{l}
\mathrm{H}^2_{\mathfrak{C}}(\mathbf{N}^{4}_{10})=
\Big \langle [\Delta_{ij}] \Big\rangle\\

(i,j) \notin \{ (1,1),(1,2),(2,2)\}
\end{array}

& 
  \begin{array}{l}
\phi_k= \begin{pmatrix}
\eta^k&0&0&0\\
0& \eta^{2k}&0&0\\
0&0& \eta^{4k}&0\\
t&0&0& \eta^{8k}
\end{pmatrix} \\
\multicolumn{1}{c}{\eta=-\eta_5, \ k=0,1,2,3,4}
\end{array}\\
\hline
\end{array}$$

Let us use the following notations:
\begin{longtable}{llll}
$\nabla_1=[\Delta_{13}], $&$ 
\nabla_2=[\Delta_{14}], $&$ 
\nabla_3=[\Delta_{23}], $&$  
\nabla_4=[\Delta_{24}], $\\$
\nabla_5=[\Delta_{33}], $&$ 
\nabla_6=[\Delta_{34}], $&$  
\nabla_7=[\Delta_{44}]. $
\end{longtable}
Take $ \theta=\sum\limits_{i=1}^{7}\alpha_i\nabla_i\in\mathrm{H}^2_{\mathfrak{C}}(\mathbf{N}^{4}_{10}) .$ Since
$$\phi_k^T\begin{pmatrix}
0&0&\alpha_1&\alpha_2\\
0&0&\alpha_3&\alpha_4\\
\alpha_1&\alpha_3&\alpha_5&\alpha_6\\
\alpha_2&\alpha_4&\alpha_6&\alpha_7

\end{pmatrix}\phi_k=
\begin{pmatrix}
\alpha^*&\alpha^{**}&\alpha^{*}_1&\alpha^*_2\\
\alpha^{**}&0&\alpha^*_3&\alpha^*_4\\
\alpha^{*}_1&\alpha^*_3&\alpha^*_5+\alpha^{**}&\alpha^*_6\\
\alpha^*_2&\alpha^*_4&\alpha^*_6&\alpha^*_7
\end{pmatrix},$$	
we have	

\begin{longtable}{llll}
$\alpha_1^*=\eta^{4 k} (\eta^k \alpha_1+t \alpha_6),$&
$\alpha_2^*=\eta^{8 k} (\eta^k \alpha_2+t \alpha_7),$&
$\alpha_3^*=\eta^{6 k} \alpha_3,$&
$\alpha_4^*=\eta^{10 k} \alpha_4,$\\
$\alpha_5^*=-t \eta^{2 k} \alpha_4+\eta^{8 k} \alpha_5,$&
$\alpha_6^*=\eta^{12 k} \alpha_6,$&
$\alpha_7^*=\eta^{16 k} \alpha_7.$
\end{longtable}	

Since $ (\alpha_2,\alpha_4,\alpha_6,\alpha_7)\neq(0,0,0,0),$ we have the following cases:

\begin{enumerate}
\item if $ \alpha_7=0, \alpha_6=0, \alpha_4=0, $ then $ \alpha_2\neq0, $ and we have the family of representatives 
\begin{center}$ \langle \alpha\nabla_1+\nabla_2+\beta\nabla_3+\gamma\nabla_5\rangle;$
\end{center}

\item if $ \alpha_7=0, \alpha_6=0, \alpha_4\neq0, $ then  we have the family of representatives 
\begin{center}$ \langle \alpha\nabla_1+\beta\nabla_2+\gamma\nabla_3+\nabla_4 \rangle;$ 
\end{center}

\item if $ \alpha_7=0, \alpha_6\neq0, $ then  we have the family of representatives 
\begin{center}$ \langle \alpha\nabla_2+\beta\nabla_3+\gamma\nabla_4+\mu\nabla_5+\nabla_6 \rangle;$
\end{center}

\item if $ \alpha_7\neq0, $ then  we have the family of representatives 
\begin{center}$ \langle \alpha\nabla_1+\beta\nabla_3+\gamma\nabla_4+\mu\nabla_5+\nu\nabla_6+\nabla_7 \rangle.$
\end{center}
\end{enumerate}

Summarizing, we have the following distinct orbits:

\begin{center}
    
    $ \langle \alpha\nabla_1+\beta\nabla_2+\gamma\nabla_3+\nabla_4 \rangle^{
   \tiny{\begin{array}{l} O(\alpha, \beta, \gamma)=
    O(\alpha, \eta^4_5\beta, -\eta_5 \gamma)=
    O(\alpha, -\eta^3_5\beta, \eta^2_5 \gamma)=\\
    O(\alpha, \eta^2_5\beta, -\eta^3_5 \gamma)=
    O(\alpha, -\eta_5\beta, \eta^4_5 \gamma)\end{array}}},$ 
    
    $ \langle \alpha\nabla_1+\nabla_2+\beta\nabla_3+\gamma\nabla_5 \rangle^{
  \tiny{\begin{array}{l}     
  O(\alpha, \beta, \gamma)=
  O(-\eta_5\alpha, \eta^2_5\beta, \eta^4_5 \gamma)=
  O(\eta^2_5\alpha, \eta^4_5\beta, -\eta^3_5 \gamma)=\\
  O(-\eta^3_5\alpha, -\eta_5\beta, \eta^2_5 \gamma)=
  O(\eta^4_5\alpha, -\eta^3_5\beta, -\eta_5 \gamma) \end{array}}},$ 
    
  $ \langle \alpha\nabla_1+\beta\nabla_3+\gamma\nabla_4+\mu\nabla_5+\nu\nabla_6+\nabla_7 \rangle^{\tiny{\begin{array}{l}
  O(\alpha, \beta, \gamma, \mu, \nu)=
  O(\eta^4_5\alpha, \beta, \eta^4_5 \gamma, \eta^2_5\mu, -\eta_5\nu)=\\
  O(-\eta^3_5\alpha, \beta, -\eta^3_5 \gamma, \eta^4_5\mu, \eta^2_5\nu)=
  O(\eta^2_5\alpha, \beta, \eta^2_5 \gamma,- \eta_5\mu, -\eta^3_5\nu)=\\
  O(-\eta_5\alpha, \beta, -\eta_5 \gamma, -\eta^3_5\mu, \eta^4_5\nu)
\end{array}}},$

   $ \langle \alpha\nabla_2+\beta\nabla_3+\gamma\nabla_4+\mu\nabla_5+\nabla_6 \rangle^{\tiny{\begin{array}{l}
   O(\alpha, \beta, \gamma, \mu)=
   O(\eta_5^2\alpha, \eta_5^4\beta, -\eta_5^3 \gamma, -\eta_5\mu)=
   O(\eta_5^4\alpha, -\eta_5^3\beta, -\eta_5\gamma, \eta^2_5\mu)=\\
   O(-\eta_5\alpha, \eta_5^2\beta, \eta_5^4 \gamma, -\eta^3_5\mu)=
   O(-\eta_5^3\alpha, -\eta_5\beta, \eta_5^2 \gamma, \eta^4_5\mu)
   \end{array}}},$
    
    \end{center}

which gives the following new algebras:
    
\begin{longtable}{llllllllllllllllll}

${\mathbf{N}}_{402}^{\alpha, \beta,\gamma}$ & $:$ &$e_1e_1=e_2$ & $e_1e_2=e_4$ & $e_1e_3=\alpha e_5$ & $e_1e_4=\beta e_5$ \\ & & $e_2e_2=e_3$   & $e_2e_3=\gamma e_5$  & $ e_2e_4=e_5$ & $e_3e_3=e_4$ \\

${\mathbf{N}}_{403}^{\alpha, \beta,\gamma}$ & $:$ &$e_1e_1=e_2$ & $e_1e_2=e_4$ & $e_1e_3=\alpha e_5$ & $e_1e_4=e_5$\\ &  & $e_2e_2=e_3$   & $e_2e_3=\beta e_5$    & $e_3e_3=e_4+\gamma e_5$ \\

${\mathbf{N}}_{404}^{\alpha, \beta,\gamma,\mu,\nu}$ & $:$ &$e_1e_1=e_2$ & $e_1e_2=e_4$ & $e_1e_3=\alpha e_5$  \\ && $e_2e_2=e_3$   & $e_2e_3=\beta e_5$    & $e_2e_4=\gamma e_5$  \\ & & $e_3e_3=e_4+\mu e_5 $  & $ e_3e_4=\nu e_5$ & $ e_4e_4=e_5 $\\

${\mathbf{N}}_{405}^{\alpha, \beta,\gamma,\mu}$ & $:$ &$e_1e_1=e_2$ & $e_1e_2=e_4$ & $e_1e_4=\alpha e_5$  & $e_2e_2=e_3$   \\ & & $e_2e_3=\beta e_5$  & $e_2e_4=\gamma e_5$  & $e_3e_3=e_4+\mu e_5 $  & $ e_3e_4=e_5$\\

\end{longtable}

\subsection{$ 1 $-dimensional central extensions of $ {\mathbf N}_{11}^{4}(\lambda) $.} Here we will collect all information about $ {\mathbf N}_{11}^{4}(\lambda): $

$$
\begin{array}{|l|l|l|l|}
\hline
\text{ }  & \text{ } & \text{Cohomology} & \text{Automorphisms} \\
\hline
{\mathbf{N}}^{4}_{11}(\lambda) & 
\begin{array}{l}
e_1e_1=e_2 \\ 
e_1e_2=\lambda e_4 \\ 
e_2e_2=e_3\\
e_2e_3=e_4 \\ 
e_3e_3=e_4
\end{array}

&
\begin{array}{l}
\mathrm{H}^2_{\mathfrak{C}}(\mathbf{N}^{4}_{11}(\lambda))=
\Big \langle [\Delta_{ij}] \Big\rangle\\
(i,j) \notin \{ (1,1),(2,2),(3,3) \}
\end{array}

&
  \phi=\begin{pmatrix}
1&0&0&0\\
0&1&0&0\\
0&0&1&0\\
t&0&0&1
\end{pmatrix}\\
\hline
\end{array}$$
Let us use the following notations:
\begin{longtable}{lll l}
$\nabla_1=[\Delta_{12}], $&$
\nabla_2=[\Delta_{13}], $&$ 
\nabla_3=[\Delta_{14}], $&$  
\nabla_4=[\Delta_{23}], $\\$
\nabla_5=[\Delta_{24}], $&$ 
\nabla_6=[\Delta_{34}], $&$  
\nabla_7=[\Delta_{44}]. $
\end{longtable}
Take $ \theta=\sum\limits_{i=1}^{7}\alpha_i\nabla_i\in\mathrm{H}^2_{\mathfrak{C}}(\mathbf{N}^{4}_{11}(\lambda)) .$ Since
$$\phi^T\begin{pmatrix}
0&\alpha_1&\alpha_2&\alpha_3\\
\alpha_1&0&\alpha_4&\alpha_5\\
\alpha_2&\alpha_4&0&\alpha_6\\
\alpha_3&\alpha_5&\alpha_6&\alpha_7

\end{pmatrix}\phi=
\begin{pmatrix}
\alpha^*&\alpha_1^{*}&\alpha^{*}_2&\alpha^*_3\\
\alpha_1^{*}&0&\alpha^*_4&\alpha^*_5\\
\alpha^{*}_2&\alpha^*_4&0&\alpha^*_6\\
\alpha^*_3&\alpha^*_5&\alpha^*_6&\alpha^*_7
\end{pmatrix}$$	
we have	

\begin{longtable}{llll}
$\alpha_1^*=\alpha_1+\alpha_5t,$&
$\alpha_2^*=\alpha_2+\alpha_6t,$&
$\alpha_3^*=\alpha_3+\alpha_7t,$&
$\alpha_4^*=\alpha_4,$\\
$\alpha_5^*=\alpha_5,$&
$\alpha_6^*=\alpha_6,$&
$\alpha_7^*=\alpha_7.$
\end{longtable}	

Since $ (\alpha_3,\alpha_5,\alpha_6,\alpha_7)\neq(0,0,0,0),$ we have the following cases:

\begin{enumerate}
	\item if $ \alpha_7=0, \alpha_6=0, \alpha_5=0 ,$ then $ \alpha_3\neq0 $ and we have the family of representatives 
	\begin{center}$ \left\langle \alpha\nabla_1+\beta\nabla_2+\nabla_3+\gamma\nabla_4 \right\rangle; $
	\end{center}
	
	\item if $ \alpha_7=0, \alpha_6=0, \alpha_5\neq0 $ then by choosing $ t=-{\alpha_1}{\alpha_5^{-1}},$ we have the family of representatives 
	\begin{center}$ \left\langle \alpha\nabla_2+\beta\nabla_3+\gamma\nabla_4+\nabla_5 \right\rangle;$
	\end{center}
	
	\item if $ \alpha_7=0, \alpha_6\neq0 $ then by choosing $ t=-{\alpha_2}{\alpha_6^{-1}},$ we have the family of representatives 
	\begin{center}$ \left\langle \alpha\nabla_1+\beta\nabla_3+\gamma\nabla_4+\mu\nabla_5+\nabla_6 \right\rangle;$	\end{center}
	
	\item if $ \alpha_7\neq0 $ then by choosing $ t=-{\alpha_3}{\alpha_7^{-1}},$ we have the family of representatives 
	\begin{center}$ \left\langle \alpha\nabla_1+\beta\nabla_2+\gamma\nabla_4+\mu\nabla_5+\nu\nabla_6+\nabla_7 \right\rangle. $	\end{center}
		
\end{enumerate}

Summarizing, we have the following distinct orbits:
\begin{center}
    $  \left\langle \alpha\nabla_1+\beta\nabla_2+\nabla_3+\gamma\nabla_4 \right\rangle,$ 
    $\left\langle \alpha\nabla_1+\beta\nabla_2+\gamma\nabla_4+\mu\nabla_5+\nu\nabla_6+\nabla_7 \right\rangle,$    
    $ \left\langle \alpha\nabla_1+\beta\nabla_3+\gamma\nabla_4+\mu\nabla_5+\nabla_6 \right\rangle,$ 
    $\left\langle \alpha\nabla_2+\beta\nabla_3+\gamma\nabla_4+\nabla_5 \right\rangle,$
    
\end{center}

which gives the following new algebras:
\begin{longtable}{llllllllllllllllll}

${\mathbf{N}}_{406}^{\lambda,\alpha, \beta,\gamma}$ & $:$ &$e_1e_1=e_2$ & $e_1e_2=\lambda e_4+\alpha e_5$ & $e_1e_3=\beta e_5$ & $e_1e_4=e_5$  \\ && $e_2e_2=e_3$   & $e_2e_3=e_4+\gamma e_5$  & $e_3e_3=e_4$ \\

${\mathbf{N}}_{407}^{\lambda,\alpha, \beta,\gamma,\mu,\nu}$ & $:$ &$e_1e_1=e_2$ & $e_1e_2=\lambda e_4+\alpha e_5$ & $e_1e_3=\beta e_5$ \\ && $e_2e_2=e_3$ & $e_2e_3=e_4+\gamma e_5$   & $e_2e_4=\mu e_5$  \\ && $e_3e_3=e_4$ & $e_3e_4=\nu e_5$ & $ e_4e_4=e_5 $\\

${\mathbf{N}}_{408}^{\lambda,\alpha, \beta,\gamma}$ & $:$ &$e_1e_1=e_2$ & $e_1e_2=\lambda e_4+\alpha e_5$ & $ e_1e_4=\beta e_5 $ & $e_2e_2=e_3$ \\ & & $e_2e_3=e_4+\gamma e_5$    & $e_2e_4=\mu e_5$  & $e_3e_3=e_4$ & $e_3e_4=e_5$ \\

${\mathbf{N}}_{409}^{\lambda,\alpha, \beta,\gamma}$ & $:$ &$e_1e_1=e_2$ & $e_1e_2=\lambda e_4$ & $e_1e_3=\alpha e_5$ & $e_1e_4=\beta e_5$  \\ && $e_2e_2=e_3$   & $e_2e_3=e_4+\gamma e_5$  & $e_2e_4=e_5$ & $e_3e_3=e_4$

\end{longtable}

\begin{remark}
Note that the algebras $\mathbf{N}^{4}_{11}(\lambda)$ and $\mathbf{N}^{4}_{11}(-\lambda)$ are isomorphic.
Hence, there are some additional isomorphism relations for algebras from the present subsection
$\mathbf{N}^{\lambda, \Xi} \cong \mathbf{N}^{-\lambda, \Xi} .$

\end{remark}

\section{Classification theorem.}

\begin{theorem}
Let $\mathbf N$ be a complex $5$-dimensional nilpotent commutative algebra. 
Then we have one of the following situations.
\begin{enumerate}
    \item If $\mathbf N$ is associative, then $\mathbf N$ is isomorphic to one algebra listed in \cite{kpv19}.
      \item If $\mathbf N$ is a non-associative Jordan algebra, then $\mathbf N$ is isomorphic to one algebra listed in \cite{ha16}.
        \item If $\mathbf N$ is a non-Jordan $\mathfrak{CD}$-algebra, then $\mathbf N$ is isomorphic to one algebra listed in \cite{jkk19}.
          \item If $\mathbf N$ is a non-$\mathfrak{CD}$-algebra, then $\mathbf N$ is isomorphic to one algebra listed in the following list.

\end{enumerate}

	\end{theorem}
 
\begin{longtable}{llllllll}
${\mathbf N}_{01}$ &$:$&  $e_1 e_1=e_2$ & $e_1 e_2=e_3$ & $e_2e_3=e_4$  \\

${\mathbf N}_{02}$ &$:$& $e_1 e_1=e_2$  & $e_1 e_2=e_3$  & $e_1e_3=e_4$ & $e_2 e_3=e_4$&& \\

${\mathbf N}_{03}$ &$:$&  $e_1 e_1 = e_2$  & $e_1 e_2=e_3$ & $e_3 e_3= e_4$&&\\

${\mathbf N}_{04}$ &$:$&  $e_1 e_1 = e_2$  & $e_1 e_2=e_3$ & $e_2e_2=e_4$ & $e_3 e_3= e_4$&&\\

${\mathbf N}_{05}$ &$:$& $e_1 e_1=e_2$ & $e_1 e_3 = e_4$ &  $e_2 e_2=e_3$&& \\

${\mathbf N}_{06}$ &$:$& $e_1 e_1=e_2$ & $e_1e_2=e_4$ & $e_1 e_3 = e_4$ &  $e_2 e_2=e_3$&&  \\

${\mathbf N}_{07}$ &$:$&  $e_1 e_1=e_2$ & $e_2 e_2=e_3$ & $e_2 e_3=e_4$ &&  \\

${\mathbf N}_{08}$ &$:$&  $e_1 e_1=e_2$ & $e_1e_3=e_4$ & $e_2 e_2=e_3$ & $e_2 e_3=e_4$  && \\
\end{longtable}

\end{document}